%% file: arxiv.tex
\documentclass[10pt]{article}
\usepackage{amsmath}
\usepackage{amsfonts}
\usepackage{amssymb}
\usepackage{amsthm}
\usepackage{mathabx}
\usepackage{enumerate}
\usepackage{mathrsfs}
\usepackage[totalwidth=450pt, totalheight=650pt]{geometry}

\begin{document}

\newcommand{\C}{{\mathbb{C}}}
\newcommand{\R}{{\mathbb{R}}}
\newcommand{\Q}{B}
\newcommand{\Z}{{\mathbb{Z}}}
\newcommand{\N}{{\mathbb{N}}}
\newcommand{\q}{\left}
\newcommand{\w}{\right}
\newcommand{\Ninf}{\N_\infty}
\newcommand{\Vol}[1]{\mathrm{Vol}\q(#1\w)}
\newcommand{\B}[4]{B_{\q(#1,#2\w)}\q(#3,#4\w)}
\newcommand{\CjN}[3]{\q\|#1\w\|_{C^{#2}\q(#3\w)}}
\newcommand{\Cj}[2]{C^{#1}\q( #2\w)}
\newcommand{\grad}{\bigtriangledown}
\newcommand{\Det}[1]{\det_{#1\times #1}}
\newcommand{\bE}{\mathbb{E}}
\newcommand{\sK}{\mathcal{K}}
\newcommand{\sKt}{\widetilde{\mathcal{K}}}
\newcommand{\sA}{\mathcal{A}}
\newcommand{\sB}{\mathcal{B}}
\newcommand{\sC}{\mathcal{C}}
\newcommand{\sD}{\mathcal{D}}
\newcommand{\sS}{\mathcal{S}}
\newcommand{\sE}{\mathcal{E}}
\newcommand{\sL}{\mathcal{L}}
\newcommand{\sI}{\mathcal{I}}
\newcommand{\sF}{\mathcal{F}}
\newcommand{\sQ}{\mathcal{Q}}
\newcommand{\sV}{\mathcal{V}}
\newcommand{\sM}{\mathcal{M}}
\newcommand{\sT}{\mathcal{T}}
\newcommand{\fD}{\mathfrak{D}}
\newcommand{\fDt}{\widetilde{\fD}}
\newcommand{\fDh}{\widehat{\fD}}
\newcommand{\sFt}{\widetilde{\sF}}
\newcommand{\fS}{\mathfrak{S}}
\newcommand{\sPh}{\widehat{\sP}}
\newcommand{\sNh}{\widehat{\sN}}
\newcommand{\cV}{\q( \sV\w)}
\newcommand{\vsig}{\varsigma}
\newcommand{\vsigt}{\widetilde{\vsig}}
\newcommand{\vsigh}{\widehat{\vsig}}
\newcommand{\ta}{\tilde{a}}
\newcommand{\dil}[2]{#1^{(#2)}}
\newcommand{\dilp}[2]{#1^{\q(#2\w)_{p}}}
\newcommand{\lA}{-\log_2 \sA}
\newcommand{\eh}{\hat{e}}
\newcommand{\Ho}{\mathbb{H}^1}
\newcommand{\sd}{\sum d}
\newcommand{\dt}{\tilde{d}}
\newcommand{\dhc}{\hat{d}}
\newcommand{\Span}[1]{\mathrm{span}\q\{ #1 \w\}}
\newcommand{\dspan}[1]{\dim \Span{#1}}
\newcommand{\K}{K_0}
\newcommand{\ad}[1]{\mathrm{ad}( #1 )}
\newcommand{\LtOpN}[1]{\q\|#1\w\|_{L^2\rightarrow L^2}}
\newcommand{\LpOpN}[2]{\q\|#2\w\|_{L^{#1}\rightarrow L^{#1}}}
\newcommand{\LpN}[2]{\q\|#2\w\|_{L^{#1}}}
\newcommand{\LpsN}[3]{\q\|#3\w\|_{L^{#1}_{#2}}}
\newcommand{\Jac}{\mathrm{Jac}\:}
\newcommand{\kapt}{\widetilde{\kappa}}
\newcommand{\gt}{\widetilde{\gamma}}
\newcommand{\gtt}{\widetilde{\widetilde{\gamma}}}
\newcommand{\gh}{\widehat{\gamma}}
\newcommand{\Sh}{\widehat{S}}
\newcommand{\Wh}{\widehat{W}}
\newcommand{\sMh}{\widehat{\sM}}
\newcommand{\Ih}{\widehat{I}}
\newcommand{\Wt}{\widetilde{W}}
\newcommand{\Xt}{\widetilde{X}}
\newcommand{\Tt}{\widetilde{T}}
\newcommand{\rh}{\hat{r}}
\newcommand{\rt}{\tilde{r}}
\newcommand{\Dt}{\widetilde{D}}
\newcommand{\Phit}{\widetilde{\Phi}}
\newcommand{\Vh}{\widehat{V}}
\newcommand{\Xh}{\widehat{X}}
\newcommand{\ch}{\widehat{c}}
\newcommand{\deltah}{\hat{\delta}}
\newcommand{\deltat}{\tilde{\delta}}
\newcommand{\sSh}{\widehat{\mathcal{S}}}
\newcommand{\sSt}{\widetilde{\mathcal{S}}}
\newcommand{\sSb}{\overline{\mathcal{S}}}
\newcommand{\sSc}{\widecheck{\mathcal{S}}}
\newcommand{\Xc}{\widecheck{X}}
\newcommand{\Yc}{\widecheck{Y}}
\newcommand{\Yb}{\overline{Y}}
\newcommand{\cd}{\check{d}}
\newcommand{\Xb}{\overline{X}}
\newcommand{\bd}{\overline{d}}
\newcommand{\sR}{\mathcal{R}}
\newcommand{\sRh}{\widehat{\sR}}
\newcommand{\sRt}{\widetilde{\sR}}
\newcommand{\sFh}{\widehat{\mathcal{F}}}
\newcommand{\sFc}{\widecheck{\mathcal{F}}}
\newcommand{\sFb}{\overline{\mathcal{F}}}
\newcommand{\thetah}{\widehat{\theta}}
\newcommand{\ct}{\widetilde{c}}
\newcommand{\at}{\tilde{a}}
\newcommand{\bt}{\tilde{b}}
\newcommand{\ah}{\hat{a}}
\newcommand{\bh}{\hat{b}}
\newcommand{\fg}{\mathfrak{g}}
\newcommand{\cC}{\q( \sC\w)}
\newcommand{\cG}{\q( \sC_{\fg}\w)}
\newcommand{\cJ}{\q( \sC_{J}\w)}
\newcommand{\cY}{\q( \sC_{Y}\w)}
\newcommand{\cYu}{\q( \sC_{Y}\w)_u}
\newcommand{\cJu}{\q( \sC_{J}\w)_u}
\newcommand{\Bb}{\overline{B}}
\newcommand{\Qb}{\overline{Q}}
\newcommand{\sP}{\mathcal{P}}
\newcommand{\sN}{\mathcal{N}}
\newcommand{\cH}{\q(\mathcal{H}\w)}
\newcommand{\Omegat}{\widetilde{\Omega}}
\newcommand{\Kt}{\widetilde{K}}
\newcommand{\sMt}{\widetilde{\sM}}
\newcommand{\denum}[2]{#1_{#2}}
\newcommand{\phit}{\tilde{\phi}}
\newcommand{\nuset}{\{1,\ldots, \nu\}}
\newcommand{\diam}{\mathrm{diam}\: }
\newcommand{\Nt}{{\widetilde{N}}}
\newcommand{\Nh}{{\hat{N}}}
\newcommand{\psit}{\widetilde{\psi}}
\newcommand{\sigt}{\widetilde{\sigma}}
\newcommand{\psih}{\widehat{\psi}}
\newcommand{\sigh}{\widehat{\sigma}}
\newcommand{\muoS}{\q\{\mu_1\w\}}
\newcommand{\Mjcutoff}{\nu+2}
\newcommand{\LplqOpN}[3]{\q\| #3 \w\|_{L^{#1}\q(\ell^{#2}\q(\N^{\nu} \w) \w)\rightarrow L^{#1}\q( \ell^{#2 }\q( \N^{\nu} \w) \w) }}
\newcommand{\LplqN}[3]{\q\| #3 \w\|_{L^{#1}\q(\ell^{#2 }\q(\N^{\nu} \w) \w) }}
\newcommand{\iinf}{\iota_\infty}
\newcommand{\jz}{\ell}
\newcommand{\Lpp}[2]{L^{#1}\q(#2\w)}
\newcommand{\Lppn}[3]{\q\|#3\w\|_{\Lpp{#1}{#2}}}
\newcommand{\Ewmu}{E\cup \q\{\mu_1\w\}}
\newcommand{\Ewmuc}{\q(\Ewmu\w)^c}
\newcommand{\ip}[2]{\q< #1, #2 \w>}
\newcommand{\hd}{\hat{d}}
\newcommand{\mset}{\q\{ 1,\ldots, m\w\}}
\newcommand{\expL}[1]{\exp_L\q( #1\w)}
\newcommand{\NLN}[3]{\q\|#3\w\|_{\NL{#1}{#2}}}
\newcommand{\NL}[2]{\mathrm{NL}^{#1}_{#2}}
\newcommand{\hD}{\widehat{D}}
\newcommand{\lap}{\bigtriangleup}
\newcommand{\Czip}[1]{C_0^\infty(#1)}
\newcommand{\supp}[1]{\mathrm{supp}\left(#1\right)}
\newcommand{\Deg}[1]{\mathrm{deg}(#1)}
\newcommand{\gammat}{\tilde{\gamma}}
\newcommand{\gammah}{\hat{\gamma}}
\newcommand{\vtt}{\tilde{t}}
\newcommand{\et}{\tilde{e}}
\newcommand{\etat}{\tilde{\eta}}
\newcommand{\etah}{\hat{\eta}}
\newcommand{\NLp}[3]{\NL{#1}{#2}(#3)}
\newcommand{\NLpN}[4]{\q\|#4\w\|_{\NLp{#1}{#2}{#3} }}
\newcommand{\nuh}{\hat{\nu}}
\newcommand{\nut}{\tilde{\nu}}
\newcommand{\td}{\tilde{d}}
\newcommand{\qh}{\hat{q}}
\newcommand{\qt}{\tilde{q}}
\newcommand{\jh}{\hat{j}}
\newcommand{\jt}{\tilde{j}}
\newcommand{\lt}{\tilde{l}}
\newcommand{\kt}{\tilde{k}}
\newcommand{\kh}{\hat{k}}
\newcommand{\phih}{\widehat{\phi}}
\newcommand{\kappat}{\tilde{\kappa}}
\newcommand{\snuvect}{\Gamma(T\Omega)\times ([0,\infty)^\nu\setminus \{0\})}
\newcommand{\sonevect}{\Gamma(T\Omega)\times (0,\infty)}
\newcommand{\snutvect}{\Gamma(T\Omega)\times ([0,\infty)^{\nut}\setminus \{0\})}
\newcommand{\snuhvect}{\Gamma(T\Omega)\times ([0,\infty)^{\nuh}\setminus \{0\})}
\newcommand{\onecompnu}{\mathfrak{d}_{\nu}}
\newcommand{\onecompnut}{\mathfrak{d}_{\nut}}
\newcommand{\onecompnuh}{\mathfrak{d}_{\nuh}}
\newcommand{\snuvectone}{\Gamma(T\Omega)\times \onecompnu}
\newcommand{\snutvectone}{\Gamma(T\Omega)\times \onecompnut}
\newcommand{\snuhvectone}{\Gamma(T\Omega)\times \onecompnuh}
\newcommand{\alphaa}{\vec{\alpha}}
\newcommand{\Na}{\vec{N}}
\newcommand{\otimesh}{\widehat{\otimes}}
\newcommand{\Lplqnu}[2]{L^{#1}\q(\ell^{#2}(\N^{\nu})\w)}
\newcommand{\LplqnuN}[3]{\q\| #3 \w\|_{ \Lplqnu{#1}{#2} }  }
\newcommand{\LplqnuOpN}[3]{\q\| #3 \w\|_{ \Lplqnu{#1}{#2}  \rightarrow \Lplqnu{#1}{#2}}  }
\newcommand{\re}[1]{\mathrm{Re}(#1)}
\newcommand{\im}[1]{\mathrm{Im}(#1)}
\newcommand{\vht}{\hat{t}}
\newcommand{\sZ}{\mathcal{Z}}
\newcommand{\gammac}{\check{\gamma}}
\newcommand{\ec}{\check{e}}
\newcommand{\Nc}{\widecheck{N}}
\newcommand{\gammab}{\overline{\gamma}}
\newcommand{\eb}{\overline{e}}
\newcommand{\Nb}{\overline{N}}
\newcommand{\tc}{\check{t}}
\newcommand{\tb}{\overline{t}}
\newcommand{\thetac}{\widecheck{\theta}}
\newcommand{\thetab}{\overline{\theta}}
\newcommand{\Wb}{\overline{W}}
\newcommand{\Wc}{\widecheck{W}}
\newcommand{\Vb}{\overline{V}}
\newcommand{\Vc}{\widecheck{V}}
\newcommand{\alphab}{\overline{\alpha}}
\newcommand{\alphac}{\check{\alpha}}
\newcommand{\etab}{\overline{\eta}}
\newcommand{\etac}{\check{\eta}}
\newcommand{\vsigc}{\widecheck{\vsig}}
\newcommand{\vsigb}{\overline{\vsig}}
\newcommand{\alphat}{\tilde{\alpha}}
\newcommand{\alphah}{\hat{\alpha}}
\newcommand{\deltab}{\overline{\delta}}
\newcommand{\deltac}{\check{\delta}}
\newcommand{\schS}{\mathscr{S}}

\newtheorem{thm}{Theorem}[section]
\newtheorem{cor}[thm]{Corollary}
\newtheorem{prop}[thm]{Proposition}
\newtheorem{lemma}[thm]{Lemma}
\newtheorem{conj}[thm]{Conjecture}

\theoremstyle{remark}
\newtheorem{rmk}[thm]{Remark}

\theoremstyle{definition}
\newtheorem{defn}[thm]{Definition}

\theoremstyle{definition}
\newtheorem{assumption}[thm]{Assumption}

\theoremstyle{remark}
\newtheorem{example}[thm]{Example}

\numberwithin{equation}{section}

\title{Sobolev spaces associated to singular and fractional Radon transforms}
\author{Brian Street\footnote{The author was partially supported by NSF DMS-1401671.}}
\date{}

\maketitle




\begin{abstract}
\input{abstract}
\end{abstract}

\tableofcontents

\section{Introduction}
\input{intro2}

\subsection{The single parameter case}\label{SectionIntroSingle}
\input{introsingle}

\subsection{Past work}\label{SectionIntroPast}
\input{intropast}

\section{Definitions: Kernels}\label{SectionResKer}
\input{resker}

\subsection{Product Kernels}\label{SectionResProdKer}
\input{resprodker}

\section{Definitions: Vector fields}
\input{rescc}

\subsection{H\"ormander's condition}\label{SectionResHorVect}
\input{rescchor}

\subsection{Real Analytic Vector Fields}\label{SectionResRealAnalVect}
\input{resccanal}

\section{Definitions: Surfaces}\label{SectionResSurf}
\input{ressurf}

\subsection{H\"ormander's condition}\label{SectionResHorSurf}
\input{ressurfhor}

\subsection{Real Analytic Surfaces}\label{SectionResRealAnalSurf}
\input{ressurfanal}

\section{Results:  Non-isotropic Sobolev Spaces}\label{SectionResSobNew}
\input{ressob2}

\subsection{Comparing Sobolev Spaces}\label{SectionResSobCompNew}
\input{rescompsob2}

\subsection{Euclidean vector fields and isotropic Sobolev spaces}
\input{ressobeuclid2}

\section{Results: Fractional Radon Transforms}\label{SectionResRadon}
\input{resradon2}

\subsection{Other geometries}
\input{resradonother2}

\subsection{H\"ormander's condition}
\input{resradonotherhor3}

\subsection{Pseudodifferential operators}\label{SectionRadonPDO}
\input{resradonpdo}

\subsubsection{Singular Integrals}
\input{resradonsing}












\section{Proofs: Schwartz space and kernels}\label{SectionSchwartzAndProd}
\input{pfker}

\subsection{Proof of Proposition \ref{PropResKerOtherDefn}}\label{SectionPfKerOther}
\input{pfkerother}

\section{The Frobenius Theorem and the Unit Scale}
\input{frob}

\subsection{Control of vector fields}
\input{frobcontrol}

\section{Proofs: Adjoints}\label{SectionPfAdjoint}
\input{pfadj}

\section{Proofs: A Technical $L^2$ Lemma}
\input{pfl2tech}

\subsection{Application I: Almost Orthogonality}
\input{pfl2techap1}

\subsection{Application II: Different Geometries}
\input{pfl2techap2}


\section{Proofs: The maximal function}
\input{pfmax}

\section{Proofs:  A single parameter Littlewood-Paley theory}
\input{pfslp}

\section{Proofs: Non-isotropic Sobolev spaces}\label{SectionPfSob}
\input{pfsob}

\subsection{Comparing Sobolev spaces}\label{SectionPfSobComp}
\input{pfcompare}

\section{Proofs:  Fractional Radon Transforms}\label{SectionPfRadon}
\input{pfradon}

\section{Optimality}\label{SectionOptimality}
\input{optimal}

\bibliographystyle{amsalpha}

\bibliography{radon}




\end{document}

%% file: abstract.tex
The purpose of this paper is to study the smoothing properties (in $L^p$ Sobolev spaces) of operators
of the form $f\mapsto \psi(x) \int f(\gamma_t(x)) K(t)\: dt$, where $\gamma_t(x)$ is a $C^\infty$
function defined on a neighborhood of the origin in $(t,x)\in \mathbb{R}^N\times \mathbb{R}^n$,
satisfying $\gamma_0(x)\equiv x$, $\psi$ is a $C^\infty$ cut-off function supported on a small neighborhood
of $0\in \mathbb{R}^n$, and $K$ is a ``multi-parameter fractional kernel'' supported on a small neighborhood of $0\in \mathbb{R}^N$.
When $K$ is a Calder\'on-Zygmund kernel these operators were studied by Christ, Nagel, Stein, and Wainger, and when
$K$ is a multi-parameter singular kernel they were studied by the author and Stein.  In both of these situations,
conditions on $\gamma$ were given under which the above operator is bounded on $L^p$ ($1<p<\infty$).  Under these same conditions,
we introduce non-isotropic $L^p$ Sobolev spaces associated to $\gamma$.  Furthermore, when $K$ is a fractional kernel
which is smoothing of an order which is close to $0$ (i.e., very close to a singular kernel) we prove mapping properties of the above
operators on these non-isotropic Sobolev spaces.  As a corollary, under the conditions introduced on $\gamma$
by Christ, Nagel, Stein, and Wainger, we prove optimal smoothing properties in isotropic $L^p$ Sobolev spaces for the above
operator when $K$ is a fractional kernel which is smoothing of very low order.

%% file: intro2.tex
In the influential paper \cite{ChristNagelSteinWaingerSingularAndMaximalRadonTransforms}, Christ, Nagel, Stein, and Wainger
studied operators of the form
\begin{equation}\label{EqnIntroCNSWOp}
Tf(x) = \psi(x) \int f(\gamma_t(x)) K(t) \: dt,
\end{equation}
where $\gamma_t(x)=\gamma(t,x):\R^{N}_0\times \R^{n}_0\rightarrow \R^n$ is a $C^\infty$ function defined on a neighborhood of $(0,0)\in \R^N\times \R^n$
satisfying $\gamma_0(x)\equiv x$, $\psi\in C_0^\infty(\R^n)$ is supported on a small neighborhood of $0\in \R^n$, and $K$ is a Calder\'on-Zygmund
kernel defined on a small neighborhood of $0\in \R^N$. They introduced conditions on $\gamma$ such that every operator of the form \eqref{EqnIntroCNSWOp}
is bounded on $L^p(\R^n)$, $1<p<\infty$.  Furthermore, they showed (under these same conditions on $\gamma$) that if $K$ is instead
a fractional kernel, smoothing of order $\delta>0$ (henceforth referred to as a kernel of order $-\delta$)\footnote{I.e., $K(t)$ satsifies estimates like $\q|\partial_t^{\alpha} K(t)\w|\lesssim |t|^{-N-|\alpha|+\delta}$.  If $\delta\geq N$, one needs a different condition, but we are mostly interested in $\delta$ small.  We address the kernels more precisely in Section \ref{SectionResKer}.}
supported near $0\in \R^{N}$, then $T:L^p\rightarrow L^p_s$ for some $s=s(p,\delta,\gamma)>0$, where $L^p_s$ denotes the $L^p$ Sobolev space of order $s$.
Left open, however, was the optimal choice of $s$.

One main consequence of this paper is that we derive the optimal formula for $s=s(p,\gamma,\delta)$, in the case when $\delta$ is sufficiently small (how small $\delta$ needs to be depends on $\gamma$ and $p$).  In fact, for this choice of $s$ we prove $T:L^p_r\rightarrow L^p_{r+s}$ so long as $r$ and $\delta$ are sufficiently small (depending on $\gamma$ and $p\in (1,\infty)$).  Here $r$ and $\delta$ can be either positive, negative, or zero.\footnote{When $\delta\leq 0$ one needs to add an additional cancellation condition on $K$ in a standard way.}
When both $r$ and $\delta$ are $0$, this is just a reprise of the $L^p$ boundedness result of \cite{ChristNagelSteinWaingerSingularAndMaximalRadonTransforms}.
Moreover, our results are sharper than this.  We introduce non-isotropic $L^p$-Sobolev spaces adapted to $\gamma$:
$\NLp{p}{r}{\gamma}$ for $r\in \R$.
Our result takes the form $T:\NLp{p}{r}{\gamma}\rightarrow \NLp{p}{r+\delta}{\gamma}$, provided $r$ and $\delta$ are sufficiently small.
In fact, we prove mapping properties for $T$ on the spaces $\NLp{p}{r}{\gammat}$, where $\gammat$ can be a different choice of $\gamma$--the
smoothing properties of $T$ on $L^p_r$ are a special case of this.
See Section \ref{SectionIntroSingle} for more precise details on this.

We proceed more generally than the above.  In the series \cite{SteinStreetA,SteinStreetI,SteinStreetII,SteinStreetIII} the author and Stein
introduced a more general framework than the one studied in \cite{ChristNagelSteinWaingerSingularAndMaximalRadonTransforms}.
Again we consider operators of the form
\begin{equation}\label{EqnIntroMultiOp}
Tf(x) = \psi(x) \int f(\gamma_t(x)) K(t)\: dt.
\end{equation}
As before $\gamma_t(x)=\gamma(t,x):\R^{N}_0\times \R^{n}_0\rightarrow \R^n$ is a $C^\infty$ function defined on a neighborhood of $(0,0)\in \R^N\times \R^n$
satisfying $\gamma_0(x)\equiv x$, $\psi\in C_0^\infty(\R^n)$ is supported on a small neighborhood of $0\in \R^n$.
$K$ is now a ``multi-parameter'' kernel.  The simplest situation to consider is when $K$ is a kernel of ``product type,'' though we will later deal with more general
kernels (see Section \ref{SectionResKer}).  To define this notion, we decompose $\R^N$ into $\nu$ factors: $\R^N= \R^{N_1}\times \cdots \times \R^{N_\nu}$;
and we write $t\in \R^{N}$ as $t=(t_1,\ldots, t_\nu)\in \R^{N_1}\times \cdots\times \R^{N_\nu}$.
A product kernel of order $\delta=(\delta_1,\ldots, \delta_\nu)\in \R^{\nu}$ satisfies estimates like
\begin{equation*}
\q|\partial_{t_1}^{\alpha_1} \cdots \partial_{t_\nu}^{\alpha_\nu} K(t_1,\ldots, t_\nu)\w|\lesssim |t_1|^{-N_1-|\alpha_1|-\delta_1} \cdots |t_\nu|^{-N_{\nu}-|\alpha_\nu|-\delta_\nu},
\end{equation*}
along with certain ``cancellation conditions'' if any of the coordinates of $\delta$ are non-negative.\footnote{Here we have reversed the role of $\delta$ and $-\delta$ from above.  In this notation, a kernel of order $\delta$ is ``smoothing'' of order $-\delta$.  When some coordinate $\delta_\mu$ of $\delta$ satisfies $\delta_\mu\leq -N_\mu$, a different definition is needed.  See Section \ref{SectionResKer} for more precise details on these kernels.}
In this situation, for $1<p<\infty$ and $r\in \R^{\nu}$ with $|r|$ sufficiently small (how small depends on $p$ and $\gamma$), we define
non-isotropic Sobolev spaces $\NLp{p}{r}{\gamma}$; and if $|\delta|$ and $|r|$ are sufficiently small (how small depends on $p$ and $\gamma$)
we prove mapping properties of the form
\begin{equation*}
T:\NLp{p}{r}{\gamma}\rightarrow \NLp{p}{r-\delta}{\gamma}.
\end{equation*}
Furthermore, we prove mapping properties for $T$ on spaces $\NLp{p}{r'}{\gammat}$, where $\gammat:\R_0^{\Nt}\times \R_0^n\rightarrow \R^n$ can be a different choice of $\gamma$--where there is an underlying  decomposition of $\R^{\Nt}$ into $\nut$ factors, and $r'\in \R^{\nut}$ is small.
In fact, the way the single-parameter results are proved is by lifting to the more general multi-parameter situation.  Thus, this more general framework is used
even if one is only interested in the single-parameter results.

In Section \ref{SectionIntroSingle} we outline the special case of some of our results when $T$ is of the form studied by Christ, Nagel, Stein, and Wainger \cite{ChristNagelSteinWaingerSingularAndMaximalRadonTransforms}.
This special case is likely the one of the most interest to many readers.  In Section \ref{SectionIntroPast} we outline some of the history of these problems
and reference related works.
In Sections \ref{SectionResKer}-\ref{SectionResSurf} we introduce all of the terminology necessary to state our results in full generality.
In Sections \ref{SectionResSobNew} and \ref{SectionResRadon} we state our main results.  Here and in the previous sections we include only the simplest and most instructive proofs.
In Sections \ref{SectionSchwartzAndProd}-\ref{SectionPfRadon} we prove all of the results whose proofs are more difficult and were not included in the previous sections.
Finally, in Section \ref{SectionOptimality} we address the question of the optimality of our results, but here we focus only on the single-parameter case.

The statement of the results in this paper are self-contained:  the reader does not need to be familiar with any previous works to understand the statement
of the main results.
The proofs, however, rely on the theories developed in several previous works including 
the series \cite{SteinStreetA,SteinStreetI,SteinStreetII,SteinStreetIII},
the papers \cite{ChristNagelSteinWaingerSingularAndMaximalRadonTransforms} and \cite{StreetMultiParameterCCBalls},
and the book \cite{StreetMultiParamSingInt}.

%% file: introsingle.tex
The setting in which our results are easiest to understand
is when $\gamma$ is of the form studied by Christ, Nagel, Stein, and Wainger in their
foundational work \cite{ChristNagelSteinWaingerSingularAndMaximalRadonTransforms}.  This falls under the single-parameter ($\nu=1$)
setting in this paper, and here we informally outline our results in this case.\footnote{Even when $\nu=1$, our setting allows more general $\gamma$ than those 
considered in \cite{ChristNagelSteinWaingerSingularAndMaximalRadonTransforms}.  For instance, our theory addresses the case when $\gamma_t(x)$
is real analytic, even if it does not satisfy the conditions of \cite{ChristNagelSteinWaingerSingularAndMaximalRadonTransforms}--see Corollary \ref{CorResSurfAnalSingle}.  In that case, the smoothing occurs only along leaves of a foliation, as opposed to smoothing in full isotropic Sobolev spaces.}
In particular, in this section we focus on the case when the kernel $K$ from \eqref{EqnIntroCNSWOp} is a
standard fractional integral kernel.  More precisely, $K(t)$ is a distribution supported on a small ball
centered at $0\in \R^N$, and  satisfies (for some $\delta\in \R$),
\begin{equation}\label{EqnIntroSingleKerEst}
\q|\partial_t^{\alpha}K(t)\w|\leq C_\alpha |t|^{-N-|\alpha|-\delta}, \quad \forall \alpha.
\end{equation}
In addition, if $\delta\geq 0$, $K$ is assumed to satisfy certain ``cancellation conditions'' which are made precise later (see Section \ref{SectionResKer} and
in particular Section \ref{SectionResProdKer}).\footnote{When $\delta\leq -N$, one needs different estimates.  We address the kernels more formally in Section \ref{SectionResKer}, but we are mostly interested in $\delta$ small, and so this is not a central point in what follows.}
 
 As mentioned above, $\gamma_t(x)=\gamma(t,x):\R^N_0\times \R^n_0\rightarrow \R^n$ is a $C^\infty$ function defined on a small neighborhood
 of $(0,0)\in \R^N\times \R^n$ and satisfies $\gamma_0(x)\equiv x$.  The paper \cite{ChristNagelSteinWaingerSingularAndMaximalRadonTransforms} shows
 that $\gamma$ can be written asymptotically as\footnote{\eqref{EqnIntroSingleCNSWExp} means that $\gamma_t(x) = \exp\q(\sum_{0<|\alpha|<M} t^{\alpha} X_\alpha\w)x+O(|t|^M)$, $\forall M$.}
 \begin{equation}\label{EqnIntroSingleCNSWExp}
 \gamma_t(x)\sim \exp\q(\sum_{|\alpha|>0} t^{\alpha} X_\alpha\w)x,
 \end{equation} 
 where each $X_\alpha$ is a $C^\infty$ vector field defined near $0$ on $\R^n$.
 The main hypothesis  studied in \cite{ChristNagelSteinWaingerSingularAndMaximalRadonTransforms} is that
 the vector fields $\{X_\alpha: |\alpha|>0\}$ satisfy H\"ormander's condition:  the Lie algebra generated by the vector fields span the tangent space
 at every point near $0$.
 
Let $\sS=\{(X_\alpha, |\alpha|):|\alpha|>0\}$.
Associated to $\sS$ are natural non-isotropic Sobolev spaces, for $1<p<\infty$ and $\delta\in \R$, which we denote by $\NLp{p}{\delta}{\sS}$.
We give the formal definition later (see Section \ref{SectionResSobNew}), but the intuitive idea is that these Sobolev spaces are defined so that for $(X,d)\in \sS$, $X$ is a differential operator of ``order'' $d$; in particular, $X:\NLp{p}{\delta}{\sS}\rightarrow \NLp{p}{\delta-d}{\sS}$.
Sobolev spaces of this type have been developed by several authors.  For settings closely related to the one here, see \cite{FollandSteinEstimatesForTheDbarComplex,FollandSubellipticEstimatesAndFunctionSpacesOnNilpotent,RothschildSteinHypoellipticDifferentialOperatorsAndNilpotentGroups,NagelRosaySteinWaingerEstimatesForTheBergmanAndSzegoKernels,KoenigOnMaximalSobolevAndHolderEstimatesForTheTangentialCR,StreetMultiParamSingInt}. 
Let $\psi\in C_0^\infty(\R^n)$ be supported on a small neighborhood of $0$ and let $K(t)$ be a kernel of order $\delta\in \R$ in the sense of \eqref{EqnIntroSingleKerEst}
which is supported near $0\in \R^N$.
Define the operator
\begin{equation*}
Tf(x) = \psi(x) \int f(\gamma_t(x)) K(t)\: dt.
\end{equation*}
Let $\sL(\sS)$ be the smallest set such that:
\begin{itemize}
\item $\sS\subseteq \sL(\sS)$.
\item If $(X_1,d_1),(X_2,d_2)\in \sL(\sS)$, then $([X_1,X_2],d_1+d_2)\in \sL(\sS)$.
\end{itemize}
By hypothesis, there is a finite set $\sF\subset \sL(\sS)$ such that $\{ X: (X,d)\in \sF\}$ spans the tangent space at every point near $0$.
Set $E=\max\{ d: (X,d)\in \sF\}$ and $e=\min\{|\alpha| :X_{\alpha}\ne 0\}$; where we have picked $\sF$ so that $E$ is minimal.\footnote{$X_\alpha\ne 0$ means
that $X_\alpha$ is not identically the zero vector field on a neighborhood of $0$.}
Note, $1\leq e\leq E$.  
Identify $\psi$ with the operator $\psi: f\mapsto \psi  f$.  Let $L^p_s$, $1<p<\infty$, $s\in \R$, denote the standard
(isotropic) $L^p$ Sobolev space of order $s$ on $\R^n$.

\begin{thm}
For $1<p<\infty$, there exists $\epsilon=\epsilon(p,\gamma)>0$ such that if $\delta,\delta_0\in (-\epsilon, \epsilon)$,
$$T:\NLp{p}{\delta_0}{\sS}\rightarrow \NLp{p}{\delta_0-\delta}{\sS}.$$
\end{thm}
\begin{proof}
This is a special case of Theorem \ref{ThmResRadonMainThm}.\footnote{In the sequel, the vector fields $X_\alpha$ are defined a slightly different way;
however this different definition is equivalent to the above definition.  See Section 17.3 of \cite{SteinStreetI}.}
\end{proof}

\begin{thm}\label{ThmIntroSingleInclusion}
Let $1<p<\infty$, $\delta\geq 0$.  Then,
\begin{equation*}
\psi: L^p_{\delta}\rightarrow \NLp{p}{\delta e}{\sS}.
\end{equation*}
Dually, we have
\begin{equation*}
\NLp{p}{-\delta e}{\sS}\hookrightarrow L^p_{-\delta}.
\end{equation*}
Also,
\begin{equation*}
\NLp{p}{\delta E}{\sS}\hookrightarrow L^p_{\delta},
\end{equation*}
and dually,
\begin{equation*}
\psi: L^p_{-\delta}\rightarrow \NLp{p}{-\delta E}{\sS}.
\end{equation*}
\end{thm}
\begin{proof}
The above results are a special cases of Theorems \ref{ThmSobEuclidNoSpanNew} and \ref{ThmSobEuclidSpanNew}.
\end{proof}

\begin{thm}\label{ThmIntroSingleBound}
Let $1<p<\infty$.  There exists $\epsilon=\epsilon(p,\gamma)>0$ such that if $s,\delta\in (-\epsilon,\epsilon)$, we have
\begin{itemize}
\item If $\delta\geq 0$, $T:L^p_s\rightarrow L^p_{s-\delta/e}$.
\item If $\delta\leq 0$, $T:L^p_s\rightarrow L^p_{s-\delta/E}$.
\end{itemize}
Furthermore, this result is optimal in the following sense (recall, $\epsilon$ depends on $p\in (1,\infty)$):
\begin{itemize}
\item There do not exist $p\in (1,\infty)$, $\delta\in [0,\epsilon)$, $s\in (-\epsilon,\epsilon)$, and $t>0$ such that for every operator $T$ of the above form we have
$T:L^p_{s}\rightarrow L^p_{s-\delta/e+t}$
\item There do not exist $p\in (1,\infty)$, $\delta\in (-\epsilon,0]$, $s\in (-\epsilon,\epsilon)$, and $t>0$ such that for every operator $T$ of the above form we have
$T:L^p_{s}\rightarrow L^p_{s-\delta/E+t}$
\end{itemize}
\end{thm}
\begin{proof}
See Corollary \ref{CorResRadonOtherHorWithEuclidNewer}.
\end{proof}

In fact, the results in this paper are more general than the above.  They include:
\begin{itemize}
\item Instead of only considering the operator $T$ acting on the spaces $\NLp{p}{\delta}{\sS}$ and $L^p_s$, we consider the operator acting
on the more general spaces $\NLp{p}{\delta'}{\sSh}$, for some other choice of $\sSh$.  Furthermore, we compare the spaces $\NLp{p}{\delta}{\sS}$
and $\NLp{p}{\delta'}{\sSh}$.
\item We consider more general kernels $K$.  This includes single-parameter kernels with nonisotropic dilations, along with more general multi-parameter kernels.
These multi-parameter kernels include fractional kernels of product type, but also more general multi-parameter kernels.  In these cases, we work with multi-parameter
non-isotropic Sobolev spaces.

\item The above results hold only for the various parameters ($\delta$, $s$, etc.) small, and in general this is necessary.  However, we also present additional conditions
on $\gamma$ under which the above results extend to all parameters.
\end{itemize}

\begin{rmk}
Our results in the single-parameter case discussed above rely heavily on the theory we develop for the multi-parameter case.  Thus, even if one is only
interested in the single-parameter case, the multi-parameter case is essential for our methods.
\end{rmk}



%% file: intropast.tex
All of the previous work on questions like the ones addressed in this paper have addressed the single parameter ($\nu=1$) case.
The work most closely related to the results in this paper is that of Greenblatt \cite{GreenblattAnAnalogueToATheoremOfFeffermanAndPhong},
who studied the case $\nu=1$, $N=1$, and $p=2$, under the additional condition that $\frac{\partial \gamma}{\partial t}(t,x)\ne 0$.  He proved
optimal smoothing in isotropic Sobolev spaces, $L^2\rightarrow L^2_s$, for such operators.\footnote{These are the results of   \cite{GreenblattAnAnalogueToATheoremOfFeffermanAndPhong} as they are stated in that paper.  However, the same methods can be extended to
some $L^p$ spaces for $p\ne 2$ via interpolation.  It also seems possible that similar methods could be used to treat some instances in the case $N>1$.}
Our results imply these results, sharpen them to
nonisotropic Sobolev spaces, generalize them to optimal estimates on all $L^p$ spaces, and remove many of the above assumptions.\footnote{\cite{GreenblattAnAnalogueToATheoremOfFeffermanAndPhong} only assumed estimates on $K$ and $\frac{d}{dt} K$, whereas we assume estimates on all derivatives of $K$.  This is not an essential point, and the methods of this paper can be modified to deal with this lesser smoothness, though we do not pursue it here.}
Greenblatt related these results to well-known results of Fefferman and Phong concerning subelliptic operators \cite{FeffermanPhongSubellipticEigenValueProblems}--see Remark \ref{RmkFeffermanPhong} for more details on this. 

An important work is that of Cuccagna \cite{CuccagnaSobolevEstimatesForFractionalAndSingularRadonTransforms} who, under strong additional hypotheses on $\gamma$, made explicit the dichotomy between
the smoothing nature of $T$ for $\delta$ small, and the smoothing properties when $\delta$ is large.  Our results are less precise than this--we only deal with $\delta$ very small, and say nothing about the case when $\delta$ is large.  However, our results hold for much more general $\gamma$ than those in \cite{CuccagnaSobolevEstimatesForFractionalAndSingularRadonTransforms}.\footnote{The situation when $\delta$ is large seems to be far
from well understood when considering the generality covered in this paper.}

Operators of the form 
\begin{equation}\label{EqnIntroPast}
f\mapsto \int f(\gamma_t(x)) K(t) \: dt, \quad \gamma_0(x)\equiv x,
\end{equation}
have a long history.  
In the discussion that follows, the distribution $K(t)$ is usually assumed to be supported near $t=0$.

When $K(t)$ is a Calder\'on-Zygmund kernel, the goal is often to prove that the above operator is bounded
on various $L^p$ spaces.
This began with the work of Fabes \cite{FabesSingularIntegralsAndPartialDifferentialEquationsOfParabolicType} who studied the case when $p=2$ and $K(t)=\frac{1}{t}$, $\gamma_t(x,y)=(x-t, y-t^2)$:  the so-called ``Hilbert transform along the parabola''; see also \cite{SteinWaingerTheEstimationOfAnIntegralArising}.
Following these initial results many papers followed. 
First, the setting where the operators were translation invariant on $\R^n$ was handled
by 
Stein \cite{SteinMaximalFunctionsPoissonIntegralsOnSymmetricSpaces,SteinMaximalFunctionsIIHomogeneousCurves},
Nagel, Rivi{\`e}re, and Wainger \cite{NagelRiviereWaingerI,NagelRiviereWaingerII,NagelRiviereWaingerIII},
and Stein and Wainger \cite{SteinWaingerProblemsInHarmonicAnalysisRelatedToCurvature}.
Moving beyond operators which were translation invariant on $\R^n$, the first results were obtained by Nagel, Stein, and Wainger \cite{NagelSteinWaingerHilbertTransformsAndMaximalFunctionsRelatedToVariableCurves}.
Next, operators which were translation invariant on a nilpotent Lie group
were handled by
Geller and Stein \cite{GellerSteinSingularConvolutionOperatorsOnTheHeisenbergGroup,GellerSteinEstimatesForSingularConvolutionOperatorsOnTheHeisenbergGroup},
M\"uller
\cite{MullerSingularKernelsSupportedByHomogeneousSubmanifolds,MullerCalderonZygmundKernelsCarriedByLinearSubsubspace,MullerTwistedConvolutionsWithCalderonZygmundKernels},
 Christ \cite{ChristHilbertTransformsAlongCurvesI},
 and Ricci and Stein \cite{RicciSteinHarmonicAnalysisOnNilpotentGroupsAndSingularIntegralsII}.
An important work that moves beyond the group translation invariant setting is that of Phong and Stein 
\cite{PhongSteinHilbertIntegrals}.
These ideas were generalized and unified by the influential work of
Christ, Nagel, Stein, and Wainger \cite{ChristNagelSteinWaingerSingularAndMaximalRadonTransforms}, who introduced
general conditions on $\gamma$ under which one can obtain $L^p$ bounds for such operators ($1<p<\infty$)--they referred to the conditions
on $\gamma$ as the {\it curvature condition}.  We refer the reader to that paper
for a more leisurely history of the work which proceeded it.

Questions regarding smoothing of operators of the form \eqref{EqnIntroPast}, when $K$ is a fractional kernel, implicitly take
their roots in H\"ormander's work on Fourier integral operators.  
This was then taken up by Ricci and Stein \cite{RicciSteinHarmonicAnalysisOnNilpotentGroupsAndSingularIntegralsIII},
Greenleaf and Uhlmann \cite{GreenleafUhlmannEstimatesForSingularRadonTransforms},
Christ in an unpublished work \cite{ChristEndpointBoundsForSingularFractionalIntegralOperators} which was
extended by
Greenleaf, Seeger, and Wainger \cite{GreenleafSeegerWaingerOnXrayTransformsForRigidLineComplexes},
 Seeger and Wainger \cite{SeegerWaingerBoundsForSingularFractionalIntegralsAndRelatedFIO}, and others.  Many of these works studied more
general operators than \eqref{EqnIntroPast} by working in the framework of Fourier integral operators.  These considerations forced the authors
to heavily restrict the class of $\gamma$ considered, though often allowed optimal smoothing estimates for \eqref{EqnIntroPast} even when $K$ is a fractional
kernel smoothing of some large order.
As with the work of Cuccagna \cite{CuccagnaSobolevEstimatesForFractionalAndSingularRadonTransforms} cited above,
one can see the difference between the smoothing properties of \eqref{EqnIntroPast} when $K$ is a fractional
kernel smoothing of a small order, and when $K$ is a fractional kernel smoothing of a larger order in these works.

One line of inquiry culminated in the above mentioned work of Christ, Nagel, Stein, and Waigner \cite{ChristNagelSteinWaingerSingularAndMaximalRadonTransforms},
who showed that their curvature condition on $\gamma$ was necessary and sufficient for the operator given in \eqref{EqnIntroPast}
to be smoothing in $L^p$ Sobolev spaces, for every fractional integral kernel $K$.   I.e., if $K$ is a fractional integral kernel which is 
smoothing of order $\delta>0$ (i.e., a kernel of order $-\delta$), the above operator maps
$L^p\rightarrow L^p_s$ for $1<p<\infty$ and some $s=s(p,\delta,\gamma)>0$.  Left open, however, was the optimal choice of $s$.
This has since been taken up by many authors.  Including the previously mentioned work of Greenblatt \cite{GreenblattAnAnalogueToATheoremOfFeffermanAndPhong}.

%

When $K$ is a multi-parameter kernel, the $L^p$ boundedness of operators of the form \eqref{EqnIntroPast}
began with the product theory of singular integrals.  This was introduced by 
R. Fefferman and Stein \cite{FeffermanSteinSingularIntegralsOnProductSpaces}.  Another early important work is that of Journ\'e
\cite{JourneCalderonZygmundOperatorsOnProductSpaces}.
This was followed by many works on the product theory of singular integrals--see, for instance,
\cite{FeffermanHarmonicAnalysisOnProductSpaces,ChangFeffermanSomeRecentDevelopmentsInFourierAnalysisAndHpTheory,RicciSteinMultiparameterSingularIntegralsAndMaximalFunctions,MullerRicciSteinMarcinkiewiczMulipliersAndMulitparameterStructure}.

Outside of the product type situation, the theory of translation invariant operators on nilpotent Lie groups given by convolution with a
flag kernel has been influential.  This started with the work of M\"uller, Ricci, and Stein \cite{MullerRicciSteinMarcinkiewiczMulipliersAndMulitparameterStructure,MullerRicciSteinMarcinkiewiczMultipliersAndMultiParameterStructuresOnHeisenbergTypeGroupsII}
and was furthered by Nagel, Ricci, and Stein \cite{NagelRicciSteinSingularIntegralsWithFlagKernels} and
Nagel, Ricci, Stein, and Wainger \cite{NagelRicciSteinWaingerSingularIntegralWithFlagKernelsOnHomogeneousGroupsI}.
See, also, \cite{GlowackiCompositionAndLtBoundednessOfFlagKernels,GlowackiCompositionAndLtBoundednessOfFlagKernelsCorrection,GlowackiLpBoundednessOfFlagKernelsOnHomogeneousGroups}.

The above product theory of singular integrals and flag theory of singular integrals only apply to very non-singular $\gamma$ when considering
operators of the form \eqref{EqnIntroPast}.  In particular, all of these operators can be thought of as a kind of singular integral.
These concepts were unified and generalized in the monograph \cite{StreetMultiParamSingInt}.
More singular forms of $\gamma$ were addressed by the author and Stein in the series \cite{SteinStreetA,SteinStreetI,SteinStreetII,SteinStreetIII},
where conditions on $\gamma$ were imposed which yielded $L^p$ boundedness ($1<p<\infty$) for operators of the form \eqref{EqnIntroPast}
when $K$ is a multi-parameter kernel.

As mentioned in the previous section, non-isotropic Sobolev spaces of the type studied in this paper have been studied by many authors.
This began with the work of Folland and Stein \cite{FollandSteinEstimatesForTheDbarComplex}, which
was soon followed by work of Folland \cite{FollandSubellipticEstimatesAndFunctionSpacesOnNilpotent}
and Rothschild and Stein \cite{RothschildSteinHypoellipticDifferentialOperatorsAndNilpotentGroups}.
See also, \cite{NagelRosaySteinWaingerEstimatesForTheBergmanAndSzegoKernels,KoenigOnMaximalSobolevAndHolderEstimatesForTheTangentialCR}.
Multi-parameter non-isotropic Sobolev spaces, like the ones studied in this paper, were studied by the author in \cite{StreetMultiParamSingInt}.
The Sobolev spaces in this paper generalize the ones found in \cite{StreetMultiParamSingInt}.

%% file: resker.tex
In this section, we state the main definitions and results concerning the class of distribution kernels $K(t)$ for which we study
operators of the form \eqref{EqnIntroMultiOp}.
$K(t)$ is a distribution on $\R^N$ which is supported in $B^N(a)=\{x\in \R^N:|x|<a\}$, where $a>0$ is small (to be chosen later).\footnote{In particular,
$a>0$ will be chosen so small that for $t\in B^N(a)$, $\gamma_t(\cdot)$ is a diffeomorphism onto its image, and we may consider $\gamma_t^{-1}(\cdot)$,
the inverse mapping.}

We suppose we are given $\nu$ parameter dilations on $\R^N$.  That is, we are given $e=(e_1,\ldots, e_N)$ with
each $0\ne e_j\in [0,\infty)^\nu$.  For $\delta\in [0,\infty)^\nu$ and $t=(t_1,\ldots, t_N)\in \R^N$, we define
\begin{equation}\label{EqnResKerDefnDil}
\delta t= (\delta^{e_1} t_1,\ldots, \delta^{e_N} t_N)\in \R^N,
\end{equation}
where $\delta^{e_j}$ is defined by standard multi-index notation: $\delta^{e_j}=\prod_{\mu=1}^\nu \delta_\mu^{e_j^\mu}.$

For $1\leq \mu\leq \nu$, let $t_\mu$ denote the \textit{vector} consisting of those coordinates $t_j$ of $t$ such that $e_j^\mu\ne 0$.  Note that $t_\mu$ and $t_{\mu'}$
may involve some of the same coordinates, even if $\mu\ne \mu'$, and every coordinate appears in at least one $t_\mu$.
Let $\schS(\R^N)$ denote Schwartz space on $\R^N$, and for $E\subseteq \nuset$ let $\schS_E$ denote the set of those
$f\in \schS(\R^N)$ such that for every $\mu\in E$,
\begin{equation*}
\int t_{\mu}^{\alpha_\mu} f(t)\: dt_\mu=0, \text{ for all multi-indices }\alpha_\mu.
\end{equation*}
$\schS_E$ is a closed subspace of $\schS(\R^N)$ and inherits the subspace topology, making $\schS_E$ a Fr\'echet space.

For $j=(j_1,\ldots, j_\nu)\in \R^\nu$, we define $2^j=(2^{j_1},\ldots, 2^{j_\nu})\in (0,\infty)^\nu$, so that it makes sense to write $2^j t$
using the above multi-parameter dilations; i.e., $2^j t = (2^{j\cdot e_1} t_1,\ldots, 2^{j\cdot e_N} t_N)$.
Given a function $\vsig:\R^N\rightarrow \C$, define 
\begin{equation}\label{EqnResKerDefFuncDil}
\dil{\vsig}{2^j}(t)=2^{j\cdot e_1+\cdots+j\cdot e_N} \vsig(2^jt).
\end{equation}
  Note that
$\dil{\vsig}{2^j}$ is defined in such a way that $\int \dil{\vsig}{2^j}(t)\: dt= \int \vsig(t)\: dt$.

\begin{defn}\label{DefnKer}
For $\delta\in \R^\nu$, we define $\sK_{\delta}=\sK_{\delta}(N,e,a)$ to be the set of all distributions $K$ of the form
\begin{equation}\label{EqnDefnsKdeltaSum}
K(t) = \eta(t) \sum_{j\in \N^\nu} 2^{j\cdot \delta} \dil{\vsig_j}{2^j}(t),
\end{equation}
where $\eta\in C_0^\infty(B^N(a))$, $\{\vsig_j:j\in \N^\nu\}\subset \schS(\R^N)$ is a bounded set with $\vsig_j\in \schS_{\{\mu:j_\mu\ne 0\}}$.
The convergence in \eqref{EqnDefnsKdeltaSum} is taken in the sense of distributions.  We will see in Lemma \ref{LemmaKerConvInDist} that every such sum converges in the
sense of distributions.
\end{defn}

Using the dilations, it is possible to assign to each multi-index $\alpha\in \N^N$ a corresponding ``degree'':
\begin{defn}\label{DefnResKerDegree}
Given a multi-index $\alpha\in \N^N$, we define
\begin{equation*}
\deg(\alpha):=\sum_{j=1}^N  \alpha_j e_j\in [0,\infty)^\nu.
\end{equation*}
\end{defn}

\begin{lemma}\label{LemmaResKerDelta}
For any $a>0$, $\delta_0\in \sK_0(N,e,a)$, where $\delta_0$ denotes the Dirac $\delta$ function at $0$.
Moreover, for any $a>0$, $\alpha\in \N^N$, $\partial_t^{\alpha} \delta_0 \in \sK_{\deg(\alpha)}(N,e,a)$.
\end{lemma}
\begin{proof}This is proved in Section \ref{SectionSchwartzAndProd}.\end{proof}

When some of the coordinates of $\delta$ are strictly negative, there is an a priori slightly weaker definition for $\sK_{\delta}$ which turns out to be equivalent.
This is presented in the next proposition.
\begin{prop}\label{PropResKerOtherDefn}
Suppose $\delta\in \R^{\nu}$, $a>0$, $\eta\in C_0^{\infty}(B^N(a))$, and $\{\vsig_j : j\in \N^{\nu}\}\subset \schS(\R^N)$ is a bounded
set with $\vsig_j\in \schS_{\{\mu : j_\mu\ne 0, \delta_\mu\geq 0\}}$.  Then the sum
\begin{equation*}
K(t):=\eta(t)\sum_{j\in \N^{\nu}} 2^{j\cdot \delta}\dil{\vsig_j}{2^j}(t),
\end{equation*}
converges in distribution and $K\in \sK_{\delta}(N,e,a)$.
\end{prop}
\begin{proof}This is proved in Section \ref{SectionPfKerOther}.\end{proof}

%% file: resprodker.tex
Definition \ref{DefnKer} is extrinsic, and in general we do not know of a simple intrinsic characterization of the kernels in $\sK_\delta$.
However, under the {\bf additional assumption} that each $e_j$ is nonzero in precisely one component, these kernels are the standard
product kernels.  Product kernels were introduced by R. Fefferman and Stein \cite{FeffermanSingularIntegralsOnProductDomains,FeffermanSteinSingularIntegralsOnProductSpaces} and studied by several authors; there are
too many to list here, but some influential papers include \cite{JourneCalderonZygmundOperatorsOnProductSpaces, FeffermanHarmonicAnalysisOnProductSpaces, ChangFeffermanSomeRecentDevelopmentsInFourierAnalysisAndHpTheory, FeffermanSomeRecentDevelopmentsInFourierAnalysisII, RicciSteinMultiparameterSingularIntegralsAndMaximalFunctions, MullerRicciSteinMarcinkiewiczMulipliersAndMulitparameterStructure}.  The definitions here follow 
ideas of Nagel, Ricci, and Stein \cite{NagelRicciSteinSingularIntegralsWithFlagKernels} and are taken from \cite{StreetMultiParamSingInt}.  We now turn to presenting the relevant definitions for this concept.  For the rest of this section,
we assume the following.

\begin{assumption}\label{AssumpKerProduct}
Each $0\ne e_j\in [0,\infty)^\nu$ is nonzero in precisely one component.  I.e., $e_j^\mu\ne 0$ for precisely one $\mu\in \nuset$.
\end{assumption}

\begin{rmk}
It is only in the following intrinsic characterization of $\sK$ that we need Assumption \ref{AssumpKerProduct}--for the rest of the results in this paper,
it is not used.
\end{rmk}

Define $t_\mu$ as before; i.e., $t_\mu$ is the vector consisting of those coordinates $t_j$ of $t$ such that $e_j^\mu\ne 0$.  Let $N_\mu$ denote
the number of coordinates in $t_\mu$.  Because of Assumption \ref{AssumpKerProduct}, this decomposes
$t=(t_1,\ldots, t_\nu)\in \R^{N_1}\times \cdots\times \R^{N_\nu}=\R^N$.

For each $\mu$, we obtain single parameter dilations on $\R^{N_\mu}$.  Indeed, we write $t_\mu=(t_\mu^1,\ldots, t_\mu^{N_\mu})$.
If the coordinate $t_\mu^k$ corresponds to the coordinate $t_{j_{\mu,k}}$ of $t$, then we write $h_\mu^k:=e_{j_{\mu,k}}^{\mu}$--the $\mu$th
component of $e_{j_{\mu,k}}$ which is nonzero by assumption.  We define, for $\delta_\mu\geq 0$,
\begin{equation*}
\delta_\mu t_\mu:= \q(\delta_\mu^{h_\mu^1} t_\mu^1,\ldots, \delta_\mu^{h_\mu^{N_\mu}} t_{\mu}^{N_\mu}\w).
\end{equation*}
In short, these dilations are defined so that if $\delta=(\delta_1,\ldots, \delta_\nu)$, then
\begin{equation*}
\delta t = (\delta_1 t_1,\ldots, \delta_\nu t_\nu).
\end{equation*}
Let $Q_\mu=h_\mu^1+\cdots+h_{\mu}^{N_{\mu}}$.  $Q_\mu$ is called the ``homogeneous dimension'' of $\R^{N_\mu}$ under these dilations.

For $t_\mu\in \R^{N_\mu}$, we write $\|t_\mu\|$ for a choice of a smooth homogenous norm on $\R^{N_\mu}$.  I.e., $\|t_\mu\|$ is smooth
away from $t_\mu=0$ and satisfies $\|\delta_\mu t_\mu\|= \delta_\mu\|t_\mu\|$, and $\|t_\mu\|\geq 0$ with $\|t_\mu\|=0\Leftrightarrow t_\mu=0$.
Any two choices for this homogenous norm are equivalent for our purposes.  For instance, we can take
\begin{equation}\label{EqnKerHomogNorm}
\|t_\mu\| = \q( \sum_{l=1}^{N_\mu} \q|t_\mu^l\w|^{\frac{2(N_\mu!) }{ h_\mu^l } }  \w)^{\frac{ 1}{2(N_\mu! ) }}.
\end{equation}

\begin{defn}
The space of product kernels of order $m=(m_1,\ldots, m_\nu)\in(-Q_1,\infty)\times \cdots\times(-Q_\nu,\infty)$ is a locally convex topological vector
space made of distributions $K(t_1,\ldots, t_\nu)\in C_0^\infty(\R^N)'$.  The space is defined recursively.  For $\nu=0$ it is defined to be $\C$,
with the usual topology.  We assume that we have defined the locally convex topological vector spaces of product kernels up to
$\nu-1$ factors, and we define it for $\nu$ factors.  The space of product kernels is the space of distribution $K\in C_0^\infty(\R^N)'$
such that the following two types of semi-norms are finite:
\begin{enumerate}[(i)]
\item (Growth Condition) For each multi-index $\alpha=(\alpha_1,\ldots, \alpha_\nu)\in \N^{N_1}\times \cdots \times \N^{N_\nu}=\N^N$
we assume there is a constant $C=C(\alpha)$ such that
\begin{equation*}
\q|\partial_{t_1}^{\alpha_1}\cdots \partial_{t_\nu}^{\alpha_\nu} K(t_1,\ldots, t_\nu) \w|\leq C \| t_1\|^{-Q_1-m_1-\alpha_1\cdot h_1}\cdots \|t_\nu\|^{-Q_\nu-m_\nu-\alpha_\nu\cdot h_\nu}.
\end{equation*}
We define a semi-norm to be the least possible $C$.  Here $\|t_\mu\|$ is the homogenous norm on $\R^{N_\mu}$ as defined in \eqref{EqnKerHomogNorm}.

\item (Cancellation Condition) Given $1\leq \mu\leq \nu$, $R>0$, and a bounded set $\sB\subset C_0^\infty(\R^{N_\mu})$ for $\phi\in \sB$ we define
\begin{equation*}
K_{\phi,R} (t_1,\ldots, t_{\mu-1},t_{\mu+1},\ldots, t_\nu):= R^{-m_\mu}\int K(t) \phi(Rt_\mu)\: dt_\mu,
\end{equation*}
where $Rt_\mu$ is defined by the single parameter dilations on $\R^{N_\mu}$.  This defines a distribution
\begin{equation*}
K_{\phi,R}\in C_0^\infty\q(\R^{N_1}\times \cdots\times \R^{N_{\mu-1}}\times \R^{N_{\mu+1}}\times \cdots\times \R^{N_\nu}\w)'
\end{equation*}
We assume that this distribution is a product kernel of order 
$$(m_1,\ldots, m_{\mu-1},m_{\mu+1},\ldots, m_\nu).$$  Let $\|\cdot\|$ be a continuous
semi-norm on the space of $\nu-1$ factor product kernels of order $$(m_1,\ldots, m_{\mu-1},m_{\mu+1},\ldots, m_\nu).$$  We define a semi-norm
on $\nu$ factor product kennels of order $m$ by $\|K\|:=\sup_{\phi\in \sB, R>0} \|K_{\phi,R}\|$, which we assume to be finite.
\end{enumerate}
We give the space of product kernels of order $m$ the coarsest topology such that all of the above semi-norms are continuous.
\end{defn}

\begin{prop}\label{PropKerEquivToProd}
Fix $a>0$ and $m\in(-Q_1,\infty)\times\cdots\times (-Q_\nu,\infty)$.  
If $K$ is a product kernel of order $m$ and $\supp{K}\subset B^N(a)$, then $K\in \sK_m$.
\end{prop}
\begin{proof}
This is a restatement of Proposition 5.2.14 of \cite{StreetMultiParamSingInt}.
\end{proof}

Thus Proposition \ref{PropKerEquivToProd} shows that, under Assumption \ref{AssumpKerProduct}, the kernels in $\sK_m$ are closely related to the standard
product kernels (at least if the coordinates of $m$ are not too negative).  See \cite{StreetMultiParamSingInt} for several generalizations of this type of result;
for example a similar result concerning flag kernels can be found in Proposition 4.2.22 and Lemma 4.2.24 of \cite{StreetMultiParamSingInt}.

%% file: rescc.tex
Before we can define the class of $\gamma$ for which we study operators of the form \eqref{EqnIntroMultiOp}, we must introduce the relevant definitions
and notation for Carnot-Carath\'eodory geometry.  
For further details on these topics, we refer the reader to \cite{StreetMultiParameterCCBalls}.

For this section, let $\Omega\subseteq \R^n$ be a fixed open set.  And let $\Gamma(T\Omega)$ denote
the space of smooth vector fields on $\Omega$.  Also, fix $\nu\in \N$, $\nu\geq 1$.

\begin{defn}
Let $X_1,\ldots, X_q$ be $C^\infty$ vector fields on $\Omega$.  We define the Carnot-Carath\'eodory ball of unit radius, centered at $x_0\in \Omega$,
with respect to the finite set $X=\{X_1,\ldots, X_q\}$ by
\begin{equation*}
\begin{split}
B_X(x_0):=\Bigg\{ y\in \Omega \:\Big|\:& \exists \gamma:[0,1]\rightarrow \Omega, \gamma(0)=x_0, \gamma(1)=y,
\\&\gamma'(t)=\sum_{j=1}^q a_j(t) X_j(\gamma(t)), a_j\in L^\infty([0,1]),
\\&\q\| \q( \sum_{1\leq j\leq q} |a_j|^2 \w)^{\frac{1}{2}} \w\|_{L^\infty([0,1])}<1
\Bigg\}.
\end{split}
\end{equation*}
In the above, we have written $\gamma'(t)=Z(t)$ to mean $\gamma(t)=\gamma(0)+\int_0^t Z(s)\: ds$.
\end{defn}

\begin{defn}\label{DefnResCCBall}
Let $X_1,\ldots, X_q$ be $C^\infty$ vector fields on $\Omega$, and to each vector field $X_j$ assign a multi-parameter formal degree $0\ne d_j\in [0,\infty)^\nu$.
Write $(X,d)=\{(X_1,d_1),\ldots, (X_q,d_q)\}$.  For $\delta\in [0,\infty)^\nu$ define the set of vector fields $\delta X:=\{\delta^{d_1} X_1,\ldots, \delta^{d_q} X_q \}$,
and for $x_0\in \Omega$, define the multi-parameter {\bf Carnot-Carath\'eodory ball}, centered at $x_0$ of ``radius'' $\delta$ by
\begin{equation*}
\B{X}{d}{x_0}{\delta}:=B_{\delta X}(x_0).
\end{equation*}
\end{defn}

Whenever we have a finite set of vector fields with $\nu$-parameter formal degrees:
$$(X,d)=\{(X_1,d_1),\ldots, (X_q,d_q)\}\subset \snuvect,$$ for $\delta\in [0,\infty)^\nu$
we write $\delta X$ to denote the set $\{\delta^{d_1}X_1,\ldots, \delta^{d_q} X_q\}$.  In addition, we identify this with an ordered list
$\delta X=(\delta^{d_1} X_1,\ldots, \delta^{d_q} X_q)$ (the particular order does not matter for our purposes).  In what follows, we use ordered multi-index notation.  If $\alpha$ is a list of elements
of $\{1,\ldots, q\}$, then we may define $(\delta X)^{\alpha}$, and we denote by $|\alpha|$ the length of the list.
For instance, $(\delta X)^{(1,3,2,1)}= \delta^{d_1} X_1 \delta^{d_3} X_3 \delta^{d_2} X_2 \delta^{d_1} X_1$ and $|(1,3,2,1)|=4$.

In what follows, let $\Omega'\Subset \Omega$ be an open relatively compact subset of $\Omega$,
$\nuh\in \N$ with $\nuh\geq 1$, and $\lambda$ a $\nuh\times \nu$ matrix whose entries are in $[0,\infty]$.

\begin{defn}
Let $(X,d)=\{(X_1,d_1),\ldots, (X_q,d_q)\}\subset \snuvect$ be a finite set consisting of $C^\infty$ vector fields on $\Omega$, $X_j$, each paired with a $\nu$-parameter formal degree $0\ne d_j\in [0,\infty)^\nu$.
Let $X_0$ be another $C^\infty$ vector field on $\Omega$, and let $h:[0,1]^{\nu}\rightarrow [0,1]$ be a function.  We say $(X,d)$ {\bf controls}
$(X_0,h)$ on $\Omega'$ if there exists $\tau_1>0$ such that for every $\delta\in [0,1]^\nu$, $x\in \Omega'$, there exists
$c_{x,j}^{\delta}\in C^0(\B{X}{d}{x}{\tau_1 \delta})$ ($1\leq j\leq q$) such that,
\begin{itemize}
\item $h(\delta) X_0 = \sum_{j=1}^q c_{x,j}^\delta \delta^{d_j} X_j, \text{ on }\B{X}{d}{x}{\tau_1\delta}.$
\item $\sup_{\substack{\delta\in [0,1]^\nu \\ x\in \Omega'}} \sum_{|\alpha|\leq m} \| (\delta X)^{\alpha} c_{x,j}^\delta \|_{C^0(\B{X}{d}{x}{\tau_1\delta} )}<\infty$, for every $m\in \N$.\footnote{For an arbitrary set $U\subseteq \R^n$, we define $\|f\|_{C^0(U)}=\sup_{y\in U} |f(y)|$, and if we say $\|f\|_{C^0(U)}<\infty$, we mean that this norm is finite and
that $f\big|_U:U\rightarrow \C$ is continuous.}
\end{itemize}
\end{defn}

\begin{defn}\label{DefnResCCControl}
Let $\sF\subset \snuvect$ be a finite set consisting of $C^\infty$ vector fields on $\Omega$, each paired with a $\nu$-parameter formal degree.
 Let $(\Xh,\hd)\in \snuhvect$ be another $C^\infty$
vector field with a $\nuh$-parameter formal degree.  We say $\sF$ $\lambda$\textbf{-controls} $(\Xh,\hd)$ on $\Omega'$ if the following holds.
Define a function $h_{\hd,\lambda}:[0,1]^{\nu}\rightarrow [0,1]$ by
\begin{equation}\label{EqnResCCDefnHlambdad}
h_{\hd,\lambda}(\delta) := \delta^{\lambda^{t}(\hd)}.
\end{equation}
Here we use the conventions that $\infty\cdot 0=0$ and $1^{\infty}=0$.
We assume $\sF$ controls $(\Xh,h_{\hd,\lambda})$ on $\Omega'$.
When $\nuh=\nu$, we say $\sF$ \textbf{controls} $(\Xh,\hd)$ on $\Omega'$ if $\sF$ $I$-controls $(\Xh,\hd)$ on $\Omega'$, where $I$ denotes the identity matrix.
\end{defn}

The notion of control is a natural one for our purposes.  In many examples, though, it arises from a stronger version, which we call smooth control.

\begin{defn}
Let $(X,d)=\{(X_1,d_1),\ldots, (X_q,d_q)\}\subset \snuvect$.
Let $X_0$ be another $C^\infty$ vector field on $\Omega$, and let $h:[0,1]^\nu\rightarrow [0,1]$.  We
say $(X,d)$ {\bf smoothly controls} $(X_0,h)$ on $\Omega'$ if there exists an open set $\Omega''$ with $\Omega'\Subset \Omega''\Subset \Omega$
such that for each $\delta\in [0,1]^\nu$ there exist functions $c_j^{\delta}\in C^\infty(\Omega'')$ ($1\leq j\leq q$) with
\begin{itemize}
\item $h(\delta) X_0=\sum_{j=1}^q c_j^\delta \delta^{d_j} X_j$, on $\Omega''$.
\item $\{c_j^\delta : \delta\in [0,1]^\nu, j\in \{1,\ldots,q\}\}\subset C^\infty(\Omega'')$ is a bounded set.
\end{itemize}
\end{defn}

\begin{defn}
Let $\sF\subset \snuvect$ be a finite set consisting of $C^\infty$ vector fields on $\Omega$, each paired with a $\nu$-parameter formal degree.
 Let $(\Xh,\hd)\in \snuhvect$ be another $C^\infty$
vector field with a $\nuh$-parameter formal degree.  We say $\sF$ \textbf{smoothly} $\lambda$\textbf{-controls} $(\Xh,\hd)$ on $\Omega'$ if
$\sF$ smoothly controls $(\Xh,h_{\hd,\lambda})$ on $\Omega'$, where $h_{\hd,\lambda}$ is as in \eqref{EqnResCCDefnHlambdad}.
When $\nuh=\nu$, we say $\sF$ \textbf{smoothly controls} $(\Xh,\hd)$ on $\Omega'$ if $\sF$ smoothly $I$-controls $(\Xh,\hd)$ on $\Omega'$,
where $I$ denotes the identity matrix.
\end{defn}


\begin{rmk}\label{RmkResCCSmoothControlControl}
It is clear that if $\sF$ smoothly $\lambda$-controls $(X_0,d_0)$ on $\Omega'$, then $\sF$ $\lambda$-controls $(X_0,d_0)$ on $\Omega'$.
The converse does not hold, even for $\lambda=I$ and $\nu=1$; see Example 5.15 of \cite{StreetMultiParameterCCBalls}.
\end{rmk}

\begin{defn}
Let $\sS\subseteq \snuvect$ be a possibly infinite set
and let $(\Xh,\hd)$ be another $C^\infty$ vector field on $\Omega$ paired with a $\nuh$-parameter formal degree $0\ne \hd\in [0,\infty)^{\nuh}$.
We say $\sS$ $\lambda$-{\bf controls} (resp. \textbf{smoothly} $\lambda$-\textbf{controls}) $(\Xh,\hd)$ on $\Omega'$ if there is a finite subset $\sF\subseteq \sS$ such that $\sF$ $\lambda$-controls (resp. smoothly $\lambda$-controls) $(\Xh,\hd)$ on $\Omega'$.
When $\nuh=\nu$, we say $\sS$ \textbf{controls} (resp. \textbf{smoothly controls}) $(\Xh,\hd)$ on $\Omega'$ if
$\sS$ $I$-controls (resp. smoothly $I$-controls) $(\Xh,\hd)$ on $\Omega'$, where $I$ denotes the identity matrix.
\end{defn}

\begin{defn}
Let $\sS\subseteq \snuvect$ and $\sT\subseteq\snuhvect$.
We say $\sS$ $\lambda$-{\bf controls} (resp. \textbf{smoothly} $\lambda$-\textbf{controls}) $\sT$ on $\Omega'$ if $\sS$ $\lambda$-controls (resp. smoothly $\lambda$-controls) $(\Xh,\hd)$ on $\Omega'$, $\forall (\Xh,\hd)\in \sT$.
When $\nuh=\nu$, we say $\sS$ \textbf{controls} (resp. \textbf{smoothly controls}) $\sT$ on $\Omega'$ if $\sS$ $I$-controls (resp. smoothly $I$-controls)
$\sT$ on $\Omega'$, where $I$ denotes the identity matrix.
\end{defn}

\begin{lemma}
The notion of control is transitive:  If $\sS_1$ controls $\sS_2$ on $\Omega'$ and $\sS_2$ controls $\sS_3$ on $\Omega'$, then $\sS_1$ controls $\sS_3$
on $\Omega'$.
\end{lemma}
\begin{proof}
This follows immediately from the definitions.
\end{proof}

\begin{defn}
Let $\sS\subseteq \snuvect$.
We say $\sS$ {\bf satisfies $\sD(\Omega')$} (resp. {\bf $\sD_s(\Omega')$}) if $\sS$ controls (resp. smoothly controls) $([X_1,X_2], d_1+d_2)$ on $\Omega'$, $\forall (X_1,d_1),(X_2,d_2)\in \sS$.
\end{defn}

\begin{defn}
Let $\sS, \sT\subseteq \snuvect$.
We say $\sS$ is {\bf equivalent} (resp. {\bf smoothly equivalent}) to $\sT$ on $\Omega'$ if $\sS$ controls (resp. smoothly controls) $\sT$ on $\Omega'$ and $\sT$ controls (resp. smoothly controls) $\sS$ on $\Omega'$.
\end{defn}

\begin{defn}
Let $\sS\subseteq \snuvect$.
We say $\sS$ is {\bf finitely generated} (resp. {\bf smoothly finitely generated}) on $\Omega'$ if there is a finite set $\sF\subset \snuvect$
such that $\sF$ is equivalent  (resp. smoothly equivalent) to $\sS$ on $\Omega'$.
If we want to make the choice of $\sF$ explicit, we say $\sS$ is finitely generated (resp. smoothly finitely generated) by $\sF$ on $\Omega'$.
\end{defn}

\begin{rmk}\label{RmkResCCCanTakesFSubset}
If $\sS$ is finitely generated (resp. smoothly finitely generated), one may always take $\sF\subseteq \sS$
such that $\sS$ is finitely generated (resp. smoothly finitely generated) by $\sF$.  However,
one need not take $\sF\subseteq \sS$.
\end{rmk}

\begin{rmk}\label{RmkResCCFiniteGenUnique}
Note that it is possible that $\sS$ be finitely generated (resp. smoothly finitely generated) on $\Omega'$ by two different finite sets
$\sF_1$, $\sF_2$, with $\sF_1\ne \sF_2$.  However, it is immediate from the definitions that $\sF_1$ and $\sF_2$
are equivalent (resp. smoothly equivalent) on $\Omega'$.  Because of this, it turns out that any two such choices will be equivalent for all of our purposes.  Thus, we may unambiguously say
$\sS$ is finitely generated (resp. smoothly finitely generated) by $\sF$ on $\Omega'$, where $\sF$ can be any such choice.
\end{rmk}

\begin{defn}
Let $\onecompnu\subset[0,\infty)^{\nu}$ denote the set of those $d\in [0,\infty)^{\nu}$ such that $d$ is nonzero in precisely one component.
\end{defn}

\begin{defn}
Let $\sS\subseteq \snuvect$.  We say $\sS$ is {\bf linearly finitely generated} (resp. {\bf smoothly linearly finitely generated}) on
$\Omega'$ if there is a finite set $\sF\subset \snuvectone$ such that $\sS$ is finitely generated (resp. smoothly finitely generated)
by $\sF$ on $\Omega'$.  If we wish to make the choice of $\sF$ explicit, we say
$\sS$ is {\bf linearly finitely generated} (resp. {\bf smoothly linearly finitely generated}) {\bf by} $\sF$ {\bf on} $\Omega'$.
\end{defn}

\begin{rmk}\label{RmkResCCnuOneFG}
Note that when $\nu=1$, $\sS$ is finitely generated (resp. smoothly finitely generated) if and only if
$\sS$ is linearly finitely generated (resp. smoothly linearly finitely generated).
\end{rmk}

\begin{defn}\label{DefnResCCsL}
Let $\sS\subseteq \snuvect$.
We define $\sL(\sS)\subseteq \snuvect$ to be the smallest set
such that:
\begin{itemize}
\item $\sS\subseteq \sL(\sS)$.
\item If $(X_1,d_1),(X_2,d_2)\in \sL(\sS)$, then $([X_1,X_2],d_1+d_2)\in \sL(\sS)$.
\end{itemize}
\end{defn}

\begin{rmk}
Note that $\sL(\sS)$ satisfies $\sD(\Omega')$, trivially.
\end{rmk}

\begin{lemma}\label{LemmaResCCsDandControl}
Let $\sS\subseteq \snuvect$.
$\sS$ satisfies $\sD(\Omega')$ (resp. $\sD_s(\Omega')$) if and only if $\sS$ controls (resp. smoothly controls) $\sL(\sS)$ on $\Omega'$.
\end{lemma}
\begin{proof}
This is immediate from the definitions.
\end{proof}

\begin{lemma}\label{LemmaResCCsLFinGenGivessD}
Let $\sS\subseteq \snuvect$.
If $\sL(\sS)$ is finitely generated (resp. smoothly finitely generated) by $\sF$ on $\Omega'$, then $\sF$ satisfies $\sD(\Omega')$ (resp. $\sD_s(\Omega')$).
\end{lemma}
\begin{proof}
This is immediate from the definitions.
\end{proof} 

\begin{example}\label{ExampleResCCFailureFGVect}
In the sequel we will be given a set $\sS\subseteq \snuvect$, and will be interested in whether or not $\sL(\sS)$ is finitely generated on $\Omega'$.
It is instructive to consider the following simple example where $\sL(\sS)$ is {\bf not} finitely generated.
Let $X_1=\frac{\partial}{\partial x}$, $X_2=e^{-\frac{1}{x^2}}\frac{\partial}{\partial y}$, 
and $\sS:=\{ (X_1,1), (X_2,1) \}\subset \Gamma(T\R^2)\times(0,\infty)$.
Let $\Omega'$ be a neighborhood of $0\in \R^2$.  Then, $\sL(\sS)$ is not finitely generated on $\Omega'$.
Indeed, any commutator of the form
\begin{equation*}
[X_1,[X_1,[X_1,\cdots,[X_1,X_2]\cdots]]],
\end{equation*}
is not spanned (with bounded coefficients) on any neighborhood of $0$ by commutators with fewer terms.
For nontrivial examples where $\sL(\sS)$ is finitely generated on $\Omega'$ see
Sections \ref{SectionResHorVect} and \ref{SectionResRealAnalVect}.
\end{example}

\begin{example}\label{ExampleResCCHormanderVF}
An important example where $\sL(\sS)$ is finitely generated but not linearly finitely generated comes from the Heisenberg group.
The Heisenberg group, $\Ho$, has a three dimensional Lie algebra spanned by vector fields $X,Y,T$, where $[X,Y]=T$ and
$T$ is in the center.  As a manifold $\Ho\cong \R^3$ and the vector fields $X,Y,T$ span the tangent space at every point.
Set $\sS := \{ (X,(1,0)), (Y,(0,1))\} \subset \Gamma(T\Ho)\times \mathfrak{d}_2$.
Then (on any non-empty open set $\Omega'\subset \Ho$) it is immediate to see that $\sL(\sS)$ is finitely generated but not linearly
finitely generated on $\Omega'$.
\end{example}

%% file: rescchor.tex
When some of the vector fields satisfy H\"ormander's condition, many of the above definitions become easier to verify.

\begin{defn}\label{DefnCCHormander}
Let $\sV$ be a collection of $C^\infty$ vector fields on $\Omega$, and let $U\subseteq \Omega$.  We say $\sV$ {\bf satisfies H\"ormander's condition} on
$U$ if
\begin{equation*}
\sV \cup [\sV,\sV]\cup [\sV,[\sV,\sV]]\cup\cdots
\end{equation*}
spans the tangent space to $\R^n$ at every point of $U$.
For $m\in \N$, we say $\sV$ {\bf satisfies H\"ormander's condition of order $m$ on} $U$ if
\begin{equation*}
\sV\cup [\sV,\sV]\cup\cdots\cup \overbrace{[\sV,[\sV,[\cdots, [\sV,\sV]]}^{m\text{ terms}}
\end{equation*}
spans the tangent space at every point of $U$.
\end{defn}

\begin{prop}\label{PropResCCHorSmoothFG}
Let $\sS\subset \snuvect$.
Suppose, for every $M$,
\begin{equation*}
\{ (X,d)\in \sS : |d|_{\infty}\leq M \}
\end{equation*}
is finite.  Also suppose for each $1\leq \mu\leq \nu$,
\begin{equation*}
\{ X : (X,d)\in \sS\text{ and }d_{\mu'}=0, \forall \mu'\ne \mu\}\text{ satisfies H\"ormander's condition on }\Omega.
\end{equation*}
Then $\sL(\sS)$ is smoothly finitely generated on $\Omega'$.
\end{prop}
\begin{proof}
Because $\Omega'$ is relatively compact in $\Omega$, for each $1\leq \mu\leq \nu$, there is a finite set
\begin{equation*}
\sF_{\mu}\subseteq \{  (X,d)\in \sL(\sS) : d_{\mu'}=0,\forall \mu'\ne \mu  \},
\end{equation*}
such that
\begin{equation*}
\{ X : \exists d, (X,d)\in \sF_\mu\}
\end{equation*}
spans the tangent space at every point on some neighborhood of the closure of $\Omega'$.
Let
\begin{equation*}
M:=\max\q\{ |d|_{\infty} : (X,d)\in \bigcup_{\mu=1}^\nu \sF_\mu\w\},
\end{equation*}
and define
\begin{equation*}
\sF:=\{(X,d)\in \sL(\sS) : |d|_{\infty}\leq M \}.
\end{equation*}
Note that $\sF$ is a finite set and $\sF_\mu\subseteq \sF$, for every $\mu$.
We claim that $\sF$ smoothly controls $\sL(\sS)$ on $\Omega'$.  Indeed, let $(X,d)\in \sL(\sS)$.
If $|d|_{\infty}\leq M$, then $(X,d)\in \sF$ and so $\sF$ smoothly controls $(X,d)$ on $\Omega'$, trivially.
If $|d|_{\infty}>M$, then there is some coordinate $\mu$ with $d_\mu>M$.  By the construction on $\sF_\mu$,
there is an open set $\Omega''$ with $\Omega'\Subset \Omega''\Subset \Omega$ and such that the vector fields in $\sF_\mu$
span the tangent space to every point of $\Omega''$.  Hence, we may write
\begin{equation*}
X = \sum_{(Y,e)\in \sF_\mu} c_{Y,e} Y,
\end{equation*}
where $c_{Y,e}\in C^\infty(\Omega'')$.
We have, for $\delta\in [0,1]^\nu$,
\begin{equation*}
\delta^{d} X = \sum_{(Y,e)\in \sF_\mu} (\delta^{d-e} c_{Y,e}) \delta^e Y.
\end{equation*}
Because $d_\mu>M$, and in the sum $e_\mu\leq M$ and $e_{\mu'}=0$ for $\mu\ne\mu'$, we have
$\{ \delta^{d-e} c_{Y,e} : \delta\in [0,1]^\nu\}\subset C^\infty(\Omega')$ is a bounded set.  Because $\sF_\mu\subseteq \sF$, the result follows.
\end{proof} 

%% file: resccanal.tex
Another situation where the concepts in this section become somewhat simpler is the case when the vector fields are real analytic.  Indeed, we have

\begin{prop}\label{PropResCCAnalSmoothFinGen}
Let $\sS\subset \snuvect$ be a finite set such that $\forall (X,d)\in \sS$, $X$ is real analytic.
Then, $\sL(\sS)$ is smoothly finitely generated on $\Omega'$.
\end{prop}
\begin{proof}
Let $\Omega''$ be an open set with $\Omega'\Subset \Omega'' \Subset \Omega$.  It follows from
Proposition 12.1 of \cite{SteinStreetIII} that $\forall x\in \Omega$,
there is a  finite set $\sF_x\subseteq \sL(\sS)$ such that $\forall (X,d)\in \sL(\sS)$, there is a neighborhood $U_{(X,d),x}$ of $x$
such that $\sF_x$ smoothly controls $(X,d)$ on $U_{(X,d),x}$.  We may assume $\sS\subseteq \sF_x$.
Define
\begin{equation*}
V_x := \bigcap_{(X_1,d_1),(X_2,d_2)\in \sF_x} U_{([X_1,X_2], d_1+d_2),x}.
\end{equation*}
We have $\sF_x$ satisfies $\sD_s(V_x)$.  Lemma \ref{LemmaResCCsDandControl} implies that
$\sF_x$ smoothly controls $\sL(\sF_x)$ on $V_x$, and therefore $\sF_x$ smoothly controls $\sL(\sS)$ on $V_x$.
$V_x$ forms an open cover of the closure of $\Omega''$, and we may extract a finite subcover $V_{x_1},\ldots, V_{x_M}$.
Set $\sF=\bigcup_{l=1}^M \sF_{x_l}$.  A simple partition of unity argument shows that $\sF$ smoothly controls $\sL(\sS)$ on $\Omega'$.
\end{proof}

%% file: ressurf.tex
In this section, we define the class of $\gamma$ for which we study operators of the form \eqref{EqnIntroMultiOp}.  We assume we
are given an open set $\Omega\subseteq \R^n$.  Fix open sets $\Omega'\Subset \Omega''\Subset \Omega'''\Subset \Omega$, 
where $A\Subset B$ denotes that $A$ is relatively compact in $B$.

\begin{defn}
A {\bf parameterization} is a triple $(\gamma,e,N)$, where $N\in \N$, $0\ne e_1,\ldots, e_N\in[0,\infty)^\nu$ define $\nu$-parameter dilations on $\R^N$,
and $\gamma(t,x)=\gamma_t(x):B^N(\rho)\times \Omega'''\rightarrow \Omega$ is a $C^\infty$ function satisfying $\gamma_0(x)\equiv x$.  Here, $\rho>0$
should be thought of as small.  If we want to make the choice of $\Omega,\Omega'''$ explicit, we write our parameterization as
$(\gamma,e,N,\Omega,\Omega''')$.
\end{defn}

Suppose we are given a parameterization $(\gamma,e,N)$.
Using the dilations, we assign to each multi-index $\alpha\in \N^N$ a corresponding degree,
$\deg(\alpha)\in [0,\infty)^{\nu}$, as in Definition \ref{DefnResKerDegree}.

When we are given a parameterization $(\gamma,e,N)$, because $\gamma_0$ is the identity map, we may always shrink $\rho$
so that $\gamma_t$ is a diffeomorphism onto its image for every $t\in B^N(\rho)$.  From now on we always do this, so that when we are given a parameterization,
it makes sense to thinking about $\gamma_t^{-1}(x)$, the inverse mapping.  Furthermore, we shrink $\rho>0$ so that $(\gamma_t^{-1},e,N,\Omega,\Omega'''')$
is a parameterization for some choice of $\Omega''''$ with $\Omega''\Subset \Omega''''\Subset \Omega$.

%
%

\begin{defn}
A {\bf vector field parameterization} is a triple $(W,e,N)$ where $N\in \N$, $0\ne e_1,\ldots, e_N\in [0,\infty)^\nu$ and $W(t,x)$ is a smooth vector field on $\Omega''$, depending smoothly on $t\in B^N(\rho)$ for some $\rho>0$, with $W(0,x)\equiv 0$.  If we want to make the choice of $\Omega''$ explicit, we write $(W,e,N, \Omega'')$.
\end{defn}

\begin{prop}\label{PropResSurfEquiv}
Parameterizations and vector field parameterizations are in bijective equivalence in the following sense:
\begin{itemize}
\item Given a parameterization $(\gamma,e,N,\Omega,\Omega''')$, we define a vector field by
\begin{equation}\label{EqnResSurfDefW}
W(t,x):=\frac{d}{d\epsilon}\bigg|_{\epsilon=1} \gamma_{\epsilon t}\circ \gamma_t^{-1}(x).
\end{equation}
For any $\Omega''\Subset \Omega'''$, if we take $\rho>0$ sufficiently small, then $(W,e,N,\Omega'')$
is a vector field parameterization on $\Omega''$.

\item Given a vector field parameterization $(W,e,N,\Omega'')$ and given $\Omega_1\Subset\Omega''$, if we take $\rho>0$ sufficiently small, there exists a unique
    $\gamma(t,x):B^N(\rho)\times \Omega_1\rightarrow \Omega''$ with $\gamma_0(x)\equiv x$ such that
    \begin{equation*}
W(t,x)=\frac{d}{d\epsilon}\bigg|_{\epsilon=1} \gamma_{\epsilon t}\circ \gamma_t^{-1}(x).
    \end{equation*}
    $(\gamma,e,N,\Omega'',\Omega_1)$ is the desired parameterization.
\end{itemize}
\end{prop}
\begin{proof}See Proposition 12.1 of \cite{SteinStreetI} for details.\end{proof}

\begin{rmk}
When we apply Proposition \ref{PropResSurfEquiv} to a vector field parameterization $(W,e,N)$,
we will take $\Omega_1$ so that $\Omega'\Subset \Omega_1$.  In this way, we may consider $\Omega'$
fixed throughout this entire paper, despite many application of Proposition \ref{PropResSurfEquiv}.
\end{rmk}

\begin{defn}
If $(\gamma,e,N)$ and $(W,e,N)$ are as in Proposition \ref{PropResSurfEquiv}, we say that
$(\gamma,e,N)$ {\bf corresponds} to $(W,e,N)$ and that $(W,e,N)$ {\bf corresponds} to $(\gamma,e,N)$.
\end{defn}

%
\begin{rmk}
If we are given a parameterization $(\gamma,e,N)$ or a vector field parameterization $(W,e,N)$,
we are (among other things) given a $\nu$-parameter dilation structure on $\R^N$.  In this case,
it makes sense to write $\delta t$ for $t\in \R^N$, $\delta\in[0,\infty)^\nu$ via \eqref{EqnResKerDefnDil}.
\end{rmk}

\begin{defn}
Let $(W,e,N,\Omega'')$ be a vector field parameterization
and let $$(X,d)=\{ (X_1,d_1),\ldots, (X_q,d_q)\} \subset \Gamma(T\Omega'')\times ([0,\infty)^{\nu}\setminus\{0\})$$
be a finite set consisting of $C^\infty$ vector fields on $\Omega''$, $X_j$, each paired with a $\nu$-parameter
formal degree $0\ne d_j\in [0,\infty)^\nu$.
We say $(X,d)$ {\bf controls} $(W,e,N)$ {\bf on} $\Omega'$ if there exists $\tau_1>0$, $\rho_1>0$ such that for every
$\delta\in[0,1]^\nu$, $x_0\in \Omega'$, there exist functions $c_l^{x_0,\delta}$ on $B^N(\rho_1)\times \B{X}{d}{x_0}{\tau_1\delta}$ satisfying:
\begin{itemize}
\item $W(\delta t, x) =\sum_{l=1}^q c_l^{x_0,\delta}(t,x) \delta^{d_l} X_l(x)$ on $B^{N}(\rho_1)\times \B{X}{d}{x_0}{\tau_1 \delta}$.
\item $\sup\limits_{\substack{x_0\in \Omega' \\ \delta\in [0,1]^\nu \\ t\in B^{N}(\rho_1)}} \sum_{|\alpha|+|\beta|\leq m } \q| (\delta X)^{\alpha} \partial_t^\beta c_l^{x_0,\delta}(t,x) \w|<\infty$, for every $m$.
\end{itemize}
\end{defn}

\begin{defn}
Let $(W,e,N,\Omega'')$ be a vector field parameterization
and let $$(X,d)=\{ (X_1,d_1),\ldots, (X_q,d_q)\} \subset \Gamma(T\Omega'')\times ([0,\infty)^{\nu}\setminus\{0\})$$
be a finite set consisting of $C^\infty$ vector fields on $\Omega''$, $X_j$, each paried with a $\nu$-parameter
formal degree $0\ne d_j\in [0,\infty)^\nu$.
We say $(X,d)$ {\bf smoothly controls} $(W,e,N)$ {\bf on} $\Omega'$ if there exists an open set $\Omega_1$
with $\Omega'\Subset\Omega_1\Subset\Omega''$ and $\rho_1>0$ such that for each $\delta\in[0,1]^\nu$ there exist
functions $c_j^\delta(t,x)\in C^\infty(B^N(\rho_1)\times\Omega_1)$ ($1\leq j\leq q$) with
\begin{itemize}
\item $W(\delta t, x) = \sum_{j=1}^q c_j^\delta(t,x) \delta^{d_j} X_j$, on $B^N(\rho_1)\times \Omega_1$.
\item $\{ c_j^\delta : \delta\in[0,1]^\nu ,j\in \{1,\ldots, q\} \}\subset C^\infty(B^{N}(\rho_1)\times \Omega_1)$ is a bounded set.
\end{itemize}
\end{defn}

\begin{rmk}
It is clear that if $(X,d)$ smoothly controls $(W,e,N)$, then $(X,d)$ controls $(W,e,N)$, though (as in Remark \ref{RmkResCCSmoothControlControl}) the converse does not hold.
\end{rmk}

\begin{defn}
Let $(W,e,N,\Omega'')$ be a vector field parameterization
and let $\sS\subseteq \Gamma(T\Omega'')\times ([0,\infty)^{\nu}\setminus\{0\})$.
We say $\sS$ {\bf controls} (resp. {\bf smoothly controls}) $(W,e,N)$ {\bf on} $\Omega'$ if there is a finite subset $\sF\subseteq \sS$
such that $\sF$ controls (resp. smoothly controls) $(W,e,N)$.
\end{defn}

\begin{defn}
Let $(\gamma,e,N,\Omega,\Omega''')$ be a parameterization and let $\sS\subseteq \Gamma(T\Omega'')\times ([0,\infty)^{\nu}\setminus\{0\})$.  We say $\sS$ {\bf controls} (resp. {\bf smoothly controls}) $(\gamma,e,N)$ {\bf on} $\Omega'$
if $\sS$ controls (resp. smoothly controls) the vector field parmeterization $(W,e,N)$, where $(W,e,N)$ corresponds to $(\gamma,e,N)$.
\end{defn}


Suppose $(W,e,N,\Omega'')$ is a vector field parameterization.  We can think of $W(t)$ as a smooth function
in the $t$ variable, taking values in smooth vector fields on $\Omega''$, satisfying $W(0)\equiv 0$.
We express $W$ as a Taylor series in the $t$ variable:
\begin{equation}\label{EqnResSurfWTaylor}
W(t)\sim \sum_{|\alpha|>0} t^{\alpha} X_\alpha,
\end{equation}
where $X_\alpha$ is a smooth vector field on $\Omega''$.
For the next definition, set
\begin{equation}\label{EqnResSurfDefnsS}
\sS:= \q\{ (X_\alpha, \deg(\alpha)) : \deg(\alpha)\in \onecompnu \w\}.
\end{equation}

\begin{defn}\label{DefnResSurfFinitelyGenVf}
Let $(W,e,N,\Omega'')$ be a vector field parameteriation and let $\sS$ be given by \eqref{EqnResSurfDefnsS}.  We say $(W,e,N)$ is {\bf finitely generated} (resp. {\bf smoothly finitely generated}) {\bf on} $\Omega'$
if
\begin{itemize}
\item $\sL(\sS)$ controls (resp. smoothly controls) $(W,e,N)$.
\item $\sL(\sS)$ is finitely generated (resp. smoothly finitely generated).
\end{itemize}
If $\sL(\sS)$ is finitely generated (resp. smoothly finitely generated) by $\sF$, 
and we wish to make this choice of $\sF$ explicit, we say $(W,e,N)$ is finitely generated (resp. smoothly finitely
generated) by $\sF$ on $\Omega'$.
\end{defn}

\begin{defn}
Let $(\gamma,e,N,\Omega,\Omega''')$ be a parameterization.  We say $(\gamma,e,N)$ is {\bf finitely generated} (resp. {\bf
smoothly finitely generated}) on $\Omega'$ if $(W,e,N)$ is finitely generated (resp. smoothly finitely generated) on $\Omega'$,
where $(\gamma,e,N)$ corresponds to $(W,e,N)$.  We say $(\gamma, e, N)$ is {\bf finitely generated} (resp. {\bf smoothly finitely generated}) by $\sF$ on $\Omega'$ if $(W,e,N)$ is finitely generated (resp. smoothly finitely generated) by $\sF$ on $\Omega'$.
\end{defn}

\begin{rmk}\label{RmkResSurfFiniteGenImpliessD}
Note that if $(\gamma,e,N)$ (or $(W,e,N)$) is finitely generated (resp. smoothly finitely generated) by $\sF$ on $\Omega'$,
then $\sF$ satisfies $\sD(\Omega')$ (resp. $\sD_s(\Omega')$)--see Lemma \ref{LemmaResCCsLFinGenGivessD}.
\end{rmk}

\begin{rmk}
If $(\gamma,e,N)$ (or $(W,e,N)$) is finitely generated (resp. smoothly finitely generated) on $\Omega'$, then there may be
many different choices of finite sets $\sF\subseteq \sL(\sS)$ such that $(\gamma,e,N)$ (or $(W,e,N)$)
is finitely generated (resp. smoothly finitely generated) by $\sF$ on $\Omega'$.  However, by Remark \ref{RmkResCCFiniteGenUnique}, any two such choices are equivalent (resp. smoothly equivalent) on $\Omega'$,
and are equivalent for all of our purposes.  Thus we may unambiguously say
$(\gamma,e,N)$ (or $(W,e,N)$) is finitely generated (resp. smoothly finitely generated) by $\sF$ on $\Omega'$,
where $\sF$ can be any such choice.
This choice of $\sF$ satisfies $\sD(\Omega')$ (resp. $\sD_s(\Omega)$)--see Remark \ref{RmkResSurfFiniteGenImpliessD}.
\end{rmk}

In Definition \ref{DefnResSurfFinitelyGenVf}, we factored the definition that $(W,e,N)$
be finitely generated into two aspects.  Sometimes it is easier to verify a slightly different characterization.
We continue to take $\sS$ as in \eqref{EqnResSurfDefnsS} and define
\begin{equation*}
\sV:=\q\{ (X_\alpha, \deg(\alpha)) : |\alpha|>0\w\}.
\end{equation*}
Note that $\sS\subseteq \sV$.

\begin{prop}\label{PropResSurfAnotherWay}
Let $(W,e,N,\Omega'')$ be a vector field parameterization.  Then $(W,e,N)$ is finitely generated (resp. smoothly
finitely generated) on $\Omega'$ if and only if
\begin{itemize}
\item $\sL(\sV)$ is finitely generated (resp. smoothly finitely generated) on $\Omega'$.
\item $\sL(\sV)$ controls (resp. smoothly controls) $(W,e,N)$ on $\Omega'$.
\item $\sL(\sS)$ controls (resp. smoothly controls) $\sV$ on $\Omega'$.
\end{itemize}
\end{prop}
\begin{proof}
Suppose the above three conditions hold.
Let $\sL(\sV)$ be finitely generated (resp. smoothly finitely generated) by $\sF$ on $\Omega'$.
Since $\sL(\sS)$ controls (resp. smoothly controls) $\sV$ on $\Omega'$, $\sL(\sS)$ controls (resp. smoothly controls) $\sL(\sV)$ on $\Omega'$, and 
therefore $\sL(\sS)$ controls (resp. smoothly controls) $\sF$ on $\Omega'$.  Since $\sF$ controls (resp. smoothly controls) $\sL(\sV)$ on $\Omega'$,
and $\sL(\sS)\subseteq \sL(\sV)$, $\sF$ controls (resp. smoothly controls) $\sL(\sS)$ on $\Omega'$.  Thus $\sL(\sS)$ and $\sF$
are equivalent (resp. smoothly equivalent) on $\Omega'$.  I.e., $\sL(\sS)$ is finitely generated (resp. smoothly finitely generated) by $\sF$ on $\Omega'$.
Since $\sL(\sV)$ controls (resp. smoothly controls) $(W,e,N)$ on $\Omega'$, and $\sF$ controls (resp. smoothly controls) $\sL(\sV)$ on $\Omega'$, it follows that $\sF$ controls (resp. smoothly controls) $(W,e,N)$ on $\Omega'$.  
Since $\sL(\sS)$ and $\sF$ are equivalent (resp. smoothly equivalent) on $\Omega'$, it follows that $\sL(\sS)$ controls (resp. smoothly controls) $(W,e,N)$
on $\Omega'$.
This shows that $(W,e,N)$ is finitely generated (resp. smoothly finitely generated)
by $\sF$ on $\Omega'$.

Conversely, suppose $(W,e,N)$ is fintiely generated (resp. smoothly finitely generated) by $\sF$ on $\Omega'$.  Becuase
$\sF$ controls (resp. smoothly controls) $(W,e,N)$ on $\Omega'$, it follows that $\sF$ controls (resp. smoothly controls) $\sV$ on $\Omega'$.  Because $\sL(\sS)$ and $\sF$ are equivalent (resp. smoothly equivalent)
on $\Omega'$, it follows that $\sL(\sS)$ controls (resp. smoothly controls) $\sV$ on $\Omega'$.  Thus, $\sL(\sS)$ controls (resp. smoothly controls) $\sL(\sV)$ on $\Omega'$.  Since $\sS\subseteq \sV$, $\sL(\sS)\subseteq \sL(\sV)$ and therefore $\sL(\sV)$ smoothly controls $\sL(\sS)$ on $\Omega'$.
Hence $\sL(\sS)$ and $\sL(\sV)$ are equivalent (resp. smoothly equivalent) on $\Omega'$.
Combining the above, we have $\sL(\sV)$ and $\sF$ are equivalent (resp. smoothly equivalent) on $\Omega'$.  Therefore,
$\sL(\sV)$ is finitelty generated (resp. smoothly finitely generated on $\Omega'$).  Futhermore, since $\sL(\sS)$ controls (resp. smoothly controls) $(W,e,N)$ on $\Omega'$,
by assumption, $\sL(\sV)$ controls (resp. smoothly controls) $(W,e,N)$ on $\Omega'$.  This completes the proof.
\end{proof}

\begin{example}
It is instructive to understand how a parameterization $(\gamma,e,N)$ can fail to be finitely generated.
When $\nu=1$, there are three main ways this can happen:
\begin{enumerate}[(i)]
\item If $\sS$ is given by \eqref{EqnResSurfDefnsS}, it could be that $\sL(\sS)$ fails to be finitely generated.
On $\R^2$ define dilations by multiplication: $\delta(s,t)= (\delta s, \delta t)$ (i.e., $N=2$ and $e_1=1$, $e_2=1$).
Let $X_1=\frac{\partial}{\partial x}$, $X_2=e^{-\frac{1}{x^2}}\frac{\partial}{\partial y}$, and $W(s,t)=sX_1+tX_2$.
Then $\sL(\sS)$ is not finitely generated on any open set containing $0$.
See Example \ref{ExampleResCCFailureFGVect}.

\item If $\sL(\sS)$ is finitely generated by $\sF$ on $\Omega'$, and if one sets $\sF_0:=\{X :(X,d)\in \sF\}$,
then for $X,Y\in \sF_0$, $[X,Y]$ is spanned by elements of $\sF_0$ (with appropriately nice coefficients).
The Frobenius theorem applies in this setting (see Section 2.2 of \cite{StreetMultiParamSingInt}) to foliate the ambient
space into leaves. Our assumptions imply that $\gamma_t(x)$ lies in the leaf passing through $x$.
This is not always the case.  For instance, if $\gamma_t(x):\R\times \R\rightarrow \R$ is given by
$\gamma_t(x)= x-e^{-\frac{1}{t^2}}$, then $X_\alpha=0$ $\forall \alpha$, and therefore the leaves
are points.  Thus, for $t\ne 0$, $\gamma_t(x)$ does not lie in the leaf passing through $x$.

\item Even if $\sL(\sS)$ is finitely generated, and $\gamma_t(x)$ lies in the appropriate leaf, it may still be that $\sL(\sS)$
does not control $\gamma_t(x)$.  Informally, this is because $\gamma_t(x)$ does not lie in the leaf
in an appropriately ``scale invariant'' way.  To create such an example we work on $\R$.  We define the vector field
$W(t,x)$ by
\begin{equation*}
W(t,x) = t e^{-\frac{1}{x^2}} \frac{\partial}{\partial x} + e^{-\frac{1}{t^2}} x \frac{\partial}{\partial x}.
\end{equation*}
Here $N=1$, $e_1=1$.  Let $(\gamma,1,1)$ be the parameterization corresponding to the vector field parameterization
$(W,1,1)$.  Note that
\begin{equation*}
\gamma_t(x)\text{ is }\begin{cases}
\text{negative,}&\text{if }x\text{ is negative,}\\
\text{zero,}&\text{if }x\text{ is zero,}\\
\text{positive}&\text{if }x\text{ is positive.}
\end{cases}
\end{equation*}
In this case, there is only one nonzero $X_\alpha$, namely $X_1=e^{-\frac{1}{x^2}}\frac{\partial}{\partial x}$.  Thus, the leaves
of the corresponding foliation are $(-\infty, 0)$, $\{0\}$, and $(0,\infty)$, and we see that $\gamma_t(x)$ does in fact
lie in the leaf passing through $x$.  Finally, we claim that $\gamma$ is not controlled by $\{ (X_1,1)\}$ on any neighborhood of $0\in \R$.
If $\gamma$ were controlled, it would imply, in particular, that there exists a $t_0\ne 0$ such that
for every $x$ near $0$, $e^{-\frac{1}{t_0^2}} x = c(x) e^{-\frac{1}{x^2}}$, with $c(x)$ bounded uniformly as $x\rightarrow 0$.
This is clearly impossible.

\end{enumerate}
The reader might note that all of the above examples used functions which vanished to infinite order.
This is necessary in the sense that when $\nu=1$ and $\gamma$ is real analytic,
$(\gamma,e,N)$ is automatically smoothly finitely generated; see Corollary \ref{CorResSurfAnalSingle}.
\end{example}
\begin{example}
When $\nu>1$, Proposition \ref{PropResSurfAnotherWay} highlights another way in which $(\gamma,e,N)$ could fail to be finitely
generated:  $\sL(\sS)$ could fail to control $\sV$ on $\Omega'$, and this can happen even if $\gamma$ is real analytic.
For instance, consider the the curve $\gamma_{(s,t)}(x) = x-st$.  Here we are using the dilation
$(\delta_1,\delta_2)(t,s)= (\delta_1 t,\delta_2 s)$.  Then every vector fields in $\sS$ is the zero vector field, however
$(\frac{\partial}{\partial x}, (1,1))\in \sV$, so $\sL(\sS)$ does not control $\sV$.  It is interesting
to note that there is a product kernel $K(s,t)$ such that the operator given by \eqref{EqnIntroCNSWOp},
with this choice of $\gamma$, is not even bounded on $L^2$.
This dates back to work of Nagel and Wainger \cite{NagelWaingerL2BoundednessOfHilbertTransformsMultiParameterGroup}.
See, also,  Section 17.5 of \cite{SteinStreetI}.
\end{example}

If $(\gamma,e,N)$ (or $(W,e,N)$) is finitely generated (resp. smoothly finitely generated) by $\sF$ on $\Omega'$, then $\sF$ satisfies $\sD(\Omega')$ (resp. $\sD_s(\Omega')$) by Lemma \ref{LemmaResCCsLFinGenGivessD}.  The next result addresses the extent to which the converse is true.

\begin{prop}\label{PropResSurfExistsAParam}
Suppose $\sS\subseteq \snuvectone$.  Suppose $\sL(\sS)$ is finitely generated by $\sF$ on $\Omega'$.  Then, there 
is a finite subset $\sS_0\subseteq \sS$ and
a vector field parameterization $(W,e,N,\Omega)$ such that $(W,e,N)$ is finitely generated by 
$\sF\cup \sS_0$ on $\Omega'$ (and
is finitely generated by 
$\sF$ on $\Omega'$).  If, in addition, $\sL(\sS)$ is smoothly finitely generated by $\sF$ on $\Omega'$, then $(W,e,N)$ is smoothly finitely generated by $\sF\cup \sS_0$ on $\Omega'$.
\end{prop}
\begin{proof}
Suppose $\sL(\sS)$ is finitely generated (resp. smoothly finitely generated) by $\sF$ on $\Omega'$.  
Then, there is a finite subset $\sS_0\subseteq\sS$ such that $\sL(\sS_0)$ is finitely generated (resp. smoothly finitely generated) by $\sF$ on $\Omega'$.
Setting
$\sF'=\sF\cup \sS_0$, we see that $\sL(\sS)$ is finitely generated (resp. smoothly finitely generated)
by $\sF'$ on $\Omega'$ and that $\sF'$ and $\sF$ are equivalent (resp. smoothly equivalent) on $\Omega'$.  Enumerate $\sF'$:
\begin{equation*}
\sF'=\{(X_1,d_1),\ldots, (X_q,d_q)\}.
\end{equation*}
Set $N=q$ and define $\nu$-parameter dilations on $\R^N$ by $e_j=d_j$.
Set
\begin{equation*}
W(t,x)=\sum_{j=1}^q t_j X_j.
\end{equation*}
Clearly, $W$ is smoothly controlled by $\sF'$ on $\Omega'$.
Furthermore, if we define $X_\alpha$ as in \eqref{EqnResSurfWTaylor} and $\sS'$
by \eqref{EqnResSurfDefnsS}, then we have $\sS_0\subseteq \sS'$, then therefore
$\sF'$ is controlled (resp. smoothly controlled) by $\sL(\sS')$ on $\Omega'$.  Since $\sF'$ clearly controls (resp. smoothly controls)
$\sL(\sS')$ on $\Omega'$, we see $W$ is finitely generated (resp. smoothly finitely generated) by $\sF'$ on $\Omega'$.
\end{proof}

\begin{rmk}\label{RmkResSurfPointExistsAParam}
The point of Proposition \ref{PropResSurfExistsAParam} is the following.  Suppose $\sS\subseteq \snuvectone$ is such that $\sL(\sS)$ is finitely generated
on $\Omega'$.  
Then, $\sL(\sS)$ ``comes from a parameterization'' in the following sense.
Applying Proposition \ref{PropResSurfExistsAParam}, we obtain a parameterization $(W,e,N,\Omega)$ which is finitely
generated on $\Omega'$ and such that
if we define $\sS'$ by 
\eqref{EqnResSurfDefnsS} with this choice of $W$, then $\sL(\sS)$ and $\sL(\sS')$ are equivalent on $\Omega'$.
I.e., $\sL(\sS)$ and $(W,e,N)$ are finitely generated by the same $\sF$ on $\Omega'$.
\end{rmk}

\begin{rmk}\label{RmkResSurfExistsParam}
The vector field parameterization exhibited in Proposition \ref{PropResSurfExistsAParam}
corresponds to a parameterization $(\gamma,e,N)$ via Proposition \ref{PropResSurfEquiv}.
In this case, $\gamma$ is easy to write down.  Indeed,
\begin{equation*}
\gamma_{(t_1,\ldots, t_q)}(x)  = e^{t_1 X_1+\cdots+t_q X_q} x.
\end{equation*}
\end{rmk} 

Some of our results can be strengthened if we assume that the parameterization involved is even better than finitely generated, which we now present.
\begin{defn}\label{DefnResSurfLinearlyFinitelyGenerated}
Let $(W,e,N,\Omega'')$ be a vector field parameterization.  We say $(W,e,N)$ is {\bf linearly finitely generated} (resp. {\bf smoothly linearly finitely generated}) on $\Omega'$
if 
\begin{itemize}
\item $\sF_0:=\{ (X_\alpha, \deg(\alpha)) : \deg(\alpha)\in \onecompnu\text{ and } |\alpha|=1\}$ satisfies $\sD(\Omega')$ (resp. $\sD_s(\Omega')$).
\item $\sF_0$ controls (resp. smoothly controls) $(W,e,N)$ on $\Omega'$. 
\end{itemize}
If $\sF\subset \snuvectone$ is another finite set such that $\sF_0$ and $\sF$ are equivalent (resp. smoothly equivalent) on $\Omega'$,
we say $(W,e,N)$ is {\bf linearly finitely generated} (resp. {\bf smoothly linearly finitely generated}) {\bf by} $\sF$ {\bf on} $\Omega'$.  (In particular, one can take $\sF=\sF_0$ in this case.)
\end{defn}

\begin{defn}
Let $(\gamma,e,N,\Omega,\Omega''')$ be a parameterization.  We say $(\gamma,e,N)$ is {\bf linearly finitely generated} (resp. {\bf
smoothly linearly finitely generated}) on $\Omega'$ if $(W,e,N)$ is linearly finitely generated (resp. smoothly linearly finitely generated) on $\Omega'$,
where $(\gamma,e,N)$ corresponds to $(W,e,N)$.  We say $(\gamma, e, N)$ is {\bf linearly finitely generated} (resp. {\bf smoothly linearly finitely generated}) {\bf by} $\sF$ {\bf on} $\Omega'$ if $(W,e,N)$ is linearly finitely generated (resp. smoothly linearly finitely generated) by $\sF$ on $\Omega'$.
\end{defn}

\begin{lemma}\label{LemmaResSurfLinFGImpliesFG}
If $(\gamma,e,N,\Omega,\Omega''')$ is linearly finitely generated (resp. smoothly linearly finitely generated) by $\sF$ on $\Omega'$, then $(\gamma,e,N,\Omega,\Omega''')$
is finitely generated (resp. smoothly finitely generated) by $\sF$ on $\Omega'$.
\end{lemma}
\begin{proof}
Suppose $(\gamma,e,N)$ is linearly finitely generated by $\sF$ on $\Omega'$.  We may take $\sF=\sF_0$,
where $\sF_0$ is as in Definition \ref{DefnResSurfLinearlyFinitelyGenerated}.
Let $(W,e,N)$ be the vector field parameterization corresponding to $(\gamma,e,N)$, and let $\sS$ be as in \eqref{EqnResSurfDefnsS}.  
By definition, $(W,e,N)$ is linearly finitely generated by $\sF$ on $\Omega'$, and therefore
$\sF$ controls $\sS$ on $\Omega'$.  Since $\sF$ satisfies $\sD(\Omega')$, $\sF$ controls $\sL(\sS)$ on $\Omega'$, and therefore
$\sL(\sS)$ is finitely generated on $\Omega'$.  Furthermore, since $\sF=\sF_0\subseteq \sS\subseteq \sL(\sS)$, $\sL(\sS)$ controls $\sF$ on $\Omega'$,
and therefore $\sL(\sS)$ controls $(W,e,N)$ on $\Omega'$.  This shows that $(W,e,N)$ and $(\gamma,e,N)$ are finitely generated by $\sF$
on $\Omega'$.  A similar proof works if $(\gamma,e,N)$ is smoothly linearly finitely generated to show $(\gamma,e,N)$ is smoothly finitely
generated.
\end{proof}

\begin{prop}\label{PropResSurfExistsLinFinGen}
Suppose $\sF\subset \snuvectone$ and satisfies $\sD(\Omega')$ (resp $\sD_s(\Omega')$).  Then, there exists a parameterization $(\gamma,e,N)$ such that
$(\gamma,e,N)$ is  linearly finitely generated (resp. smoothly linearly finitely generated) by $\sF$ on $\Omega'$.
\end{prop}
\begin{proof}
The example from Proposition \ref{PropResSurfExistsAParam} and Remark \ref{RmkResSurfExistsParam} works.
I.e., write $$\sF=\{(X_1,d_1),\ldots, (X_q,d_q)\}\subset \snuvectone,$$ and define dilations on $\R^q$ by $e_j=d_j$.
Then, $(\gamma,e,q$) satisfies the conclusions of the proposition, where
\begin{equation*}
\gamma_t(x) = e^{t_1 X_1+\cdots+ t_q X_q}x.
\end{equation*}
\end{proof}

\begin{example}
The most basic example of a linearly finitely generated parameterization is arises when $\nu=1$, $N=n$ and we use the standard dilations
$\delta (t_1,\ldots, t_n) = (\delta t_1,\ldots, \delta t_n)$; i.e., $e_1=\cdots=e_n=1$.  Take 
\begin{equation*}
\gamma_t(x) = x-t.
\end{equation*}
Then $(\gamma,(1,\ldots, 1), n)$ is linearly finitely generated.  In this special case, the operators we consider are just standard
pseudodifferential operators on $\R^n$.
Thus, the linearly finitely generated case will help us to generalize the setting of pseudodifferential operators
to a non-translation invariant, non-Euclidean setting.  See Section \ref{SectionRadonPDO} for further details on this.
\end{example}

\begin{example}
For $(\gamma,e,N)$ to be linearly finitely generated is a much stronger hypothesis than for $(\gamma,e,N)$ to be merely finitely generated,
even when $\nu=1$.  We present a few examples which help to elucidate the difference.
\begin{enumerate}[(i)]
\item When $N=n=2$, $\nu=1$, and we take the standard dilations $\delta(s,t) = (\delta s, \delta t)$, then if
\begin{equation*}
\gamma_{(s,t)}(x) = x-(s,t),
\end{equation*}
we have $(\gamma,(1,1), 2)$ is linearly finitely generated (on any open set in $\R^2$).  However, if
\begin{equation*}
\gammat_{(s,t)}(x)=x-(s,t^2),
\end{equation*}
$(\gammat,(1,1),2)$ is finitely generated but not linearly finitely generated.

\item When $N=n=1$ and we use multiparameter dilations $e_1=(1,1)$ (i.e., $(\delta_1,\delta_2)t = \delta_1\delta_2 t$),
then if
\begin{equation*}
\gamma_t(x) = x-t,
\end{equation*}
$(\gamma, ((1,1)), 2)$ is neither finitely generated nor linearly finitely generated (on any open subset of $\R$).

\item On the Heisenberg group $\Ho$ (see Example \ref{ExampleResCCHormanderVF}),
we define
\begin{equation*}
W((s,t)) = sX+tY.
\end{equation*}
If we use the standard dilations $e_1=1$, $e_2=1$, then $(W,e,2)$ is finitely generated, but not linearly finitely generated:
$(X,1)$ and $(Y,1)$ do not control $([X,Y],2)=(T,2)$.  Here, if $X$ and $Y$ are taken to be right invariant vector fields,
\begin{equation*}
\gamma_{s,t}(x) = (s,t,0)x,
\end{equation*}
where $x\in \Ho$, and $(s,t,0)x$ denotes group multiplication.  See \cite{SteinHarmonicAnalysis} for an exposition of the Heisenberg group.

\item As in the previous example, we use the Heisenberg group $\Ho$, but now take $N=3$
and define $e_1=1$, $e_2=1$, $e_3=2$. If we define
\begin{equation*}
W((t_1,t_2,t_3))= t_1 X+ t_2 Y + t_3 T,
\end{equation*}
then $(W,(1,1,2),3)$ is linearly finitely generated.  Here,
\begin{equation*}
\gamma_{t_1,t_2,t_3}(x) = (t_1,t_2,t_3)x,
\end{equation*}
where, again, $(t_1,t_2,t_3)x$ denotes group multiplication.

\item If we take $\gamma_{t_1,t_2,t_3}(x)$ as in the previous example, but use multiparameter dilations 
$e_1=(1,0)$, $e_2=(0,1)$, and $e_3=(1,1)$, then $(\gamma, e, 3)$ is finitely generated but not linearly finitely generated.
\end{enumerate}
\end{example}

%% file: ressurfhor.tex
One situation where Definition \ref{DefnResSurfFinitelyGenVf} is particularly easy to verify is when some of the
vector fields satisfy H\"ormander's condition (see Definition \ref{DefnCCHormander}).  Let $(W,e,N,\Omega'')$ be
a vector field parameterization with $\nu$ parameter dilations.  As in \eqref{EqnResSurfWTaylor}, let $W(t)\sim \sum_{|\alpha|>0} t^{\alpha} X_\alpha$, so that $X_\alpha$ is a smooth vector field on $\Omega''$.
Let $\sS$ be as in \eqref{EqnResSurfDefnsS}.

\begin{prop}\label{PropResSurfHorMainProp}
Suppose for each $1\leq \mu\leq \nu$,
\begin{equation}\label{EqnResSurfHorMainHyp}
\begin{split}
\{ X_\alpha : &\deg(\alpha)\text{ is nonzero in only the }\mu\text{ component}\} 
\\&\text{ satisfies H\"ormander's condition on }\Omega'',
\end{split}
\end{equation}
and suppose $\sL(\sS)$ controls (resp. smoothly controls) $(X_\alpha, \deg(\alpha))$ on $\Omega'$ for every $\alpha$.
Then $W$ is finitely generated (resp. smoothly finitely generated) on $\Omega'$.
\end{prop}
\begin{proof}
As before, let $\sV:=\{(X_\alpha, \deg(\alpha)): |\alpha|>0\}$.  Note that $\sS\subseteq \sV$.
By Proposition \ref{PropResCCHorSmoothFG}, $\sL(\sV)$ is smoothly finitely generated on $\Omega'$.
In light of Proposition \ref{PropResSurfAnotherWay}, the proof will be complete if we show
$\sL(\sV)$  smoothly controls $(W,e,N)$ on $\Omega'$.

Fix an open set $\Omega_1$ with $\Omega'\Subset \Omega_1\Subset \Omega''$.
Because $\Omega_1$ is relatively compact in $\Omega''$, and because of \eqref{EqnResSurfHorMainHyp}, for each $1\leq \mu\leq \nu$, there is a finite set
\begin{equation*}
\sF_\mu \subseteq \{ (X,d)\in \sL(\sV) : d_{\mu'}=0, \forall \mu'\ne \mu \}
\end{equation*}
such that
\begin{equation*}
\{ X: \exists d, (X,d)\in \sF_\mu\}
\end{equation*}
spans the tangent space at every point on some neighborhood of the closure of $\Omega_1$.
Let
\begin{equation*}
M:=\max \q\{ |d|_{\infty} : (X,d)\in \bigcup_{\mu=1}^\nu \sF_\mu \w\},
\end{equation*}
and define
\begin{equation*}
\sF:=\{ (X,d)\in \sL(\sV) : |d|_{\infty}\leq M \}.
\end{equation*}
Note that $\sF$ is a finite set and $\sF_\mu\subseteq \sF$, for every $\mu$.  We claim
$\sF$ smoothly controls $(W,e,N)$ on $\Omega'$.

By using the Taylor series of $W$, there is a finite set of multi-indices $\sA\subset \N^N$ such that
$\forall \alpha \in \sA$, $|\deg(\alpha)|_\infty>M$, and such that we may write
\begin{equation*}
W(t,x)=\sum_{|\deg(\alpha)|_{\infty}\leq M} t^{\alpha} X_\alpha + \sum_{\alpha\in \sA} t^{\alpha} W_{\alpha}(t,x),
\end{equation*}
where $W_{\alpha}(t)$ is a smooth vector field on $\Omega''$, depending smoothly on $t$.
It is clear that for $|\deg(\alpha)|_{\infty}\leq M$, $t^{\alpha} X_\alpha$ is smoothly controlled
by $\sF$ on $\Omega'$, as $(X_\alpha,\deg(\alpha))\in \sF$ by construction.  The proof will be complete if we
show, for $\alpha\in \sA$, $t^{\alpha} W_{\alpha}(t)$ is smoothly controlled by $\sF$ on $\Omega'$.

Fix $\alpha\in \sA$, since $|\deg(\alpha)|_{\infty}>M$, there is a $\mu\in \nuset$ such that
$\deg(\alpha)_{\mu}>M$.  Using that $\sF_\mu$ spans the tangent space at every point of
$\Omega_1$, we may write
\begin{equation*}
t^{\alpha} W_\alpha(t,x) = \sum_{(Z,d)\in \sF_\mu} t^{\alpha} c_Z(t,x) Z(x),
\end{equation*}
where $c_Z\in C^\infty(B^N(\rho)\times \Omega_1)$ (where the domain of $W(t)$ in the $t$ variable is
$B^N(\rho)$).
Thus, we have, for $\delta\in [0,1]^\nu$,
\begin{equation*}
(\delta t)^{\alpha} W_\alpha(\delta t, x)= \sum_{(Z,d)\in \sF_\mu} t^{\alpha} \delta^{\deg(\alpha) - d} c_Z(\delta t, x) \delta^d Z(x).
\end{equation*}
Using that for $(Z,d)\in \sF_\mu$, $\deg(d)$ is nonzero in only the $\mu$ component, and $\deg(d)_\mu\leq M$, we see that $\deg(\alpha)-d$ is nonnegative in every component.  It follows that
$$\{ t^{\alpha} \delta^{\deg(\alpha)-d} c_Z(\delta t, x) : \delta\in [0,1]^\nu, (Z,d)\in \sF_\mu\}\subset C^\infty(B^N(\rho)\times \Omega_1)$$
is a bounded set.  Hence, $t^{\alpha} W_\alpha(t)$ is smoothly controlled by $\sF_\mu$ on $\Omega'$, and therefore
it is smoothly controlled
by $\sF$ on $\Omega'$.  This completes the proof.
\end{proof}

\begin{cor}\label{CorResSurfHorSmoothlyFG}
Suppose that $\nu=1$ and that
\begin{equation*}
\{X_\alpha : |\alpha|>0\}\text{ satisfies H\"ormander's condition on }\Omega''.
\end{equation*}
Then $W$ is smoothly finitely generated on $\Omega'$.
\end{cor}
\begin{proof}
In this case, $(X_\alpha, \deg(\alpha))\in \sS$ for every $\alpha$, and therefore $\sL(\sS)$ smoothly
controls $(X_\alpha,\deg(\alpha))$ for every $\alpha$, trivially.  The result follows from
Proposition \ref{PropResSurfHorMainProp}.
\end{proof} 

%% file: ressurfanal.tex
Another situation where Definition \ref{DefnResSurfFinitelyGenVf} is easy to verify is when the vector field
$W$ is real analytic.  Indeed, let $(W,e,N,\Omega'')$ be a vector field parameterization with $\nu$ parameter dilations
and with $W$
real analytic.  We write $W$ as a power series in the $t$ variable, so that for $t$ small,
$W(t,x)=\sum_{|\alpha|>0} t^{\alpha} X_\alpha$, where $X_\alpha$ is a real analytic
vector field on $\Omega''$ .  Let $\sS$ be as in \eqref{EqnResSurfDefnsS}.

\begin{prop}\label{PropResSurfRealAnalMain}
Suppose $\sL(\sS)$ controls (resp. smoothly controls) $(X_\alpha,\deg(\alpha))$ on $\Omega'$ for every $\alpha$.
Then $W$ is finitely generated (resp. smoothly finitely generated) on $\Omega'$.
\end{prop}
\begin{proof}
As before, let $\sV:=\{(X_\alpha, \deg(\alpha): |\alpha|>0\}$.  Note that $\sS\subseteq \sV$.
Theorem 9.1 of \cite{SteinStreetIII} shows for each $x_0\in \Omega''$, exists a neighborhood
$U_{x_0}$ containing $x_0$ and a finite set $F_{x_0}\subset \N^N$ such that
\begin{equation*}
W(t,x)= \sum_{\alpha\in F_{x_0}} c_{\alpha}^{x_0}(t,x) t^{\alpha} X_{\alpha}(x),  \text{ on }U_{x_0},
\end{equation*}
where $c_{\alpha}^{x_0}(t,x):B^N(\rho_{x_0})\times U_{x_0}\rightarrow \R$ is real analytic, and $\rho_{x_0}>0$.

$U_{x_0}$ forms a cover of the closure of $\Omega'$, which is a compact set.  Extract a finite subcover,
$U_{x_1},\ldots, U_{x_M}$, and set
\begin{equation*}
F= \bigcup_{l=1}^M F_{x_l}, \quad \rho_0:=\min\q\{\rho_{x_l}: 1\leq l\leq M\w\}.
\end{equation*}
A partition of unity argument shows that there is an open set $\Omega_1$ with
$\Omega'\Subset \Omega_1\Subset \Omega'''$ and with
\begin{equation}\label{EqnResSurfAnalFirstControl}
W(t,x)= \sum_{\alpha\in F} c_{\alpha}(t,x) t^{\alpha} X_{\alpha}(x),  \text{ on }\Omega_1,
\end{equation}
with $c_{\alpha}\in C^\infty( B^N(\rho_0)\times \Omega_1)$.  Set $\sF:=\{(X_\alpha, \deg(\alpha)) : \alpha\in F\}$.  \eqref{EqnResSurfAnalFirstControl} shows that $\sF$ smoothly controls $(W,e,N)$ on $\Omega'$,
and therefore $\sL(\sV)$ smoothly controls $(W,e,N)$ on $\Omega'$.

Furthermore, because
\begin{equation*}
X_\beta = \frac{1}{\beta!} \partial_t^{\beta} W(t)\bigg|_{t=0},
\end{equation*}
\eqref{EqnResSurfAnalFirstControl} shows that $X_\beta$ is a $C^\infty(\Omega_1)$ linear
combination of $\{ X_{\alpha} : \alpha\in F,\alpha\leq \beta\}$ where $\alpha\leq \beta$ means the inequality
holds coordinatewise.  Thus, $X_\beta$ is a $C^\infty(\Omega_1)$ linear
combination of $\{ X_{\alpha} : \alpha\in F,\deg(\alpha)\leq \deg(\beta)\}$,
and it follows that $\sF$ smoothly controls $(X_\beta, \deg(\beta))$ on $\Omega'$, $\forall \beta$.
I.e., $\sF$ smoothly controls $\sV$ on $\Omega'$.  Thus, $\sL(\sF)$ smoothly controls $\sL(\sV)$
on $\Omega'$.  Because $\sF\subseteq \sV$, this shows that $\sL(\sF)$ and $\sL(\sV)$ are smoothly 
equivalent.  Proposition \ref{PropResCCAnalSmoothFinGen} shows that
$\sL(\sF)$ is smoothly finitely generated on $\Omega'$, and therefore $\sL(\sV)$ is smoothly finitely generated on $\Omega'$.

The result now follows by combining the above with Proposition \ref{PropResSurfAnotherWay}.
\end{proof}

\begin{cor}\label{CorResSurfAnalSingle}
When $\nu=1$ and when $W$ is real analytic, then $W$ is smoothly finitely generated on $\Omega'$.
\end{cor}
\begin{proof}
In this case $(X_\alpha, \deg(\alpha))\in \sS$, $\forall \alpha$, so the conditions
of Proposition \ref{PropResSurfRealAnalMain} hold automatically.
\end{proof}

%% file: ressob2.tex
Let $\Omega\subseteq \R^n$ be an open set, and fix open sets $\Omega_0\Subset \Omega'\Subset\Omega$.  Let
$\sS\subset \snuvectone$.  We wish to define Sobolev spaces where for each $(X,d)\in \sS$, $X$ is viewed
as a differential operator of ``order'' $0\ne d\in [0,\infty)^{\nu}$.  We restrict attention to functions supported in
$\Omega_0$.  There are two main assumptions that we deal with:

\begin{enumerate}[\bf{Case} I:]
\item $\sL(\sS)$ is finitely generated on $\Omega'$.
\item $\sL(\sS)$ is linearly finitely generated on $\Omega'$.
\end{enumerate}

Notice that Case II implies Case I, 
and our results in Case II will be stronger than in Case I.
In Case II, for $1<p<\infty$ and $\delta\in \R^{\nu}$ we define non-isotropic Sobolev spaces
consisting of functions supported in $\Omega_0$, denoted by $\NL{p}{\delta}$.  In Case I,
we do the same, but must restrict to $|\delta|$ small in a way which is made precise in what follows.

In what follows, we make several choices in defining the norm which induces the space $\NL{p}{\delta}$.  Different
choices yields comparable norms:  for all $\delta$ in Case II, and for $|\delta|$ small in Case I.  In Case I, how small
$|\delta|$ needs to be depends on the various choices made.  See Theorem \ref{ThmResSobWellDefNew} where this is made precise.

\begin{defn}
An ordered list $\fD=\q(\nu,(\gamma,e,N,\Omega,\Omega'''), a, \eta,\{\vsig_j\}_{j\in \N^{\nu}}, \psi\w)$ is called
{\bf Sobolev data on} $\Omega'$ if:
\begin{itemize}
\item $0\ne \nu\in \N$.
\item $\Omega$, $\Omega'$, and $\Omega'''$ are open with $\Omega'\Subset \Omega'''\Subset \Omega\subseteq \R^n$.
\item $(\gamma,e,N,\Omega,\Omega''')$ is a parameterization, with $\nu$-parameter dilations $$0\ne e_1,\ldots, e_N\in [0,\infty)^{\nu}.$$
Here, $\gamma(t,x):B^{N}(\rho)\times \Omega'''\rightarrow \Omega$, for some $\rho>0$.
\item $0<a\leq \rho$ is a small number (how small $a$ must be depends on $\gamma$, and will be detailed later).
\item $(\gamma,e,N,\Omega,\Omega''')$ is finitely generated on $\Omega'$.
\item $\eta\in C_0^{\infty}(B^N(a))$ and $\{\vsig_j \}_{j\in \N^{\nu}}\subset \schS(\R^N)$ is a bounded set with
$\vsig_j\in \schS_{\{\mu : j_\mu\ne 0\}}$ and satisfies $\delta_0(t)=\eta(t)\sum_{j\in \N^{\nu} } \dil{\vsig_j}{2^j}(t)$.  Here, 
$\dil{\vsig_j}{2^j}$ is defined by the dilations $e$--see \eqref{EqnResKerDefFuncDil}.  Note that such a choice of $\eta$ and $\vsig_j$ always exists by Lemma \ref{LemmaResKerDelta}.
\item $\psi\in C_0^\infty(\Omega')$ with $\psi\equiv 1$ on a neighborhood of the closure of $\Omega_0$.
\end{itemize}
We say $\fD$ is {\bf finitely generated by} $\sF$ {\bf on} $\Omega'$ if $(\gamma,e,N,\Omega,\Omega''')$ is finitely generated by $\sF$ on $\Omega'$.
We say $\fD$ is {\bf linearly finitely generated on} $\Omega'$ if $(\gamma,e,N, \Omega,\Omega''')$ is linearly finitely generated on $\Omega'$, and we say 
$\fD$ is {\bf linearly finitely generated by} $\sF$ {\bf on} $\Omega'$ if $(\gamma,e,N,\Omega,\Omega''')$ is linearly finitely generated by $\sF$ on $\Omega'$.
\end{defn}

Given Sobolev data $\fD=(\nu,(\gamma,e,N,\Omega,\Omega'''),a,\eta,\{\vsig_j : j\in \N^\nu\},\psi )$,
define $D_j=D_j(\fD)$, for $j\in \N^\nu$, by
\begin{equation}\label{EqnResSobDefnDjNew}
D_j f(x) = \psi(x) \int f(\gamma_t(x)) \psi(\gamma_t(x))\eta(t) \dil{\vsig_j}{2^j}(t)\: dt.
\end{equation}
Note that $\sum_{j\in \N^\nu} D_j f = \psi^2 f$; in particular, if $\supp{f}\subset \Omega_0$,
$\sum_{j\in \N^\nu} D_j f =f$.

\begin{defn}
Given Sobolev data $\fD$, for $1<p<\infty$, $\delta\in \R^\nu$, we define (for $f\in C_0^\infty(\Omega_0)$),
\begin{equation*}
\NLpN{p}{\delta}{\fD}{f}:= \LpN{p}{ \q(  \sum_{j\in \N^\nu} \q| 2^{j\cdot \delta} D_j f \w|^2 \w)^{\frac{1}{2}}  },
\end{equation*}
where $D_j=D_j(\fD)$ is defined in \eqref{EqnResSobDefnDjNew}.
We define the Banach space $\NLp{p}{\delta}{\fD}$ to be the closure of $C_0^\infty(\Omega_0)$ in the norm
$\NLpN{p}{\delta}{\fD}{\cdot}$.
\end{defn}

If $\fD$ is finitely generated by $\sF$ on $\Omega'$, then the next result shows that the equivalence class of $\NLpN{p}{\delta}{\fD}{\cdot}$ depends
only on $\sF$, for $|\delta|$ sufficiently small.  If $\fD$ is linearly finitely generated by $\sF$ on $\Omega'$, it shows that the equivalence
class depends only on $\sF$ for all $\delta\in \R^{\nu}$.

\begin{thm}\label{ThmResSobWellDefNew}
Let 
\begin{equation*}
\begin{split}
&\fD=(\nu,(\gamma,e,N,\Omega,\Omega'''),a,\eta,\{\vsig_j\}_{j\in \N^{\nu}}, \psi) \text{ and }
\\& \fDt=(\nu,(\gammat,\et,\Nt,\Omega,\Omegat'''),\at,\etat,\{\vsigt_j\}_{j\in \N^{\nu}}, \psit)
\end{split}
\end{equation*}
both be Sobolev data.  
If $a>0$ and $\at>0$ are chosen sufficiently small (depending on the parameterizations) we have:
\begin{enumerate}[(I)]
\item If $\fD$ and $\fDt$ are both finitely generated by the same $\sF$ on $\Omega'$, then for $1<p<\infty$, $\exists  \epsilon=\epsilon(p,(\gamma,e,N),(\gammat,\et,\Nt))>0$
such that for $\delta\in \R^{\nu}$ with $|\delta|<\epsilon$,
\begin{equation*}
\NLpN{p}{\delta}{\fD}{f}\approx \NLpN{p}{\delta}{\fDt}{f},\quad \forall f\in C_0^{\infty}(\Omega_0).
\end{equation*}
Here, the implicit constants depend on $\fD$, $\fDt$, and $p$.

\item If $\fD$ and $\fDt$ are both linearly finitely generated by the same $\sF$ on $\Omega'$, then for $1<p<\infty$, $\delta\in \R^{\nu}$,
\begin{equation*}
\NLpN{p}{\delta}{\fD}{f}\approx \NLpN{p}{\delta}{\fDt}{f}, \quad \forall f\in C_0^{\infty}(\Omega_0).
\end{equation*}
Here, the implicit constants depend on $\fD$, $\fDt$, $p$, and $\delta$.
\end{enumerate}
\end{thm}
\begin{proof}The proof of this result is completed in Section \ref{SectionPfSob}.\end{proof}

Theorem \ref{ThmResSobWellDefNew} implies a few properties of the norm $\NLpN{p}{\delta}{\fD}{\cdot}$.
When $\fD$ is merely finitely generated,\footnote{Recall, it is part of the definition of Sobolev data the $\fD$ be finitely generated.} it shows that the equivalence class of the norm $\NLpN{p}{\delta}{\fD}{\cdot}$ does not depend on 
the choices of $a$, $\eta$, $\{\vsig_j\}$, and $\psi$, for $|\delta|<\epsilon$ for some $\epsilon=\epsilon(p,(\gamma,e,N))>0$.
When $\sD$ is linearly finitely generated, it shows that the equivalence class does not depend on 
the choices of $a$, $\eta$, $\{\vsig_j\}$, and $\psi$, for any $\delta\in \R^{\nu}$.
We are led to the following definition.

\begin{defn}\label{DefnResSobForGamma}
Let $(\gamma,e,N,\Omega,\Omega''')$ (with $\Omega'\Subset \Omega'''$) be a parameterization
which is finitely generated on $\Omega'$.  For $f\in C_0^{\infty}(\Omega_0)$,
we write $\NLpN{p}{\delta}{\gamma,e,N}{f}$ to denote $\NLpN{p}{\delta}{\fD}{f}$, where $\fD$ can be any Sobolev data of the form
$$\fD=(\nu,(\gamma,e,N,\Omega,\Omega'''),a,\eta,\{\vsig_j\}_{j\in \N^{\nu}}, \psi),$$
with $a>0$ small.
By Theorem \ref{ThmResSobWellDefNew}, equivalence class of $\NLpN{p}{\delta}{\gamma,e,N}{\cdot}$
is well-defined for $|\delta|<\epsilon$, for some $\epsilon=\epsilon(p,(\gamma,e,N))>0$.
If, in addition, $(\gamma,e,N)$ is linearly finitely generated on $\Omega'$, then the equivalence
class of $\NLpN{p}{\delta}{\gamma,e,N}{\cdot}$ is well-defined for all $\delta\in \R^{\nu}$.
\end{defn}

Now consider the setting at the start of this section.  We are given a finite set $\sS\subset \snuvectone$.  We assume either
$\sL(\sS)$ is finitely generated by some $\sF\subset \snuvect$ on $\Omega'$ (Case I), 
or $\sL(\sS)$ is linearly finitely generated by some $\sF\subset \snuvectone$ on $\Omega'$ (Case II).
In Case I, Proposition \ref{PropResSurfExistsAParam} and Remark \ref{RmkResSurfPointExistsAParam} show that there is a parameterization $(\gamma,e,N)$ such that
$(\gamma,e,N)$ is finitely generated by by $\sF$ on $\Omega'$.  In Case II, Proposition \ref{PropResSurfExistsLinFinGen} shows that
there is a parameterization $(\gamma,e,N)$ such that $(\gamma,e,N)$ is linearly finitely generated by $\sF$ on $\Omega'$.
Theorem \ref{ThmResSobWellDefNew} shows that, given $\sS$, any two such choices of $(\gamma,e,N)$ and $(\gammat,\et,\Nt)$
yield comparable norms for all $\delta$ in Case II--when $(\gamma,e,N)$ and $(\gammat,\et,\Nt)$ are both linearly
finitely generated by $\sF$ on $\Omega'$, and in Case I for all $|\delta|$ sufficiently small (depending on $p$ and the choices of $(\gamma,e,N)$ and $(\gammat,\et,\Nt)$)--when
$(\gamma,e,N)$ and $(\gammat,\et,\Nt$) are both finitely generated by $\sF$ on $\Omega'$.\footnote{Note that the choice
of $\sF$ which finitely generates (resp. linearly finitely generates) $\sL(\sS)$ is irrelevant--any two such choices
are equiavalent on $\Omega'$.}
In Case II, this means that the equivalence class of $\NLpN{p}{\delta}{\gamma,e,N}{\cdot}$ depends only on $\sS$, $\forall \delta\in \R^{\nu}$.
In Case I, this implies that, when thought of as a germ of a function near $0$ in the $\delta$ variable, the equivalence class of the norm
$\NLpN{p}{\delta}{\gamma,e,N}{\cdot}$ depends only on $\sS$.

\begin{defn}
For $\nu\in \N$, we write $f:\R^{\nu}_0\rightarrow \R$ to denote that $f$ is a germ of a function defined near $0\in \R^{\nu}$.
If we write $x\in \R^{\nu}_0$, we mean that $x$ is a variable defined on as small a neighborhood of $0$ as necessary for the application.
Thus, it makes sense to write $f(x)$, for $x\in \R^{\nu}_0$, for $f:\R^{\nu}_0\rightarrow \R$.
\end{defn}

\begin{defn}\label{DefnResSobDependOnsS}
Suppose $\sS\subset \snuvectone$.
\begin{enumerate}[(I)]
\item If $\sL(\sS)$ is finitely generated on $\Omega'$, then for $1<p<\infty$ and $\delta\in \R^{\nu}_0$, we write $\NLpN{p}{\delta}{\sS}{f}:=\NLpN{p}{\delta}{\gamma,e,N}{f}$ for $f\in C_0^{\infty}(\Omega')$.  The equivalence class of the norm $\NLpN{p}{\delta}{\sS}{\cdot}$ is well-defined as a germ of a function in
the $\delta\in \R^{\nu}_0$ variable.

\item If $\sL(\sS)$ is linearly finitely generated 
on $\Omega'$, then for  $1<p<\infty$ and $\delta\in \R^{\nu}$, we write $\NLpN{p}{\delta}{\sS}{f}:=\NLpN{p}{\delta}{\gamma,e,N}{f}$,
for $f\in C_0^{\infty}(\Omega')$.  The equivalence class of the norm $\NLpN{p}{\delta}{\sS}{\cdot}$ is well-defined $\forall \delta\in \R^{\nu}$.
\end{enumerate}
\end{defn}

\begin{rmk}
Suppose $\fD$ is Sobolev data which is finitely generated by $\sF$ on $\Omega'$.  Then, for $1<p<\infty$, $\delta\in \R^{\nu}_0$,
\begin{equation*}
\NLpN{p}{\delta}{\fD}{f}\approx \NLpN{p}{\delta}{\sF}{f}, \quad f\in C_0^{\infty}(\Omega_0).
\end{equation*}
If $\fD$ is linearly finitely generated by $\sF$ on $\Omega'$, the above holds $\forall \delta\in \R^{\nu}$.
\end{rmk}

\begin{prop}\label{PropResSobLp}
Let $\fD$ be Sobolev data.  Then for $1<p<\infty$,
\begin{equation*}
\NLpN{p}{0}{\fD}{f}\approx \LpN{p}{f}, \quad f\in C_0^\infty(\Omega_0),
\end{equation*}
where the implicit constants depend on $p$ and $\fD$.
\end{prop}
\begin{proof}The proof is contained in Section \ref{SectionPfSob}.\end{proof}

%% file: rescompsob2.tex
Fix open sets $\Omega_0\Subset \Omega'\Subset\Omega'''\Subset \Omega\subseteq \R^n$.
Let $\sSt\subset \snutvectone$ and $\sSh\subset\snuhvectone$ be finite sets.  As in the previous section, we separate our results into two cases:
\begin{enumerate}[\bf{Case} I:]
\item $\sL(\sSt)$ and $\sL(\sSh)$ are finitely generated on $\Omega'$.
\item $\sL(\sSt)$ and $\sL(\sSh)$ are linearly finitely generated
on $\Omega'$.
\end{enumerate}
In light of Definition \ref{DefnResSobDependOnsS}, it makes sense to talk about $\NLp{p}{\deltat}{\sSt}$ and $\NLp{p}{\deltah}{\sSh}$
for $\deltat\in \R^{\nut}_0$ and $\deltah\in \R^{\nuh}_0$ in Case I, and $\deltat\in \R^{\nut}$ and $\deltah\in \R^{\nuh}$ in Case II.

From $\sSt$ and $\sSh$ we create a set of smooth vector fields on $\Omega$,
paired with $\nu=\nut+\nuh$ parameter formal degrees by
\begin{equation*}
\sS:=\q\{ \q(\Xh, (\hd,0_{\nut})\w) : (\Xh, \hd)\in \sSh \w\} \bigcup\q\{ \q(\Xt, (0_{\nuh},\td)\w) : (\Xt, \td)\in \sSt \w\},
\end{equation*}
where $0_{\nut}$ denotes the $0$ vector in $\R^{\nut}$ and $0_{\nuh}$ denotes the $0$ vector in $\R^{\nuh}$.
We now introduce the main hypothesis of this section.

\begin{assumption}\label{AssumptionResCompSobMainNew}
We assume
\begin{itemize}
\item In Case I, we assume $\sL(\sS)$ is finitely generated on $\Omega'$.
\item In Case II, we assume $\sL(\sS)$ is linearly finitely generated 
on $\Omega'$.
\end{itemize}
\end{assumption}

We assume Assumption \ref{AssumptionResCompSobMainNew} for the remainder of the section.
As before, in light of Definition \ref{DefnResSobDependOnsS}, it makes sense to talk about the norm
$\NLpN{p}{\delta}{\sS}{\cdot}$ for $\delta\in \R^{\nu}_0$ in Case I, and for $\delta\in \R^{\nu}$ in Case II.

\begin{rmk}\label{RmkResCompSobImpliesNew}
Assuming that $\sL(\sS)$ is finitely generated (resp. linearly finitely generated) on $\Omega'$ {\it implies} that $\sL(\sSt)$ and $\sL(\sSh)$ are finitely generated (resp. linearly finitely generated) on $\Omega'$.
Thus, Assumption \ref{AssumptionResCompSobMainNew}
contains all the assumptions of this section.
\end{rmk}

So that we may precisely state our results, in Case I pick Sobolev data 
\begin{equation*}
\begin{split}
&\fDt=(\nut, (\gammat,\et,\Nt,\Omega, \Omega'''), \at, \etat, \{\vsigt_j \}_{ j\in \N^{\nut}}, \psit)\\
&\fDh=(\nuh, (\gammah,\eh,\Nh,\Omega, \Omega'''), \ah, \etah, \{\vsigh_j \}_{j\in \N^{\nuh}}, \psih)\\
&\fD=(\nu,(\gamma,e,N,\Omega,\Omega'''), a, \eta, \{\vsig_j\}_{j\in \N^{\nu}}, \psi)
\end{split}
\end{equation*}
so that $\fDt$ and $\sL(\sSt)$ are finitely generated by
the same $\sFt$ on $\Omega'$,  $\fDh$ and $\sL(\sSh)$ are finitely generated by the same $\sFh$ on $\Omega'$, and
$\fD$ and $\sL(\sS)$ are finitely generated by the same $\sF$ on $\Omega'$.  This is always possible--see Proposition \ref{PropResSurfExistsAParam} and Remark \ref{RmkResSurfPointExistsAParam}.

\begin{thm}\label{ThmResCompSobDropParamsNew}
\begin{itemize}
\item In Case I,
for $1<p<\infty$, $\deltah\in \R^{\nuh}_0$ and $\deltat\in \R^{\nut}_0$ we have
\begin{equation*}
\NLpN{p}{(\deltah,0_{\nut})}{\sS}{f}\approx \NLpN{p}{\deltah}{\sSh}{f}, \quad \NLpN{p}{(0_{\nuh}, \deltat)}{\sS}{f}\approx \NLpN{p}{\deltat}{\sSt}{f},\quad \forall f\in C_0^\infty(\Omega_0).
\end{equation*}
More precisely,
for $1<p<\infty$, $\exists \epsilon=\epsilon(p,(\gamma,e,N),(\gammat,\et,\Nt),(\gammah,\eh,\Nh))>0$ such that
for $\deltah\in \R^{\nuh}$, $\deltat\in \R^{\nut}$ with $|\deltah|,|\deltat|<\epsilon$, we have
\begin{equation*}
\NLpN{p}{(\deltah,0_{\nut})}{\fD}{f}\approx \NLpN{p}{\deltah}{\fDh}{f}, \quad \NLpN{p}{(0_{\nuh}, \deltat)}{\fD}{f}\approx \NLpN{p}{\deltat}{\fDt}{f},\quad \forall f\in C_0^\infty(\Omega_0),
\end{equation*}
where the implicit constants depend on $p$, $\deltat$, $\deltah$, $\fD$, $\fDt$, and $\fDh$.

\item In Case II, for $1<p<\infty$, $\deltah\in \R^{\nuh}$, and $\deltat\in \R^{\nut}$, we have
\begin{equation*}
\NLpN{p}{(\deltah,0_{\nut})}{\sS}{f}\approx \NLpN{p}{\deltah}{\sSh}{f}, \quad \NLpN{p}{(0_{\nuh}, \deltat)}{\sS}{f}\approx \NLpN{p}{\deltat}{\sSt}{f},\quad \forall f\in C_0^\infty(\Omega_0).
\end{equation*}
where the implicit constant depends on $p$, $\deltah$, $\deltat$, and the choices made in defining the above norms.
\end{itemize}
\end{thm}
\begin{proof}This is proved in Section \ref{SectionPfSobComp}.\end{proof}

In what follows, it will be convenient to write an element of $[0,1]^{\nu}$ as $2^{-j}$, where $j\in [0,\infty]^{\nu}$.
Here, $2^{-j}=(2^{-j_1},\ldots, 2^{-j_\nu})$ (and similarly for $\nu$ replaced by $\nuh$ or $\nut$).

Let $\lambda$ be a $\nut\times \nuh$ matrix whose entries are all in $[0,\infty]$.  
In both Case I and Case II, we assume:
\begin{equation*}
\sL(\sSh)\text{ }\lambda\text{-controls }\sSt\text{ on }\Omega'.
\end{equation*}
In what follows we write $\lambda^{t}(\deltat)$ for $\deltat\in [0,\infty)^{\nut}$.  Here, we use the convention
that $\infty\cdot 0=0$ but $\infty\cdot x=\infty$ for $x>0$.

\begin{thm}\label{ThmResCompSobMainThmNew}
Under the above hypotheses, we have
\begin{itemize}
\item In Case I, for $1<p<\infty$,
$\delta\in \R^{\nu}_0$, $\deltat\in \R^{\nut}_0\cap [0,\infty)^{\nut}$, and such that $\lambda^{t}(\deltat)$ is not equal to $\infty$ in any coordinate,
\begin{equation}\label{EqnResCompSobMainThmInfin}
\NLpN{p}{\delta+(-\lambda^{t}(\deltat),\deltat )}{\sS}{f}\lesssim \NLpN{p}{\delta}{\sS}{f}, \quad f\in C_0^\infty(\Omega_0).
\end{equation}
More precisely, for $1<p<\infty$, $\exists \epsilon=\epsilon(p, (\gamma,e,N),\lambda)>0$, such that for $\delta\in \R^{\nu}$ and $\deltat\in [0,\infty)^{\nut}$
with $|\delta|,|\deltat|<\epsilon$, and such that $\lambda^{t}(\deltat)$ is not equal to $\infty$ in any coordinate,
\begin{equation}\label{EqnResCompSobMainThmSmall}
\NLpN{p}{\delta+(-\lambda^{t}(\deltat),\deltat )}{\fD}{f}\lesssim \NLpN{p}{\delta}{\fD}{f}, \quad f\in C_0^\infty(\Omega_0).
\end{equation}

\item In Case II, for $1<p<\infty$,
$\delta\in \R^{\nu}$, $\deltat\in \R^{\nut}\cap [0,\infty)^{\nut}$, and such that $\lambda^{t}(\deltat)$ is not equal to $\infty$ in any coordinate,
\begin{equation}\label{EqnResCompSobMainThmAll}
\NLpN{p}{\delta+(-\lambda^{t}(\deltat),\deltat )}{\sS}{f}\lesssim \NLpN{p}{\delta}{\sS}{f}, \quad f\in C_0^\infty(\Omega_0).
\end{equation}

\end{itemize}
\end{thm}
\begin{proof}This is proved in Section \ref{SectionPfSobComp}.\end{proof}

\begin{rmk}
By changing $\delta$, \eqref{EqnResCompSobMainThmInfin} and \eqref{EqnResCompSobMainThmAll} can be equivalently written as
\begin{equation}\label{EqnResCompSobMainThmReWrite}
\NLpN{p}{\delta+(0,\delta_0 )}{\sS}{f}\lesssim \NLpN{p}{\delta+(\lambda^{t}(\delta_0),0)}{\sS}{f}, \quad f\in C_0^\infty(\Omega_0).
\end{equation}
A similar remark holds for \eqref{EqnResCompSobMainThmSmall}.
\end{rmk}

\begin{cor}\label{CorCompMainCor}
Under the above hypotheses, we have
\begin{itemize}
\item In Case I, for $1<p<\infty$,
$\deltat\in \R^{\nut}_0\cap [0,\infty)^{\nut}$ and such that $\lambda^{t}(\deltat)$ is not equal to $\infty$ in any coordinate,
\begin{equation*}
\NLpN{p}{\deltat}{\sSt}{f}\lesssim \NLpN{p}{\lambda^{t}(\deltat)}{\sSh}{f}, \quad f\in C_0^{\infty}(\Omega_0),
\end{equation*}
and dually,
\begin{equation*}
\NLpN{p}{-\lambda^{t}(\deltat)}{\sSh}{f}\lesssim \NLpN{p}{-\deltat}{\sSt}{f}, \quad  f\in C_0^{\infty}(\Omega_0).
\end{equation*}
More precisely, for $1<p<\infty$, $\exists \epsilon=\epsilon(p,(\gammah,\eh,\Nh),(\gammat,\et,\Nt),\lambda)>0$ such that
for $\deltat\in [0,\infty)^{\nut}$ with $|\deltat|<\epsilon$ and such that $\lambda^{t}(\deltat)$ is not equal to $\infty$
in any coordinate,
\begin{equation}\label{EqnResCompSobCorToShowO}
\NLpN{p}{\deltat}{\fDt}{f}\lesssim \NLpN{p}{\lambda^{t}(\deltat)}{\fDh}{f}, \quad f\in C_0^{\infty}(\Omega_0),
\end{equation}
and dually,
\begin{equation}\label{EqnResCompSobCorToShowT}
\NLpN{p}{-\lambda^{t}(\deltat)}{\fDh}{f}\lesssim \NLpN{p}{-\deltat}{\fDt}{f}, \quad  f\in C_0^{\infty}(\Omega_0).
\end{equation}

\item  In Case II,
for $1<p<\infty$,
$\deltat\in \cap [0,\infty)^{\nut}$ and such that $\lambda^{t}(\deltat)$ is not equal to $\infty$ in any coordinate,
\begin{equation*}
\NLpN{p}{\deltat}{\sSt}{f}\lesssim \NLpN{p}{\lambda^{t}(\deltat)}{\sSh}{f}, \quad f\in C_0^{\infty}(\Omega_0),
\end{equation*}
and dually,
\begin{equation*}
\NLpN{p}{-\lambda^{t}(\deltat)}{\sSh}{f}\lesssim \NLpN{p}{-\deltat}{\sSt}{f}, \quad  f\in C_0^{\infty}(\Omega_0).
\end{equation*}
\end{itemize}
\end{cor}
\begin{proof}
\eqref{EqnResCompSobCorToShowO} follows by taking $\delta=0$ in \eqref{EqnResCompSobMainThmReWrite}
and applying Theorem \ref{ThmResCompSobDropParamsNew}.  \eqref{EqnResCompSobCorToShowT} follows similarly
by taking $\delta=(-\lambda^{t}(\deltat),-\deltat)$ in \eqref{EqnResCompSobMainThmReWrite} and applying Theorem \ref{ThmResCompSobDropParamsNew}.
The result in Case II follows by a similar proof.
\end{proof} 

%% file: ressobeuclid2.tex
An important special case of our Sobolev spaces comes when we consider the finite set of vector fields with single parameter formal degrees on $\R^n$ given by
\begin{equation}\label{EqnResSobEuclidPartialOneNew}
(\partial,1):=\q\{\q(\frac{\partial}{\partial x_1 },1\w),\ldots, \q(\frac{\partial }{\partial x_n},1\w)\w\}.
\end{equation}
Fix open sets $\Omega_0\Subset \Omega'\Subset \Omega'''\Subset \Omega\subseteq \R^n$.
Clearly, $(\partial,1)$ satisfies $\sD(\Omega')$, and (in particular) $\sL(\partial,1)$ is linearly finitely generated by $(\partial,1)$ on $\Omega'$ (in fact,
it is smoothly linearly finitely generated by $(\partial, 1)$ on $\Omega'$).
Thus, it makes sense to talk about $\NLpN{p}{s}{\partial,1}{\cdot}$ for any $1<p<\infty$ and $s\in \R$.
Let $L^p_s$ denote the standard, isotropic, Sobolev space of order $s\in \R$ on $\R^n$.

We have
\begin{thm}\label{ThmResSobEuclidIsNonIsoNew}
For $1<p<\infty$, $s\in \R$,
\begin{equation*}
\NLpN{p}{s}{\partial,1}{f}\approx \LpsN{p}{s}{f}, \quad f\in C_0^\infty(\Omega_0),
\end{equation*}
where the implicit constants depend on $p$ and $s$ (and, of course, on the choices made in defining $\NLpN{p}{s}{\partial,1}{\cdot}$).
\end{thm}
\begin{proof}
This is exactly the statement of Lemma 5.8.9 of \cite{StreetMultiParamSingInt}.
\end{proof}

Using the theorems earlier in this section, in combination with Theorem \ref{ThmResSobEuclidIsNonIsoNew}, we can compare
the standard isotropic $L^p$ Sobolev spaces with our non-isotropic Sobolev spaces.  Thus, suppose we are
given a set $\sSt\subset \snutvectone$.  Let $\nu=1+\nut$ and define
\begin{equation*}
\begin{split}
\sS&:=\q\{ \q(\frac{\partial}{\partial x_1}, (1,0_{\nut})\w),\ldots,\q(\frac{\partial}{\partial x_n}, (1,0_{\nut})\w) \w\} \bigcup\q\{ \q(\Xt, (0,\td)\w) : (\Xt, \td)\in \sSt \w\}
\\&\subset \snuvectone.
\end{split}
\end{equation*}
For the rest of this section, we assume the following.
\begin{assumption}
We assume one of the following two cases.
\begin{enumerate}[\bf{Case} I:]
\item $\sL(\sS)$ is finitely generated on $\Omega'$.
\item $\sL(\sS)$ is linearly finitely generated on $\Omega'$.
\end{enumerate}
\end{assumption}

Note that, if $\sL(\sS)$ is (linearly) finitely generated on $\Omega'$, then the same is true of $\sL(\sSt)$.
Thus, in Case I, it makes sense to talk about the norms $\NLpN{p}{\delta}{\sS}{\cdot}$ and $\NLpN{p}{\deltat}{\sSt}{\cdot}$
for $\delta\in \R^{\nu}_0$, $\deltat\in \R^{\nut}_0$.  In Case II, it makes sense to talk about the same norms
for all $\delta\in \R^{\nu}$, $\deltat\in \R^{\nut}$.

The next lemma helps to elucidate situations where the above assumptions hold.
\begin{lemma}\label{LemmaResSobEuclidWhenHolds}
If $\sL(\sSt)$ is smoothly finitely generated (resp. smoothly linearly finitely generated) on $\Omega'$, then $\sL(\sS)$ is smoothly finitely
generated (resp. smoothly linearly finitely generated) on $\Omega'$.
\end{lemma}
\begin{proof}
Suppose $\sL(\sSt)$ is smoothly finitely generated by $\sFt$ on $\Omega'$.
Define
$$\sF:=\q\{ \q(\frac{\partial}{\partial x_1}, (1,0_{\nut})\w),\ldots,\q(\frac{\partial}{\partial x_n}, (1,0_{\nut})\w) \w\} \bigcup\q\{ \q(\Xt, (0,\td)\w) : (\Xt, \td)\in \sFt \w\}.$$
Let $(X,d)\in \sL(\sS)$.  We wish to show $(X,d)$ is smoothly controlled by $\sF$.  There are two possibilities.

The first possibility is that $d$ equals $0$ in the first component.  In this case $(X,d)$ is of the form $(X,(0,\td))$
for some $(X,\td)\in \sL(\sSt)$.  Since $\sFt$ smoothly controls $(X,\td)$, by assumption, it is immediate from the definitions
that $\sF$ smoothly controls $(X,d)$.

The other possibility is that the first component of $d$ is $\geq 1$.  In this case, we use that
\begin{equation*}
X=\sum_{j=1}^n c_j \frac{\partial}{\partial x_j}, \quad c_j\in C^{\infty}(\Omega),
\end{equation*}
and therefore for $\delta\in [0,1]^{\nu}$,
\begin{equation*}
\delta^d X=\sum_{j=1}^n \q(\delta^{d-(1,0_{\nut})}c_j\w) \delta^{(1,0_{\nut})} \frac{\partial}{\partial x_j}.
\end{equation*}
Since $d-(1,0_{\nut})\in [0,\infty)^\nu$, we see that $\q\{\delta^{d-(1,0_{\nut})}c_j : \delta\in [0,1]^\nu\w\}\subset C^\infty(\Omega)$ is a bounded set, and therefore $\sF$ smoothly controls $(X,d)$.  Thus, $\sL(\sS)$ is smoothly finitely generated by $\sF$.

If $\sFt\subset \snutvectone$, then $\sF\subset \snuvectone$, and it follows that $\sL(\sS)$ is smoothly linearly finitely generated by $\sF$, completing the proof.
\end{proof}

\begin{example}
When $\sSt$ is finite and the vector fields in $\sSt$ are real analytic then $\sL(\sSt)$ is smoothly finitely generated on $\Omega'$ (see Section \ref{SectionResRealAnalVect}), and therefore $\sL(\sS)$ is smoothly finitely generated on $\Omega'$ (Lemma \ref{LemmaResSobEuclidWhenHolds}).  If $\sSt$ is finite and for each $1\leq \mu\leq \nut$,
$$\q\{\Xt : (\Xt,\td)\in \sSt\text{ and }\td\text{ is nonzero in the }\mu\text{ component}\w\}$$
satisfies H\"ormander's condition, then $\sL(\sSt)$ is smoothly finitely generated (see Section \ref{SectionResHorVect}), and therefore
$\sL(\sS)$ is smoothly finitely generated (Lemma \ref{LemmaResSobEuclidWhenHolds}).  When $\nut=1$ and in either of these settings, $\sL(\sSt)$
is smoothly linearly finitely generated, and the same is true of $\sL(\sS)$ (Lemma \ref{LemmaResSobEuclidWhenHolds}).
\end{example}

So that we may precisely state our results, in Case I, pick Sobolev data 
$$\fDt=(\nut,(\gammat,\et,\Nt),\Omega,\Omega'''), \at, \etat, \{\vsigt_j\}_{j\in \N^{\nu}}, \psit)$$
so that $\sL(\sSt)$ and $\fDt$ are finitely generated by the same $\sFt$ on $\Omega'$.

Let $E\in [0,\infty]^{\nut}$ be a vector such that $\forall (\Xt,\td)\in \sSt$, with $\Xt$ not the zero vector field, $E\cdot \td\geq 1$.  We have,
\begin{thm}\label{ThmSobEuclidNoSpanNew}
\begin{itemize}
\item In Case I, for $1<p<\infty$ and $\deltat\in \R^{\nut}_0\cap [0,\infty)^{\nut}$ such that $E\cdot \deltat<\infty$, we have
\begin{equation*}
\NLpN{p}{\deltat}{\sSt}{f}\lesssim \LpsN{p}{E\cdot \deltat}{f}, \quad f\in C_0^{\infty}(\Omega_0),
\end{equation*}
and dually,
\begin{equation*}
\LpsN{p}{-E\cdot \deltat}{f}\lesssim \NLpN{p}{-\deltat}{\sSt}{f}, \quad  f\in C_0^{\infty}(\Omega_0).
\end{equation*}
More precisely, for $1<p<\infty$, $\exists \epsilon=\epsilon(p, (\gammat,\et,\Nt), E)>0$ such that for $\deltat\in [0,\infty)^{\nut}$ with $|\deltat|<\epsilon$
and such that $E\cdot \deltat<\infty$,
\begin{equation*}
\NLpN{p}{\deltat}{\fDt}{f}\lesssim \LpsN{p}{E\cdot \deltat}{f}, \quad f\in C_0^{\infty}(\Omega_0),
\end{equation*}
and dually,
\begin{equation*}
\LpsN{p}{-E\cdot \deltat}{f}\lesssim \NLpN{p}{-\deltat}{\fDt}{f}, \quad  f\in C_0^{\infty}(\Omega_0).
\end{equation*}

\item In Case II, for $1<p<\infty$ and $\deltat\in [0,\infty)^{\nut}$ such that $E\cdot \deltat<\infty$, we have
\begin{equation*}
\NLpN{p}{\deltat}{\sSt}{f}\lesssim \LpsN{p}{E\cdot \deltat}{f}, \quad f\in C_0^{\infty}(\Omega_0),
\end{equation*}
and dually,
\begin{equation*}
\LpsN{p}{-E\cdot \deltat}{f}\lesssim \NLpN{p}{-\deltat}{\sSt}{f}, \quad  f\in C_0^{\infty}(\Omega_0).
\end{equation*}
\end{itemize}
\end{thm}
\begin{proof}
For $(\Xt,\td)\in \sSt$, we may write $\Xt=\sum_{l=1}^n c_l \frac{\partial}{\partial x_l}$, where $c_l\in C^{\infty}(\Omega)$.
Thus, for $\jh\in [0,\infty]$, we have
\begin{equation*}
2^{-(E \jh)\cdot \td} \Xt = \sum_{l=1}^n \q(2^{\jh (1-E\cdot \td)} c_l\w) 2^{-\jh }\frac{\partial}{\partial x_l}.
\end{equation*}
By the assumption on $E$, $\{2^{\jh (1-E\cdot \td)} c_l : 1\leq l\leq n, \jh\in[0,\infty]\}\subset C^\infty(\Omega)$
is a bounded set.  
Thus we have $(\partial, 1)$ smoothly $E$-controls $(\Xt,\td)$ on $\Omega'$.
From here the result follows from an application of Corollary \ref{CorCompMainCor} and Theorem
\ref{ThmResSobEuclidIsNonIsoNew}.
\end{proof}

For the reverse inequalities, in Case I pick $\sF$ so that $\sL(\sS)$ is finitely generated by $\sF$ on $\Omega'$, and in
Case II, pick $\sF$ so that $\sL(\sS)$ is linearly finitely generated by $\sF$ on $\Omega'$.
Define $\sFt\subset \snutvect$ by
\begin{equation*}
\sFt:=\q\{(\Xt,\td):  (\Xt, (0,\td))\in \sF\w\}.
\end{equation*}
It immediately follows that in Case I, $\sL(\sSt)$ is finitely generated by $\sFt$ in $\Omega'$, and in Case II, $\sL(\sSt)$ is linearly
finitely generated by $\sFt$ on $\Omega'$.  We further assume
\begin{equation*}
\q\{ \Xt : (\Xt,\td)\in \sFt\w\}
\end{equation*}
spans the tangent space at every point of $\Omega'$.

For $1\leq k\leq n$, let $\sRt_k\subseteq \sFt$ be such that
\begin{equation}\label{EqnResSobEuclidUseSpanNew}
\frac{\partial}{\partial x_k} = \sum_{(\Xt,\td)\in \sRt_k} c_{k,\Xt,\td} \Xt,
\end{equation}
where $c_{k,\Xt,\td}\in C^\infty(\Omega')$.  Set $\sRt=\bigcup_{k=1}^n \sRt_k$.
And define a vector $F=(F_1,\ldots, F_{\nut})\in [0,\infty)^{\nut}$ by
\begin{equation*}
F_\mu = \max_{(\Xt,\td)\in \sRt} \td_\mu.
\end{equation*}

\begin{thm}\label{ThmSobEuclidSpanNew}
\begin{itemize}
\item In Case I, for $1<p<\infty$, and $\deltah\in \R_0\cap [0,\infty)$ we have
\begin{equation*}
\LpsN{p}{\deltah}{f}\lesssim \NLpN{p}{\deltah F}{\sSt}{f}, \quad f\in C_0^{\infty}(\Omega_0),
\end{equation*}
and dually,
\begin{equation*}
\NLpN{p}{-\deltah F}{\sSt}{f}\lesssim \LpsN{p}{-\deltah}{f}, \quad f\in C_0^\infty(\Omega_0).
\end{equation*}
More precisely, for $1<p<\infty$, $\exists \epsilon=\epsilon(p, (\gammat,\et,\Nt),F)>0$ such that for $\deltah\in [0,\infty)$ with $|\delta|<\epsilon$,
\begin{equation*}
\LpsN{p}{\deltah}{f}\lesssim \NLpN{p}{\deltah F}{\fDt}{f}, \quad f\in C_0^{\infty}(\Omega_0),
\end{equation*}
and dually,
\begin{equation*}
\NLpN{p}{-\deltah F}{\fDt}{f}\lesssim \LpsN{p}{-\deltah}{f}, \quad f\in C_0^\infty(\Omega_0).
\end{equation*}

\item In Case II, for $1<p<\infty$, and $\deltah\in [0,\infty)$ we have
\begin{equation*}
\LpsN{p}{\deltah}{f}\lesssim \NLpN{p}{\deltah F}{\sSt}{f}, \quad f\in C_0^{\infty}(\Omega_0),
\end{equation*}
and dually,
\begin{equation*}
\NLpN{p}{-\deltah F}{\sSt}{f}\lesssim \LpsN{p}{-\deltah}{f}, \quad f\in C_0^\infty(\Omega_0).
\end{equation*}
\end{itemize}
\end{thm}
\begin{proof}
By \eqref{EqnResSobEuclidUseSpanNew}, for $\jt\in [0,\infty]^{\nut}$,  we have
\begin{equation*}
2^{-\jt\cdot F} \frac{\partial}{\partial x_k} = \sum_{(\Xt, \td)\in \sRt_k} \q(2^{\jt\cdot( \td - F )} c_{k,\Xt,\td}\w) 2^{-j\cdot \td} \Xt.
\end{equation*}
By the definition of $F$, $\td-F$ is nonpositive in every component (for $(\Xt,\td)\in \sRt_k$), and therefore
$$\q\{2^{\jt\cdot( \td - F )} c_{k,\Xt,\td} : (\Xt,\td)\in \sRt_k, \jt\in [0,\infty]^{\nut}\w\}\subset C^\infty(\Omega')$$
 is a bounded set.
 This shows that $\sFt$ smoothly $F$-controls $(\frac{\partial}{\partial x_k},1)$ on $\Omega'$.
From here the result follows from an application of Corollary \ref{CorCompMainCor} and Theorem
\ref{ThmResSobEuclidIsNonIsoNew}.
\end{proof}

%% file: resradon2.tex
As in Section \ref{SectionResSobNew}, we fix open sets $\Omega_0\Subset\Omega'\Subset\Omega''\Subset\Omega'''\Subset\Omega\subseteq\R^n$.
Let $(\gamma,e,N,\Omega,\Omega''')$ be a parameterization, with $\nu$-parameter dilations.
Fix $a>0$.
For $\psi_1,\psi_2\in C_0^{\infty}(\Omega_0)$, $\kappa(t,x)\in C^{\infty}(B^{N}(a)\times \Omega'')$, $\delta\in \R^{\nu}$,
and $K\in \sK_\delta(N,e,a)$, define an operator
\begin{equation}\label{EqnResRadonMainOpNew}
Tf(x)=\psi_1(x) \int f(\gamma_t(x)) \psi_2(\gamma_t(x)) \kappa(t,x) K(t)\: dt.
\end{equation}

\begin{defn}
If $T$ is an operator as in \eqref{EqnResRadonMainOpNew}, we say that $T$ is a {\bf fractional Radon transform of order} $\delta$ {\bf corresponding
to} $(\gamma,e,N)$.  If we wish to make the choice of $a>0$ explicit, we say $T$ is a {\bf fractional Radon transform of order} $\delta$ {\bf
corresponding to} $(\gamma,e,N)$ {\bf on} $B^N(a)$.
\end{defn}

\begin{thm}\label{ThmResRadonMainThm}
\begin{itemize}
\item Let $(\gamma,e,N,\Omega,\Omega''')$ be a parameterization which is finitely generated on $\Omega'$.
Then, there exists $a>0$, such that for $1<p<\infty$, there exists $\epsilon=\epsilon(p,(\gamma,e,N))>0$, such that
for any $\delta,\delta'\in \R^{\nu}$ with $|\delta|,|\delta'|<\epsilon$, and 
any $T$ a fractional Radon transform of order $\delta$ corresponding to $(\gamma,e,N)$ on $B^N(a)$, we have
\begin{equation*}
\NLpN{p}{\delta'}{\gamma,e,N}{Tf}\lesssim \NLpN{p}{\delta+\delta'}{\gamma,e,N}{f}, \quad f\in C_0^{\infty}(\Omega_0).
\end{equation*}

\item Let $(\gamma,e,N,\Omega,\Omega''')$ be a parameterization which is linearly finitely generated on $\Omega'$.
Then, there exists $a>0$, such that for $1<p<\infty$, $\delta,\delta'\in \R^{\nu}$, and 
any $T$ a fractional Radon transform of order $\delta$ corresponding to $(\gamma,e,N)$ on $B^N(a)$, we have
\begin{equation*}
\NLpN{p}{\delta'}{\gamma,e,N}{Tf}\lesssim \NLpN{p}{\delta+\delta'}{\gamma,e,N}{f}, \quad f\in C_0^{\infty}(\Omega_0).
\end{equation*}

\end{itemize}
\end{thm}
\begin{proof}This is proved in Section \ref{SectionPfRadon}.\end{proof}

\begin{rmk}
In Theorem \ref{ThmResRadonMainThm}, we have used that when $(\gamma,e,N)$ is finitely generated on $\Omega'$,
the equivalence class of the norm $\NLpN{p}{\delta}{\gamma,e,N}{\cdot}$ is well-defined on $C_0^{\infty}(\Omega_0)$
for $\delta$ sufficiently small (depending on $p$ and $(\gamma,e,N)$), and is well-defined for all $\delta\in \R^{\nu}$
when $(\gamma,e,N)$ is linearly finitely generated on $\Omega'$.  See Definition \ref{DefnResSobForGamma} for further details.
\end{rmk}

\begin{cor}\label{CorResRadonIntroCor}
Let $(\gamma,e,N,\Omega,\Omega''')$ be a parameterization which is finitely generated on $\Omega'$.
Then for $1<p<\infty$, there exists $\epsilon=\epsilon(p,(\gamma,e,N))>0$, such that for every $\psi\in C_0^{\infty}(\Omega_0)$,
there exists $a>0$ such that for every $\delta,\delta'\in \R^{\nu}$ with $|\delta|,|\delta'|<\epsilon$, and every $K\in \sK_{\delta}(N,e,a)$,
the operator
\begin{equation}\label{EqnResRadonCorOpNew}
T f(x)=\psi(x)\int f(\gamma_t(x)) K(t)\: dt
\end{equation}
satisfies
\begin{equation*}
\NLpN{p}{\delta'}{\gamma,e,N}{Tf}\lesssim \NLpN{p}{\delta+\delta'}{\gamma,e,N}{f}, \quad f\in C_0^{\infty}(\Omega_0).
\end{equation*}
If, in addition, $(\gamma,e,N)$ is linearly finitely generated on $\Omega'$, we may take $\epsilon=\infty$.
\end{cor}
\begin{proof}
Given $\psi$, take $\psi_2\equiv 1$ on a neighborhood of $\supp{\psi}$ and take $\kappa=1$.  It is easy to see that
if $K$ is supported on a sufficiently small neighborhood of $0$ (i.e., if $a>0$ is sufficiently small), then 
$\psi_2(\gamma_t(x))\equiv 1$ on the domain of integration of \eqref{EqnResRadonMainOpNew} (with $\psi$ replaced by $\psi_1$).
Thus, $T$ is of the form \eqref{EqnResRadonMainOpNew}.  From here, the corollary follows from Theorem \ref{ThmResRadonMainThm}.
\end{proof}

\begin{rmk}
The main reason we work with the more general operators in Theorem \ref{ThmResRadonMainThm} (instead of the operators in Corollary \ref{CorResRadonIntroCor}) is that the class of operators in
\eqref{EqnResRadonMainOpNew} is closed under adjoints, while the class of operators in \eqref{EqnResRadonCorOpNew} is not.  See Section \ref{SectionPfAdjoint}
and
Section 12.3 of \cite{SteinStreetI} for details.
\end{rmk}

%% file: resradonother2.tex
Suppose $(\gammat,\et,\Nt,\Omega,\Omega''')$ is a parameterization with $\nut$-parameter dilations which is finitely generated (resp. linearly finitely generated) on $\Omega'$.
If $T$ is a fractional Radon transform corresponding to $(\gammat,\et,\Nt)$ of order $\deltat$, then
Theorem \ref{ThmResRadonMainThm} shows, for $p\in (1,\infty$) and  $\deltat,\deltat'\in \R^{\nut}$ sufficiently small (resp. for all $\deltat,\deltat'\in \R^{\nu}$),
$T:\NLp{p}{\deltat+\deltat'}{(\gammat,\et,\Nt)}\rightarrow \NLp{p}{\deltat'}{(\gammat,\et,\Nt)}$.
In other words, if $\sSt$ is defined in terms of $(\gammat,\et,\Nt)$ by \eqref{EqnResSurfDefnsS}, then
$T:\NLp{p}{\deltat+\deltat'}{\sSt}\rightarrow \NLp{p}{\deltat'}{\sSt}$ for $\deltat,\deltat'\in \R_0^{\nut}$ (resp. for all $\deltat,\deltat'\in \R^{\nu}$).

Now suppose $\sSh\subseteq \snuhvectone$ is such that $\sL(\sSh)$ is finitely generated (resp. linearly finitely generated) on $\Omega'$.  If $\deltat\in \R_0^{\nut}$ (resp. $\deltat\in \R^{\nut}$) it makes sense to ask
for what $\deltah_1,\deltah_2\in \R_0^{\nuh}$ (resp. $\deltah_1,\deltah_2\in \R^{\nuh}$), if any, do we have mappings of the form
$T:\NLp{p}{\deltah_1}{\sSh}\rightarrow \NLp{p}{\deltah_2}{\sSh}$.  Moreover, we wish to have a formula for $\deltah_2$ in terms of $\deltah_1$ and $\deltat$.
To this end, we make the following assumptions for the rest of the section.


\begin{enumerate}[\bf{Case} I:]
\item
\begin{itemize}
\item$(\gammat,\et,\Nt,\Omega,\Omega''')$ is a parameterization (with $\nut$-parameter dilations) which is finitely generated
on $\Omega'$.
\item Let $\sSt\subset \snutvectone$ be defined in terms of $(\gammat,\et,\Nt)$ by \eqref{EqnResSurfDefnsS}, and let $\sSh\subseteq \snuhvectone$. 
\item Let\begin{equation}\label{EqnAssumptionResRadonOtherMainsSNew}
\begin{split}
&\sS:=
\\&\quad\quad\q\{ \q(\Xh, (\hd,0_{\nut})\w) : (\Xh, \hd)\in \sSh \w\} 
\\&\quad\quad\bigcup\q\{ \q(\Xt, (0_{\nuh},\td)\w) : (\Xt, \td)\in \sSt \w\}
\\&\subset \snuvectone,
\end{split}
\end{equation}
where $\nu=\nut+\nuh$.  We assume $\sL(\sS)$ is finitely generated on $\Omega'$.  Note, this implies $\sL(\sSh)$ is finitely generated
on $\Omega'$.
\item On $\R^{\Nt}$, define $\nu$-parameter dilatations $\et'$ given by $\et_j'=(0_{\nuh}, \et_j)$.\footnote{I.e., for $\delta=(\deltah,\deltat)\in [0,1]^{\nuh}\times [0,1]^{\nut}$, we define $\delta \vtt=\deltat\vtt$, where $\deltat\vtt$ is defined by the $\nut$ parameter dilations $\et$.}  We assume
$\sL(\sS)$ controls $(\gammat,\et',\Nt)$ on $\Omega'$.
\end{itemize}
\item This is the same as Case I, but everywhere ``finitely generated'' is replaced by ``linearly finitely generated''.
\end{enumerate}

\begin{rmk}\label{RmkResRadonOtherSmoothWorks}
Case I comes up in many situations.  For instance, if $\sL(\sS)$ is finitely generated on $\Omega'$ and if $(\gammat,\et,\Nt)$ is
smoothly finitely generated on $\Omega'$, then the assumptions in Case I hold.\footnote{The point here is that 
if $(\gammat,\et,\Nt)$ is merely finitely generated on $\Omega'$, then it does not necessarily follow that $\sL(\sS)$ controls
$(\gammat,\et,\Nt)$ on $\Omega'$.  However, if $(\gammat,\et,\Nt)$ is {\textit smoothly} finitely generated on $\Omega'$, then
it does follow that $\sL(\sS)$ controls $(\gammat,\et,\Nt)$ on $\Omega'$.  This is one of the main conveniences of
smoothly finitely generated over finitely generated.}
This happens automatically in many situations of interest.
See, e.g.,
Sections \ref{SectionResHorVect}, \ref{SectionResRealAnalVect}, \ref{SectionResHorSurf}, and \ref{SectionResRealAnalSurf}.
Similarly, if $\sL(\sS)$ is linearly finitely generated on $\Omega'$ and if $(\gammat,\et,\Nt)$ is
smoothly linearly finitely generated on $\Omega'$, then the assumptions in Case II hold.  This also arises in some cases of interest;
see Section \ref{SectionRadonPDO}.
\end{rmk}

So that we may precisely state our results we need to make a few choices.
Pick $\sF$ so that in Case I, $\sL(\sS)$ is finitely generated by $\sF$ on $\Omega'$, and in Case II, $\sL(\sS)$ is linearly finitely generated
by $\sF$ on $\Omega'$.
Because $\sL(\sS)$ is finitely generated (resp. linearly finitely generated) on $\Omega'$, it follows that $\sL(\sSh)$ is finitely generated (resp. linearly finitely
generated) by some $\sFh$ on $\Omega'$.
In Case I, pick parameterizations $(\gamma,e,N)$ and $(\gammah,\eh,\Nh)$ which are
finitely generated on $\Omega'$ by $\sF$ and $\sFh$, respectively.  For instance, one may use the choice in Proposition \ref{PropResSurfExistsAParam} and Remark \ref{RmkResSurfPointExistsAParam}.
In what follows, in Case I, we use the norms $\NLpN{p}{\delta}{\gamma,e,N}{\cdot}$ and $\NLpN{p}{\deltah}{\gammah,\eh,\Nh}{\cdot}$,
for $\delta\in \R^{\nu}$ and $\deltah\in \R^{\nut}$ small.  In light of Theorem \ref{ThmResSobWellDefNew} and Definition \ref{DefnResSobForGamma},
these norms are well defined, and depend only on $\sS$ and $\sSh$ for $\delta\in \R^{\nu}_0$, $\deltah\in \R^{\nut}_0$.
In Case II, the equivalence class of the norms $\NLpN{p}{\delta}{\sS}{\cdot}$, $\NLpN{p}{\deltah}{\sSh}{\cdot}$ are well-defined
for all $\delta\in \R^{\nu}$, $\deltah\in \R^{\nuh}$--see Definition \ref{DefnResSobDependOnsS}.

\begin{prop}\label{PropResRadonOtherAddingParamsNew}
Suppose $T$ is a fractional Radon transform of order $\deltat\in \R^{\nut}$ corresponding to $(\gammat,\et,\Nt)$ on $B^{\Nt}(a)$.
Then, there exists a $\nu$-parameter parameterization $(\gammat',\et',\Nt')$ which is finitely generated (resp. linearly finitely generated)
by $\sF$ on $\Omega'$ in Case I (resp. in Case II), and such that $T$ is a fractional Radon transform
of order $(0_{\nuh},\deltat)\in \R^{\nuh}\times \R^{\nut}=\R^{\nu}$ corresponding to $(\gammat',\et',\Nt')$ on $B^{\Nt'}(a)$.
\end{prop}
\begin{proof}This is proved in Section \ref{SectionPfRadon}.\end{proof}

\begin{thm}\label{ThmResRadonOtherMainBound}
There exists $a>0$ such that for all $1<p<\infty$, the following holds.
\begin{itemize}
\item In Case I, for every $\deltat\in \R^{\nut}_0$, $\delta\in \R^{\nu}_0$, and every fractional Radon transform, $T$, of order
$\deltat$ corresponding to $(\gammat,\et,\Nt)$ on $B^{\Nt}(a)$, we have
\begin{equation*}
\NLpN{p}{\delta}{\sS}{Tf}\lesssim \NLpN{p}{\delta+(0_{\nuh},\deltat)}{\sS}{f}, \quad f\in C_0^{\infty}(\Omega_0).
\end{equation*}
More precisely, there exists $\epsilon=\epsilon(p,(\gamma,e,N),(\gammat,\et,\Nt))>0$ such that for any $\deltat\in \R^{\nut}$, $\delta\in \R^{\nu}$
with $|\deltat|,|\delta|<\epsilon$, and every fractional Radon transform, $T$, of order
$\deltat$ corresponding to $(\gammat,\et,\Nt)$ on $B^{\Nt}(a)$, we have
\begin{equation}\label{EqnResRadonOtherToShowBoundMain}
\NLpN{p}{\delta}{\gamma,e,N}{Tf}\lesssim \NLpN{p}{\delta+(0_{\nuh},\deltat)}{\gamma,e,N}{f}, \quad f\in C_0^{\infty}(\Omega_0).
\end{equation}

\item In Case II, for every $\deltat\in \R^{\nut}$, $\delta\in \R^{\nu}$, and every fractional Radon transform, $T$, of order
$\deltat$ corresponding to $(\gammat,\et,\Nt)$ on $B^{\Nt}(a)$, we have
\begin{equation*}
\NLpN{p}{\delta}{\sS}{Tf}\lesssim \NLpN{p}{\delta+(0_{\nuh},\deltat)}{\sS}{f}, \quad f\in C_0^{\infty}(\Omega_0).
\end{equation*}

\end{itemize}
\end{thm}
\begin{proof}
First we consider Case I.
Proposition \ref{PropResRadonOtherAddingParamsNew} combined with Theorem \ref{ThmResRadonMainThm} proves
\eqref{EqnResRadonOtherToShowBoundMain} with $(\gamma,e,N)$ replaced by $(\gammat',\et',\Nt')$.
Because $(\gamma,e,N)$ and $(\gammat',\et',\Nt')$ are both finitely generated by $\sF$ on $\Omega'$,
Theorem \ref{ThmResSobWellDefNew} shows that for $\delta\in \R^{\nu}$ sufficiently small,
\begin{equation*}
\NLpN{p}{\delta}{\gamma,e,N}{f}\approx \NLpN{p}{\delta}{\gammat',\et',\Nt'}{f}, \quad f\in C_0^{\infty}(\Omega').
\end{equation*}
\eqref{EqnResRadonOtherToShowBoundMain} follows and completes the proof in Case I.
The same proof goes through in Case II, where in each step we do not need to restrict to $\delta$, $\deltat$ small.
\end{proof}

Using Theorem \ref{ThmResRadonOtherMainBound}, we proceed as in Section \ref{SectionResSobCompNew} to conclude
mapping properties of fractional Radon transforms on the space $\NLp{p}{\deltah}{\sSh}$ for $\deltah\in \R^{\nuh}_0$ in Case I,
and for $\deltah\in \R^{\nuh}$ in Case II.

Let $\lambda_1$ and $\lambda_2$ be two matrices with entries in $[0,\infty]$.  $\lambda_1$ a $\nut\times \nuh$ matrix and $\lambda_2$
a $\nuh\times \nut$ matrix.  We impose the following additional assumptions in both Case I and Case II.
\begin{enumerate}[(i)]
\item\label{ItemResRadonOtherAssumpi} 
We assume $\sL(\sSh)$ $\lambda_1$-controls $\sSt$ on $\Omega'$.

\item 
We assume $\sL(\sSt)$ $\lambda_2$-controls $\sSh$ on $\Omega'$.
%
\end{enumerate}
As before, in what follows we define $\infty\cdot 0=0$ but $\infty\cdot x=\infty$ for $x>0$.

\begin{cor}\label{CorResRadonOtherUpTopCor}
Under the above hypotheses, there exists $a>0$ such that for $1<p<\infty$, the following holds.
\begin{itemize}
\item In Case I, for all $\deltat\in \R^{\nut}_0\cap [0,\infty)^{\nut}$ and $\deltah\in \R^{\nuh}_0\cap [0,\infty)^{\nuh}$
with $\lambda_1^{t}(\deltat)$ and $\lambda_2^{t}(\deltah)$ not $\infty$ in any coordinate, and every $\delta\in \R^{\nu}_0$,
we have for every fractional Radon transform $T$ of order $\deltat-\lambda_2^t(\deltah)$ corresponding to $(\gammat,\et,\Nt)$
on $B^{\Nt}(a)$,
\begin{equation*}
\NLpN{p}{\delta}{\sS}{Tf}\lesssim \NLpN{p}{\delta+(\lambda_1^t(\deltat)-\deltah,0_{\nut})}{\sS}{f},\quad f\in C_0^\infty(\Omega_0).
\end{equation*}
More precisely, there exists $\epsilon=\epsilon(p,(\gamma,e,N),(\gammat,\et,\Nt),\lambda_1,\lambda_2)>0$ such that for all
$\deltat\in [0,\infty)^{\nut}$, $\deltah\in [0,\infty)^{\nuh}$, and $\delta\in \R^{\nu}$ with $|\deltat|,|\deltah|,|\delta|<\epsilon$
and such that $\lambda_1^{t}(\deltat)$ and $\lambda_2^{t}(\deltah)$ are not $\infty$ in any coordinate,
we have for every fractional Radon transform $T$ of order $\deltat-\lambda_2^t(\deltah)$ corresponding to $(\gammat,\et,\Nt)$
on $B^{\Nt}(a)$,
\begin{equation*}
\NLpN{p}{\delta}{\gamma,e,N}{Tf}\lesssim \NLpN{p}{\delta+(\lambda_1^t(\deltat)-\deltah,0_{\nut})}{\gamma,e,N}{f},\quad f\in C_0^\infty(\Omega_0).
\end{equation*}

\item In Case II, for all $\deltat\in  [0,\infty)^{\nut}$ and $\deltah\in  [0,\infty)^{\nuh}$
with $\lambda_1^{t}(\deltat)$ and $\lambda_2^{t}(\deltah)$ not $\infty$ in any coordinate, and every $\delta\in \R^{\nu}$,
we have for every fractional Radon transform $T$ of order $\deltat-\lambda_2^t(\deltah)$ corresponding to $(\gammat,\et,\Nt)$
on $B^{\Nt}(a)$,
\begin{equation*}
\NLpN{p}{\delta}{\sS}{Tf}\lesssim \NLpN{p}{\delta+(\lambda_1^t(\deltat)-\deltah,0_{\nut})}{\sS}{f},\quad f\in C_0^\infty(\Omega_0).
\end{equation*}
\end{itemize}
\end{cor}
\begin{proof}
In Case I, using Theorem \ref{ThmResRadonOtherMainBound} and two applications of Theorem \ref{ThmResCompSobMainThmNew},
we have
\begin{equation*}
\begin{split}
&\NLpN{p}{\delta}{\gamma,e,N}{Tf} \lesssim \NLpN{p}{\delta+(0_{\nuh}, \deltat-\lambda_2^{t}(\deltah))}{\gamma,e,N}{f}
\\&\lesssim \NLpN{p}{\delta+(\lambda_1^{t}(\deltat),-\lambda_2^{t}(\deltah) )}{\gamma,e,N}{f}
\lesssim \NLpN{p}{\delta+(\lambda_1^{t}(\deltat)-\deltah,0_{\nut} )}{\gamma,e,N}{f},
\end{split}
\end{equation*}
as desired.
A similar proof yields the result in Case II.
\end{proof}

\begin{cor}\label{CorResRadonOtherDownBelowCor}
Under the above hypotheses, there exists $a>0$ such that for $1<p<\infty$, the following holds.
\begin{itemize}
\item In Case I, for all $\deltat\in \R^{\nut}_0\cap [0,\infty)^{\nut}$ and $\deltah\in \R^{\nuh}_0\cap [0,\infty)^{\nuh}$
with $\lambda_1^{t}(\deltat)$ and $\lambda_2^{t}(\deltah)$ not $\infty$ in any coordinate,  and every $\delta\in \R^{\nuh}_0$,
we have for every fractional Radon transform $T$ of order $\deltat-\lambda_2^t(\deltah)$ corresponding to $(\gammat,\et,\Nt)$
on $B^{\Nt}(a)$,
\begin{equation*}
\NLpN{p}{\delta}{\sSh}{Tf}\lesssim \NLpN{p}{\delta+\lambda_1^{t}(\deltat)-\deltah}{\sSh}{f},\quad f\in C_0^\infty(\Omega_0).
\end{equation*}
More precisely, there exists $\epsilon=\epsilon(p,(\gammah,\eh,\Nh),(\gammat,\et,\Nt),\lambda_1,\lambda_2)>0$, such that
for every $\deltat\in  [0,\infty)^{\nut}$ and $\deltah\in  [0,\infty)^{\nuh}$
with $\lambda_1^{t}(\deltat)$ and $\lambda_2^{t}(\deltah)$ not $\infty$ in any coordinate,  and every $\delta\in \R^{\nuh}_0$,
with $|\deltat|,|\deltah|,|\delta|<\epsilon$,
we have for every fractional Radon transform $T$ of order $\deltat-\lambda_2^t(\deltah)$ corresponding to $(\gammat,\et,\Nt)$
on $B^{\Nt}(a)$,
\begin{equation*}
\NLpN{p}{\delta}{\gammah,\eh,\Nh}{Tf}\lesssim \NLpN{p}{\delta+\lambda_1^{t}(\deltat)-\deltah}{\gammah,\eh,\Nh}{f},\quad f\in C_0^\infty(\Omega_0).
\end{equation*}

\item In Case II, for all $\deltat\in [0,\infty)^{\nut}$ and $\deltah\in [0,\infty)^{\nuh}$
with $\lambda_1^{t}(\deltat)$ and $\lambda_2^{t}(\deltah)$ not $\infty$ in any coordinate,  and every $\delta\in \R^{\nuh}$,
we have for every fractional Radon transform $T$ of order $\deltat-\lambda_2^t(\deltah)$ corresponding to $(\gammat,\et,\Nt)$
on $B^{\Nt}(a)$,
\begin{equation*}
\NLpN{p}{\delta}{\sSh}{Tf}\lesssim \NLpN{p}{\delta+\lambda_1^{t}(\deltat)-\deltah}{\sSh}{f},\quad f\in C_0^\infty(\Omega_0).
\end{equation*}

\end{itemize}
\end{cor}
\begin{proof}
In Case I, pick $(\gamma,e,N)$ as above (e.g., the choice given in Proposition \ref{PropResSurfExistsAParam} when applied to $\sS$).  Using Theorem \ref{ThmResCompSobDropParamsNew}
and Corollary \ref{CorResRadonOtherUpTopCor} we have
\begin{equation*}
\begin{split}
&\NLpN{p}{\delta}{\gammah,\eh,\Nh}{Tf}\approx \NLpN{p}{(\delta, 0_{\nut})}{\gamma,e,N}{Tf}
\\&\lesssim
\NLpN{p}{(\delta+\lambda_1^{t}(\deltat)-\deltah,0_{\nut})}{\gamma,e,N}{f}\approx \NLpN{p}{\delta+\lambda_1^{t}(\deltat)-\deltah}{\gammah,\eh,\Nh}{f},
\end{split}
\end{equation*}
for $f\in C_0^{\infty}(\Omega_0)$, yielding the result in Case I.  A similar proof yields the result in Case II.
\end{proof}

\begin{rmk}
Though it is not necessary, in Corollaries \ref{CorResRadonOtherUpTopCor} and \ref{CorResRadonOtherDownBelowCor} one often
wishes to choose $\deltat$ and $\deltah$ so that $\deltat$ is zero in every coordinate in which $\lambda_2^{t}(\deltah)$ is nonzero,
and $\lambda_2^{t}(\deltah)$ is zero in every coordinate where $\deltat$ is nonzero.
\end{rmk}

\begin{rmk}
In Corollaries \ref{CorResRadonOtherUpTopCor} and \ref{CorResRadonOtherDownBelowCor}
if one takes one of the matricies ($\lambda_1$ or $\lambda_2$) to be $+\infty$ in every component,
then it is as if that matrix were not present in the assumptions and conclusions at all.
For instance, if one takes $\lambda_1$ to be $+\infty$ in every component, then one is forced
to take $\deltat=0$ and the assumption (\ref{ItemResRadonOtherAssumpi}), above, holds automatically.
Most previous work on this topic (e.g., \cite{CuccagnaSobolevEstimatesForFractionalAndSingularRadonTransforms} and \cite{GreenblattAnAnalogueToATheoremOfFeffermanAndPhong}) does not involve $\lambda_1$ (i.e., takes $\lambda_1$ to
be $+\infty$ in every component), and deals only with $\lambda_2$ in very special cases.
\end{rmk} 

%% file: resradonotherhor3.tex
The special case of Case I of Corollary \ref{CorResRadonOtherDownBelowCor} which is likely of most interest is when $\nut=\nuh=1$ and the
$\Xt$ and $\Xh$ vector fields each satisfy H\"ormander's condition.  Below we present this situation.

We start with a parameterization $(\gammat,\et,\Nt,\Omega,\Omega''')$, where $\et_1,\ldots, \et_{\Nt}\in (0,\infty)$--i.e., we have
{\it single}-parameter dilations.  Let $(\Wt,\et,\Nt)$ be the vector field parameterization corresponding to $(\gammat,\et,\Nt)$.
Expanding $\Wt(\vtt,x)$ as a Taylor series in the $\vtt$ variable,
\begin{equation*}
\Wt(\vtt)\sim \sum_{|\alpha|>0} \vtt^{\alpha} \Xt_{\alpha},
\end{equation*}
where $\Xt_{\alpha}$ is a smooth vector field on some $\Omega''$ with $\Omega'\Subset \Omega''$.  We suppose $\{\Xt_{\alpha}:|\alpha|>0\}$
satisfies H\"ormander's condition on $\Omega''$.
Fix $\Omegat$ with $\Omega'\Subset \Omegat\Subset \Omega''$.
 By Corollary \ref{CorResSurfHorSmoothlyFG}, $(\gammat,\et,\Nt)$ is smoothly finitely generated
by some $\sFt$ on $\Omegat$.

We suppose we are given a finite set of vector fields $\sSh\subset \sonevect$, such that
$\{ \Xh :\exists \hd, (\Xh,\hd)\in \sSh\}$ satisfies H\"ormander's condition on $\Omega''$.
By Proposition \ref{PropResCCHorSmoothFG} and Remark \ref{RmkResCCnuOneFG},
$\sL(\sSh)$ is smoothly linearly finitely generated by some $\sFh$ on $\Omegat$.
By Theorem \ref{ThmResSobWellDefNew} (see also
Definition \ref{DefnResSobDependOnsS})
it makes sense to talk about the norm
$\NLpN{p}{\delta}{\sSh}{\cdot}$ for all $1<p<\infty$, $\delta\in(0,\infty)$.

For each $(\Xt,\td)\in \sFt$, Let $\sFh_{(\Xt,\td)}\subseteq \sFh$ be such that
$\Xt$ is in the $C^\infty(\Omega'')$ module generated by
$\{\Xh : \exists \hd, (\Xh,\hd)\in \sFh_{(\Xt,\td)}\}$.\footnote{This is always possible because the vector fields
in $\sFh$ span the tangent space at every point, by assumption.}
Define
\begin{equation}\label{EqnResRadonOtherHorDefnLambdaOXD}
\lambda_1^{(\Xt,\td)} :=\frac{\min \max \{ \hd : \exists \Xh, (\Xh,\hd)\in \sFh_{(\Xt,\td)}\}}{\td},
\end{equation}
where the minimum is taken over all possible choices of $\sFh_{(\Xt,\td)}$.
And set
\begin{equation*}
\lambda_1 := \max \{\lambda_1^{(\Xt,\td)} : (\Xt,\td)\in \sFt \}>0.
\end{equation*}
Define $\lambda_2>0$ in the same way by reversing the roles of $\sFh$ and $\sFt$ throughout. 

\begin{cor}\label{CorResRadonOtherHor3Main}
Under the above hypotheses, there exists $a>0$ such that for every $1<p<\infty$, there exists $\epsilon=\epsilon(p,(\gammat,\et,\Nt),\sSh)>0$
such that for every $\deltat,\deltah\in [0,\epsilon)$, $\delta\in (-\epsilon,\epsilon)$, we have for every fractional Radon transform, $T$,
of order $\deltat-\lambda_2\deltah$ corresponding to $(\gammat,\et,\Nt)$ on $B^{\Nt}(a)$,
\begin{equation*}
\NLpN{p}{\delta}{\sSh}{Tf}\lesssim \NLpN{p}{\delta+\lambda_1\deltat-\deltah}{\sSh}{f}, \quad f\in C_0^{\infty}(\Omega_0).
\end{equation*}
\end{cor}
\begin{proof}
By Proposition \ref{PropResCCHorSmoothFG}, if $\sS$ is given by \eqref{EqnAssumptionResRadonOtherMainsSNew}, then $\sL(\sS)$
is linearly finitely generated on $\Omega'$.  Corollary \ref{CorResSurfHorSmoothlyFG} shows
$(\gammat,\et,\Nt)$ is smoothly finitely generated, and it follows that $\sL(\sS)$ controls $(\gammat,\et',\Nt)$ on $\Omega'$  where $\et'$
is as in the assumptions from Case I, above.  See, also, Remark \ref{RmkResRadonOtherSmoothWorks}.

The result will follow from Corollary \ref{CorResRadonOtherDownBelowCor}
once we show:
\begin{enumerate}[(i)]
\item\label{ItemResRadonOtherHorToShow} $\sL(\sSh)$ $\lambda_1$-controls $\sFt$ on $\Omega'$.

\item\label{ItemResRadonOtherHorToShowAlso} $\sFt$ $\lambda_2$-controls $\sSh$ on $\Omega'$.
\end{enumerate}
We begin with (\ref{ItemResRadonOtherHorToShow}).
Let $(\Xt,\td)\in \sFt$.  To show $(\Xt,\td)$ is $\lambda_1$-controlled by $\sL(\sSh)$ on $\Omega'$, it suffices to show
$(\Xt,\td)$ is $\lambda_1$-controlled by $\sFh$ on $\Omega'$.  Let $\sFh_{(\Xt,\td)}$ achieve the minimum
in the definition of $\lambda_1^{(\Xt,\td)}$ in \eqref{EqnResRadonOtherHorDefnLambdaOXD}.
We will show $\sFh_{(\Xt,\td)}$ $\lambda_1^{(\Xt,\td)}$-controls $(\Xt,\td)$ on $\Omega'$,
and it then follows that $\sFh$ $\lambda_1$-controls $(\Xt,\td)$ on $\Omega'$, as desired.

By the definition of $\sFh_{(\Xt,\td)}$, we may write
\begin{equation*}
\Xt = \sum_{(\Xh,\hd)\in \sFh_{(\Xt,\td)}} c_{\Xh,\hd} \Xh,
\end{equation*}
where $c_{\Xh}\in C^{\infty}(\Omegat)$.  Multiplying both sides by $2^{-\lambda_1^{(\Xt,\td)}\jh\td}$ we obtain,
\begin{equation*}
2^{-\lambda_1^{(\Xt,\td)}\jh\td} \Xt = \sum_{(\Xh,\hd)\in \sFh_{(\Xt,\td)}} \q(2^{\jh\hd-\lambda_1^{(\Xt,\td)}\jh\td }c_{\Xh,\hd}\w) 2^{-\jh \hd} \Xh.
\end{equation*}
The choice of $\lambda_1^{(\Xt,\td)}$ shows $\jh \hd-\lambda_1\jh\td\leq 0$ for all $(\Xh,\hd)\in \sFh_{(\Xt,\td)}$.  From here, (\ref{ItemResRadonOtherHorToShow})
follows immediately.

For (\ref{ItemResRadonOtherHorToShowAlso}), note that (using $\sFh$ controls $\sSh$ on $\Omega'$) to show $\sFt$ $\lambda_2$-controls $\sSh$ on $\Omega'$,
it suffices to show $\sFt$ $\lambda_2$-controls $\sFh$ on $\Omega'$ (because $\sFh$ controls $\sSh$ on $\Omega'$).
From here, the proof follows just as in the proof for (\ref{ItemResRadonOtherHorToShow}).
\end{proof}

\begin{rmk}
Corollary \ref{CorResRadonOtherHor3Main} is often optimal.  See
Theorem \ref{ThmOptimalMain} for details.
\end{rmk}

Define $\lambda_1',\lambda_2'\in (0,\infty)$ by
\begin{equation}\label{EqnResRadonOtherHor3New}
\lambda_1' := \frac{\max\{ \hd : \exists (\Xh,\hd)\in \sFh\} }{\min\{ \td : \exists (\Xt,\td)\in \sFt\} }, \quad
\lambda_2' := \frac{\max\{ \td : \exists (\Xt,\td)\in \sFt\}}{\min\{ \hd : \exists (\Xh,\hd)\in \sFh\} }.
\end{equation}

\begin{cor}\label{CorResRadonOtherHor3Follows}
Under the above hypotheses, there exists $a>0$ such that for every $1<p<\infty$, there exists $\epsilon=\epsilon(p,(\gammat,\et,\Nt),\sSh)>0$
such that for every $\deltat,\deltah\in [0,\epsilon)$, $\delta\in (-\epsilon,\epsilon)$, we have for every fractional Radon transform, $T$,
of order $\deltat-\lambda_2'\deltah$ corresponding to $(\gammat,\et,\Nt)$ on $B^{\Nt}(a)$,
\begin{equation*}
\NLpN{p}{\delta}{\sSh}{Tf}\lesssim \NLpN{p}{\delta+\lambda_1'\deltat-\deltah}{\sSh}{f}, \quad f\in C_0^{\infty}(\Omega_0).
\end{equation*}
\end{cor}
\begin{proof}
Because $\lambda_1\leq \lambda_1'$ and $\lambda_2\leq \lambda_2'$, this follows immediately from Corollary \ref{CorResRadonOtherHor3Main}.
\end{proof}

\begin{rmk}\label{RmkResRadonOtherHor3WhatAreLambdas}
The conclusion in 
Corollary \ref{CorResRadonOtherHor3Follows} depends on the choice of $\sFt$ and $\sFh$.  One wishes to pick them so that $\lambda_1'$ and $\lambda_2'$
are as small as possible.  The conclusion of the stronger result in Corollary \ref{CorResRadonOtherHor3Main} does not depend on the choices of $\sFt$ and $\sFh$.
In an application of Corollary \ref{CorResRadonOtherHor3Follows}, one can pick $\sFh$ so that $\max\{ \hd : \exists (\Xh,\hd)\in \sFh\}$ is equal to:
\begin{equation*}
\min_{\sF} \max\q\{\hd : \exists (\Xh,\hd)\in \sF\w\}
\end{equation*}
where the minimum is taken over all $\sF\subseteq \sL(\sSh)$ such that the vector fields in $\sF$ span the tangent space to every point of $\Omega''$;
and so that $\min\{\hd :\exists (\Xh,\hd)\in \sFh\}$ is equal to
\begin{equation*}
\min\q\{\hd : \exists (\Xh,\hd)\in \sSh, \Xh\text{ is not the zero vector field}\w\}.
\end{equation*}
Similarly for $\sFt$.  See the proof of Proposition \ref{PropResCCHorSmoothFG} for how to choose such $\sFt$ and $\sFh$.
\end{rmk}

An important special case of Corollary \ref{CorResRadonOtherHor3Follows} comes when $\sSh=(\partial,1)$ (see \eqref{EqnResSobEuclidPartialOneNew} for this
notation).  In this case, $\NLpN{p}{\delta}{\partial,1}{\cdot}\approx \LpsN{p}{\delta}{\cdot}$ (Theorem \ref{ThmResSobEuclidIsNonIsoNew}).
To present this case, we change perspective and state the result just near some fixed point $x_0\in  \Omega'''\Subset \Omega\subseteq \R^n$.
We suppose we are given a parameterization $(\gammat,\et,\Nt,\Omega,\Omega''')$ with single-parameter dilations $\et$, and with corresponding vector field parameterization
$(\Wt,\et,\Nt)$.  We expand $\Wt(\vtt)$ into a Taylor series in the $\vtt$ variable:
\begin{equation*}
\Wt(\vtt)\sim \sum_{|\alpha|>0} \vtt^{\alpha} \Xt_{\alphat}.
\end{equation*}
We assume the following.

\noindent\textbf{Assumption:}  The Lie algebra generated by $\{ \Xt_{\alphat} : |\alphat|>0\}$ spans the tangent space at $x_0$.

Let $\sSt:=\{(\Xt,\deg(\alphat)) : |\alphat|>0\}$ as before.\footnote{Here, $\deg(\alphat)$ is defined using the single parameter dilations $\et$; see Definition \ref{DefnResKerDegree}.}
We define two numbers:
\begin{equation*}
E := \min_{\sF_0} \max \{ \td : \exists (\Xt,\td)\in \sF_0\}, 
\end{equation*}
and the minimum is taken over all $\sF_0\subseteq \sL(\sSt)$ such that the vector fields in $\sF_0$ span the tangent space at $x_0$.
Also set
\begin{equation*}
e:= \min\q\{\deg(\alphat) :  \Xt_{\alphat}\text{ not the zero vector on a neighborhood of }x_0\w\}.
\end{equation*}

\begin{cor}\label{CorResRadonOtherHorWithEuclidNewer}
Under the above hypotheses, there exists an open set $\Omega'\Subset \Omega'''$ with $x_0\in \Omega'$ and $a>0$ such that
for $1<p<\infty$, there exists $\epsilon=\epsilon(p,(\gammat,\et,\Nt))>0$ such that for every $r,s\in (-\epsilon,\epsilon)$,
if $T$ is a fractional Radon transform of order $r$ corresponding to $(\gammat,\et,\Nt)$ on $B^{\Nt}(a)$ (with this choice of $\Omega'$
in the definition of a fractional Radon transform),
\begin{itemize}
\item If $r\geq 0$,
\begin{equation}\label{EqnCorResRadonOtherHorWithEuclidNew1Newer}
\LpsN{p}{s-\frac{r}{e}}{Tf}\lesssim \LpsN{p}{s}{f}, \quad f\in C_0^\infty(\Omega').
\end{equation}
Furthermore, this result is optimal in the sense that there do not exist $p\in (1,\infty)$, $t>0$, $s\in (-\epsilon,\epsilon)$,\footnote{Recall, $\epsilon$ depends on $p\in (1,\infty)$.} and $r\in [0,\epsilon)$ such that for every fractional Radon transform, $T$, of order $r$ corresponding to $(\gammat,\et,\Nt)$ on $B^{\Nt}(a)$  we have
\begin{equation*}
\LpsN{p}{s-\frac{r}{e}+t}{Tf}\lesssim \LpsN{p}{s}{f}, \quad f\in C_0^\infty(\Omega').
\end{equation*}

\item If $r\leq 0$,
\begin{equation}\label{EqnCorResRadonOtherHorWithEuclidNew2Newer}
\LpsN{p}{s-\frac{r}{E}}{Tf}\lesssim \LpsN{p}{s}{f}, \quad f\in C_0^\infty(\Omega').
\end{equation}
Furthermore, this result is optimal in the sense that there do not exist $p\in (1,\infty)$, $t>0$, $s\in (-\epsilon,\epsilon)$ and $r\in (-\epsilon,0]$ such that for every fractional Radon transform, $T$, of order $r$ corresponding to $(\gammat,\et,\Nt)$ on $B^{\Nt}(a)$  we have
\begin{equation*}
\LpsN{p}{s-\frac{r}{E}+t}{Tf}\lesssim \LpsN{p}{s}{f}, \quad f\in C_0^\infty(\Omega').
\end{equation*}
\end{itemize}
\end{cor}
\begin{proof}
To prove \eqref{EqnCorResRadonOtherHorWithEuclidNew1Newer} and \eqref{EqnCorResRadonOtherHorWithEuclidNew2Newer} we wish to
apply Corollary \ref{CorResRadonOtherHor3Follows}.  We are taking $\sFh=(\partial,1)$, and therefore
$\min\{ \hd : \exists (\Xh,\hd)\in \sFh\}=1=\max\{ \hd : \exists (\Xh,\hd)\in \sFh\}$.
By the discussion in Remark \ref{RmkResRadonOtherHor3WhatAreLambdas},
we may pick a small neighborhood $\Omega'$ of $x_0$ so that we may take $\sFt$ 
with $\min\{\td : \exists (\Xt,\td)\in \sFt\}=e$, $\max\{\td :\exists(\Xt,\td)\in \sFt\}=E$.
Thus, $\lambda_1'=e^{-1}$ and $\lambda_2'=E$.

If $r\geq 0$, set $\deltah=0$ and $\deltat=r$.  Then if $s=\delta+\lambda_1'\deltat-\deltah$, we have $\delta=s- \lambda_1'r$.
Plugging these choices into Corollary \ref{CorResRadonOtherHor3Follows} yields \eqref{EqnCorResRadonOtherHorWithEuclidNew1Newer}.

If $r\leq 0$, set $\deltat=0$ and $r=-\lambda_2' \deltah$.  Then, if $s=\delta+\lambda_1' \deltat-\deltah$, we have $\delta=s-r/\lambda_2'$.
Plugging these choices into Corollary \ref{CorResRadonOtherHor3Follows} yields \eqref{EqnCorResRadonOtherHorWithEuclidNew2Newer}.

For the proof of optimality, see Section \ref{SectionOptimality}.
\end{proof}

%% file: resradonpdo.tex
When $(\gamma,e,N,\Omega,\Omega''')$ is linearly finitely generated on $\Omega'$, and $T$ is a fractional
Radon transform of order $\delta\in \R^{\nu}$ corresponding to $(\gamma,e,N)$, it is sometimes useful to think
of $T$ as a generalized kind of ``pseudodifferential operator''.
Actually, for this we consider a slightly more general kind of operator:
\begin{equation}\label{EqnResRadonPDODefnPDO}
Tf(x) = \int f(\gamma_t(x)) \psi(\gamma_t(x)) K(x,t)\: dt,
\end{equation}
where $\psi\in C_0^{\infty}(\Omega_0)$, $K(x,t)$ is a distribution which can be written as
\begin{equation*}
K(x,t) = \eta(t) \sum_{j\in \N^{\nu}} 2^{j\cdot \delta} \dil{\vsig_j}{2^j}(x,t),
\end{equation*}
 $\dil{\vsig_j}{2^j}(x,t)=2^{j\cdot e_1+\cdots + j\cdot e_N} \vsig_j(x, 2^{j}t)$, $2^j t$ is defined by the dilations $e$,
$\{\vsig_j : j\in \N^{\nu}\}\subset C^{\infty}_0(\Omega_0 ; \schS(\R^{N}))$ is a bounded set, with
$\vsig_j\in C^{\infty}_0(\Omega_0; \schS_{\{\mu : j_\mu\ne 0\}})$.
The results above actually extend to this more general situation, automatically.  Indeed, because
$C_0^{\infty}(\Omega_0; \schS(\R^{N}))\cong C_0^{\infty}(\Omega_0)\otimesh \schS(\R^{N})$ (where $\otimesh$ denotes the
completed tensor product of these nuclear spaces), all of our results for fractional Radon transforms extend to the more
general operators given by \eqref{EqnResRadonPDODefnPDO}.  See \cite{TrevesTopologicalVectorSpaces} for more details on tensor products,
and Theorem 2.14.16 of \cite{StreetMultiParamSingInt} for a similar result using these ideas.

\begin{rmk}\label{RmkResRadonPDOEuclidPDO}
In the case when $\sF=(\partial,1)$ (see \eqref{EqnResSobEuclidPartialOneNew}), and when 
\begin{equation*}
\gamma_{t_1,\ldots, t_n}(x) = e^{-t_1 \frac{\partial}{\partial x_1}-\cdots - t_n \frac{\partial}{\partial x_n}} x = x-t,
\end{equation*}
then fractional Radon transforms of order $\delta\in \R$ corresponding to $(\gamma,(1,\ldots, 1), n)$ are standard pseudodifferential
operators order order $\delta$ whose Schwarz kernels are supported in $\Omega_0\times \Omega_0$.
\end{rmk}

We saw in Remark \ref{RmkResRadonPDOEuclidPDO} that a particular special case of fractional Radon transforms
corresponding to a linearly finitely generated parameterization yields standard pseudodifferential operators on $\R^n$.
The analogy with pseudodifferential operators does not end there, though.  Indeed, a basic use of standard pseudodifferential
operators is to create parametricies for elliptic operators (e.g., the Laplacian on $\R^n$).  When $\nu=1$, the fractional
Radon transforms here can be used to create a parametrix for H\"omander's sub-Laplacian.
This idea was developed by Rothschild and Stein \cite{RothschildSteinHypoellipticDifferentialOperatorsAndNilpotentGroups}
and was based on previous work by Folland and Stein \cite{FollandSteinEstimatesForTheDbarComplex} and Folland
\cite{FollandSubellipticEstimatesAndFunctionSpacesOnNilpotent}.  This was further developed by Goodman
\cite{GoodmanNilpotentLieGroups}; see also \cite{GellerSteinEstimatesForSingularConvolutionOperatorsOnTheHeisenbergGroup,PhongSteinHilbertIntegrals,ChirstGellerGlowackiPolinPseudodifferentialOperatorsOnGroupsWithDilations}.  See \cite{StreetMultiParamSingInt} for more details; in particular, Theorem 2.14.28.
Combining this with the other results in this paper, gives regularity results for H\"ormander's sub-Laplacian
on various non-isotropic Sobolev spaces corresponding to geometries other than
the associated Carnot-Carath\'eodory geoemtry.\footnote{In the case of {\it isotropic} Sobolev spaces, this idea was already
present in the work of Rothschild and Stein \cite{RothschildSteinHypoellipticDifferentialOperatorsAndNilpotentGroups}.}

In fact, the operators discussed here are closely related to a far reaching generalization of ellipticity, known as
maximal hypoellipticity.  See Chapter 2 of \cite{StreetMultiParamSingInt} for this concept, its relationship with these pseudodifferential operators,
and a history of these ideas.

\begin{rmk}\label{RmkFeffermanPhong}
In \cite{GreenblattAnAnalogueToATheoremOfFeffermanAndPhong} results concerning fractional Radon transforms were connected
to the well-known results of Fefferman and Phong on subelliptic operators
\cite{FeffermanPhongSubellipticEigenValueProblems}.
Here, we can make the analogy more explicit:  the results of \cite{FeffermanPhongSubellipticEigenValueProblems} are closely related
to the case when $\gamma$ is linearly finitely generated (e.g., when studying H\"ormander's sub-Laplacian), while the results of \cite{GreenblattAnAnalogueToATheoremOfFeffermanAndPhong}
are for when $\gamma$ is finitely generated.
\end{rmk}

%% file: resradonsing.tex
There is a close relationship here between the smoothing properties for Radon transforms and the corresponding smoothing properties
for singular integrals.  Indeed, suppose $(\gamma,e,N)$ is finitely generated by $\sF\subset \snuvectone$ on $\Omega'$.  Note, we
are taking $\sF\subset \snuvectone$, but not assuming $\gamma$ is linearly finitely generated on $\Omega'$.
Corresponding to $\sF$ we obtain nonisotropic Sobolev spaces $\NLp{p}{r}{\sF}$, $p\in (1,\infty)$, $r\in \R^{\nu}$.
We assume, in addition, $\{ X: (X,d)\in \sF\}$ spans the tangent space to every point of $\Omega'$.

Under the above assumptions, corresponding to $\sF$, there is an algebra of singular integral operators (see \cite{StreetMultiParamSingInt}).\footnote{When
$\nu=1$ these singular integral operators are the NIS operators introduced by Nagel, Rosay, Stein, and Wainger \cite{NagelRosaySteinWaingerEstimatesForTheBergmanAndSzegoKernels} and later studied by
Koenig \cite{KoenigOnMaximalSobolevAndHolderEstimatesForTheTangentialCR}
and Chang, Nagel, and Stein \cite{ChangNagelSteinEstaimtesForDbarNeumannProblem}.}
If $S$ is a singular integral operator of order $\delta\in \R^{\nu}$ (corresponding to $\sF$), then for $1<p<\infty$, $r\in \R^\nu$, 
$S: \NLp{p}{r}{\sF}\rightarrow \NLp{p}{r-\delta}{\sF}$ (see Theorem 5.1.23 of \cite{StreetMultiParamSingInt}).
Furthermore, if $(\gamma,e,N)$ is {\it linearly} finitely generated by $\sF$ on $\Omega'$, and if $T$ is a fractional Radon transform of 
order $r\in \R^{\nu}$ corresponding to $(\gamma,e,N)$, then $T$ is a singular integral operator of order $r$
(the results in Section 5.2.1 of \cite{StreetMultiParamSingInt} can be adapted to this situation).
Thus, Theorem \ref{ThmResRadonMainThm} in the case when $(\gamma,e,N)$ is linearly finitely generated on $\Omega'$ (and $\sF$ satisfies the above hypotheses) is really a result about singular integrals;
and is therefore essentially a special case of  Theorem 5.1.23 of \cite{StreetMultiParamSingInt}.

If $(\gamma,e,N)$ is only finitely generated on $\Omega'$, then $T$ is not necessarily a singular integral operator.
However, we do have $T: \NLp{p}{r}{\sF}\rightarrow \NLp{p}{r-\delta}{\sF}$ for $r,\delta\in \R_0^{\nu}$.
Thus, one way of informally restating Theorem \ref{ThmResRadonMainThm} (at least in the case when $\sF$ satisfies the above hypotheses)
is that the mapping properties of fractional Radon transforms on $\NLp{p}{r}{\sF}$ are the same as the mapping properties
of the corresponding singular integral operators, so long as $r$ and $\delta$ are sufficiently small.

%% file: pfker.tex
Fix $N\in \N$ and $0\ne e_1,\ldots, e_N\in [0,\infty)^\nu$--$\nu$ parameter dilations on $\R^N$.  For $a>0$ and $\delta\in \R^{\nu}$,
we wish to understand the space $\sK_{\delta}(N,e,a)$.  The first step is to understand the spaces $\schS_{E}$, where $E\subseteq \nuset$.\footnote{Recall, $\schS(\R^N)$ is Schwartz space, and the definition of $\schS_E\subseteq \schS(\R^N)$ is given at the beginning of Section \ref{SectionResKer}.}
$\schS_E$ is clearly a closed subspace of $\schS(\R^N)$, and therefore inherits the Fr\'echet topology.
Using the dilations $e$, it makes sense to write $\dil{\vsig}{2^j}$, for $j\in [0,\infty)^\nu$, as in \eqref{EqnResKerDefFuncDil}.

For each $\mu\in \nuset$, recall the variable $t_\mu$, which denotes the vector consisting of those coordinates $t_j$ of $t$ such that $e_j^{\mu}\ne 0$.
Let $t_\mu^{\perp}$ denote the vector consisting of those coordinates of $t$ which are not in $t_\mu$, so that $t=(t_\mu, t_\mu^{\perp})$.
For $f\in \schS(\R^N)$, let $\hat{f}\in \schS(\R^N)$ denote its Fourier transform.  Let $\xi_\mu$ denote the dual variable to $t_\mu$,
and define $\xi_\mu^{\perp}$ so that $\xi=(\xi_\mu, \xi_{\mu}^{\perp})$; i.e., $\xi_\mu^{\perp}$ is the dual variable to $t_\mu^{\perp}$.

\begin{lemma}\label{LemmaKerFTOfsS}
For $f\in \schS(\R^N)$, the following are equivalent:
\begin{itemize}
\item $f\in \schS_E$.
\item $\forall \mu\in E$, $\partial_{\xi_\mu}^{\alpha} \hat{f}(\xi_\mu, \xi_\mu^{\perp})\big|_{\xi_\mu=0}=0$, for every multi-index $\alpha$.
\end{itemize}
\end{lemma}
\begin{proof}
This is immediate from the definitions.
\end{proof}

Decompose the $t_\mu$ variable as $t_\mu = (t_\mu^1,\ldots, t_\mu^{N_\mu})$, so that $t_\mu^1,\ldots, t_\mu^{N_\mu}\in \R$ are an enumeration
of those coordinates $t_j\in \R$ of $t\in \R^{N}$ such that $e_j^{\mu}\ne 0$.
Let $\lap_{\mu}$ denote the positive Laplacian in the $t_\mu$ variable, so that
$$\lap_\mu=-\sum_{j=1}^{N_\mu} \q(\frac{\partial}{\partial t_\mu^j}\w)^2=-\grad_{t_{\mu}}\cdot \grad_{t_{\mu}}.$$

On $t_\mu$ there are $\nu$ parameter dilations, defined so that for $\delta\in [0,\infty)^{\nu}$, $\delta t = (\delta t_{\mu}, \delta t_{\mu}^\perp)$.
I.e., we define dilations $0\ne \eh^{\mu}_1,\ldots,\eh^{\mu}_{N_{\mu}}\in [0,\infty)^\nu$ on $\R^{N_\mu}$
by $\eh^{\mu}_k = e_j$ if $t_\mu^k$ corresponds to the coordinate $t_j$ of $t$.
Notice that each $\eh^{\mu}_k\in [0,\infty)^\nu$ is nonzero in the $\mu$ component; thus
when we compute $\delta t_\mu$, each coordinate of $t_\mu$ is multiplied by a power of $\delta_\mu$ (and possibly by powers $\delta_{\mu'}$ for $\mu'\ne \mu$, as well).
Define a $\nu$ parameter dilation on $\lap_\mu$ by
\begin{equation}\label{EqnPfKerDilLap}
\delta\cdot \lap_\mu := -\sum_{j=1}^{N_\mu} \delta^{2\eh_j^{\mu}}\q(\frac{\partial}{\partial t_\mu^j}\w)^2.
\end{equation}
This is defined in such a way that
\begin{equation}\label{EqnKerPullOutLapDil}
\dil{\q(\lap_{\mu} \vsig\w)}{2^j} = \q(2^{-j}\cdot \lap_\mu\w)\dil{\vsig}{2^j}.
\end{equation}
Letting
\begin{equation}\label{EqnKerPutOnLapDilDefnh0}
h_0:= \min\{e_j^{\mu} : 1\leq j\leq N, 1\leq \mu\leq \nu, e_j^{\mu}\ne 0\}>0
\end{equation} we see, for every $M\in \N$,
\begin{equation}\label{EqnKerPutOnLapDil}
\q|\q(2^{-j}\cdot \lap_\mu\w)^M f\w|\leq C_M 2^{-2M j_\mu h_0} \sum_{|\alpha|\leq 2M} \q|\q(\frac{\partial}{\partial t_{\mu}}\w)^\alpha f\w|,
\end{equation}
where $C_M$ does not depend on $j$ or $f$.

\begin{defn}
Let $f\in \schS_E$, with $\mu\in E$.  For $s\in \R$ define $\lap_\mu^s f$ by $\widehat{\lap_\mu^s f} = |\xi_\mu|^{2s} \hat{f}(\xi)$.
\end{defn}

\begin{lemma}\label{LemmaKersSEPullOut}
For $s\in \R$, and $\mu\in E$, $\lap_\mu^s: \schS_E\rightarrow \schS_E$ is an automorphism of the Fr\'echet space $\schS_E$.
\end{lemma}
\begin{proof}
It is a simple consequence of Lemma \ref{LemmaKerFTOfsS} that $\lap_\mu^s:\schS_E\rightarrow \schS_E$.  It follows from the closed graph theorem
that it is continuous.  The continuous inverse is $\lap_\mu^{-s}$, thereby making $\lap_\mu^s$ an automorphism.
\end{proof}

\begin{lemma}\label{LemmaKerConvInDist}
Let $\{\vsig_j : j\in \N^{\nu}\}\subset \schS(\R^N)$ be a bounded set such that $\vsig_j \in \schS_{\{\mu:j_\mu\ne 0\}}$.
Then, for any $s\in \R^{\nu}$, the sum
\begin{equation*}
\sum_{j\in \N^{\nu}} 2^{j\cdot s} \dil{\vsig_j}{2^j}
\end{equation*}
converges in the sense of tempered distributions.
\end{lemma}
\begin{proof}
Let $\vsig\in \schS(\R^N)$.  We show, for every $L$,
\begin{equation}\label{EqnToShowKerSum}
\q|\int \dil{\vsig_j}{2^j}(t) \vsig(t)\: dt\w|\lesssim 2^{-L|j|_{\infty}},
\end{equation}
and the result will follow.

If $j=0$, \eqref{EqnToShowKerSum} is trivial, so we assume $|j|_{\infty}>0$.  Take $\mu$ so that $j_\mu=|j|_{\infty}$, and take $M$ so large
$2M h_0\geq L$ (where $h_0$ is as in \eqref{EqnKerPutOnLapDilDefnh0}).  Lemma \ref{LemmaKersSEPullOut} shows we may write $\vsig_j = \lap_\mu^{M} \vsigt_j$, where $\{\vsigt_j : j\in \N^\nu, j_\mu\ne 0\}\subset \schS(\R^N)$
is a bounded set (by setting $\vsigt_j = \lap_\mu^{-M} \vsig_j$).
Using \eqref{EqnKerPullOutLapDil}, we have
\begin{equation*}
\int \dil{\vsig_j}{2^j}(t)\vsig(t)\: dt = \int \dil{\vsigt_j}{2^j}(t) (2^{-j}\cdot  \lap_\mu)^{M} \vsig(t) \: dt.
\end{equation*}
Using \eqref{EqnKerPutOnLapDil} and the choice of $M$, we have that $\{ 2^{L|j|_{\infty}} (2^{-j}\cdot \lap_\mu)^M \vsig : j\in \N^{\nu}, j\ne 0, j_\mu=|j|_\infty\}\subset \schS(\R^N)$ is a bounded set.
\eqref{EqnToShowKerSum} follows, completing the proof.
\end{proof}

\begin{lemma}\label{LemmaResKerBasicDelta}
There is a bounded set $\{\vsig_j : j\in \N^\nu\}\subset \schS(\R^N)$ with $\vsig_j\in \schS_{\{\mu:j_\mu\ne 0\}}$ and
$\delta_0=\sum_{j\in \N^{\nu}} \dil{\vsig_j}{2^j}$, with the convergence taken in the sense of tempered distributions.
\end{lemma}
\begin{proof}
We decompose $\delta_0$ on the Fourier transform side.  Indeed, let $\phih\in C_0^\infty(\R^N)$ equal $1$ on a neighborhood of $0$.
Define, for $j\in \N^\nu$,
\begin{equation*}
\psih_j(\xi):=\sum_{\substack{ (p_1,\ldots, p_\nu)\in \{0,1\}^\nu \\ j_\mu \geq p_\mu }} (-1)^{p_1+\cdots+p_\nu} \phih(2^p \xi).
\end{equation*}
Notice, $\sum_{j\in \N^{\nu}} \psih_j(2^{-j}\xi)=1$, in the sense of tempered distributions.  Also notice
$\{\psih_j : j\in \N^\nu\}\subset \schS(\R^N)$ is a finite set and therefore bounded.
Letting $\psi_j$ denote the inverse Fourier transform of $\psih_j$, we have
$\sum_{j\in \N^{\nu}} \dil{\psi_j}{2^j}=\delta_0$ and $\{\psi_j : j\in \N^{\nu}\}\subset \schS(\R^N)$ is a bounded set (indeed, it is a finite set).
Finally, notice that, if $j_\mu>0$, $\psih_j(\xi)$ vanishes to infinite order at $0$ in the $\xi_\mu$ variable (it is identically zero on a neighborhood of
$0$ in the $\xi_\mu$ variable).
Lemma \ref{LemmaKerFTOfsS} shows $\psi_j\in \schS_{\{\mu : j_\mu\ne 0\}}$ and completes the proof.
\end{proof}

\begin{proof}[Proof of Lemma \ref{LemmaResKerDelta}]
Let $\eta\in C_0^\infty(B^N(a))$, with $\eta\equiv 1$ on a neighborhood of $0$.  Take $\vsig_j$ as in Lemma \ref{LemmaResKerBasicDelta},
so that $\delta_0 = \sum_{j\in \N^{\nu}} \dil{\vsig_j}{2^j}$.  Thus, $\delta_0 = \eta \delta_0 =\eta \sum_{j\in \N^{\nu}} \dil{\vsig_j}{2^j}$, which proves
the result for $\alpha=0$.  For $|\alpha|>0$, note that
if $\beta_1\ne 0$, $(\partial_{t}^{\beta_1} \eta) \partial_t^{\beta_2} \sum_{j\in \N^{\nu}} \dil{\vsig_j}{2^j} = (\partial_t^{\beta_1} \eta) \partial_t^{\beta_2} \delta_0=0$,
because $\partial_t^{\beta_1} \eta\equiv 0$ on a neighborhood of $0$.
Thus,
$$\partial_t^{\alpha} \delta_0 = \partial_t^{\alpha} \q(\eta \sum_{j\in \N^{\nu}} \dil{\vsig_j}{2^j}\w) = \eta \sum_{j\in \N^{\nu}} \partial_t^{\alpha} \dil{\vsig_j}{2^j}
= \eta \sum_{j\in \N^{\nu}} 2^{j\cdot \deg(\alpha)} \dil{(\partial_t^{\alpha} \vsig_j)}{2^j},
$$
which completes the proof.
\end{proof}

In later sections, we will need some decomposition results about functions in $\schS_E$.  We record them here.

\begin{lemma}\label{LemmaPfKerDecompsS}
Fix $a>0$, and let $\sB_1\subset \schS(\R^N)$ and $\sB_2\subset C_0^{\infty}(B^N(a))$ be bounded sets.  Let $j\in [0,\infty)^{\nu}$,
$\vsig\in \sB_1$, and $\eta\in \sB_2$.  Then, there exist $\{\gamma_k : k\in \N^{\nu}, k\leq j\}\subset C_0^{\infty}(B^N(a))$,
such that\footnote{We have written $k\leq j$ to denote the inequality holds coordinatewise.}
\begin{equation*}
\eta(t) \dil{\vsig}{2^j}(t) = \sum_{\substack{k\leq j \\ k\in \N^\nu}} \eta(t) \dil{\gamma_k}{2^k}(t).
\end{equation*}
Furthermore, for every $M\in \N$, the set
\begin{equation*}
\q\{ 2^{M|j-k|} \gamma_k : j\in [0,\infty)^\nu, k\leq j, k\in \N^\nu, \vsig\in \sB_1, \eta\in \sB_2 \w\}\subset C_0^{\infty}(B^N(a))
\end{equation*}
is bounded.
\end{lemma}
\begin{proof}
Let $\eta'\in C_0^{\infty}(B^N(a))$ equal $1$ on a neighborhood of  $$\text{the closure of }\bigcup_{\eta\in \sB_2} \supp{\eta}.$$
For $k\in \N^{\nu}$ with $k\leq j$, let
\begin{equation*}
\delta_k(t) := \sum_{\substack{ p \in \{0,1\}^\nu \\ k+p\leq j }} (-1)^{p_1+\cdots+p_\nu} \eta'(2^{p}t),
\end{equation*}
so that $\eta'(t) = \sum_{0\leq k\leq j} \delta_k(2^{k}t).$  Note that $\delta_k(t) =0$ if $k_\mu\leq j_\mu-1$ and $|t_\mu|$ is sufficiently small (independent of $j,k$).
Define, for $k\in \N^{\nu}$, $k\leq j$,
\begin{equation*}
\gamma_k(t) = \delta_k(t) \dil{\vsig}{2^{j-k}}(t).
\end{equation*}
If $|j-k|_{\infty}< 1$, it is easy to see that $\| \gamma_k\|_{C^r}\lesssim 1$, for every $r$.  Suppose $|j-k|_{\infty}\geq 1$.  Take $\mu$ so that $|j-k|_{\infty}=j_{\mu}-k_{\mu}$.
Because $\delta_k(t)$ is $0$ for $|t_\mu|$ sufficiently small (independent of $j,k$), and by the Schwartz bounds on $\vsig$, we have for any $\alpha$ and $L$,
\begin{equation*}
\q|\partial_t^{\alpha} \gamma_k(t)\w| \lesssim \chi_{\{|t_\mu|\approx 1\}}  (1+|2^{j_\mu-k_\mu} t_\mu|)^{-L}\lesssim 2^{-h_0|j-k|_{\infty}L},
\end{equation*}
where $h_0$ is as in \eqref{EqnKerPutOnLapDilDefnh0}.
Taking $L=L(M)$ sufficiently large shows that for every $M$,
\begin{equation*}
\q\{ 2^{M|j-k|} \gamma_k : j\in [0,\infty)^\nu, k\leq j, k\in \N^\nu, \vsig\in \sB_1, \eta\in \sB_2 \w\}\subset C_0^{\infty}(B^N(a))
\end{equation*}
is bounded.
Since $\eta \dil{\vsig}{2^j} = \eta \eta' \dil{\vsig}{2^j} = \sum_{k\leq j} \eta \dil{\gamma_k}{2^k}$, the result follows.
\end{proof}

In what follows, $\alpha_\mu\in \N^{N_\mu}$ will denote a multi-index in the $t_\mu$ variable.  Let $\Na=N_1+\cdots+N_{\nu}$.
$\alphaa=(\alpha_1,\ldots, \alpha_\nu)\in \N^{N_1}\times \cdots\times \N^{N_\nu}=\N^{\Na}$.  We write
$\partial_t^{\alphaa}=\partial_{t_1}^{\alpha_1}\cdots \partial_{t_\nu}^{\alpha_\nu}$.  This differs from multi-index notation, because
the $t_\mu$ variables may overlap and, therefore, $\Na$ may be greater than $N$.
For $j\in [0,\infty]^{\nu}$, we write $2^{-j} \partial_{t_\mu}= (2^{-j\cdot \eh_1^{\mu}} \partial_{t_\mu^1} ,\ldots, 2^{-j\cdot \eh_{N_\mu}^\mu} \partial_{t_\mu^{N_\mu}})$,
and $2^{-j} \partial_t^{\alphaa} = (2^{-j} \partial_{t_1})^{\alpha_1} \cdots (2^{-j} \partial_{t_\nu})^{\alpha_\nu}$.  In particular, this is defined so that
$(2^{-j} \partial_t^{\alphaa}) \dil{\vsig}{2^j} = \dil{(\partial_t^{\alphaa}\vsig)}{2^j}$.


\begin{prop}\label{PropPfKerDecompVsig}
Let $a>0$, $M\in \N$ and let $\sB_1\subset \schS(\R^N)$ and $\sB_2\subset C_0^{\infty}(B^N(a))$ be bounded sets.
Let $j\in [0,\infty)^{\nu}$, $\vsig\in \sB_1\cap \schS_{\{\mu : j_\mu\ne 0\}}$, and $\eta\in \sB_2$.  There exists
\begin{equation*}
\begin{split}
&\q\{\gamma_{k,\alphaa} : k\leq j, k\in \N^{\nu}, |\alpha_\mu|=M\text{ when }k_\mu\ne 0, |\alpha_\mu|=0\text{ when }k_\mu=0 \w\}
\\&\subset C_0^\infty(B^N(a))
\end{split}
\end{equation*}
such that if we set
\begin{equation*}
\vsig_k:=\sum_{\substack{\alphaa\in \N^{\Na} \\ |\alpha_\mu|=M\text{ when }k_\mu \ne 0 \\ |\alpha_\mu|=0\text{ when }k_\mu=0 }} \partial_t^{\alphaa} \gamma_{k,\alphaa},
\end{equation*}
we have
\begin{equation*}
\eta(t) \dil{\vsig}{2^j}(t) = \eta(t)\sum_{\substack{ k\leq j \\ k\in \N^{\nu}}} \dil{\vsig_k}{2^k}(t).
\end{equation*}
Furthermore, for every $L\in \N$, the following set is bounded:
\begin{equation*}
\begin{split}
\big\{
2^{L|j-k|} \gamma_{k,\alphaa} :& j\in [0,\infty)^{\nu}, k\leq j, k\in \N^{\nu}, \vsig\in \sB_1\cap \schS_{\{\mu : j_\mu\ne 0\}},\eta\in \sB_2,
\\&
|\alpha_\mu|=M\text{ when }k_\mu\ne 0,
|\alpha_\mu|=0 \text{ when }k_\mu=0
\big\}
\subset C_0^{\infty}(B^N(a)).
\end{split}
\end{equation*}
\end{prop}
\begin{proof}
Let $j\in [0,\infty)^{\nu}$, $\vsig\in \sB_1\cap \schS_{\{\mu : j_\mu\ne 0\}}$, and $\eta\in \sB_2$.  We prove the result for $M$ replaced by $2M$ (because the result
for $M$ follows from the result for $2M$, this is sufficient).
Define $E_0=\{\mu : j_\mu\ne 0\}$.  Because $\vsig\in \schS_{E_0}$, Lemma \ref{LemmaKersSEPullOut} shows
\begin{equation*}
\vsig = \q[\prod_{\mu\in E_0} \lap_\mu^M\w]\vsigt,
\end{equation*}
where $\{\vsigt : j\in[0,\infty)^\nu, \vsig\in \sB_1\cap \schS_{\{\mu : j_\mu\ne 0\}}\}\subset \schS(\R^N)$ is a bounded set.
Let $\eta'\in C_0^{\infty}(B^N(a))$ be such that $\eta'$ equals $1$ on a neighborhood of the closure of $\bigcup_{\eta\in \sB_2} \supp{\eta}$.
We apply Lemma \ref{LemmaPfKerDecompsS} to $\eta' \dil{\vsigt}{2^j}$ to write
\begin{equation*}
\eta'(t) \dil{\vsigt}{2^j}(t) = \sum_{\substack{k\leq j \\ k\in \N^\nu}} \eta'(t) \dil{\gammat_k}{2^k}(t),
\end{equation*}
where for every $L$,
\begin{equation}\label{EqnPfKerBoundGammat}
\q\{ 2^{L|j-k|} \gammat_k : j\in [0,\infty)^\nu, k\leq j, k\in \N^\nu, \vsig\in \sB_1\cap \schS_{\{\mu : j_\mu\ne 0\}} \w\}\subset C_0^\infty(B^N(a))
\end{equation}
is a bounded set.
Consider, using \eqref{EqnKerPullOutLapDil},
\begin{equation*}
\begin{split}
&\eta\dil{\vsig}{2^j} = \eta \q[\prod_{\mu\in E_0} (2^{-j} \cdot \lap_\mu)^M \w] \dil{\vsigt}{2^j} = \eta \q[\prod_{\mu\in E_0} (2^{-j_\mu} \cdot\lap_\mu)^M \w]  \eta' \dil{\vsigt}{2^j}
\\& = \sum_{\substack{k\leq j \\ k\in \N^{\nu}}} \eta \q[\prod_{\mu \in E_0} (2^{-j}\cdot \lap_\mu)^M \w] \eta' \dil{\gammat_k}{2^k}
= \sum_{\substack{k\leq j \\ k\in \N^{\nu}}} \eta \q[\prod_{\mu \in E_0} (2^{-j}\cdot \lap_\mu)^M \w]  \dil{\gammat_k}{2^k}
\end{split}
\end{equation*}
We expand $\prod_{\mu\in E_0} \lap_\mu^M$ as
\begin{equation*}
\prod_{\mu\in E_0} \lap_\mu^M = \sum_{\substack{\alphaa\in \N^{\Na} \\ |\alpha_\mu|=2M\text{ when }\mu\in E_0 \\ |\alpha_\mu|=0\text{ when }\mu\not \in E_0 }} c_{\alphaa} \partial_t^{\alphaa},
\end{equation*}
where $c_{\alphaa}$ is a constant depending on $\alphaa$.  Consider,
\begin{equation*}
(2^{-j} \partial_t)^{\alphaa} \dil{\gammat_k}{2^k} = \dil{( 2^{(k-j)\cdot e_{\alphaa}} \partial_t^{\alphaa} \gammat_k )}{2^k},
\end{equation*}
where $e_{\alphaa}$ is a vector depending on $\alphaa$.  Setting $\gamma_{k,\alphaa} = 2^{(k-j)\cdot e_{\alphaa}} c_{\alphaa} \gammat_k$
and using \eqref{EqnPfKerBoundGammat} completes the proof.
\end{proof}

For the next lemma, we move to the single-parameter case.  Thus, we assume we have single-parameter dilations $e_1,\ldots, e_N\in (0,\infty)$
on $\R^N$.  When we write $\dil{\vsig}{2^j}(t)$ for $\vsig:\R^N\rightarrow \C$ and $j\in \R$, we are using these dilations in the definition.
\begin{lemma}\label{LemmaPfKerDecompSingleParam}
Fix $\delta\in \R$.  There exists $M=M(\delta)\in \N$ such that the following holds.  For every bounded set $\sB\subset \schS(\R^N)$ and $j\in [0,\infty)$
if $\vsig_0\in \sB$ and $\vsig=\partial_t^{\alpha} \vsig_0$ with $|\alpha|=M$, we may write
\begin{equation*}
2^{j \delta} \dil{\vsig}{2^j} = \sum_{\substack{k\leq j \\k\in \N}} 2^{k\delta} \dil{\vsig_{k,j}}{2^k}
\end{equation*}
where $\vsig_{k,j}\in \schS_0(\R^N)$ for $k>0$ and
\begin{equation*}
\q\{ \vsig_{k,j} : j\in [0,\infty), k\in \N, k\leq j, \vsig_0\in \sB, |\alpha|=M\w\}\subset \schS(\R^N)
\end{equation*}
is a bounded set.
\end{lemma}
\begin{proof}
Let $\phih\in C_0^\infty(\R^n)$ equal $1$ on a neighborhood of $0\in \R^n$, and set $\psih(\xi) = \phih(\xi)-\phih(2\xi)$ (here $2\xi$ is defined using
the given single-parameter dilations on $\R^n$).  Let $j'\in \N$ be the largest integer $\leq j$.  Consider, where $\hat{f}$ denotes the Fourier transform of $f$,
\begin{equation*}
\widehat{\dil{\vsig}{2^j} }(\xi) = \q(1-\phih(2^{-j'} \xi)\w)\vsigh(2^{-j}\xi) + \phih(2^{-j'}\xi) \vsigh(2^{-j}\xi).
\end{equation*}
Let
\begin{equation*}
2^{j'\delta} \vsigh_{j',j,1}(2^{-j'}\xi):= 2^{j\delta} \q( 1-\phih(2^{-j'\xi} ) \w)\vsigh(2^{-j}\xi).
\end{equation*}
Let $\vsig_{j',j,1}$ be the inverse Fourier transform of $\vsigh_{j',j,1}$.  Clearly
\begin{equation*}
\{ \vsig_{j',j,1} : j\in [0,\infty), \vsig_0\in \sB, |\alpha|=M\} \subset \schS_0(\R^n)
\end{equation*}
is a bounded set.

Thus, we need only deal with the term $2^{j\delta} \phih(2^{-j'}\xi) \vsigh(2^{-j}\xi)$.
Consider,
\begin{equation*}
\begin{split}
&\phih(2^{-j'}\xi) \vsigh(2^{-j}\xi) = 2^{j\delta} \phih(2^{-j'}\xi)  (2^{-j} \xi)^{\alpha} \vsigh_0(2^{-j}\xi)
\\&=2^{j\delta} \phih(\xi) (2^{-j}\xi)^{\alpha} \vsigh_0(2^{-j}\xi) + \sum_{k=1}^{j'} 2^{j\delta} \psih(2^{-k} \xi) (2^{-j}\xi)^{\alpha} \vsigh_0(2^{-j}\xi)
\\&=2^{j\delta-j\deg(\alpha)} \xi^{\alpha} \phih(\xi) \vsigh_0(2^{-j}\xi)
\\&\quad+ \sum_{k=1}^{j'} 2^{k\delta} 2^{(j-k)\delta +(k-j)\deg(\alpha)} (2^{-k}\xi)^{\alpha} \psih(2^{-k}\xi) \vsigh_0(2^{-j} \xi).
\end{split}
\end{equation*}
By taking $M=M(\delta)$ large, we have $\deg(\alpha)\geq \delta+1$.  From here,
the result follows by
taking the inverse Fourier transform of the above expression.
\end{proof}

%% file: pfkerother.tex
In this section, we prove Proposition \ref{PropResKerOtherDefn}.  The ideas here are not used elsewhere in the paper.

Fix $\mu\in \nuset$.  
Using the $\mu$ parameter dilations on $t_\mu\in \R^{N_\mu}$ discussed at the start of this section, we obtain single parameter
dilations on $t_\mu$ by
\begin{equation*}
\delta_\mu t_\mu = (1,\ldots,1, \delta_\mu,1,\ldots, 1)t_\mu,
\end{equation*}
where $(1,\ldots, 1, \delta_\mu,1,\ldots, 1)\in [0,\infty)^\nu$ denotes the vector which is $\delta_\mu$ in the $\mu$ component,
and $1$ in all the other components. 
Let $\phih_{\mu}\in C_0^{\infty}(\R^{N_\mu})$ equal $1$ on a neighborhood of $0\in \R^{N_\mu}$, and for $j\in \N^{\nu}$ define
\begin{equation*}
\psih_{j,\mu}(\xi_\mu):=
\begin{cases}
\phih_{\mu}(\xi_{\mu})-\phih_{\mu}(2\xi_{\mu}),&\text{ if }j>0,\\
\phih_{\mu}(\xi_{\mu}),&\text{ if }j=0,
\end{cases}
\end{equation*}
where $2\xi_\mu$ is defined using the single parameter dilations on $\R^{N_{\mu}}$ (and so is {\it not} just standard multiplication).
Note that $\sum_{j\in \N} \psih_{j,\mu}(2^{-j} \xi_{\mu})=1$ in the sense of tempered distributions.
Let $\psi_{j,\mu}$ be the inverse Fourier transform of $\psih_{j,\mu}$, and define $\dil{\psi_{j,\mu}}{2^j}$ ($j\in \N$) in the usual way, i.e. $\dil{\psi_{j,\mu}}{2^j}$ is defined
so that
\begin{equation*}
\int \dil{\psi_{j,\mu}}{2^j}(t_\mu) \eta(t_\mu)\: dt_\mu = \int \psi_{j,\mu}(t_\mu) \eta(2^{-j} t_\mu) \: dt_\mu,
\end{equation*}
and $2^{-j} t_\mu$ is defined by the single parameter dilations.
We have,
\begin{equation}\label{EqnPfKerOtherDecompDeltaz}
\delta_0(t_\mu)=\sum_{j\in \N} \dil{\psi_{j,\mu}}{2^j}(t_\mu),
\end{equation}
where $\delta_0(t_\mu)$ denotes the Dirac $\delta$ function at $0$ in the $t_\mu$ variable.
In the next lemma, we take $\mu=\nu$.

\begin{lemma}\label{LemmaPfKerOtherConvSum}
Let $\delta<0$ and fix $\sE\subseteq \{1,2,\ldots, \nu-1\}$.  Suppose $\sB\subset \schS_{\sE}$ is a bounded set.
For $j\in \N$ let $\{\vsig_l : l\in \N, l\geq j\}\subset \sB$ be bounded.
Define $\vsigt_j$ by
\begin{equation}\label{EqnPfKerOtherToShowLem}
\vsigt_j(t) := \sum_{\substack{l,k\in \N \\ l\wedge k=j}} 2^{(l-j)\delta} \q(\dil{\vsig_l}{2^{(0,0,\ldots,0,l-j) }}*\dil{\psi_{k,\nu}}{2^{k-j}} \w)(t),
\end{equation}
where $*$ denotes convolution in the $t_\nu$ variable.  Then,
\begin{equation*}
\q\{ \vsigt_j : j\in \N, \{\vsig_l : l\geq j\}\subseteq \sB\w\}\subset \schS(\R^N)
\end{equation*}
is a bounded set.
Furthermore, $\vsigt_0\in \schS_\sE$ and $\vsigt_j\in \schS_{\sE\cup \{\nu\}}$ for $j>0$.
\end{lemma}
\begin{proof}
We separate the sum in \eqref{EqnPfKerOtherToShowLem} into two parts:
\begin{equation*}
\vsigt_j(t) =  \sum_{l\geq j} 2^{(l-j)\delta} \q(\dil{\vsig_l}{2^{(0,0,\ldots,0,l-j)}}*\psi_{j,\nu}\w)(t) + \sum_{k>j} \q(\vsig_j * \dil{\psi_{k,\nu}}{2^{k-j}}\w)(t)  =:\vsigt_{j,1}+\vsigt_{j,2}.
\end{equation*}
We first deal with $\vsigt_{j,1}$.
It follows from a standard computation that 
\begin{equation*}
\q\{ \q( \dil{\vsig}{2^{(0,0,\ldots,0,l)}}*\psi_{j,\nu}\w)(t) : \vsig\in \sB, j,l\in \N\w\}\subset \schS(\R^N)
\end{equation*}
is a bounded set.  
Since $\delta<0$, it follows that 
\begin{equation*}
\q\{ \vsigt_{j,1} :j\in \N, \{\vsig_l : l\geq j\}\subseteq \sB\w\}\subset \schS(\R^N)
\end{equation*}
is a bounded set.
If $j>0$, because $\vsig_l\in \schS_\sE$ and $\psi_{j,\nu}\in \schS_0(\R^{N_{\nu}})$ it follows that $\vsig_l*\psi_{j,\nu}\in \schS_{\sE\cup \{\nu\}}$;
and therefore $\vsigt_{j,1}\in \schS_{\sE\cup \{\nu\}}$.
If $j=0$, because $\vsig_l\in \schS_\sE$ it follows that $\vsig_l*\psi_{0,\nu}\in \schS_{\sE}$;
and therefore $\vsigt_{0,1}\in \schS_{\sE}$.

We now turn to $\vsigt_{j,2}$.  For $k>0$, because $\psi_{k,\nu}\in \schS_0(\R^{N_\nu})$ we have $\psi_{k,\nu}=\lap_{\nu} \psit_{\nu}$,
where $\psit_{\nu}\in \schS(\R^{N_\nu})$ (here we have used that $\psi_{k,\nu}$ does not depend on $k$ for $k>0$).  Integrating by parts, we have
\begin{equation*}
\vsig_j *\dil{ \psi_{k,\nu}}{2^{k-j}} = \q(2^{j-k} \cdot \lap_{\nu} \vsig_j\w) *\psit_{\nu};
\end{equation*}
here
\begin{equation*}
2^{j-k} \cdot \lap_{\nu} := 2^{(1,\ldots, 1, j-k)}\cdot \lap_\nu,
\end{equation*}
where $ 2^{(1,\ldots, 1, j-k)}\cdot \lap_\nu$ is defined via \eqref{EqnPfKerDilLap}.
A standard estimate shows that there exists $c>0$ (depending on the dilations; see \eqref{EqnKerPutOnLapDil}) with
\begin{equation*}
\q\{ 2^{c|j-k|}\q(2^{j-k} \cdot \lap_{\nu} \vsig_j\w) *\psit_{\nu} :j\in \N, \{\vsig_l : l\geq j\}\subseteq \sB  \w\}\subset \schS(\R^N)
\end{equation*}
is a bounded set.  It follows that 
\begin{equation*}
\q\{ \vsigt_{j,2} :j\in \N, \{\vsig_l : l\geq j\}\subseteq \sB\w\}\subset \schS(\R^N)
\end{equation*}
is a bounded set.  Furthermore, since for $k>0$ we have $\psi_k\in \schS_0(\R^{N_\nu})$ it follows that
$\vsig_j * \dil{\psi_{k,\nu}}{2^{k-j}}\in \schS_{\sE\cup\{\nu\}}$ for all $k>0$.  Hence
$\vsigt_{j,2}\in \schS_{\sE\cup\{\nu\}}$, as desired.
\end{proof}

\begin{lemma}\label{LemmaPfKerOtherConvg}
Let $\{\vsig_j : j\in \N^{\nu}\}\subset \schS(\R^N)$ be a bounded set and let $\delta\in \R^{\nu}$.  Suppose $\vsig_j\in \schS_{\{\mu : j_\mu\ne 0, \delta_\mu\geq 0\}}$. 
Then the sum
\begin{equation*}
\sum_{j\in \N^{\nu}} 2^{j\cdot \delta} \dil{\vsig_j}{2^j}
\end{equation*}
converges in the sense of tempered distribution.
\end{lemma}
\begin{proof}
Let $\vsig\in \schS(\R^N)$.
Let $N(j) = \max\{ j_\mu : \delta_\mu\geq 0\}$.
The same proof as in Lemma \ref{LemmaKerConvInDist} shows that for every $L$,
\begin{equation*}
\q|\int \dil{\vsig_j}{2^j}(t) \vsig(t)\: dt\w| \lesssim 2^{-LN(j)}. 
\end{equation*}
Hence, if $c=\min\{-\delta_\mu : \delta_\mu<0\}$ we have
\begin{equation*}
\q|\int 2^{j\cdot \delta}\dil{\vsig_j}{2^j}(t) \vsig(t)\: dt\w| \lesssim 2^{-c|j|_\infty}. 
\end{equation*}
The result follows.
\end{proof}


\begin{lemma}\label{LemmaPfKerOtherLemmaInducStep}
Suppose $\delta\in \R^{\nu}$.  Let $\sE\subseteq \nuset$ be such that $\{\mu : \delta_\mu\geq 0\}\subseteq \sE$.
Suppose $\{\vsig_j : j\in \N^{\nu}\}\subset \schS(\R^N)$ is a bounded set with 
$\vsig_j\in \schS_{\{\mu\in \sE : j_\mu\ne 0\}}$.  Define a distribution by
\begin{equation*}
K(t):=\sum_{j\in \N^{\nu}} 2^{j\cdot \delta} \dil{\vsig_j}{2^j}(t).
\end{equation*}
Let $\mu_0\not \in \sE$.  Then, there is a bounded set $\{\vsigt_j : j\in \N^{\nu}\}\subset \schS(\R^N)$
with $\vsigt_j\in  \schS_{\{\mu\in \sE\cup \{\mu_0\} : j_\mu\ne 0\}}$ such that
\begin{equation*}
K(t)=\sum_{j\in \N^{\nu}} 2^{j\cdot \delta} \dil{\vsigt_j}{2^j}(t).
\end{equation*}
\end{lemma}
\begin{proof}
By relabeling, we may without loss of generality take $\mu_0=\nu$ in the statement of the lemma (and therefore $\delta_{\nu}<0$).
For the proof, we need some new notation.  Separate $t=(t_{\nu}, t_{\nu}^{\perp})$ 
and define $\nu$ parameter dilations on $t_\nu$ as was done at the start of the section.  
Note that (for $k\in \R$), $2^{k} t_{\nu} = 2^{(0,\ldots, 0,k)} t_{\nu}$, where $2^{k}t_\nu$ is defined via the  single parameter dilations on $\R^{N_{\nu}}$
defined above.
For $\psi\in \schS(\R^{N_\nu})$, $j\in \R^\nu$, we define $\dil{\psi}{2^j}$ as usual; i.e., for $\eta\in \schS(\R^{N_{\nu}})$, we have
\begin{equation*}
\int \dil{\psi}{2^j}(t_{\nu}) \eta(t_\nu)\: dt_\nu = \int \psi(t_\nu) \eta(2^{-j} t_\nu)\: dt_\nu.
\end{equation*}
For $j\in \R^{\nu}$ we decompose $j=(j', j_\nu)\in \R^{\nu-1}\times \R$, where $j'=(j_1,\ldots, j_{\nu-1})$.  Notice, with this notation and applying \eqref{EqnPfKerOtherDecompDeltaz},
we have for $j'\in \R^{\nu-1}$ fixed
\begin{equation*}
\sum_{k\in \N} \dil{\psi_{k,\nu}}{2^{(j',k)}}= \delta_0(t_\nu).
\end{equation*}
Also, we have, for $\vsigt,\vsig\in \schS(\R^N)$, $\psi\in \schS(\R^{N_{\nu}})$, $j,k,l\in \R^{\nu}$,
\begin{equation}\label{EqnPfKerOtherUndil}
\dil{\vsigt}{2^l} = \dil{\vsig}{2^j}*\dil{\psi}{2^k} \Leftrightarrow \vsigt = \dil{\vsig}{2^{j-l}} *\dil{\psi}{2^{k-l}}.
\end{equation}

Thus, writing $j\in \N^{\nu}$ as $j=(j',j_{\nu})$, we have
\begin{equation*}
\begin{split}
&\sum_{j\in \N^{\nu}} 2^{j\cdot \delta} \dil{\vsig_j}{2^j} = \sum_{j\in \N^{\nu}} \sum_{k\in \N} 2^{j\cdot \delta} \q(\dil{\vsig_j}{2^j}* \dil{\psi_{k,\nu}}{2^{(j',k)\cdot\delta}}\w)
\\&=\sum_{j'\in \N^{\nu-1}} \sum_{l\in \N} \sum_{j_\nu\wedge k=l} 2^{j\cdot \delta} \q(\dil{\vsig_j}{2^j}* \dil{\psi_{k,\nu}}{2^{(j',k)\cdot\delta}} \w)
=:\sum_{j'\in \N^{\nu-1}}\sum_{l\in \N} 2^{(j',l)\cdot\delta} \dil{\vsigt_{(j',l)}}{2^l},
\end{split}
\end{equation*}
where,
\begin{equation*}
2^{(j',l)\cdot\delta} \dil{\vsigt_{(j',l)}}{2^l} := \sum_{j_\nu\wedge k=l} 2^{j\cdot \delta} \q(\dil{\vsig_j}{2^j}* \dil{\psi_{k,\nu}}{2^{(j',k)\cdot\delta}} \w).
\end{equation*}
Using \eqref{EqnPfKerOtherUndil}, this is equivalent to
\begin{equation*}
\vsigt_{(j',l)} = \sum_{j_\nu\wedge k =l} 2^{(j_\nu-l)\delta_{\nu}} \q(\dil{\vsig_j}{2^{0,\ldots, 0, j_\nu-l}} * \dil{\psi_{k,\nu}}{2^{k-l}}\w).
\end{equation*}
Rewriting $(j',l)$ as $j$, and
using Lemma \ref{LemmaPfKerOtherConvSum}, we have
$\vsigt_j\in  \schS_{\{\mu\in \sE\cup \{\nu\} : j_\mu\ne 0\}}$ and $\{\vsigt_j : j\in \N^{\nu}\} \subset \schS(\R^N)$ is a bounded set, which completes the proof.
%
%
\end{proof}

\begin{proof}[Proof of Proposition \ref{PropResKerOtherDefn}]
The convergence of the sum follows from Lemma \ref{LemmaPfKerOtherConvg}.  From here, the result follows from Lemma \ref{LemmaPfKerOtherLemmaInducStep}
and a simple induction.
\end{proof}

%% file: frob.tex
In this section, we review the quantitative Frobenius theorem from \cite{StreetMultiParameterCCBalls}; the reader is also referred to
Section 2.2 of \cite{StreetMultiParamSingInt} for further details and a more leisurely introduction to these concepts.\footnote{This quantitate version
of the Frobenius theorem takes its roots in work of Nagel, Stein, and Wainger \cite{NagelSteinWaingerBallsAndMetricsDefinedByVectorFields}
and Tao and Wright \cite{TaoWrightLpImprovingBounds} on Carnot-Carath\'eodory geometry.}

Before we can introduce the setting of the Frobenius theorem, we need to introduce a few additional pieces of notation.
Let $\Omega\subseteq \R^n$ be an open set and let $U\subseteq \Omega$ be an arbitrary set.  For $m\in \N$, we define
\begin{equation*}
\CjN{f}{m}{U}:= \sum_{|\alpha|\leq m} \sup_{x\in U} |\partial_x^{\alpha} f(x)|,
\end{equation*}
and if we say $\CjN{f}{m}{U}<\infty$ we mean that all the above partial derivatives exist and are are continuous on $U$.
If $V$ is a vector field on $\Omega$ and $V=\sum_{j=1}^n a_j \partial_{x_j}$, we denote by $\CjN{V}{m}{U}:=\sum_{j=1}^n \CjN{a_j}{m}{U}$.
If $A$ is a $n\times q$ matrix, and $n_0\leq n\wedge q$, we let $\det_{n_0\times n_0} A$ denote the {\it vector} whose entries consist
of the determinants of the $n_0\times n_0$ submatrices of $A$.  It is not important in which order these determinants are arranged.

Let $\Omega\subseteq \R^n$ be an open set.  We are given $C^\infty$ vector fields $Z_1,\ldots, Z_q$ on $\Omega$.  
For $x\in \Omega$ and $\delta>0$, we let $B_Z(x,\delta):=\B{Z}{1}{x}{\delta}$, where $\B{Z}{1}{x}{\delta}$ denotes the 
Carnot-Carath\'eodory ball from Definition \ref{DefnResCCBall} using the set $(Z,1)=\{(Z_1,1),\ldots, (Z_q,1)\}$.

Fix $x_0\in \Omega$.
We assume that there exists $0<\xi_1\leq 1$ such that:
\begin{enumerate}[(a)]
\item For every $a_1,\ldots, a_q\in L^\infty([0,1])$, with $\LpN{\infty}{\sum_{j=1}^q |a_j|^2}<1$, there exists a solution to the ODE
\begin{equation*}
\gamma'(t) = \sum_{j=1}^q a_j(t) \xi_1 Z_j(\gamma(t)), \quad \gamma(0)=x_0, \quad \gamma:[0,1]\rightarrow \Omega.
\end{equation*}
Notice, by the Picard-Lindel\"of theorem for existence of ODEs, this condition holds so long as we take $\xi_1$ small enough, depending on the 
$C^1$ norms of $Z_1,\ldots, Z_q$ and the Euclidean distance between $x_0$ and $\partial \Omega$.

\item For each $m$, there is a constant $C_m$ such that
\begin{equation}\label{EqnFrobDefnCm}
\CjN{Z_j}{m}{B_{Z}({x_0},{\xi_1})} \leq C_m.
\end{equation}

\item $[Z_j,Z_k]=\sum_{l=1}^q c_{j,k}^l Z_l$ on $\B{Z}{d}{x_0}{\xi_1}$, where for every $m$ there is a constant $D_m$ such that
\begin{equation}\label{EqnFrobDefnDm}
\sum_{|\alpha|\leq m} \CjN{Z^{\alpha} c_{j,k}^l }{0}{B_{Z}(x_0,\xi_1)}\leq D_m.
\end{equation}
\end{enumerate}

For $m\geq 2$, we say $C$ is an $m$-admissible constant if $C$ can be chosen to depend only on upper bounds for $m$, $n$, $q$, $C_m$
from \eqref{EqnFrobDefnCm}, $D_m$ from \eqref{EqnFrobDefnDm}, 
and a positive lower bound for $\xi_1$.
  For $m<2$, we say $C$ is an $m$-admissible constant
if $C$ is a $2$-admissible constant.  We write $A\lesssim_m B$ if $A\leq C B$, where $C$ is an $m$-admissible constant, and we
write $A\approx_m B$ if $A\lesssim_m B$ and $B\lesssim_m A$.

Let $n_0=\dim \Span {Z_1(x_0),\ldots, Z_q(x_0)}$.  
In light of the classical Frobenius theorem, there is a unique, maximal, injectively immersed submanifold of $\Omega$ passing through
$x_0$ whose tangent space equals $\Span{Z_1,\ldots, Z_q}$--called a leaf.  $B_{Z}({x_0},{\delta})$ is an open subset of this leaf, and for a
subset, $S$, of this leaf, we use the notation $\Vol{S}$ to denote the volume of $S$ with respect to the induced Lebesgue measure on this leaf.  See
Section 2.2 of \cite{StreetMultiParamSingInt} for more details.

We state, without proof, the quantitative Frobenius theorem.  The proof can be found in \cite{StreetMultiParameterCCBalls}.
\begin{thm}[The quantitative Frobenius theorem]\label{ThmQuantFrob}
There exist $2$-admissible constants $\xi_2,\xi_3,\eta>0$, $\xi_3<\xi_2<\xi_1$ and a $C^{\infty}$ map
\begin{equation*}
\Phi:B^{n_0}(\eta)\rightarrow B_{Z}({x_0},{\xi_2}),
\end{equation*}
such that
\begin{itemize}
\item $\Phi(0)=x_0$.
\item $\Phi$ is injective.
\item $B_{Z}({x_0},{\xi_3})\subseteq \Phi(B^{n_0}(\eta))\subseteq B_{Z}({x_0},{\xi_2})$.
\item For $u\in B^{n_0}(\eta)$,
\begin{equation}\label{EqnQuantFrobJacVol}
\q|\det_{n_0\times n_0} d\Phi(u)\w| \approx_2 \q|\det_{n_0\times n_0} (Z_1(x_0) | \cdots | Z_q(x_0))\w| \approx_2 \Vol{B_{Z}({x_0},{\xi_2})},
\end{equation}
where $(Z_1(x_0)| \cdots| Z_q(x_0))$ is the matrix whose columns are given by the vectors $Z_1(x_0),\ldots, Z_q(x_0)$.
\end{itemize}
Furthermore, if $Y_1,\ldots, Y_q$ are the pullback of $Z_1,\ldots, Z_q$ to $B^{n_0}(\eta)$, then
\begin{equation}\label{EqnFrobYsSmooth}
\CjN{Y_j}{m}{B^{n_0}(\eta)}\lesssim_m 1,
\end{equation}
and 
\begin{equation}\label{EqnFrobYsSpan}
\inf_{u\in B^{n_0}(\eta)} \q|\det_{n_0\times n_0} (Y_1(u)| \cdots | Y_q(u))\w| \approx_2 1.
\end{equation}
\end{thm}

%% file: frobcontrol.tex
Let $\Omega'\subseteq \Omega\subseteq \R^n$ be open sets.
Let $Z=\{Z_1,\ldots, Z_q\}$ be a finite set of smooth vector fields on $\Omega$,
which 
satisfy the assumptions of Theorem \ref{ThmQuantFrob} at some point $x_0\in \Omega$.  Let $Z_0$ be another smooth vector field on $\Omega$.
  We introduce the following conditions on $Z_0$ which turn out to be equivalent.  All parameters
that follow are assumed to be strictly positive real numbers.
\begin{enumerate}[1.]
\item $\sP_1(\tau_1, \{\sigma_1^m\}_{m\in \N})$, ($\tau_1\leq \xi_1$):  There exists $c_j\in C^0\q(B_{Z}({x_0},{\tau_1})\w)$ such that
\begin{itemize}
\item $Z_0=\sum_{j=1}^q c_j Z_j$, on $B_{Z}({x_0},{\tau_1})$.
\item $\sum_{|\alpha|\leq m} \CjN{Z^{\alpha} c_j}{0}{B_{Z}({x_0},{\tau_1})}\leq \sigma_1^m$.
\end{itemize}

\item $\sP_2(\eta_2,\{\sigma_2^m\}_{m\in \N})$, ($\eta_2\leq \eta_1$):  $Z_0$ is tagent to the leaf passing through $x_0$
generated by $Z_1,\ldots, Z_q$, and moreover, if $Y_0$ is the pullback of $Z_0$ via the map $\Phi$ from
Theorem \ref{ThmQuantFrob}, then we have $\CjN{Y_0}{m}{B^{n_0}(\eta_2)}\leq \sigma_2^m$.
\end{enumerate}

\begin{prop}\label{PropFrobControlEquivsP}
$\sP_1\Leftrightarrow \sP_2$ in the following sense.
\begin{itemize}
\item $\sP_1(\tau_1, \{\sigma_1^m\}_{m\in \N}) \Rightarrow$ there exists a $2$-admissible constant $\eta_2=\eta_2(\tau_1)>0$ and
$m$-admissible constants $\sigma_2^m=\sigma_2^m(\sigma_1^m)$ such that $\sP_2(\eta_2, \{\sigma_2^m\}_{m\in \N})$ holds.
\item $\sP_2(\eta_2, \{\sigma_2^m\}_{m\in \N})\Rightarrow$ there exists a $2$-admissible constant $\tau_1=\tau_1(\eta_2)>0$
and $m$-admissible constants $\sigma_1^m=\sigma_1^m(\sigma_2^m)$ such that $\sP_1(\tau_1, \{\sigma_1^m\}_{m\in \N})$ holds.
\end{itemize}
\end{prop}
\begin{proof}
This is contained in Proposition 11.6 of \cite{SteinStreetI}.
\end{proof}

\begin{rmk}\label{RmkExtraPropFrobControlEquivsP}
The proof of  Proposition 11.6 of \cite{SteinStreetI} actually gives more than is stated in Proposition \ref{PropFrobControlEquivsP}:
namely, that $\sigma_1^m$ is small if and only if $\sigma_2^m$ is small.  More precisely
for all $\epsilon>0$, $m\in \N$, there exists an $m$-admissible constant $\delta_m>0$ such that:
\begin{itemize}
\item $\sP_1(\tau_1, \{\sigma_1^m\}_{m\in \N}) \Rightarrow$ there exists a $2$-admissible constant $\eta_2=\eta_2(\tau_1)>0$ and
$m$-admissible constants $\sigma_2^m=\sigma_2^m(\sigma_1^m)$ such that $\sP_2(\eta_2, \{\sigma_2^m\}_{m\in \N})$ holds.  Furthermore,
if $\sigma_1^m<\delta_m$, then $\sigma_2^m<\epsilon$.

\item $\sP_2(\eta_2, \{\sigma_2^m\}_{m\in \N})\Rightarrow$ there exists a $2$-admissible constant $\tau_1=\tau_1(\eta_2)>0$
and $m$-admissible constants $\sigma_1^m=\sigma_1^m(\sigma_2^m)$ such that $\sP_1(\tau_1, \{\sigma_1^m\}_{m\in \N})$ holds.
Furthermore, if $\sigma_2^m<\delta_m$, then $\sigma_1^m<\epsilon$.
\end{itemize}
We will use this in Section \ref{SectionOptimality}.
\end{rmk}

\begin{defn}
Let $V=\{V_1,\ldots, V_r\}$ be a finite set of smooth vector fields on $\Omega$.
We say $V$ {\bf controls} $Z_0$ {\bf at the unit scale near} $x_0$ if $\sP_1$ holds (with $Z$ replaced by $V$).
We say $V$ {\bf controls} $Z_0$ {\bf at the unit scale on} $\Omega'$ if $\sP_1$ holds $\forall x_0\in \Omega'$,
where the constants $\tau_1$, $\{\sigma_1^m\}_{m\in \N}$ can be chosen independent of $x_0\in \Omega'$.
\end{defn}

\begin{rmk}
Proposition \ref{PropFrobControlEquivsP} shows that $Z$ controls $Z_0$ at the unit scale near $x_0$ if and only if $\sP_2$ holds.
This uses that $Z$ satisfies the assumptions of Theorem \ref{ThmQuantFrob}
\end{rmk}


Let $\gammah_t(x)=\gammah(t,x):B^N(\rho)\times \Omega\rightarrow \Omega$ be a smooth function, satisfying $\gammah_0(x)\equiv x$.
Given $\gammah$, we can define a vector field as in \eqref{EqnResSurfDefW} given by
\begin{equation}\label{EqnFrobControlWh}
\Wh(t,x)=\frac{d}{d\epsilon}\bigg|_{\epsilon=1} \gammah_{\epsilon t}\circ \gammah_t^{-1}(x),
\end{equation}
which is defined for $|t|$ sufficiently small (we shrink $\rho$ so that $W(t,x)$ is defined for all $t\in B^N(\rho)$).
Just as above, we introduce two conditions, which turn out to be equivalent.  All parameters that follow are assumed to be strictly
positive real numbers.
\begin{enumerate}[1.]
\item $\sQ_1(\rho_1, \tau_1, \{\sigma_1^m\}_{m\in \N})$, ($\rho_1\leq \rho$, $\tau_1\leq \xi_1$):
\begin{itemize}
\item $\Wh(t,x)=\sum_{l=1}^q c_l(t,x) Z_l(x)$, on $B_{Z}({x_0},{\tau_1})$.
\item $\sum_{|\alpha|+|\beta|\leq m} \CjN{Z^{\alpha} \partial_t^{\beta} c_l}{0}{B^N(\rho_1)\times B_{Z}({x_0},{\tau_1})}\leq \sigma_1^m$.
\end{itemize}
\item $\sQ_2(\rho_2,\tau_2,\{\sigma_2^m\})$, ($\rho_2\leq \rho$, $\tau_2\leq \xi_1$):
\begin{itemize}
\item $\gammah\q(B^N(\rho_2)\times B_{Z}({x_0},{\tau_2})\w)\subseteq B_{Z}({x_0},{\xi_1})$.
\item If $\eta'=\eta'(\tau_2)>0$ is an $2$-admissible constant so small\footnote{Such a $2$-admissible $\eta'$ always exists.  See Proposition 11.2 of \cite{SteinStreetI}.}
that
\begin{equation*}
\Phi(B^{n_0}(\eta'))\subseteq B_{Z}({x_0},{\tau_2})\subseteq \Phi(B^{n_0}(\eta_1)),
\end{equation*}
then if we define a new map
\begin{equation*}
\theta_t(u) = \Phi^{-1} \circ \gamma_t \circ \Phi(u):B^N(\rho_2)\times B^{n_0}(\eta')\rightarrow B^{n_0}(\eta_1),
\end{equation*}
we have $\CjN{\theta}{m}{B^N(\rho_2)\times B^{n_0}(\eta')}\leq \sigma_2^m$.
\end{itemize}
\end{enumerate}

\begin{prop}\label{PropFrobControlEquivsQ}
$\sQ_1\Leftrightarrow \sQ_2$ in the following sense:
\begin{itemize}
\item $\sQ_1(\rho_1, \tau_1, \{\sigma_1^m\}_{m\in \N})\Rightarrow$ there exists a $2$-admissible constant $$\rho_2=\rho_2(\rho_1,\tau_1,\sigma_1^1,N)>0$$
and $m+1$-admissible constants $\sigma_2^m=\sigma_2^m(\sigma_1^{m+1}, N)$ such that
$$\sQ_2(\rho_2,\tau_1/2, \{\sigma_2^m\}_{m\in \N})$$ holds.
\item $\sQ_2(\rho_2,\tau_2,\{\sigma_2^m\}_{m\in \N})\Rightarrow$ there exists a $2$-admissible constant $\tau_1=\tau_1(\tau_2)>0$
and $m$-admissible constants $\sigma_1^m=(\sigma_2^{m+1}, N)$ such that $$\sQ_1(\rho_2,\tau_1,\{\sigma_1^m\}_{m\in \N})$$ holds.
\end{itemize}
\end{prop}
\begin{proof}
This is Proposition 12.3 of \cite{SteinStreetI}.
\end{proof}

\begin{defn}
Let $V=\{V_1,\ldots, V_r\}$ be a finite set of smooth vector fields on $\Omega$.
We say $V$ {\bf controls} $\gammah$ {\bf at the unit scale near} $x_0$ if $\sQ_1$ holds (with $Z$ replaced by $V$).
We say $V$ {\bf controls} $\gammah$ {\bf at the unit scale on} $\Omega'$ if $\sQ_1$ holds 
with $\rho_1$, $\tau_1$, and $\{\sigma_1^m\}_{m\in \N}$ independent of $x_0\in \Omega'$.
\end{defn}

\begin{rmk}
Proposition \ref{PropFrobControlEquivsQ} shows $Z$ controls $\gammah$ at the unit scale near $x_0$ if and only if $\sQ_2$ holds.
This uses that $Z$ satisfies the assumptions of Theorem \ref{ThmQuantFrob}
\end{rmk}


Now let $\sI$ be some index set (of arbitrary cardinality).

\begin{defn}\label{DefnFrobControlUniform}
For each $\iota\in \sI$, let $V^\iota=\{V^{\iota}_1,\ldots, V^{\iota}_r\}$ be a finite set of smooth vector fields
on $\Omega$, and let $Z^{\iota}_0$ be another smooth vector field on $\Omega$.
We say $V^{\iota}$ {\bf controls} $Z^{\iota}_0$
{\bf at the unit scale on} $\Omega'$, {\bf uniformly in} $\iota\in \sI$ if for each $x_0\in \Omega'$ and $\iota\in \sI$, $\sP_1$ holds with parameters ($\tau_1,\{\sigma_1^m\}$)
independent of $x_0\in \Omega'$ and $\iota\in \sI$ (with $Z$ replaced by $V^{\iota}$).  
\end{defn}


\begin{defn}
For each $\iota\in \sI$ let $V^\iota=\{V^{\iota}_1,\ldots, V^{\iota}_r\}$ be a finite set of smooth vector fields
on $\Omega$ each paired with single-parameter formal degrees.
Let $\Omega'\Subset\Omega''\Subset \Omega$ be open sets.
Suppose, for each $\iota\in \sI$, $\gammah^{\iota}:B^N(\rho)\times \Omega\rightarrow \Omega$ is a smooth function satisfying $\gammah_0^{\iota}(x)\equiv 0$,
and such that for $|t|<\rho$, $\gamma_t^{\iota}:\Omega''\rightarrow \Omega$ is a diffeomophism onto its image.  We say $V^{\iota}$ {\bf controls}
$\gammah^{\iota}$ {\bf at the unit scale on} $\Omega'$, {\bf uniformly in} $\iota \in \sI$ if for each $x_0\in \Omega'$ and $\iota \in \sI$, $\sQ_1$ holds with parameters  ($\rho_1,\tau_1,\{\sigma_1^m\}$)
independent of $x_0\in \Omega'$ and $\iota \in \sI$ (with $Z$ replaced by $V^{\iota}$).  
\end{defn}


\begin{rmk}\label{RmkFrobControlCanUseUniformt}
Fix $q\in \N$ and for each $\iota \in \sI$ let
$Z^\iota=\q\{Z_1^{\iota},\ldots, Z_q^{\iota}\w\}$ be a finite set of smooth vector fields
on $\Omega$ (where $q$ does not depend on $\iota$).
We assume that the assumptions of Theorem \ref{ThmQuantFrob} hold uniformly for $Z^{\iota}$ uniformly for
$x_0\in \Omega'$, in the sense that $\xi_1>0$
can be chosen independent of $\iota$,
and $C_m$ and $D_m$ from \eqref{EqnFrobDefnCm} and \eqref{EqnFrobDefnDm} can be chosen independent of $\iota$.
Note that $m$-admissible constants can be chosen independent of $\iota$.
In this case,
$Z^{\iota}$ controls $Z^{\iota}_0$ at the unit scale on $\Omega'$, uniformly in $\iota$, if and only if $\sP_2$
holds with parameters independent of $x_0\in \Omega'$ and $\iota\in \sI$ (Proposition \ref{PropFrobControlEquivsP}).
$Z^{\iota}$ controls $\gammah^{\iota}$ at the unit scale on $\Omega'$, uniformly in $\iota$, if and only if
$\sQ_2$ holds with parameters independent of $x_0\in \Omega'$ and $\iota\in \sI$ (Proposition \ref{PropFrobControlEquivsQ}).
\end{rmk}

\begin{prop}\label{PropFrobControlUnitScaleVsAll}
Let $(X,d)=\{(X_1,d_1),\ldots, (X_q,d_q)\}\subset \snuvect$ be a finite set.  For $\delta\in [0,1]^{\nu}$, let
$Z^{\delta}:=\{\delta^{d_1}X_1,\ldots, \delta^{d_q} X_q\}$.
\begin{itemize}
\item Let $X_0$ be another $C^{\infty}$ vector field on $\Omega$ and let $h:[0,1]^{\nu}\rightarrow [0,1]^{\nu}$ be a function.
Then, $(X,d)$ controls $(X_0,h)$ on $\Omega'$ if and only if $Z^{\delta}$ controls $h(\delta) X_0$ at the unit scale on $\Omega'$, uniformly
in $\delta\in [0,1]^{\nu}$.

\item $(X,d)$ satisfies $\sD(\Omega')$ if and only if $Z^{\delta}$ satisfies the conditions of Theorem \ref{ThmQuantFrob}, uniformly
for $\delta\in [0,1]^{\nu}$ and $x_0\in \Omega'$.

\item Let $(\gamma,e,N,\Omega,\Omega''')$ be a parameterization with $\nu$-parameter dilations (where $\Omega'\Subset\Omega'''\Subset\Omega$).  Then $(X,d)$ controls $(\gamma,e,N,\Omega,\Omega''')$ on $\Omega'$
if and only if $Z^{\delta}$ controls $\gamma_{\delta t}(x)$ at the unit scale on $\Omega'$, uniformly in $\delta\in [0,1]^{\nu}$.  Here, $\delta t$ is defined
with the $\nu$-parameter dilations, $e$.
\end{itemize}
\end{prop}
\begin{proof}
Let $\hd_1=|d_1|_1,\ldots, \hd_q=|d_q|_1$, and let $(Z^\delta,\hd)=\{ (Z^{\delta}_1, \hd_1),\ldots, (Z^{\delta}_q, \hd_q)\}$.  If, in all of the definitions ``at the unit scale'' and in the conditions of Theorem \ref{ThmQuantFrob},
we replace $B_{Z^\delta}(x,\cdot)$ with $\B{Z^{\delta}}{\hd}{x}{\cdot}$, then the above results follow immediately from the definitions.
To complete the proof, it suffices to note, given $\tau_1>0$, there exists $\tau_2=\tau_2(\tau_1)>0$, depending on the $C^1$ norms of
$X_1,\ldots, X_q$, and upper and lower bounds for $\hd_1,\ldots, \hd_q$ such that
\begin{equation*}
B_{Z^{\delta}}(x, \tau_2)\subseteq \B{Z^\delta}{\hd}{x}{\tau_1},\quad \B{Z^\delta}{\hd}{x}{\tau_2}\subseteq B_{Z^{\delta}}(x,\tau_1).
\end{equation*}
The result follows.
\end{proof}

\begin{rmk}
Combining Proposition \ref{PropFrobControlUnitScaleVsAll} with Remark \ref{RmkFrobControlCanUseUniformt} shows that
we can use the characterizations $\sP_2$ and $\sQ_2$ at each scale, uniformly in the scale, when working with the notions of control.
This is the key point behind these definitions.
\end{rmk}


%% file: pfadj.tex
A key point in the proofs that follow is that the class of operators we consider is closed under adjoints.
More precisely, we have the following results.  Fix open sets $\Omega_0\Subset \Omega'\Subset \Omega''\Subset\Omega'''\Subset \Omega\subseteq \R^n$.

\begin{prop}\label{PropPfAdjGamma}
Let $(\gamma,e,N,\Omega,\Omega''')$ be a parameterization.  For $t\in \R^N$ with $|t|$ sufficiently small, we may consider
$\gamma_t^{-1}(x)$; the inverse of $\gamma_t(\cdot)$.  Then, $$(\gamma_t^{-1}, e, N,\Omega'',\Omega)$$ is a parameterization.
Furthermore, if $(\gamma,e,N)$ is finitely generated (resp. linearly finitely generated) by $\sF$ on $\Omega'$,
then $(\gamma_t^{-1},e,N)$ is finitely generated (resp. linearly finitely generated) by $\sF$ on $\Omega'$.
\end{prop}
\begin{proof}
It is clear that $(\gamma_t^{-1},e,N,\Omega'',\Omega)$ is a parameterization.  When $(\gamma,e,N)$ is finitely generated by $\sF$ on $\Omega'$,
that $(\gamma_t^{-1},e,N)$ is finitely generated by $\sF$ on $\Omega'$ is exactly the statement of Lemma 12.20 of \cite{SteinStreetI}.

%

Now suppose $(\gamma,e,N)$ is linearly finitely generated by $\sF$ on $\Omega'$.  First we claim that $\sF$ controls $(\gamma_t^{-1},e,N)$ on
$\Omega'$.  Indeed, by Lemma \ref{LemmaResSurfLinFGImpliesFG}, $(\gamma,e,N)$ is finitely generated by $\sF$ on $\Omega'$,
and by the above $(\gamma_t^{-1},e,N)$ is finitely generated by $\sF$ on $\Omega'$, and therefore $(\gamma_t^{-1},e,N)$ is controlled
by $\sF$ on $\Omega'$.

Let $(\gamma,e,N)$ correspond to the vector field parameterization $(W,e,N)$, and let $(\gamma_t^{-1},e,N)$ correspond to the vector field
parameterization $(W_i,e,N)$.  We know that
\begin{equation*}
W(t) = \sum_{j=1}^N t_j X_j + O(|t|^2),
\end{equation*}
where $\{X_1,\ldots, X_N\}$ and $\sF$ are equivalent on $\Omega'$.  
By the definition of linearly finitely generated, the proof will be complete once we show the same is true for $W_i$ (with $X_j$ replaced by
$-X_j$).
We have
\begin{equation*}
0=\frac{d}{d\epsilon}\bigg|_{\epsilon=1} \gamma_{\epsilon t}\circ \gamma_{\epsilon t}^{-1}\circ\gamma_t(x) = W(t,\gamma_t(x)) + \q(d_x \gamma_t(x)\w) W_i(t,x).
\end{equation*}
I.e.,
\begin{equation*}
W_i(t,x) = -(d_x \gamma_t(x))^{-1} W(t,\gamma_t(x)).
\end{equation*}
Using that $W(t,x)\equiv 0$, and that $\gamma_0(x)\equiv x$, we have
\begin{equation*}
W_i(t,x) = -(d_x \gamma_t(x))^{-1} W(t,\gamma_t(x)) = -W(t,x) + O(|t|^2)=\sum_{j=1}^N t_j (-X_j) + O(|t|^2).
\end{equation*}
The result follows.
\end{proof}

\begin{thm}\label{ThmPfAdjAdjThm}
Let $(\gamma,e,N,\Omega''', \Omega)$ be a parameterization.  There exists $a>0$ (depending on the parameterization)
such that if $T$ is a fractional Radon transform of order $\delta\in \R^{\nu}$ corresponding to $(\gamma,e,N)$ on $B^N(a)$,
then $T^{*}$ is a fractional Radon transform of order $\delta$ corresponding to $(\gamma_t^{-1},e,N)$ on $B^N(a)$.
\end{thm}
\begin{proof}
Suppose $T$ is given by \eqref{EqnResRadonMainOpNew}.  Then a simple change of variables shows, if $K\in \sK_\delta(N,e,a)$ for $a>0$ sufficiently small,
\begin{equation*}
T^{*} f(x) = \overline{\psi_2(x)} \int f(\gamma_t^{-1}(x))\overline{\psi_1(\gamma_t^{-1}(x))} \kappat(t,x) K(t)\: dt,
\end{equation*}
where 
$$\kappat(t,x)=\overline{\kappa(t,\gamma_t^{-1}(x))} \q|\q(\det d_x \gamma_t\w)(t, \gamma_t^{-1}(x))\w|^{-1}.$$
Using that $\gamma_0(x)\equiv x$, for $|t|<a$ (if $a>0$ is chosen small enough), $$\q(\det d_x \gamma_t\w)(t, \gamma_t^{-1}(x))$$ stays
bounded away from zero, and therefore $\kappat\in C^{\infty}(B^N(a)\times \Omega'')$.  The result follows.
\end{proof}

In the sequel, we need the following version of Theorem \ref{ThmPfAdjAdjThm}  which happens ``at each scale''.

\begin{prop}\label{PropPfAdjEachScale}
Let $(\gamma,e,N,\Omega,\Omega''')$ be a parameterization.  There exists $a>0$ such that the following holds.
Let $\sB_1\subset \schS(\R^N)$, $\sB_2\subset C_0^\infty(B^N(a))$, $\sB_3\subset C_0^\infty(\Omega_0)$,
and $\sB_4\subset C^\infty(B^N(a)\times \Omega'')$ be bounded sets.
For $j\in [0,\infty)^{\nu}$, $\vsig\in \sB_1\cap \schS_{\{\mu : j_{\mu}>0\}}$, $\eta\in \sB_2$, $\psi_1,\psi_2\in \sB_3$, and $\kappa\in \sB_4$,
define
\begin{equation*}
T_j[(\gamma,e,N),\vsig,\eta,\psi_1,\psi_2,\kappa]f(x):=\psi_1(x)\int f(\gamma_t(x)) \psi_2(\gamma_t(x)) \kappa(t,x) \eta(t)\dil{\vsig}{2^j}(t)\: dt.
\end{equation*}
Then, $$T_j[(\gamma,e,N),\vsig,\eta,\psi_1,\psi_2,\kappa]^{*}$$ is ``of the same form'' as $$T_j[(\gamma,e,N),\vsig,\eta,\psi_1,\psi_2,\kappa],$$
with $(\gamma,e,N)$ replaced by $(\gamma_t^{-1},e,N)$.  More precisely,
\begin{equation*}
T_j[(\gamma,e,N),\vsig,\eta,\psi_1,\psi_2,\kappa]^{*} = T_j[(\gamma_t^{-1},e,N), \overline{\vsig}, \overline{\eta}, \overline{\psi_2},\overline{\psi_1}, \kappat],
\end{equation*}
where $\kappat=\kappat(\kappa,(\gamma,e,N))$ and 
\begin{equation}\label{EqnPfAdjToShowKapt}
\{\kappat : \kappa\in \sB_4\}\subset C^\infty(B^N(a)\times \Omega'')
\end{equation}
 is a bounded set.  In the above, $\overline{\vsig}$,  $\overline{\eta}$, $\overline{\psi_2}$, and $\overline{\psi_1}$ denote the complex conjugate
 of the respective function.
\end{prop}
\begin{proof}
A straightforward change of variables shows
\begin{equation*}
T_j[(\gamma,e,N),\vsig,\eta,\psi_1,\psi_2,\kappa]^{*} = T_j[(\gamma_t^{-1},e,N), \overline{\vsig}, \overline{\eta}, \overline{\psi_2},\overline{\psi_1}, \kappat],
\end{equation*}
where 
$$\kappat(t,x)=\overline{\kappa(t,\gamma_t^{-1}(x))} \q|\q(\det d_x \gamma_t\w)(t, \gamma_t^{-1}(x))\w|^{-1}.$$
Using that $\gamma_0(x)\equiv x$, for $|t|<a$ (if $a>0$ is chosen small enough), $$\q(\det d_x \gamma_t\w)(t, \gamma_t^{-1}(x))$$ stays
bounded away from zero, and therefore the set in \eqref{EqnPfAdjToShowKapt} is bounded.
\end{proof}

%% file: pfl2tech.tex
In this section, we present a basic $L^2$ lemma, which lies at the heart of our analysis.  The lemma has two parts.  The first part is a consequence of
the basic $L^2$ result which forms the heart
of the work in \cite{SteinStreetI}; we use this to study fractional Radon transforms corresponding to finitely generated parameterizations.
The second part is an analogous version which is appropriate for studying fractional Radon transforms corresponding to linearly finitely generated parameterizations.

Fix open sets $\Omega'\Subset\Omega''\Subset\Omega'''\Subset \Omega\subseteq \R^n$.  Let $Z=\{Z_1,\ldots, Z_q\}$ be a finite set of
smooth vector fields on $\Omega$.
We assume these vector fields satisfy the conditions of Theorem \ref{ThmQuantFrob}, uniformly for $x_0\in \Omega'$.  It, therefore, makes
sense to talk about $m$-admissible constants as in that section; and these constants can be chosen to be independent of $x_0\in \Omega'$.

Let $\gammah(t,x):B^N(\rho)\times \Omega'''\rightarrow \Omega$ be controlled at the unit scale on $\Omega'$ by $Z$,
with $\gammah(0,x)\equiv x$.
Fix $a>0$ small (to be chosen later). Let $\kappa(t,x)\in C^{\infty}(B^{N}(a)\times \Omega'')$, $\psi_1,\psi_2\in C_0^{\infty}(\Omega')$, and $\vsig\in C_0^{\infty}(B^N(a))$.
Decompose $\R^{N}=\R^{N_1}\times \R^{N_2}$ and write $t=(t_1,t_2)\in \R^{N_1}\times \R^{N_2}$.
Define $\Wh$ as in \eqref{EqnFrobControlWh}.
We separate our assumptions on $\Wh$ into two possible cases.

\noindent{\bf Case I:}  Setting $t_2=0$ and expanding $\Wh(t_1,0)$ as a Taylor series in the $t_1$ variable
\begin{equation}\label{EqnPfLtTSeries}
\Wh(t_1,0)\sim \sum_{|\alpha|>0} t_1^{\alpha} Z_{\alpha},
\end{equation}
where the $Z_\alpha$ are smooth vector fields on $\Omega''$.  We assume that there exists $M\in \N$ such that the following holds.
Let $\sF_1^M:=\{Z_\alpha : |\alpha|\leq M\}$.  Recursively define
\begin{equation*}
\sF_j^M := \sF_{j-1}^M \bigcup\q[ \bigcup_{1\leq k,l\leq j} \q\{ [Y_1,Y_2] : Y_1\in \sF_k^M, Y_2\in \sF_{l}^M\w\}\w].
\end{equation*}
We assume that $\sF_M^M$ controls $Z$ at the unit scale on $\Omega'$.

\noindent{\bf Case II:}  Setting $t_2=0$ and expanding $\Wh(t_1,0)$ as a Taylor series in the $t_1$ variable
as in \eqref{EqnPfLtTSeries}, we assume $\{ Z_{\alpha} : |\alpha|=1\}$ controls $Z$ at the unit scale on $\Omega'$.

We also separate our assumptions on $\vsig$ into the same two cases:

\noindent{\bf Case I:} $\int \vsig(t_1,t_2)\: dt_2=0$.

\noindent{\bf Case II:} For some $L\in \N$, $\vsig(t_1,t_2)=\sum_{|\alpha|=L} \partial_{t_2}^{\alpha} \vsig_\alpha(t_1,t_2)$,
where $\vsig_{\alpha}\in C_0^{\infty}(B^N(a))$.

For $\zeta\in(0,1]$, define an operator
\begin{equation*}
D_{\zeta} f(x) = \psi_1(x) \int f(\gammah_{t_1,\zeta t_2}(x)) \psi_2(\gammah_{t_1,\zeta t_2}(x)) \kappa(t_1, \zeta t_2,x) \vsig(t)\: dt.
\end{equation*}

\begin{lemma}\label{LemmaPfLtMainLemma}
If $a>0$ is chosen sufficiently small, then we have:
\begin{itemize}
\item In Case I, there exists $\epsilon>0$ so that $\LpOpN{2}{D_\zeta}\lesssim \zeta^{\epsilon}$.
\item In Case II, $\LpOpN{2}{D_\zeta}\lesssim \zeta^{L/2}$.
\end{itemize}
\end{lemma}

In our applications of Lemma \ref{LemmaPfLtMainLemma} it is important that we be explicit about what parameters the constants may depend on.
Because $Z$ satisfies the conditions of Theorem \ref{ThmQuantFrob}, it makes sense to talk about $m$-admissible constants as in that theorem.
In this section, we say $C$ is an admissible constant, if $C$ can be chosen to depend only on the following:
\begin{itemize}
\item Anything an $m$-admissible constant can depend on, where $m$ can be chosen to depend only on $q$, $n$, and $M$ in Case I, or $q$, $n$, and $L$
in Case II.
\item  $\|\kappa\|_{C^K(B^{N}(a)\times \Omega')}$, $\|\psi_1\|_{C^K(\Omega')}$, and $\|\psi_2\|_{C^K(\Omega')}$, where $K$ can be chosen
to depend only on $q$, $n$, $N$, and $M$ in Case I, or $q$, $n$, $N$, and $L$ in Case II.

\item In Case I, $\| \vsig\|_{C^K(B^{N}(a))}$,  where $K$ can be chosen
to depend only on $q$, $n$, $N$, and $M$.

\item In Case II, $\|\vsig_{\alpha}\|_{C^K(B^N(a))}$, $|\alpha|=L$, where $K$ can be chosen to depend only on  $q$, $n$, $N$, and $L$.

\item Because $Z$ controls $\gammah$ at the unit scale on $\Omega'$, we have that $\sQ_1(\rho_1,\tau_1, \{\sigma_1^m\}_{m\in \N})$
holds at each point $x_0\in \Omega'$, for some $\rho_1$, $\tau_1$, and $\{\sigma_1^m\}$ independent of $x_0$.
$C$ can depend on $\rho_1$, $\tau_1$ and $\sigma_1^m$, where $m$ can be chosen to depend only on $q$, $n$, and $M$ in Case I,
or $q$, $n$, and $L$ in Case II.
\end{itemize}
In Lemma \ref{LemmaPfLtMainLemma}, $a$ and the implicit constants in $\lesssim$ are all admissible constants.
In Case I, $\epsilon$ can be chosen to depend only on $q$, $N$, and $M$.
For the remainder of this section, we write $A\lesssim B$ to denote $A\leq CB$, where $C$ is an admissible constant.

\begin{proof}[Proof of Lemma \ref{LemmaPfLtMainLemma} in Case I]
We first claim that we may, without loss of generality, replace $Z$ with $\sF_M^M$.
Indeed, because $\sF_M^M$ controls $Z_j$ at the unit scale on $\Omega'$, $\forall Z_j\in Z$, we have that $\sF_M^M$ controls $\gammah$ at the unit scale on $\Omega'$.
Because $Z$ controls $\gammah$ at the unit scale on $\Omega'$, we have that $Z$ controls $Y$ at the unit scale on $\Omega'$, $\forall Y\in \sF_M^M$.
Combining this with the fact that $\sF_M^M$ controls $Z_j$ at the unit scale on $\Omega'$, $\forall Z_j\in Z$, it follows that $\sF_M^M$ satisfies
the conditions of Theorem \ref{ThmQuantFrob}: indeed, if $X,Y\in \sF_M^M$, then $Z$ controls $[X,Y]$ at the unit scale on $\Omega'$,
and therefore $\sF_M^M$ controls $[X,Y]$ at the unit scale on $\Omega'$.
Hence, $\sF_M^M$ satisfies all the hypothesis that $Z$ satisfies in our assumptions and  
we may assume $Z=\sF_M^M$ in what follows.

It suffices to show that there exists $\epsilon>0$ (depending only on $q$, $n$, and $M$) such that
\begin{equation*}
\LpOpN{2}{D_\zeta^{*} D_\zeta D_\zeta^{*} D_\zeta}\lesssim \zeta^{\epsilon}.
\end{equation*}
Because $\LpOpN{2}{D_{\zeta}^{*}}\lesssim 1$, it suffices to show
\begin{equation}\label{EqnPfltTechCaseIToShow}
\LpOpN{2}{D_\zeta D_\zeta^{*} D_\zeta}\lesssim \zeta^{\epsilon}.
\end{equation}
Set $S_1=D_\zeta$, $S_2=D_\zeta^{*}$, $R_1=D_\zeta$, and $R_2=0$.
Note,
\begin{equation*}
R_2 f(x) = \psi_2(x) \int f(\gamma_{t_1, 0}(x)) \psi_2(\gamma_{t_1,0}(x)) \kappa(t_1,0,x) \vsig(t_1,t_2)\: dt_1 dt_2=0,
\end{equation*}
where we have used the fact that $\int \vsig(t_1,t_2)\: dt_2=0$.  I.e., $R_2$ is the same as $R_1$ but with $\zeta=0$.
\eqref{EqnPfltTechCaseIToShow} is equivalent to
\begin{equation}\label{EqnPfltTechCaseIToShowT}
\LpOpN{2}{S_1 S_2 (R_1-R_2)}\lesssim \zeta^{\epsilon},
\end{equation}
and this follows directly from Theorem 14.5 of \cite{SteinStreetI}.  See the proofs of Proposition 15.1 of \cite{SteinStreetI} and Theorem 10.1 of \cite{SteinStreetII} for similar
arguments with more details.
\end{proof}

\begin{proof}[Proof of Lemma \ref{LemmaPfLtMainLemma} in Case II]
Without loss of generality, we may take $Z=\{ Z_{\alpha} : |\alpha|=1\}$, where $Z_\alpha$ is as in \eqref{EqnPfLtTSeries}:
this follows as in Case I, replacing $\sF_M^M$ with $\{Z_{\alpha} : |\alpha|=1\}$ throughout the first paragraph of the proof of Case I.
Just as in Case I, we have that this choice of $Z$ satisfies all of the assumptions; in particular, it satisfies the assumptions of Theorem \ref{ThmQuantFrob}.

It follows from a simple change of variables that (for $a\gtrsim 1$ sufficiently small), $\LpOpN{1}{D_\zeta}\lesssim 1$.
Thus, by interpolation, it suffices to show $\LpOpN{\infty}{D_\zeta}\lesssim \zeta^{L}$.
Let $f$ be a bounded measurable function with $\sup_{x} |f(x)|=\LpN{\infty}{f}$.
Fix $x_0\in \Omega'$.  We will show
\begin{equation}\label{EqnPfLtTechToShowBigGain}
\q|D_\zeta f(x_0)\w|\lesssim \zeta^{L}\LpN{\infty}{f},
\end{equation}
with implicit constant not depending on $x_0$, and the result will follow.\footnote{It suffices to only consider $f$ such that $\sup_{x} |f(x)|=\LpN{\infty}{f}$ because
$\LpOpN{1}{D_{\zeta}}<\infty$, and therefore if $f$ is supported on a set of measure $0$, then $D_\zeta f=0$ almost everywhere.}

Because $\gammah$ is controlled by $Z$, it is easy to see that $D_{\zeta} f(x_0)$ depends only on the values of $f$ on $B_Z(x_0,\xi_3)$
provided we take $a\gtrsim 1$ to be sufficiently small,
and $\xi_3$ is the $2$-admissible constant from Theorem \ref{ThmQuantFrob}.  Let $\Phi$ be the map from Theorem \ref{ThmQuantFrob}
with this choice of $x_0$ and $Z$.  Let $\Phi^{\#} f = f\circ \Phi$ (the pullback via $\Phi$), and consider the operator
$T_\zeta:=  \Phi^{\#} D_{\zeta}\q(\Phi^{\#}\w)^{-1}$.  
Let $\eta>0$ be the $2$-admissible constant of the same name from Theorem \ref{ThmQuantFrob}, and let $n_0=\dim \Span{Z_1(x_0),\ldots, Z_q(x_0)}$.
Because $B_Z(x_0,\xi_3)\subseteq \Phi(B^{n_0}(\eta))$, and because $\Phi(0)=x_0$, to show \eqref{EqnPfLtTechToShowBigGain} it suffices to show,
for $\supp{g}\subseteq B^{n_0}(\eta)$,
\begin{equation}\label{EqnPfLtTeachToShowBigGaint}
\q|T_\zeta g(0)\w|\lesssim \zeta^{L} \LpN{\infty}{g}.
\end{equation}
Let $\theta_t(u) = \Phi^{-1}\circ \gammah_t\circ \Phi(u)$.  By Proposition \ref{PropFrobControlEquivsQ}, we have
\begin{equation}\label{EqnPfLtTClTheta}
\|\theta\|_{C^{L+1} (B^{N}(a)\times B^{n_0}(\eta) )}\lesssim 1.
\end{equation}
Also,
\begin{equation*}
\begin{split}
T_{\zeta} g(0) &= \psi_1\circ\Phi(0) \int g( \theta_t(0)) \psi_2(\theta_t(0)) \kappa(t,\Phi(0)) \zeta^{-N_2} \vsig(t_1, \zeta^{-1} t_2)\: dt.
\\&=\zeta^{L} \sum_{|\alpha|=L} \psi_1\circ \Phi(0) \int g(\theta_t(0)) \psi_2(\theta_t(0)) \kappa(t,\Phi(0)) \zeta^{-N_2} \partial_{t_2}^{\alpha} \vsig_\alpha(t_1,\zeta^{-1} t_2)\: dt.
\end{split}
\end{equation*}
Fix $\alpha$ with $|\alpha|=L$.  Set
\begin{equation*}
I_\alpha :=  \int g(\theta_t(0)) \psi_2(\theta_t(0)) \kappa(t,\Phi(0)) \zeta^{-N_2} \partial_{t_2}^{\alpha} \vsig_\alpha(t_1,\zeta^{-1} t_2)\: dt.
\end{equation*}
We will show 
\begin{equation}\label{EqnPfLtTeachToShowBigGainr}
|I_\alpha|\lesssim \LpN{\infty}{g},
\end{equation}
and then \eqref{EqnPfLtTeachToShowBigGaint} will follow, which completes the proof.
Because $\vsig_\alpha(t_1,t_2)$ can be written as a rapidly converging sum of functions of the form $\vsig_1(t_1)\vsig_2(t_2)$ with
$\vsig_1\in C_0^{\infty}(B^{N_1}(a))$, $\vsig_2\in C_0^{\infty}(B^{N_2}(a))$,\footnote{We are using the fact that
$C_0^{\infty}(B^{N_1}(a))\widehat{\otimes} C_0^{\infty}(B^{N_2}(a))\cong C_0^{\infty}(B^{N_1}(a)\times B^{N_2}(a))$, where $\widehat{\otimes}$
denotes the tensor product of these nuclear spaces.  See \cite{TrevesTopologicalVectorSpaces} for more details.}
and so it suffices to assume $\vsig_\alpha(t_1, t_2)$ is of the form $\vsig_\alpha(t_1,t_2)=\vsig_1(t_1)\vsig_2(t_2)$.  Thus we are considering
\begin{equation*}
I_\alpha :=  \int g(\theta_t(0)) \psi_2(\theta_t(0)) \kappa(t,\Phi(0))   \vsig_1(t_1) \zeta^{-N_2} \partial_{t_2}^{\alpha} \vsig_2(\zeta^{-1} t_2)\: dt.
\end{equation*}
Let $Y_1,\ldots, Y_q$ be the pullback of $Z_1,\ldots, Z_q$ via $\Phi$ to $B^{n_0}(\eta)$. 
Because we are assuming $Z=\{ Z_{\alpha} : |\alpha|=1\}$, our we have that $q=N_1$ and 
if we write $t_1=(t_1^1,\ldots, t_1^{q})$,
\begin{equation}\label{EqnPfLtDerivZero}
\frac{\partial}{\partial t_1^j} \theta_t(0)\bigg|_{t=0} = Y_j (0), \quad 1\leq j\leq q.
\end{equation}
By possibly rearranging the coordinates of $t_1$ and using $\eqref{EqnFrobYsSpan}$, we may assume
\begin{equation}\label{EqnPfLtTYsSpan}
\q|\det \q(Y_1(0)| \cdots | Y_{n_0}(0)\w)\w|\gtrsim 1.
\end{equation}
Separating $t_1=(t_{1,1},t_{1,2})$ where $t_{1,1}=(t_1^1,\ldots, t_1^{n_0})$ and $t_{1,2}=(t_1^{n_0+1},\ldots, t_1^{q})$,
\eqref{EqnPfLtTYsSpan} combined with \eqref{EqnPfLtDerivZero} shows
\begin{equation*}
\q| \det \frac{\partial}{\partial t_{1,1} }\theta_{t}(0)\big|_{t=0} \w|\gtrsim 1.
\end{equation*}
Applying a change of variables in the $t_{1,1}$ variable (setting $u=\theta_t$), taking $a\gtrsim 1$ sufficiently small, and using \eqref{EqnPfLtTClTheta}, we have
\begin{equation*}
I_{\alpha} = \int_{|u|,|t_{1,2}|<a} g(u) K(u,t_{1,2}, t_2) \zeta^{-N_2} \partial_{t_2}^{\alpha} \vsig_2(\zeta^{-1} t_2)\: du \: dt_{1,2}\: dt_2,
\end{equation*}
where $\|K\|_{C^L}\lesssim 1$.\footnote{See Theorem B.3.1 of \cite{StreetMultiParamSingInt} for a precise statement of this kind of change of variables.}
Integrating by parts, we have
\begin{equation*}
I_{\alpha} = (-1)^L\int_{|u|,|t_{1,2}|<a} g(u) \q[\partial_{t_2}^{\alpha} K(u,t_{1,2}, t_2)  \w] \zeta^{-N_2} \vsig_2(\zeta^{-1} t_2)\: du \: dt_{1,2}\: dt_2.
\end{equation*}
That $|I_{\alpha}|\lesssim \LpN{\infty}{g}$ now follows immediately.
\end{proof}

In the next two subsections, we present our main applications of Lemma \ref{LemmaPfLtMainLemma}.

%% file: pfl2techap1.tex
Fix open sets $\Omega_0\Subset\Omega'\Subset\Omega''\Subset\Omega'''\Subset \Omega \subseteq \R^n$.
For a set $F\subset \R^{\nu}$ define
\begin{equation*}
\diam F =\sup_{j,k\in F} |j-k|.
\end{equation*}
Suppose $(\gamma^1,e^1,N_1,\Omega,\Omega'''),\ldots, (\gamma^L,e^L, N_L, \Omega, \Omega''')$ are $L$
parameterizations, each with $\nu$-parameter dilations $0\ne e_1^l,\ldots e_{N_L}^l\in [0,\infty)^\nu$.
We separate our assumptions in this section into two cases:
\begin{enumerate}[\bf{Case} I:]
\item There exists a finite set $\sF\subset \snuvect$ such that $$(\gamma^1,e^1,N_1),\ldots, (\gamma^L,e^L,N_L)$$ are all finitely
generated by $\sF$ on $\Omega'$.

\item There exists a finite set $\sF\subset \snuvectone$ such that $$(\gamma^1,e^1,N_1),\ldots, (\gamma^L,e^L,N_L)$$
are all linearly finitely generated by $\sF$ on $\Omega'$.
\end{enumerate}

Fix $a>0$ small (how small to be chosen later).  For $1\leq l\leq L$, $\vsig\in \schS(\R^{N_l})$,
$\eta\in C_0^{\infty}(B^{N_l}(a))$, $j\in [0,\infty)^{\nu}$, $\psi_1,\psi_2\in C_0^{\infty}(\Omega_0)$,
and $\kappa\in C^{\infty}(B^{N_l}(a)\times \Omega'')$, define the operator
\begin{equation}\label{EqnPfLtDefnOfOpNew}
\begin{split}
&T^l_j[\vsig,\eta,\psi_1,\psi_2,\kappa] f(x)
\\&:=\psi_1(x) \int f(\gamma^l_{t^l}(x)) \psi_2(\gamma^l_{t^l}(x)) \kappa(t^l,x) \eta(t^l)\dil{\vsig}{2^j}(t^l)\: dt^l.
\end{split}
\end{equation}

\begin{prop}\label{PropPfLtTechApO}
For each $l\in\{1,\ldots, L\}$, let $\sB_1^l\subset \schS(\R^{N_l})$, $\sB_2^l\subset C_0^\infty(B^{N_l}(a))$, $\sB_3^l\subset C_0^\infty(\Omega_0)$,
and $\sB_4^l\subset C^\infty(B^{N_l}(a)\times \Omega'')$ be bounded sets.  For $j_1,\ldots, j_L\in [0,\infty)^{\nu}$,
$\vsig_l\in \schS_{\{\mu : j_{l,\mu}>0\}}\cap \sB_1^l$, $\eta_l\in \sB_2^l$, $\psi_{1,l},\psi_{2,l}\in \sB_3^l$, and $\kappa_l\in \sB_4^l$,
define
\begin{equation*}
T_{j_l}^l := T_{j_l}^l[\vsig_l,\eta_l,\psi_{1,l},\psi_{2,l},\kappa_l].
\end{equation*}
Then, there exists $a>0$ (depending on the parameterizations) such that the following holds.
\begin{itemize}
\item In Case I, there exists $\epsilon>0$ (depending on the parameterizations) such that
\begin{equation*}
\LpOpN{2}{ T_{j_1}^1 \cdots T_{j_L}^L} \leq C 2^{-\epsilon\diam \{j_1,\ldots, j_L\}},
\end{equation*}
where $C$ depends on the above parameterizations, and the sets $\sB_1^l$, $\sB_2^l$, $\sB_3^l$, and $\sB_4^l$.

\item In Case II, for every $N$, there exists $C_N$ (depending on the above parameterizations, and the sets $\sB_1^l$, $\sB_2^l$, $\sB_3^l$, and $\sB_4^l$)
such that
\begin{equation*}
\LpOpN{2}{ T_{j_1}^1 \cdots T_{j_L}^L} \leq C_N 2^{-N\diam \{j_1,\ldots, j_L\}}.
\end{equation*}
\end{itemize}
\end{prop}

In this section, we prove of Proposition \ref{PropPfLtTechApO}.

\begin{defn}\label{DefnPfLtTechApOPiSigma}
Given $\sS\subseteq \snuvect$, define the following:
\begin{itemize}
\item For $\mu\in \nuset$, let
\begin{equation*}
\pi_{\mu}(\sS):=\{ (X,d)\in \sS : d\text{ is nonzero in only the }\mu\text{ component} \}.
\end{equation*}

\item For $K\geq 0$, let
\begin{equation*}
\sigma_{K}(\sS) := \{ (X,d)\in \sS : |d|_1\leq K\}.
\end{equation*}

\item For $K\in \N$, let $\sL_K(\sS)$ be defined inductively as follows:
$\sL_1(\sS)=\sS$ and for $j\geq 1$,
\begin{equation*}
\begin{split}
&\sL_{j+1}(\sS):=
\\&\sL_{j}(\sS)\bigcup
\q[\bigcup_{\substack{k+l=j+1\\ 1\leq k,l}} \{ ([X_1,X_2], d_1+d_2) : (X_1,d_1)\in \sL_{k}(\sS), (X_2,d_2)\in \sL_{l}(\sS) \}\w].
\end{split}
\end{equation*}
Note, $\sL(\sS)=\bigcup_{j\in \N} \sL_j(\sS)$.
\end{itemize}
\end{defn}

\begin{lemma}\label{LemmaPfLtsLControl}
Suppose $\sS_1, \sS_2\subset \snuvectone$ are such that there exists a finite set $\sF\subset \snuvect$
with both $\sL(\sS_1)$ and $\sL(\sS_2)$ finitely generated by $\sF$ on $\Omega'$.
Then, there exists $L\in \N$ such that the following holds.  Pick any two subsets $\sP_1,\sP_2\subseteq \nuset$ with
$\sP_1\bigcup \sP_2=\nuset$.
Then, the set
\begin{equation}\label{EqnPfLtToShowsLControl}
\sL_L\q( \bigcup_{l\in \{1,2\}} \bigcup_{\mu\in \sP_l} \sigma_L(\pi_{\mu} \sS_l ) \w)
\end{equation}
is equivalent to $\sF$ on $\Omega'$.
\end{lemma}
\begin{proof}
Because $\sL(\sS_1)$ and $\sL(\sS_2)$ are both equivalent to $\sF$ on $\Omega'$, taking $\delta=(0,\ldots, 0,\delta_\mu,0,\ldots, 0)\in [0,1]^{\nu}$
in the definitions
shows $\sL(\pi_\mu \sS_1)$ and $\sL(\pi_{\mu} \sS_2)$ are equivalent to $\pi_\mu \sF$ on $\Omega'$.  In particular, $\pi_{\mu} \sF$ controls
$\pi_{\mu} \sS_1$ and $\pi_{\mu} \sS_2$ on $\Omega'$.  For $l=1,2$, $\sS_l=\bigcup_{\mu} \pi_{\mu} \sS_l$ and so,
\begin{equation*}
\sL\q(\bigcup_{\mu\in \nuset} \pi_{\mu} \sF \w)
\end{equation*}
controls $\sL(\sS_1)$ and $\sL(\sS_2)$ on $\Omega'$, and therefore, $\sL\q(\bigcup_{\mu\in \nuset} \pi_{\mu} \sF\w)$ controls $\sF$ on $\Omega'$.

Because $\sL(\pi_\mu \sS_1)$ and $\sL(\pi_{\mu} \sS_2)$ are both equivalent to $\pi_\mu \sF$ on $\Omega'$,
we have
$$\sL\q( \bigcup_{l\in \{1,2\}} \bigcup_{\mu\in \sP_l} \pi_{\mu} \sS_l  \w)$$
controls $\bigcup_{\mu} \pi_{\mu} \sF$ on $\Omega'$, and therefore controls $\sL\q(\bigcup_{\mu\in \nuset} \pi_{\mu} \sF\w)$ on $\Omega'$,
and therefore controls $\sF$ on $\Omega'$.  Because $\sF$ is finite, there exists $L$ such that the set in \eqref{EqnPfLtToShowsLControl}
controls $\sF$ on $\Omega'$.  It follows that, with this choice of $L$, the set in \eqref{EqnPfLtToShowsLControl} and $\sF$ are equivalent on $\Omega'$.
\end{proof}

\begin{lemma}\label{LemmaPfLtLinControl}
Suppose $\sS_1,\sS_2, \sF\subset \snuvectone$ are finite sets such that $\sS_1$, $\sS_2$, and $\sF$ are all equivalent on $\Omega'$.
Let $\sP_1,\sP_2\subseteq \nuset$ be such that $\sP_1\bigcup \sP_2=\nuset$.  Then
\begin{equation*}
\bigcup_{l\in \{1,2\}} \bigcup_{\mu\in \sP_l} \pi_{\mu} \sS_l
\end{equation*}
is equivalent to $\sF$ on $\Omega'$.
\end{lemma}
\begin{proof}
It follows immediately from the definition of control, by taking $$\delta=(0,\ldots, 0,\delta_{\mu},0,\ldots, 0)\in [0,1]^{\nu},$$
that $\pi_{\mu} \sS_1$, $\pi_{\mu} \sS_2$, and $\pi_{\mu} \sF$ are all equivalent on $\Omega'$.
Because $\sF=\bigcup_{\mu} \pi_{\mu} \sF$ (similarly for $\sS_1$ and $\sS_2$) by assumption, the result follows.
\end{proof}

The heart of the proof of Proposition \ref{PropPfLtTechApO} lies in the next lemma.
For it, we need some notation from Section \ref{SectionSchwartzAndProd}.  For $l=1,\ldots, L$, we let $t^l\in \R^{N_l}$.  For $1\leq \mu\leq \nu$,
we let $t^l_\mu$ denote the vector consisting of those coordinates $t^l_j$ of $t^l$ such that $e^l_{j,\mu}\ne 0$ (i.e., the $\mu$th component of $e^l_j$ is strictly positive).  We let $N_l^\mu$ denote the dimension of the $t^l_\mu$ variable,
and let $\Na_l=N_l^1+\cdots+N_l^\nu$.
We write $\alphaa_l = (\alpha_1^l,\ldots,\alpha_\nu^l)\in\N^{N_l^1}\times \cdots \times \N^{N_l^{\nu}}=\N^{\Na_l}$, and define $\partial_{t^l}^{\alphaa}$ as
in Section \ref{SectionSchwartzAndProd}.

\begin{lemma}\label{LemmaPfLtApOBaiscBound}
There exists $a>0$ (depending on $(\gamma^1, e^1, N_1)$ and $(\gamma^2, e^2,N_2)$) such that the following holds.
For $l=1,2$, let $\sB_1^l\subset C_0^\infty(B^{N_l}(a))$, $\sB_2^l \subset C_0^\infty(\Omega_0)$, and $\sB_3^l\subset C^\infty(B^{N_l}(a)\times \Omega'')$
be bounded sets.  Fix $M\in \N$.  Let $j_1,j_2\in [0,\infty)^{\nu}$.  For $l=1,2$ suppose the following.
$$\vsig_l=\sum_{\substack{ \alphaa_l\in \N^{\Na_l} \\ |\alpha_\mu^{l}|=M\text{ when }j_l^{\mu}\ne 0 \\ |\alpha_\mu^{l}|=0\text{ when }  j_{l}^{\mu}=0}}\partial_{t^l}^{\alphaa} \gamma_{l,\alphaa},
$$
where $\gamma_{l,\alphaa}\in \sB_1^l$.  Let $\psi_{1,l},\psi_{2,l}\in \sB_2^l$, and $\kappa_l\in \sB_3^l$.  Define,
\begin{equation*}
T_{j_l}^{l}f(x) = \psi_{1,l}(x) \int f(\gamma_{t^l}^l(x)) \kappa_l(t^l,x) \psi_{2,l}(\gamma_{t^l}^{l}(x)) \dil{\vsig}{2^{j_l}}(t^l)\: dt^l.
\end{equation*}
Then,
\begin{itemize}
\item In Case I, if $M\geq 1$, there exists $\epsilon>0$ (depending on $(\gamma^1, e^1, N_1)$ and $(\gamma^2, e^2,N_2)$) such that
\begin{equation*}
\LpOpN{2}{T_{j_1}^1 T_{j_2}^2}\leq C 2^{-\epsilon|j_1-j_2|},
\end{equation*}
where $C$ depends on $(\gamma^1, e^1, N_1)$, $(\gamma^2,e^2, N_2)$, $\sB_l^1$, $\sB_l^2$, and $\sB_l^3$, $l=1,2$.

\item In Case II, for every $N$, there exists a choice of $M$ and a constant $C_N$ (depending on  $(\gamma^1, e^1, N_1)$, $(\gamma^2,e^2, N_2)$, $\sB_l^1$, $\sB_l^2$, and $\sB_l^3$, $l=1,2$) such that
\begin{equation*}
\LpOpN{2}{T_{j_1}^1 T_{j_2}^2}\leq C_N 2^{-N|j_1-j_2|}.
\end{equation*}
\end{itemize}
\end{lemma}
\begin{proof}
We prove the two cases simultaneously.
If $j_1=j_2$, the result is obvious, so we assume $|j_1-j_2|>0$.  Set $j=j_1\wedge j_2\in [0,\infty)^{\nu}$ (the coordinatewise minimum of $j_1$ and $j_2$).
Set $Z=\{2^{-j\cdot d} X : (X,d)\in \sF\}$.  Because $\sF$ satisfies $\sD(\Omega')$, $Z$ satisfies the conditions of Theorem \ref{ThmQuantFrob}, uniformly
for $x_0\in \Omega'$.

Consider,
\begin{equation*}
\begin{split}
T_{j_1}^1 T_{j_2}^2 f(x) =
\psi_{1,1}(x) \int &f(\gamma^2_{2^{-j_2} t^2}\circ \gamma^1_{2^{-j_1} t^1}(x)) \psi_{2,2}(\gamma^2_{2^{-j_2} t^2}\circ \gamma^1_{2^{-j_1} t^1}(x) )
\\&\kappa(2^{-j_1} t^1, 2^{-j_2} t^2,x) \vsig_1(t^1) \vsig_2(t^2)\: dt^1 dt^2,
\end{split}
\end{equation*}
where $\kappa\in C^\infty(B^{N_1}(a)\times B^{N_2}(a)\times \Omegat)$ and $\Omega'\Subset \Omegat\Subset \Omega''$, provided $a>0$ is chosen small,
and $\kappa$ ranges over a bounded set as the various parameters in the problem vary.

Take $\mu$ so that $|j_1^{\mu}-j_2^{\mu}|=|j_1-j_2|_\infty>0$.  We assume $j_1^{\mu}>j_2^{\mu}$--the reverse case is nearly identical, and we leave
it to the reader.  Separate $t^1$ into two variables, $t^1=(t^1_{\mu}, t^1_{\mu^{\perp}})$, where $t^{1}_{\mu}$ denotes the vector consisting of those coordinates $t_j^1$ of $t^1$
so that $e^1_{j,\mu}\ne 0$, and $t^1_{\mu^{\perp}}$ denotes the vector consisting of the rest of the coordinates.  We write $2^{-j_1} t^1_{\mu}$ and $2^{-j_1} t^1_{\mu^{\perp}}$
for the corresponding coordinates of $2^{-j_1} t^1$.  Let $\zeta=2^{-(j_{1}^{\mu} -j_2^{\mu})c}$, where $c=\min\{e^1_{j,\mu} : 1\leq j\leq N_1, e^{1}_{j,\mu}\ne 0\}>0$.
Note that if $t^{1}_{\mu}=(t^{1}_{\mu,1},\ldots, t^{1}_{\mu,l_\mu})$, we have
\begin{equation*}
2^{-j_1} t^{1}_\mu = \zeta 2^{-j} (\xi_1 t^{1}_{\mu,1}, \ldots, \xi_{l_{\mu}} t^{1}_{\mu,l_\mu} ),
\end{equation*}
where $\xi_1,\ldots, \xi_{l_{\mu}}\in (0,1]$ (here, $\xi_1,\ldots, \xi_{l_{\mu}}$ depend on $j_1,j_2$). By our hypotheses, $\gamma_{2^{-j} t^1}^1(x)$ and $\gamma_{2^{-j} t^2}^2$ are controlled at the unit scale
by $Z$ on $\Omega'$ (here, and everywhere else in this proof, all such conclusions are uniform in $j_1,j_2$).
It follows immediately from the definitions (see also Proposition 12.7 of \cite{SteinStreetI}) that
\begin{equation*}
\gamma^1_{(2^{-j}(\xi_1  t^{1}_{\mu,1}, \ldots, \xi_{l_{\mu}} t^{1}_{\mu,l_\mu} ), 2^{-j_1}t^{1}_{\mu^{\perp}})}(x), \quad \gamma^2_{2^{-j_2} t^2}(x)
\end{equation*}
are controlled at the unit scale by $Z$.
Let $s_1=(t^{1}_{\mu^{\perp}}, t^2)$ and $s_2=t^{1}_\mu$.
Applying Proposition 12.6 of \cite{SteinStreetI}, we have
\begin{equation*}
\gammah_{s_1,s_2}(x):=  \gamma^2_{2^{-j_2} t^2} \circ \gamma^1_{(2^{-j}(\xi_1  t^{1}_{\mu,1}, \ldots, \xi_{l_{\mu}} t^{1}_{\mu,l_\mu} ), 2^{-j_1}t^{1}_{\mu^{\perp}})}(x)
\end{equation*}
is controlled at the unit scale by $Z$ on $\Omega'$.
Set $\vsigt(s_1,s_2):= \vsig_1(t^1) \vsig_2(t^2)$.  By our hypotheses,
\begin{equation*}
\vsigt(s_1,s_2) = \sum_{|\alpha|=M} \partial_{s_2}^{\alpha} \vsigt_{\alpha}(s_1,s_2),
\end{equation*}
where $\{\vsigt_{\alpha}\}\subset C_0^{\infty}(B^{N_1+N_2}(2a))$ is a bounded set (where $\{\vsigt_{\alpha}\}$ denotes the set of all $\vsigt_\alpha$
as the various parameters in the problem vary).
Also, define
\begin{equation*}
\kappat(s_1,s_2):= \psi_{2,2}( \gammah_{s_1,s_2}(x)) \kappa( (2^{-j}(\xi_1  t^{1}_{\mu,1}, \ldots, \xi_{l_{\mu}} t^{1}_{\mu,l_\mu} ), 2^{-j_1}t^{1}_{\mu^{\perp}}), 2^{-j_2} t^{2}, x),
\end{equation*}
so that $\{\kappat\}\subset C^{\infty}(B^{N_1}(a)\times B^{N_2}(a)\times \Omegat)$ is a bounded set (again, $\{\kappat\}$ denotes the set of all such $\kappat$
as the various parameters vary).  Using this notation we have
\begin{equation*}
T_{j_1}^1 T_{j_2}^2 f(x) = \psi_{1,1}(x) \int f(\gammah_{s_1,\zeta s_2}(x)) \kappat(s_1, \zeta s_2, x) \vsig(s_1,s_2)\: ds_1 \: ds_2.
\end{equation*}
The goal is to apply Lemma \ref{LemmaPfLtMainLemma} to this operator.

Define
\begin{equation*}
\Wh(s,x) = \frac{\partial}{\partial \epsilon}\bigg|_{\epsilon=1} \gammah_{\epsilon s}\circ \gammah_{s}^{-1}(x),
\end{equation*}
and suppose the parameterizations $(\gamma^1,e^1,N_1)$ and $(\gamma^2,e^2,N_2)$ correspond to the vector field
parameterizations $(W^1, e^1, N_1)$ and $(W^2, e^2, N_2)$, respectively.  Returning to the $t^1$, $t^2$ coordinates
and treating $\Wh$ as a function of $(t^1, t_2)=((t_{\mu}^1, t^{1}_{\mu^{\perp}}), t^2)$, we have
\begin{equation}\label{EqnPfLtWhDecomp}
\Wh((0,t^1_{\mu^{\perp}}),0) = W^1(0,2^{-j_1} t^1_{\mu^{\perp}}), \quad \Wh(0,t^2)=W^2(2^{-j_2} t^2).
\end{equation}

We now separate the proof into the two cases.  We begin with Case I, and we wish to show that the hypotheses of Case I of Lemma \ref{LemmaPfLtMainLemma}
hold in this situation, with the above choice of $Z$.  At this point it is a matter of unravelling the definitions.  We have already seen
that $Z$ controls $\gammah$ at the unit scale on $\Omega'$, and it is clear that when $M\geq 1$, $\int \vsig(s_1,s_2)\: ds_2=0$.
Write $W_1(t^{1})\sim \sum_{|\alpha|>0} (t^1)^{\alpha} X_{\alpha}^1$, and $W_2(t^2)\sim \sum_{|\alpha|>0} (t^2)^{\alpha} X_\alpha^2$.
For $l\in \{1,2\}$ let $\sS_l:=\{ (X_\alpha^l, \deg(\alpha)) : \deg(\alpha)\in \onecompnu\}$.  By our assumptions,
both $\sL(\sS_1)$ and $\sL(\sS_2)$ are equivalent to $\sF$ on $\Omega'$.
For $l\in {1,2}$, let $\sP_l:=\{ \mu'\in \nuset : j_{l}^{\mu'}=j_{\mu'}\}$; i.e., $\mu'\in \sP_l$ if and only if $j_l^{\mu'}=j_1^{\mu'}\wedge j_2^{\mu'}$.
Clearly $\sP_1\bigcup \sP_2=\nuset$ and $\mu\not \in \sP_1$.  Lemma \ref{LemmaPfLtsLControl} shows that there exists $L$ such that the set in
\eqref{EqnPfLtToShowsLControl} is equivalent to $\sF$ on $\Omega'$.  In particular, the set in \eqref{EqnPfLtToShowsLControl} controls
$\sF$ on $\Omega'$.
Notice that there exists $K=K(L)$ such that for each
$$(X,d)\in \bigcup_{l\in \{1,2\}} \bigcup_{\mu'\in \sP_l} \sigma_L(\pi_{\mu} \sS_l ),$$
the following holds:
\begin{itemize}
\item $d$ is nonzero in precisely one component:  $d_{\mu'}\ne 0$ for only one $\mu'\in \nuset$.
\item  This $\mu'$ satisfies $\mu'\in \sP_1\cup \sP_2=\nuset$.
\item If $\mu'\in \sP_1$, $X$ appears as a Taylor coefficient in  $ W^1(0,t^1_{\mu^{\perp}})$ corresponding to
some multi-index $\alpha$ with $|\alpha|\leq K$ and $\deg(\alpha)=d$.  Here we have used $\mu\not\in \sP_1$, so that $\mu'\ne \mu$ in this case.
\item If $\mu'\in \sP_2$, $X$ appears as a Taylor coefficient in  $W^2(t^2)$ corresponding to
some multi-index $\alpha$ with $|\alpha|\leq K$ and $\deg(\alpha)=d$.
\end{itemize}

Combining this with \eqref{EqnPfLtWhDecomp}, and using that $d$ is nonzero in only one component, shows that $2^{-j\cdot d} X$ appears as a Taylor coefficient
in $\Wh(s_1,s_2)$ corresponding to a mulit-index of order $\leq K$.
Combining this with the fact that the set in \eqref{EqnPfLtToShowsLControl} controls $\sF$ on $\Omega'$ shows that the hypotheses of Lemma \ref{LemmaPfLtsLControl}
hold in this case.  It follows that there exists $\epsilon,\epsilon',\epsilon''>0$ with
\begin{equation*}
\LpOpN{2}{T_{j_1}^1 T_{j_2}^2}\lesssim \zeta^{\epsilon}\lesssim  2^{-\epsilon'(j_1^{\mu}-j_2^{\mu})}\lesssim2^{-\epsilon''|j_1-j_2|},
\end{equation*}
as desired.

We now turn to Case II, and we wish to show that the hypotheses of Case II of Lemma \ref{LemmaPfLtMainLemma} hold in this situation with
the above choice of $Z$, and it again is just a matter of unravelling the definitions.  We have already seen that
$Z$ controls $\gammah$ at the unit scale on $\Omega'$, and we have $\vsig(s_1,s_2)=\sum_{|\alpha|=M} \partial_{s_2}^{\alpha} \vsig_{\alpha}(s_1,s_2)$,
where $\vsig_{\alpha}\in C_0^{\infty}(B^{N_1}(a)\times B^{N_2}(a))$ ranges over a bounded set as the various parameters in the problem vary.
Proceeding as in Case I, we write $W_1(t^1)$ and $W_2(t^2)$ as Taylor series.  For $l=1,2$, we set
$\sS_l:=\{ (X_\alpha^l, \deg(\alpha)) : \deg(\alpha)\in \onecompnu\text{ and } |\alpha|=1\}$.  Our hypotheses imply that
$\sS_1$, $\sS_2$, and $\sF$ are all equivalent.  Define $\sP_1$ and $\sP_2$ as in Case I; i.e., $\sP_l:=\{ \mu'\in \nuset : j_{l}^{\mu'}=j_{\mu'}\}$.
Note $\mu\not\in \sP_1$.  By Lemma \ref{LemmaPfLtLinControl},
\begin{equation}\label{EqnPfLtToShowLinControl}
\bigcup_{l\in \{1,2\}} \bigcup_{\mu'\in \sP_l} \pi_{\mu'} \sS_l
\end{equation}
is equivalent to $\sF$ on $\Omega'$.  Note that for $(X,d)\in \bigcup_{l\in \{1,2\}} \bigcup_{\mu'\in \sP_l} \pi_{\mu'} \sS_l$,
the following holds:
\begin{itemize}
\item $d$ is nonzero in precisely one component: $d_{\mu'}\ne 0$ for only one $\mu'\in \nuset$.
\item This $\mu'$ satisfies $\mu'\in \sP_1\cup \sP_2=\nuset$.
\item If $\mu'\in \sP_1$, $X$ appears as a Taylor coefficient in $ W^1(0,t^1_{\mu^{\perp}})$ corresponding to
some multi-index $\alpha$ with $|\alpha|=1$ and $\deg(\alpha)=d\in \onecompnu$.  Here we have used $\mu\not\in \sP_1$,
so $\mu'\ne \mu$ in this case.
\item If $\mu'\in \sP_2$, $X$ appears as a Taylor coefficient in  $W^2(t^2)$ corresponding to
some multi-index $\alpha$ with $|\alpha|=1$ and $\deg(\alpha)=d\in \onecompnu$.
\end{itemize}
Combining this with \eqref{EqnPfLtWhDecomp}, and using that $d$ is nonzero in only one component, shows that $2^{-j\cdot d} X$ appears as a Taylor coefficient
in $\Wh(s_1,s_2)$ corresponding to a mulit-index of order $1$.
Using that the set in \eqref{EqnPfLtToShowLinControl} controls $\sF$ on $\Omega'$ shows the the hypotheses of Lemma \ref{LemmaPfLtsLControl}
hold in this case.  Thus we have, for some $c,c'>0$ (independent of any relevant parameters)
\begin{equation*}
\LpOpN{2}{T_{j_1}^1 T_{j_2}^2}\lesssim \zeta^{M/2}\lesssim  2^{-cM(j_1^{\mu}-j_2^{\mu})}\lesssim2^{-c' M |j_1-j_2|}.
\end{equation*}
The result follows.
\end{proof}

\begin{proof}[Proof of Proposition \ref{PropPfLtTechApO}]
The result for $L=1$ is trivial.  The result for $L>2$ follows from the result for $L=2$ and $L=1$, so we prove only the result for $L=2$.
Fix $M\in \N$.  In Case I, we take $M=1$.  In Case II, we take $M=M(N)$ large to be chosen later.
For $l\in \{1,2\}$, we apply Proposition \ref{PropPfKerDecompVsig} to write
\begin{equation*}
\eta_l(t) \dil{\vsig_l}{2^{j_l}}(t)=\eta_l(t)\sum_{\substack{k\leq j_l\\ k\in \N^{\nu} }} \dil{\vsig_{l,k}}{2^k}(t),
\end{equation*}
where $t\in N_l$ and we are using the dilations $e^l$ to define $\dil{\vsig_l}{2^{j_l}}$.  Here, using the notation of Proposition \ref{PropPfKerDecompVsig},
we have
\begin{equation*}
\vsig_{l,k}=\sum_{\substack{\alphaa\in \N^{\Na_l} \\ |\alpha_{\mu}|=M\text{ when }k_\mu\ne 0 \\ |\alpha_\mu|=0\text{ when }k_\mu=0 }} \partial_t^{\alphaa} \gamma_{k,\alphaa},
\end{equation*}
where, for every $L\in \N$,
\begin{equation}\label{EqnPfLtApODecomp}
\begin{split}
&\big\{ 2^{L|j_l-k|} \gamma_{k,\alphaa} : j_l\in [0,\infty)^{\nu}, \vsig_{l}\in \schS_{\{\mu : j_{l,\mu}>0 \}}\cap \sB_1^l, \eta_l\in \sB_2^l, k\leq j_l, k\in \N^{\nu}
\\&\quad\quad\quad\quad\quad\quad\quad
|\alpha_\mu|=M\text{ when }k_\mu\ne 0,
|\alpha_\mu|=0 \text{ when }k_\mu=0
\big\}
\\&\subset C_0^{\infty}(B^N(a))
\end{split}
\end{equation}
is a bounded set.

Define $T^l_{j_l,k}$, for $k\in \N^{\nu}$ with $k\leq j_l$, by
\begin{equation*}
T_{j_l,k}^l f(x) = \psi_{1,l}(x) \int f(\gamma_{t^l}^l(x)) \psi_{2,l}(\gamma_{t^l}^l(x)) \kappa(t^{l},x) \eta_l(t_l)\dil{\vsig_{l,k}}{2^k}(t^l)\: dt^l.
\end{equation*}
We have,
\begin{equation*}
T_{j_l}^l = \sum_{\substack{k\in \N^{\nu} \\ k\leq j_l}} T_{j_l,k}^l,
\end{equation*}
and it follows that
\begin{equation}\label{EqnPfLtDecompTs}
\LpOpN{2}{T_{j_1}^1 T_{j_2}^2}\leq \sum_{\substack{k_1,k_2\in \N^{\nu} \\ k_1\leq j_1, k_2\leq j_2}} \LpOpN{2}{T_{j_1,k_1}^1 T_{j_2,k_2}^2}.
\end{equation}

In Case I, we apply Lemma \ref{LemmaPfLtApOBaiscBound} and use \eqref{EqnPfLtApODecomp} (with $L=1$) to see that (if $a>0$ is chosen small enough, depending only on the parameterizations),
there exists $0<\epsilon\leq \frac{1}{2}$ such that
\begin{equation*}
\LpOpN{2}{T_{j_1,k_1}^1 T_{j_2,k_2}^2}\lesssim 2^{-|j_1-k_1|-|j_2-k_2|-\epsilon|k_1-k_2|}\leq 2^{-(1/2) (|j_1-k_2|+|j_2-k_2|) - \epsilon|j_1-j_2|}.
\end{equation*}
Plugging this into \eqref{EqnPfLtDecompTs}, we have
\begin{equation*}
\LpOpN{2}{T_{j_1}^1 T_{j_2}^2}\lesssim \sum_{\substack{k_1,k_2\in \N^{\nu} \\ k_1\leq j_1, k_2\leq j_2}} 2^{-(1/2) (|j_1-k_2|+|j_2-k_2|) - \epsilon|j_1-j_2|}\lesssim 2^{-\epsilon|j_1-j_2|},
\end{equation*}
as desired.

In Case II, we take $M=M(N)$ large, apply Lemma \ref{LemmaPfLtApOBaiscBound}, and use \eqref{EqnPfLtApODecomp} (with $L=N+1$) to see that (if $a>0$
is chosen small enough, depending only on the parameterizations),
\begin{equation*}
\begin{split}
\LpOpN{2}{T_{j_1,k_1}^1 T_{j_2,k_2}^2}&\lesssim 2^{-(N+1)|j_1-k_1|-(N+1)|j_2-k_2|-N|k_1-k_2|}
\\&\lesssim 2^{-|j_1-k_1|-|j_2-k_2|-N|j_1-j_2|}.
\end{split}
\end{equation*}
Plugging this into \eqref{EqnPfLtDecompTs}, we have
\begin{equation*}
\LpOpN{2}{T_{j_1}^1 T_{j_2}^2}\lesssim \sum_{\substack{k_1,k_2\in \N^{\nu} \\ k_1\leq j_1, k_2\leq j_2}}  2^{-|j_1-k_1|-|j_2-k_2|-N|j_1-j_2|} \lesssim 2^{-N|j_1-j_2|},
\end{equation*}
as desired.
\end{proof}

%% file: pfl2techap2.tex
In this section we present another application of Lemma \ref{LemmaPfLtMainLemma}.  Here the setting
is the same as in Section \ref{SectionResSobCompNew}.
Fix open sets $\Omega_0\Subset\Omega'\Subset\Omega''\Subset\Omega'''\Subset \Omega\subseteq \R^n$.
Let $\nut,\nuh\in \N$, and suppose $\sSt\subset \snutvectone$, $\sSh\subset \snuhvectone$ are finite sets.
Let $\nu=\nut+\nuh$ and define
\begin{equation*}
\sS:=\q\{ \q(\Xh, (\hd,0_{\nut})\w) : (\Xh, \hd)\in \sSh \w\} \bigcup\q\{ \q(\Xt, (0_{\nuh},\td)\w) : (\Xt, \td)\in \sSt \w\}\subset \snuvectone.
\end{equation*}
We separate our assumptions into two cases:
\begin{enumerate}[\bf{Case} I:]
\item $\sL(\sS)$ is finitely generated by some $\sF\subset \snuvect$ on $\Omega'$.
\item $\sL(\sS)$ is linearly finitely generated by some $\sF\subset \snuvectone$ on $\Omega'$.
\end{enumerate}

Let $\lambda$ be a $\nut\times\nuh$ matrix whose entries are all in $[0,\infty]$.
 In both Case I and Case II, we assume
\begin{equation*}
\sL(\sSh)\text{ }\lambda\text{-controls }\sSt\text{ on }\Omega'.
\end{equation*}
We suppose we are given a parameterization $(\gamma,e,N,\Omega,\Omega''')$ with $\nu$-parameter dilations such that
\begin{enumerate}[\bf{Case} I:]
\item $(\gamma,e,N)$ is finitely generated by $\sF$ on $\Omega'$.
\item $(\gamma,e,N)$ is linearly finitely generated by $\sF$ on $\Omega'$.
\end{enumerate}

Fix $a>0$ small (how small to be chosen later).  For $\vsig\in \schS(\R^{N})$, $\eta\in C_0^{\infty}(B^{N}(a))$,
$j\in [0,\infty)^{\nu}$, $\psi_1,\psi_2\in C_0^\infty(\Omega_0)$, and $\kappa\in C^\infty(B^{N}(a)\times \Omega'')$, define the operator
\begin{equation*}
T_j[\vsig,\eta,\psi_1,\psi_2,\kappa]f(x):=\psi_1(x) \int f(\gamma_{t}(x)) \psi_2(\gamma_t(x)) \kappa(t,x)\dil{\vsig}{2^j}(t)\: dt.
\end{equation*}

\begin{prop}\label{PropPfLtApTMainProp}
Let $\sB_1\subset \schS(\R^N)$, $\sB_2\subset C_0^\infty(B^N(a))$, $\sB_3\subset C_0^\infty(\Omega_0)$,  and $\sB_4\subset C^\infty(B^N(a)\times \Omega'')$
be bounded sets.  For $j=(\jh,\jt)\in [0,\infty)^{\nuh}\times [0,\infty)^{\nut}= [0,\infty)^\nu$, $\vsig\in \schS_{\{\mu:j_\mu>0\}}\cap \sB_1$, $\eta\in \sB_2$, $\psi_1,\psi_2\in \sB_3$, and $\kappa\in \sB_4$,
define
\begin{equation*}
T_{j}:=T_j[\vsig,\eta,\psi_1,\psi_1,\kappa].
\end{equation*}
Then, there exists $a>0$ (depending on $(\gamma,e,N)$) such that the following holds.
\begin{itemize}
\item In Case I, there exists $\epsilon>0$ (depending on $(\gamma,e,N)$) such that
\begin{equation*}
\LpOpN{2}{T_j}\leq C 2^{-\epsilon|\jt\vee  \lambda(\jh) - \lambda(\jh)| },
\end{equation*}
where $C$ depends on $(\gamma,e,N)$, $\sB_1$, $\sB_2$, $\sB_3$, and $\sB_4$.

\item In Case II, for every $L$, there exists $C_L$ (depending on $(\gamma,e,N)$ and the sets $\sB_1$, $\sB_2$, $\sB_2$, and $\sB_4$) such that
\begin{equation*}
\LpOpN{2}{T_j}\leq C_L 2^{-L |\jt\vee  \lambda(\jh) - \lambda(\jh)|}.
\end{equation*}
\end{itemize}
In the above, for vectors $j, k$ we have written $j\vee k$ to denote the coordinatewise maximum.
\end{prop}

The rest of this section is devoted to the proof of Proposition \ref{PropPfLtApTMainProp}.  The key lies in the next lemma.
We use some notation from Section \ref{SectionSchwartzAndProd}.  For $t\in \R^{N}$, and for $1\leq \mu\leq \nu$,
we let $t_{\mu}$ denote the vector consisting  of those coordinates $t_j$ of $t$ such that $e_{j}^{\mu}\ne 0$.  We let $N_\mu$ denote the dimension of the
$t_{\mu}$ variable and let $\Na=N_1+\cdots+N_{\nu}$.  We write
$\alphaa=(\alpha_1,\ldots, \alpha_{\nu})\in \N^{N_1}\times \cdots \times \N^{N_{\nu}}=\N^{\Na}$,
and define $\partial_t^{\alphaa}$ as in Section \ref{SectionSchwartzAndProd}.

\begin{lemma}\label{LemmaPfLtAptMainLemma}
There exists $a>0$ (depending on $(\gamma,e,N$)) such that the following holds.  Let $\sB_1\subset C_0^{\infty}(B^N(a))$, $\sB_2\subset C_0^\infty(\Omega_0)$,
and $\sB_3\subset C^{\infty}(B^{N}(a)\times \Omega'')$ be bounded sets.  Fix $M\in \N$ and let $j\in [0,\infty)^{\nu}$.  Suppose
\begin{equation*}
\vsig = \sum_{\substack{\alphaa\in \N^N \\ |\alpha_{\mu}|=M\text{ when }j_{\mu}\ne 0 \\ |\alpha_\mu|=0\text{ when }j_\mu=0 }} \partial_t^{\alphaa} \gamma_{\alphaa},
\end{equation*}
where $\gamma_{\alphaa}\in \sB_1$.  Let $\psi_1,\psi_2\in \sB_2$, and $\kappa\in \sB_3$.  Define,
\begin{equation*}
T_j f(x) = \psi_1(x) \int f(\gamma_t(x)) \kappa(t,x) \psi_2(\gamma_t(x)) \dil{\vsig}{2^j}(t)\: dt.
\end{equation*}
Then,
\begin{itemize}
\item In Case I, if $M\geq 1$, there exists $\epsilon>0$ (depending on $(\gamma,e,N)$) such that
\begin{equation*}
\LpOpN{2}{T_j}\leq C 2^{-\epsilon|\jt\vee  \lambda(\jh) - \lambda(\jh)| },
\end{equation*}
where $C$ depends on $(\gamma,e,N)$, $\sB_1$, $\sB_2$, and $\sB_3$.

\item In Case II, for every $L$, there exists a  choice of $M=M(L)$ and a constant $C_L$ (depending on $(\gamma,e,N$), $\sB_1$, $\sB_2$, and $\sB_3$) such that
\begin{equation*}
\LpOpN{2}{T_j}\leq C_L 2^{-L |\jt\vee  \lambda(\jh) - \lambda(\jh)|}.
\end{equation*}
\end{itemize}
\end{lemma}
\begin{proof}
We prove the two cases simultaneously.  First notice that the assumption
that
$\sL(\sSh)$ $\lambda$-controls $\sSt$ on $\Omega'$ can be rephrased in the following way.
For $0\ne \td\in [0,\infty)^{\nut}$, define $h_{\td,\lambda}:[0,1]^{\nut}\rightarrow [0,1]$
by $h_{\td,\lambda}(2^{-\jh})=2^{-\lambda(\jh)\cdot \td}$.
Then we are assume that for each $(\Xt,\td)\in \sSt$, $\sL(\sSh)$ controls $(\Xt, h_{\td,\lambda})$ on
$\Omega'$.\footnote{Here we make the convention that $\infty\cdot 0=\infty$ in the definition of
$\lambda(\jh)$, but $\infty\cdot 0=0$ in the definition of dot product $\lambda(\jh)\cdot \td$.}

If $\jt_{\mu}\leq \lambda(\jh)_{\mu}$ for all $\mu\in \nuset$, the result it obvious.  Thus we assume for some $\mu$,
$\jt_{\mu}-\lambda(\jh)_{\mu}>0$ and we pick $\mu$ so that $\jt_{\mu}-\lambda(\jh)_{\mu}=|\jt\vee \lambda(\jh)-\lambda(\jh)|_{\infty}$.
Let $\lt=\jt\wedge \lambda(\jh)\in [0,\infty)^{\nut}$ (i.e., $\lt$ is the coordinatewise minimum of $\jt$ and $\lambda(\jh)$) and set $l=(\jh,\lt)\in [0,\infty)^{\nu}$.
Let $Z=\{2^{-l\cdot d} X: (X,d)\in \sF\}$.  Because $\sF$ satisfies $\sD(\Omega')$, $Z$ satisfies the conditions
of Theorem \ref{ThmQuantFrob}, uniformly for $x_0\in \Omega'$.
Note that $(\jh,\jt)_{\mu_0+\nuh}=\jt_{\mu_0}$ for $1\leq \mu_0\leq \nut$.

We decompose $t\in \R^{N}$ into two variables:  $t=(s_1,s_2)$.  $s_2$ is the vector consisting of those coordinates of $t$ which, when we compute
$2^{-j}t=2^{-(\jh,\jt)}t$ are multiplied by a power of $2^{-\jt_{\mu'}}$, where $\jt_{\mu'}>\lambda(\jh)_{\mu'}$.  $s_1$ is the vector consisting of the rest of the coordinates.
More precisely, $s_2$ is the vector consisting of those coordinates $t_k$ of $t$ such that $e_{k}^{\mu_0+\nuh}\ne 0$, where $\mu_0$ is such that
$\jt_{\mu_0}>\lambda(\jh)_{\mu_0}$; and $s_1$ is the vector consisting of the rest of the coordinates--note that $\mu$ is such a $\mu_0$.
We decompose $s_2$ into two variables: $s_2=(s_{2,1},s_{2,2})$.  $s_{2,2}$ is the vector consisting of those coordinates of $t$ which, when we compute
$2^{-j} t$ are multiplied by a power of $2^{-\jt_{\mu}}$--note that every such coordinate is a coordinate of $s_2$.  $s_{2,1}$ is the vector consisting of the rest of the coordinates
of $s_2$.  More precisely, $s_{2,2}$ consists of those coordinates $t_k$ of $t$ such that $e_{k}^{\mu+\nuh}\ne 0$; and $s_{2,1}$ consists of the rest of the coordinates of $s_2$.


Set $\zeta=2^{-c(\jt_{\mu}-\lambda(\jh)_{\mu})}$, where $c=\min\{ e_{l}^{\mu+\nuh} :1\leq l\leq N, e_{l}^{\mu+\nuh}\ne 0\}>0$.  Write $s_{2,1}=(s_{2,1}^1,\ldots, s_{2,1}^{N_1})$ and $s_{2,2}=(s_{2,2}^1,\ldots, s_{2,2}^{N_2})$.
The dilations $2^{-j}t$ induce dilations in the $(s_1,s_2)$ variables, which we again denote by $2^{-j}(s_1,s_2)$.
Using these choices, we may write
\begin{equation*}
2^{-j}(s_1,s_2) = 2^{-l} ( s_1,  (\xi_{1,1} s_{2,1}^1,\ldots, \xi_{1,N_1} s_{2,1}^{N_1}), \zeta (\xi_{2,1} s_{2,2}^1,\ldots, \xi_{2,N_2} s_{2,2}^{N_2} )  ),
\end{equation*}
where $\xi_{1,1},\ldots, \xi_{1,N_1},\xi_{2,1},\ldots, \xi_{2,N_2}\in (0,1]$ (here, $\xi_{k_1,k_2}$  depends on $\jt,\jh$).

By our hypotheses, $\gamma_{2^{-l}t}(x)$ is controlled at the unit scale by $Z$ on $\Omega'$ (here, and everywhere else in the proof, all such conclusions
are uniform in $\jt$ and $\jh$; see Definition \ref{DefnFrobControlUniform}).  It follows immediately from the definitions (see also Proposition 12.7 of \cite{SteinStreetI})
that
\begin{equation*}
\gammah_{s_1,s_{2,1}, s_{2,2}}(x):= \gamma_{2^{-l}  ( s_1,  (\xi_{1,1} s_{2,1}^1,\ldots, \xi_{1,N_1} s_{2,1}^{N_1}), (\xi_{2,1} s_{2,2}^1,\ldots, \xi_{2,N_2} s_{2,2}^{N_2} )  )}(x)
\end{equation*}
is controlled at the unit scale by $Z$ on $\Omega'$.
By our hypotheses, and using the coordinates $s_1,s_{2,1},s_{2,2}$ and the fact that $\jt_{\mu}>0$, we have
\begin{equation}\label{EqnPfLtAptVsigCancel}
\vsig(s_1,s_{2,1},s_{2,2})=\sum_{|\alpha|=M} \partial_{s_{2,2}}^{\alpha} \vsig_{\alpha}(s_1,s_{2,1},s_{2,2}),
\end{equation}
where $\{\vsig_{\alpha}\}\subset C_0^{\infty}(B^N(a))$ is a bounded set (and $\{\vsig_\alpha\}$ denotes the set of all $\vsig_{\alpha}$ as the various
parameters in the problem vary).  Using the above choices, we have
\begin{equation}\label{EqnPfLtAptBasicOp}
\begin{split}
T_{j} f(x) = \psi_1(x) \int &f(\gammah_{s_1,s_{2,1},\zeta s_{2,2}}(x)) \psi_2(\gammah_{s_1,s_{2,1},\zeta s_{2,2}}(x))
\\&\kappa(s_1, s_{2,1}, \zeta s_{2,2}, x) \vsig(s_1,s_{2,1},s_{2,2})\: ds_1\: ds_{2,1} \: ds_{2,2}.
\end{split}
\end{equation}
The goal is to apply Lemma \ref{LemmaPfLtMainLemma} to this operator.

Let $s=(s_1,s_{2,1},s_{2,2})$ and define
\begin{equation*}
\Wh(s,x)=\frac{\partial}{\partial \epsilon}\bigg|_{\epsilon=1} \gammah_{\epsilon s}\circ \gammah_s^{-1}(x).
\end{equation*}
and suppose the parameterization $(\gamma,e,N)$ corresponds to the vector field paramaterization $(W,e,N)$. Note, we have
\begin{equation}\label{EqnPfLtAptWsEqual}
\Wh((s_1,0,0),x)= W(2^{-l}(s_1,0,0),x).
\end{equation}

The $\nu$-parameter dilations $e_1,\ldots, e_N$ assign to each multi-index $\alpha\in \N^{N}$ a degree $\deg(\alpha)\in [0,\infty)^{\nu}$
by Definition \ref{DefnResKerDegree}.  This induces the same for multi-indicies when we consider $(s_1,s_2)^{\alpha}$:  i.e., if
we write $(s_1,s_2)^{\alpha}$, then $\deg(\alpha)$ is defined to be $\deg(\beta)$ where $(s_1,s_2)^{\alpha}$ corresponds to $t^{\beta}$
under the change of coordinates.

We now separate the proof into the two cases, and we begin with Case I.  We wish to show that the hypotheses of Case I of Lemma \ref{LemmaPfLtMainLemma}
hold in this situation, with the above choice of $Z$.
We have already seen that $Z$ controls $\gammah$ at the unit scale on $\Omega'$, and it is clear from \eqref{EqnPfLtAptVsigCancel}
that $\int \vsig \: ds_{2,2}=0$.
Decompose $W(s)$ as a Taylor series in the $s=(s_1,s_2)$ variable: $W(s)\sim \sum_{|\alpha|>0} s^{\alpha} X_{\alpha}$.
Let
\begin{equation*}
\sS_0:=\{ (X_{\alpha},\deg(\alpha)): \deg(\alpha)\in \onecompnu \}.
\end{equation*}
Our hypotheses show that $\sL(\sS_0)$ is equivalent to $\sF$ on $\Omega'$.
Using that $\sS_0=\bigcup_{\mu\in \nuset} \pi_{\mu} \sS_0$, Lemma \ref{LemmaPfLtsLControl} shows
that there exists $L\in \N$ such that $\sL_L(\sigma_L \sS_0)$ is equivalent to $\sF$ on $\Omega'$.  In particular,
$\sL_L(\sigma_L \sS_0)$ satisfies $\sD(\Omega')$.
Let
\begin{equation}\label{EqnPfLtAptsPo}
\sP_1:=\{\mu_0+\nuh\in \nuset :1\leq \mu_0\leq \nut\text{ and } \jt_{\mu_0}>\lambda(\jh)_{\mu_0} \},
\end{equation}
and let $\sP_2=\nuset\setminus \sP_1$.
Because $\sL_L(\sigma_L (\sS_0))$ is equivalent to $\sF$ on $\Omega'$ we have that $\sL_L(\sigma_L(\sS_0))$ is equivalent to $\sL(\sS)$ on $\Omega'$.  It follows (by taking $\delta_{\mu'}=0$ for $\mu'\in \sP_1$
in the definitions) that
\begin{equation*}
\sL_L\q(\sigma_L \bigcup_{\mu'\in \sP_2} \pi_{\mu'}\sS_0\w) \text{ and } \sL\q(\bigcup_{\mu'\in \sP_2} \pi_{\mu'} \sS\w)
\end{equation*}
are equivalent on $\Omega'$.
Set $\sS_1 = \bigcup_{\mu'\in \sP_2} \pi_{\mu'}\sS_0$.  Note that $(X,d)\in \sS_1$ if and only if $X$ appears as a Taylor coefficient
in $W(s_1,0,0)$ corresponding to a multi-index $\alpha$ with $\deg(\alpha)=d$; and therefore by \eqref{EqnPfLtAptWsEqual},
$2^{-l\cdot d}X$ is a Taylor coefficient of $\Wh(s_1,0,0)$.

Using the above, we have
\begin{equation*}
\q\{ 2^{-l\cdot d} X: (X,d)\in \sL_L\q(\sigma_L \sS_1\w)\w\}
\end{equation*}
controls $\{ 2^{-l\cdot d} X : (X,d)\in \sF\text{ and }d_{\mu'}=0, \forall \mu'\in \sP_1\}$ at the unit scale on $\Omega'$.
In particular, since for each $(\Xh,\hd)\in \sL(\sSh)$,
$\{ 2^{-l\cdot d} X : (X,d)\in \sF\text{ and }d_{\mu'}=0, \forall \mu'\in \sP_1\}$
controls $2^{-\jh\cdot \hd} \Xh$ at the unit scale on $\Omega'$, we have that $\q\{ 2^{-l\cdot d} X: (X,d)\in \sL_L\q(\sigma_L \sS_1\w)\w\}$
controls $2^{-\jh\cdot \hd} \Xh$ at the unit scale on $\Omega'$.

If $1\leq \mu'\leq \nut$ is such that $\jt_{\mu'}\leq \lambda(\jh)_{\mu'}$, then $\mu'+\nuh\in \sP_2$.  Thus, for such a $\mu'$,
if $(\Xt,\td)\in \sSt$ has $\td$ nonzero in only the $\mu'$ coordinate, since $\sF$ controls $\sL(\sS)$ on $\Omega'$,
it follows that $\sL_L\q(\sigma_L \sS_1\w)$ controls $(\Xt,(0_{\nuh},\td))$ on $\Omega'$, and therefore,
$\q\{ 2^{-l\cdot d} X: (X,d)\in \sL_L\q(\sigma_L \sS_1\w)\w\}$ controls $2^{-\jt\cdot \td} \Xt=2^{- (\jt\wedge \lambda(\jh) )\cdot \td} \Xt$ at the unit scale on $\Omega'$.

By the hypothesis that for each $(\Xt,\td)\in \sSt$, $(\Xt, h_{\td,\lambda})$ is controlled by $\sL(\sSh)$ on $\Omega'$, it follows that
$\q\{ 2^{-l\cdot d} X: (X,d)\in \sL_L\q(\sigma_L \sS_1\w)\w\}$ controls $2^{-\lambda(\jh)\cdot \td} \Xt$ at the unit scale on $\Omega'$.
Hence, if $\td$ is nonzero in only the $\mu'$ coordinate for some $\mu'$ with $\jt_{\mu'}>\lambda(\jh)_{\mu'}$, then
 $\q\{ 2^{-l\cdot d} X: (X,d)\in \sL_L\q(\sigma_L \sS_1\w)\w\}$ controls $2^{-\lambda(\jt)\cdot \td} \Xt=2^{- (\jt\wedge \lambda(\jh) )\cdot \td} \Xt$ at the unit scale on $\Omega'$.

 Combining the above three paragraphs (and using that for each $(X,d)\in \sS$, $d$ is nonzero in only one component) shows that for all $(X_0,d_0)\in\sS$,  $$\q\{ 2^{-l\cdot d} X: (X,d)\in \sL_L\q(\sigma_L \sS_1\w)\w\}$$
 controls $2^{-l\cdot d_0} X_0$ at the unit scale on $\Omega'$.  By taking commutators of this, we see for every $(X_0,d_0)\in \sL(\sS)$,
 there exists $L'=L'(X_0,d_0)$ such that $$\q\{ 2^{-l\cdot d} X: (X,d)\in \sL_{L'}\q(\sigma_L \sS_1\w)\w\}$$ controls
 $2^{-l\cdot d_0} X_0$ at the unit scale on $\Omega'$.  Since $\sL(\sS)$ controls $\sF$ on $\Omega'$, we see that for every $(X_0,d_0)\in \sF$,
 there exists $L''=L''(X_0,d_0)$ such that $$\q\{ 2^{-l\cdot d} X: (X,d)\in \sL_{L''}\q(\sigma_L \sS_1\w)\w\}$$ controls
 $2^{-l\cdot d_0} X_0$ at the unit scale on $\Omega'$.   Set $L'''=\max\{ L''(X_0,d_0) : (X_0,d_0)\in \sF\}$.  We therefore have
 for all $(X_0,d_0)\in \sF$, $$\q\{ 2^{-l\cdot d} X: (X,d)\in \sL_{L'''}\q(\sigma_L \sS_1\w)\w\}$$ controls
 $2^{-l\cdot d_0} X_0$ at the unit scale on $\Omega'$.  This is the same as saying for all $Z_0\in Z$,
 $$\q\{ 2^{-l\cdot d} X: (X,d)\in \sL_{L'''}\q(\sigma_L \sS_1\w)\w\}$$ controls
 $Z_0$ at the unit scale on $\Omega'$.

 Because for every $(X,d)\in \sS_1$, $X$ appears as a Taylor coefficient of $W(s_1,0,0)$ corresponding to a multi-index $\alpha$ with $\deg(\alpha)=d$
 (and therefore by \eqref{EqnPfLtAptWsEqual}, $2^{-l\cdot d}X$ appears as a Taylor coefficient of $\Wh(s_1,0,0)$),
 this shows that the hypotheses of Case I of Lemma \ref{LemmaPfLtMainLemma} hold when applied to the operator \eqref{EqnPfLtAptBasicOp}.
 It follows that there exists $\epsilon,\epsilon',\epsilon''>0$ with
 \begin{equation*}
 \LpOpN{2}{T_j}\lesssim \zeta^{\epsilon} \lesssim 2^{-\epsilon' (\jt_{\mu}-\lambda(\jh)_{\mu})} = 2^{-\epsilon' |\jt\vee \lambda(\jh)-\lambda(\jh)|_{\infty}}
 \lesssim 2^{-\epsilon'' |\jt\vee \lambda(\jh)-\lambda(\jh)|},
 \end{equation*}
 as desired, completing the proof in Case I.

 We now turn to Case II.  We wish to shows that the hypotheses of Case II of Lemma \ref{LemmaPfLtMainLemma} hold when applied to the operator
 \eqref{EqnPfLtAptBasicOp}, with the above choice of $Z$.  We have already seen that $Z$ controls $\gammah$ at the unit scale on $\Omega'$,
 and it is clear from \eqref{EqnPfLtAptVsigCancel} that
 \begin{equation*}
 \vsig(s_1,s_{2,1},s_{2,2}) = \sum_{|\alpha|=M} \partial_{s_{2,2}}^{\alpha} \vsig_{\alpha}(s_1,s_{2,1},s_{2,2}),
 \end{equation*}
 where $\vsig_\alpha\in C_0^\infty(B^{N}(a))$ ranges over a bounded set as the various parameters of the problem vary.
 As before, decompose $W(s)$ into a Taylor series $W(s)\sim \sum_{|\alpha|>0} s^{\alpha} X_\alpha$, and now let
 \begin{equation*}
 \sS_0:=\{(X_\alpha,\deg(\alpha)) : \deg(\alpha)\in \onecompnu\text{ and }|\alpha|=1\}.
 \end{equation*}
 Our hypotheses show that $\sS_0$ and $\sF$ are equivalent on $\Omega'$.  Define $\sP_1$ as in \eqref{EqnPfLtAptsPo}
and let $\sP_2=\nuset\setminus \sP_1$ as before.  Set $\sS_1=\bigcup_{\mu'\in \sP_2} \pi_{\mu'} \sS_0$.
Note, $(X,d)\in \sS_1$ if and only if $X$ appears as a Taylor coefficient of $W(s_1,0,0)$ corresponding to a multi-index
$\alpha$ with $|\alpha|=1$ and $\deg(\alpha)=d\in \onecompnu$.  Using \eqref{EqnPfLtAptWsEqual}, we see that if $(X,d)\in \sS_1$,
then $2^{-l\cdot d} X$ appears as a Taylor coefficient of $\Wh(s_1,0,0)$ corresponding to a multi-index
$\alpha$ with $|\alpha|=1$.

Recall, by the hypothesis that $(\gamma,e,N)$ is linearly finitely generated by $\sF$, we have $\sF\subset \snuvectone$.
Let $(X_0,d_0)\in \sF$.  We know that $d_0\in \onecompnu$, i.e. $d_0$ is nonzero in precisely one component.
Suppose $\mu'\in \sP_2$, and $d_{0,\mu'}\ne 0$.  Using the fact that $\sS_0$ controls $\sF$ on $\Omega'$, we have
$\sS_1$ controls $(X_0,d_0)$ on $\Omega'$, and therefore $\{2^{-l\cdot d} X: (X,d)\in \sS_1\}$ controls
$2^{-l\cdot d_0} X_0$ at the unit scale on $\Omega'$.

Because $\sS_0$ is equivalent to $\sF$ on $\Omega'$ and $\sF$ controls $\sS$ on $\Omega'$, we have for each $(\Xh,\hd)\in \sSh$,
$2^{-\jh \cdot \hd} \Xh$ is controlled by $\{2^{-l\cdot d} X: (X,d)\in \sS_0\}$ at the unit scale on $\Omega'$.
Because $\{1,\ldots, \nuh\}\subseteq \sP_2$, it follows that $2^{-\jh \cdot \hd} \Xh$ is controlled by $\{2^{-l\cdot d} X: (X,d)\in \sS_1\}$ at the unit scale on $\Omega'$.
Since $\sS_0$ is equivalent to $\sF$ on $\Omega'$, $\sS_0$ satisfies $\sD(\Omega')$.  By taking $\delta$ so that $\delta_{\mu'}=0$ for all $\mu'\in \sP_1$
in the definition of $\sD(\Omega')$, it follows that $\sS_1$ satisfies $\sD(\Omega')$.  Hence, for each $(\Xh,\hd)\in \sL(\sSh)$, we have
 $2^{-\jh \cdot \hd} \Xh$ is controlled by $\{2^{-l\cdot d} X: (X,d)\in \sS_1\}$ at the unit scale on $\Omega'$.
Let $(\Xt,\td)\in \sSt$, and suppose $\td_{\mu'}\ne 0$ with $\mu'+\nuh\in \sP_1$ (because $\sSt\subset \snutvectone$, $\td$ is nonzero in precisely one
component).  By hypothesis, $\sL(\sSh)$ controls $(\Xt,h_{\td,\lambda})$ on $\Omega'$; therefore, $2^{-\lambda(\jh)\cdot \td} \Xt=2^{-(\jt \wedge \lambda(\jh))\cdot \td} \Xt$ is controlled by $\{2^{-l\cdot d} X: (X,d)\in \sS_1\}$ at the unit scale on $\Omega'$.

Let $(X_0,d_0)\in \sF$ with $d_0$ nonzero in only the $\mu'$ component, where $\mu'\in \sP_1$.  Because $\sL(\sS)$ controls $\sF$ on $\Omega'$,
$(X_0,d_0)$ is controlled by $\sL(\pi_{\mu'} \sS)$.  For each $(X_1,d_1)\in \pi_{\mu'}\sS$, $2^{-l\cdot d_1} X_1=2^{-\lambda(\jh)\cdot \td} \Xt$ for some
$\Xt\in \sSt$.  Thus, applying the conclusion of the previous paragraph and using that $\sS_1$ satisfies $\sD(\Omega')$, we
have $2^{-l\cdot d_0} X_0$ is controlled by $\{2^{-l\cdot d} X: (X,d)\in \sS_1\}$ at the unit scale on $\Omega'$.

Combining the above, we see that for any $(X_0,d_0)\in \sF$, $2^{-l\cdot d_0} X_0$ is controlled by $\{2^{-l\cdot d} X: (X,d)\in \sS_1\}$ at the unit scale on $\Omega'$.
This is the same as saying for all $Z_0\in Z$, $Z$ is controlled by $\{2^{-l\cdot d} X: (X,d)\in \sS_1\}$ at the unit scale on $\Omega'$.  Since for each
$(X,d)\in \sS_1$, $2^{-l\cdot d} X$ appears as a Taylor coefficient of $\Wh(s_1,0,0)$ corresponding to a multi-index $\alpha$ with $|\alpha|=1$,
this shows that the hypotheses of Case II of Lemma \ref{LemmaPfLtMainLemma} hold when applied to the operator
 \eqref{EqnPfLtAptBasicOp}.  Thus, we have for some $c, c'>0$ (independent of any relevant parameters),
 \begin{equation*}
 \LpOpN{2}{T_j}\lesssim \zeta^{M/2}\lesssim 2^{-cM(\jt_{\mu}-\lambda(\jh)_{\mu})} \lesssim 2^{-c'M |\jt\vee \lambda(\jh)-\lambda(\jh)|}.
 \end{equation*}
The result follows.
\end{proof}

\begin{proof}[Proof of Proposition \ref{PropPfLtApTMainProp}]
Fix $M\in \N$.  In Case I, we take $M=1$.  In Case II, we take $M=M(L)$ large to be chosen later.  We apply Proposition \ref{PropPfKerDecompVsig} to write
\begin{equation*}
\eta(t)\dil{\vsig}{2^j}(t)=\eta(t)\sum_{\substack{k\leq j \\ k\in \N^{\nu}}} \dil{\vsig_k}{2^k}(t).
\end{equation*}
Here, using the notation of Proposition \ref{PropPfKerDecompVsig}, we have
\begin{equation*}
\vsig_k=\sum_{\substack{\alphaa\in \N^{\Na} \\ |\alpha_{\mu}|=M\text{ when }k_\mu\ne 0 \\ |\alpha_\mu|=0\text{ when }k_\mu=0 }} \partial_t^{\alphaa} \gamma_{k,\alphaa},
\end{equation*}
where for every $M'\in \N$,
\begin{equation}\label{EqnPfLtAptBddSet}
\begin{split}
&\big\{ 2^{M'|j-k|} \gamma_{k,\alphaa} : j\in [0,\infty)^{\nu}, \vsig\in \schS_{\{\mu:j_{\mu}>0\}}\cap \sB_1, \eta\in \sB_2, k\leq j, k\in \N^{\nu},
\\&\quad\quad\quad\quad\quad\quad\quad|\alpha_{\mu}|=M\text{ when }k_\mu\ne 0, |\alpha_{\mu}|=0\text{ when }k_\mu=0
\big\}
\\&\subset C_0^{\infty}(B^N(a))
\end{split}
\end{equation}
is a bounded set.

Define $T_{j,k}$ for $k\in \N^{\nu}$ with $k\leq j$ by
\begin{equation*}
T_{j,k} f(x) = \psi_1(x) \int f(\gamma_t(x))\psi_2(\gamma_t(x))\kappa(t,x) \eta(t) \dil{\vsig_k}{2^k}(t)\: dt.
\end{equation*}
We have,
\begin{equation*}
T_j = \sum_{\substack{k\in \N^{\nu} \\ k\leq j}} T_{j,k},
\end{equation*}
and it follows that
\begin{equation}\label{EqnPfLtAptTriIneq}
\LpOpN{2}{T_j}\leq \sum_{\substack{k\in \N^{\nu} \\ k\leq j}} \LpOpN{2}{T_{j,k}}.
\end{equation}

In Case I, we apply Lemma \ref{LemmaPfLtAptMainLemma} and use \eqref{EqnPfLtAptBddSet} to see that (if $a>0$ is chosen small enough,
depending only on the parameterization), there exists $0<\epsilon\leq 1$ such that for any $M'$,
\begin{equation*}
\LpOpN{2}{T_{j,k}}\leq 2^{-M'|j-k|-\epsilon|\kt\vee \lambda(\kh) - \lambda(\kh)|}.
\end{equation*}
Taking $M'=M'(\lambda)$ sufficiently large, we have,
\begin{equation*}
\LpOpN{2}{T_{j,k}}\leq 2^{-|j-k|-\epsilon|\jt\vee \lambda(\jh) - \lambda(\jh)|}.
\end{equation*}
Plugging this into \eqref{EqnPfLtAptTriIneq}, we have
\begin{equation*}
\LpOpN{2}{T_j}\lesssim \sum_{\substack{k\in \N^{\nu} \\ k\leq j}}2^{-|j-k|-\epsilon|\jt\vee \lambda(\jh) - \lambda(\jh)|} \lesssim 2^{-\epsilon|\jt\vee \lambda(\jh) - \lambda(\jh)|},
\end{equation*}
as desired.

In Case II, we take $M=M(L)$ large, apply Lemma \ref{LemmaPfLtAptMainLemma}, and use \eqref{EqnPfLtAptBddSet} to see that (if $a>0$ is chosen small enough,
depending only on the parameterization), we have for any $M'\in \N$,
\begin{equation*}
\LpOpN{2}{T_{j,k}} \lesssim 2^{-M' |j-k| - L|\kt\vee \lambda(\kh) - \lambda(\kh)|}.
\end{equation*}
Taking $M'=M'(\lambda,L)$ sufficiently large, we have
\begin{equation*}
\LpOpN{2}{T_{j,k}}\lesssim 2^{-|j-k| - L|\jt\vee \lambda(\jh) - \lambda(\jh)|}.
\end{equation*}
Plugging this into \eqref{EqnPfLtAptTriIneq}, we have
\begin{equation*}
\LpOpN{2}{T_j}\lesssim \sum_{\substack{k\in \N^{\nu} \\ k\leq j}} 2^{-|j-k| - L|\jt\vee \lambda(\jh) - \lambda(\jh)|} \lesssim 2^{-L|\jt\vee \lambda(\jh) - \lambda(\jh)|},
\end{equation*}
as desired.
\end{proof}

%% file: pfmax.tex
A key tool in the proofs that follow is the maximal function corresponding to a finitely generated parameterization.
This was studied in \cite{SteinStreetII} and we present those results here.
Fix open sets $\Omega_0\Subset\Omega'\Subset \Omega'''\Subset \Omega\subseteq \R^n$.
Let $(\gamma,e,N,\Omega,\Omega''')$ be a parameterization, with $\nu$-parameter dilations $0\ne e_1,\ldots, e_N\in [0,\infty)^{\nu}$.
For $a>0$ a small number, and $\psi_1,\psi_2\in C_0^{\infty}(\Omega_0)$ with $\psi_1,\psi_2\geq 0$, define
\begin{equation*}
\sM_{(\gamma,e,N),\psi_1,\psi_2} f(x) = \sup_{\delta\in [0,1]^{\nu}}  \psi_1(x) \int_{|t|<a} \q|f(\gamma_{\delta t}(x))\w| \psi_2(\gamma_{\delta t}(x)) \: dt.
\end{equation*}

\begin{thm}\label{ThmMaxThm}
If $(\gamma,e,N)$ is finitely generated on $\Omega'$, there exists $a>0$ (depending on the parameterization $(\gamma,e,N)$) such that for $1<p\leq \infty$,
\begin{equation*}
\LpN{p}{\sM_{(\gamma,e,N),\psi_1,\psi_2} f}\lesssim \LpN{p}{f},
\end{equation*}
where the implicit constant depends on $p$, the parameterization, and the choice of $\psi_1$ and $\psi_2$.
\end{thm}
\begin{proof}
This follows from Theorem 5.4 of \cite{SteinStreetII}.
\end{proof}

The maximal function often arises via the following proposition.
\begin{prop}\label{PropPfMaxMaxFuncBound}
Let $a>0$ be as in the definition of the maximal function.  
Let $\sB_1\subset \schS(\R^N)$, $\sB_2\subset C_0^\infty(B^N(a))$, $\sB_3\subset C_0^{\infty}(B^n(\Omega_0))$, and $\sB_4\subset C^\infty(B^N(a)\times \Omega'')$ be bounded sets.
For each $\vsig\in \sB_1$, $\eta\in \sB_2$, $\psi_1,\psi_2\in \sB_3$, $\kappa\in \sB_4$, and $j\in [0,\infty)^{\nu}$ define
\begin{equation*}
T_j[\vsig,\eta,\psi_1,\psi_2,\kappa] f(x) :=\psi_1(x) \int f(\gamma_t(x)) \psi_2(\gamma_t(x)) \kappa(t,x) \eta(t) \dil{\vsig}{2^j}(t)\: dt.
\end{equation*}
Then, there exists $C=C(\sB_1,\sB_2,\sB_3,\sB_4,e,N)$ and $\psi_1',\psi_2'\in C_0^\infty(\Omega_0)$ ($\psi_1',\psi_2'$ depending only on $\sB_3$) such that
\begin{equation*}
\q|T_j[\vsig,\eta,\psi_1,\psi_2,\kappa] f(x)\w|\leq C \sM_{(\gamma,e,N),\psi_1',\psi_2'} |f| (x).
\end{equation*}
\end{prop}
\begin{proof}
Because $\sB_3\subset C_0^{\infty}(B^n(\Omega_0))$ is bounded, if we define
\begin{equation*}
K:=\overline{\bigcup_{\psi\in \sB_3} \supp{\psi}},
\end{equation*}
then $K\Subset \Omega_0$.  Let $\psi_1',\psi_2'\in C_0^\infty(\Omega_0)$ be such that $\psi_1',\psi_2'\geq 0$,
and $\psi_1',\psi_2'\equiv 1$ on $K$.
For $\vsig\in \sB_1$, $\eta\in \sB_2$, and $j\in [0,\infty)^{\nu}$, we apply 
Lemma \ref{LemmaPfKerDecompsS} to write
\begin{equation*}
\eta(t) \dil{\vsig}{2^j}(t) = \sum_{\substack{k\leq j \\ k\in \N^{\nu}}} \eta(t) \dil{\vsig_k}{2^k}(t),
\end{equation*}
where
\begin{equation}\label{EqnPfMaxDecompBound}
\q\{ 2^{|j-k|}\vsig_k : j\in [0,\infty)^{\nu}, k\in \N^{\nu}, \vsig\in \sB_1,\eta\in \sB_2 \w\}\subset C_0^\infty(B^N(a))
\end{equation}
is bounded.

For $k\in \N^{\nu}$, $k\leq j$, define $T_{j,k}=T_{j,k}[\vsig,\eta,\psi_1,\psi_2,\kappa]$ by
\begin{equation*}
T_{j,k} f(x) = \psi_1(x) \int f(\gamma_t(x)) \psi_2(\gamma_t(x)) \kappa(t,x) \dil{\vsig_k}{2^k}(t)\: dt
\end{equation*}
so that
\begin{equation}\label{EqnPfMaxDecomp}
T_j[\vsig,\eta,\psi_1,\psi_2,\kappa] =\sum_{\substack{k\leq j \\ k\in \N^{\nu}}} T_{j,k}.
\end{equation}
We have
\begin{equation*}
\begin{split}
|T_{j,k} f(x)| &\leq |\psi_1(x)| \int \q| f(\gamma_{2^{-k} t}(x)) \psi_2(\gamma_{2^{-k}t}(x)) \kappa(2^{-k} t,x) \vsig_k(t)\w|\: dt
\\&\lesssim \psi_1'(x) \int_{|t|<a} \q|f(\gamma_{2^{-k}t}(x)) \w| \psi_2'(\gamma_{2^{-k}t}(x)) \q|\vsig_k(t)\w|\: dt 
\\&\lesssim 2^{-|j-k|} \sM_{(\gamma,e,N),\psi_1',\psi_2'} |f|(x),
\end{split}
\end{equation*}
where in the last line we used \eqref{EqnPfMaxDecompBound}.
Plugging this into \eqref{EqnPfMaxDecomp}, we have
\begin{equation*}
\q|T_j[\vsig,\eta,\psi_1,\psi_2,\kappa] f(x)  \w|\lesssim \sum_{\substack{k\leq j \\ k\in \N^{\nu}}}2^{-|j-k|} \sM_{(\gamma,e,N),\psi_1',\psi_2'} |f|(x) \lesssim \sM_{(\gamma,e,N),\psi_1',\psi_2'} |f|(x),
\end{equation*}
as desired, completing the proof.
\end{proof}

%% file: pfslp.tex
Fix open sets $\Omega_0\Subset \Omega' \Subset\Omega'''\Subset \Omega$.  Suppose
$\sF\subset \sonevect$ is a finite set satisfying $\sD(\Omega')$.
Enumerate $\sF$:
\begin{equation*}
\sF=\{ (X_1,d_1),\ldots, (X_q,d_q)\}.
\end{equation*}
On $\R^q$ define single parameter dilations by, for $\delta\in [0,\infty)$:
\begin{equation*}
\delta(t_1,\ldots,t_q)= (\delta^{d_1} t_1,\ldots, \delta^{d_1} t_q).
\end{equation*}
We denote these single parameter dilations by $d$.
Define
\begin{equation*}
\gamma_t(x):=e^{t_1X_1+\cdots+ t_q X_q} x,
\end{equation*}
so that $(\gamma,d,q,\Omega,\Omega''')$ is a parameterization.

\begin{prop}\label{PropPfsLp}
There exists $a>0$, depending on $(\gamma,d,q)$, such that the following holds.
Let $\fD=\q(1,(\gamma,d,q,\Omega,\Omega'''), a, \eta,\{\vsig_j\}_{j\in \N}, \psi\w)$ be Sobolev data on $\Omega'$,
and let $\{\epsilon_j\}_{j\in \N}$ be a sequence of i.i.d. random variables of mean $0$ taking values $\pm 1$.
For $j\in \N$ let $D_j=D_j(\fD)$ be given by \eqref{EqnResSobDefnDjNew}.
Then, for $1<p<\infty$,
\begin{equation*}
\LpN{p}{f} \approx\LpN{p}{\q( \sum_{j\in \N}\q|D_j f\w|^2 \w)^{\frac{1}{2} }} \approx \q(\bE \LpN{p}{ \sum_{j\in \N}  \epsilon_j D_j f }^{p} \w)^{\frac{1}{p}}, \quad f\in C_0^\infty(\Omega_0).
\end{equation*}
Here, the implicit constants depend on $p\in (1,\infty)$ and $\fD$, and $\bE$ denotes the expectation with respect to the (suppressed) variable with respect
to which $\epsilon_j$ is a random variable.
\end{prop}
\begin{proof}
This is exactly the statement of Corollary 2.15.54 of \cite{StreetMultiParamSingInt}, and we refer the reader there for the full details.
We make a few comments on the proof here.
The estimate
\begin{equation*}
\LpN{p}{\q( \sum_{j\in \N}\q|D_j f\w|^2 \w)^{\frac{1}{2} }} \approx \q(\bE \LpN{p}{ \sum_{j\in \N}  \epsilon_j D_j f }^{p} \w)^{\frac{1}{p}}
\end{equation*}
is an immediate consequence of the Khintchine inequality.

In the case that $X_1,\ldots, X_q$ span the tangent space at every point, to prove
\begin{equation*}
\q(\bE \LpN{p}{ \sum_{j\in \N}  \epsilon_j D_j f }^{p} \w)^{\frac{1}{p}} \lesssim \LpN{p}{f},
\end{equation*}
one shows the (a priori stronger) estimate
\begin{equation*}
\sup_{\epsilon_j\in \{\pm 1\}}  \LpN{p}{ \sum_{j\in \N}  \epsilon_j D_j f } \lesssim \LpN{p}{f}
\end{equation*}
by showing that $\sum_{j\in \N}  \epsilon_j D_j$ is a Calder\'on-Zygmund operator, uniformly in the choice of sequence $\epsilon_j$,
so that the estimate follows from classical theorems.
When $X_1,\ldots, X_q$ do not span the tangent space one wishes to use the Frobenius theorem to foliate the ambient space
into leaves; and apply the above idea to each leaf.  This can be achieved using Theorem \ref{ThmQuantFrob}.

Finally, the estimate
\begin{equation*}
\LpN{p}{f}\lesssim \LpN{p}{\q( \sum_{j\in \N}\q|D_j f\w|^2 \w)^{\frac{1}{2} }}
\end{equation*}
follows from the above and a Calder\'on reproducing type formula (which can be proved using the above estimates and an almost orthogonality argument).  We refer the reader to \cite{StreetMultiParamSingInt} for all the details.
\end{proof}

%% file: pfsob.tex
In this section, we prove the results from Section \ref{SectionResSobNew}.  For this, we use vector valued $L^p$ spaces.
For $1\leq p\leq \infty$, $1\leq q\leq\infty$, we define the space $\Lplqnu{p}{q}$ to be the Banach space consisting of sequences
of measurable functions $\{f_j(x)\}_{j\in \N^{\nu}}$, $f_j:\R^{n}\rightarrow \C$, with the norm:
\begin{equation*}
\LplqnuN{p}{q}{\{ f_j\}_{j\in \N^{\nu}}}:=
\begin{cases}
\LpN{p}{\q(\sum_{j\in \N^{\nu}} |f_j(x)|^q\w)^{\frac{1}{q}}}, &\text{ if }q\in[1,\infty),\\
\LpN{p}{\sup_{j\in \N^{\nu}} |f_j(x)|}, &\text{ if } q=\infty.
\end{cases}
\end{equation*}

In the proofs that follow, we use the following convention.  If $T_j$, $j\in \N^{\nu}$ is a sequence of operators,
then, for $j\in \Z^{\nu}\setminus \N^{\nu}$ we define $T_j=0$.  
Let $\Omega_0\Subset\Omega'\Subset\Omega'''\Subset \Omega\subseteq \R^n$ be open sets.
We need the next result concerning vector valued operators.

\begin{prop}\label{PropPfSobVVOps}
Suppose $(\gamma^1,e^1,N_1,\Omega,\Omega'''),\ldots,(\gamma^K,e^K,N_K,\Omega,\Omega''')$ be $K$ parameterizations,
each with $\nu$-parameter dilations $0\ne e_1^l,\ldots, e_{N_l}^l\in [0,\infty)^{\nu}$.  We separate our assumptions into two cases:
\begin{enumerate}[\bf{Case} I]
\item There exists a finite set $\sF\subset \snuvect$ such that $$(\gamma^1,e^1,N_1),\ldots,(\gamma^{K},e^K,N_K)$$ are all finitely
generated by $\sF$ on $\Omega'$.

\item There exists a finite set $\sF\subset \snuvectone$ such that $$(\gamma^1,e^1,N_1),\ldots,(\gamma^{K},e^K,N_K)$$ are all linearly finitely
generated by $\sF$ on $\Omega'$.
\end{enumerate}
Then, there exists $a>0$ (depending on the parameterizations) such that the following holds. 
For each $l\in\{1,\ldots, K\}$, let $\sB_1^l\subset \schS(\R^{N_l})$, $\sB_2^l\subset C_0^\infty(B^{N_l}(a))$, $\sB_3^l\subset C_0^\infty(\Omega_0)$,
and $\sB_4^l\subset C^\infty(B^{N_l}(a)\times \Omega'')$ be bounded sets.  For each $j_1,\ldots, j_K\in [0,\infty)^{\nu}$,
let $\vsig_{l,j_l}\in \schS_{\{\mu : j_{l,\mu}>0\}}\cap \sB_1^l$, $\eta_{l,j_l}\in \sB_2^l$, $\psi_{1,l,j_l},\psi_{2,l,j_l}\in \sB_3^l$, and $\kappa_{l,j_l}\in \sB_4^l$.
Define an operator $T_{j_l}^l$ by
\begin{equation*}
T_{j_l}^l f(x):=\psi_{1,l,j_l}(x) \int f(\gamma^l_{t^l}(x)) \psi_{2,l,j_l}(\gamma^l_{t^l}(x)) \kappa_{l,j_l}(t^l,x) \eta_{l,j_l}(t^l)\dil{\vsig_{l,j_l}}{2^{j_l}}(t^l)\: dt^l.
\end{equation*}
For $k_2,\ldots, k_K\in \Z^{\nu}$ define a vector valued operator by
\begin{equation*}
\sT_{k_2,\ldots, k_K} \{ f_j \}_{j\in \N^{\nu}} := \q\{ T_{j}^1 T_{j+k_2}^2 T_{j+k_3}^3\cdots T_{j+k_K}^K f_j \w\}_{j\in \N^{\nu}},
\end{equation*}
where for $j_l\in \Z^{\nu}\setminus \N^{\nu}$, $T^l_{j_l}=0$, by convention.  Then, for $1<p<\infty$, we have the following.
\begin{itemize}
\item In Case I, there exists $\epsilon_p>0$ (depending on $p$ and the parameterizations) such that
\begin{equation}\label{EqnPropPfSobVVCaseI}
\LplqnuOpN{p}{2}{\sT_{k_2,\ldots, k_K}}\leq C_p 2^{-\epsilon_p ( |k_2|+\cdots+|k_K|)},
\end{equation}
where $C_p$ depends on $p$, the above parameterizations, and the sets $\sB_1^l$, $\sB_2^l$, $\sB_3^l$, and $\sB_4^l$.

\item In Case II, for every $L$, there exists $C_{p,L}$ (depending on $p$, $L$, the above parameterizations, and the sets $\sB_1^l$, $\sB_2^l$, $\sB_3^l$, and $\sB_4^l$)
such that
\begin{equation}\label{EqnPropPfSobVVCaseII}
\LplqnuOpN{p}{2}{\sT_{k_2,\ldots, k_K}}\leq C_{p,L} 2^{-L ( |k_2|+\cdots+|k_K|)}.
\end{equation}

\end{itemize}
The above results hold even when $K=1$.  When $K=1$ one takes, by convention, $|k_2|+\cdots+|k_K|=0$.
\end{prop}
\begin{proof}
Applying Proposition \ref{PropPfLtTechApO}, and using that $|k_2|+\cdots+|k_K|\lesssim \diam\{j,j+k_2,\ldots, j+k_K\}$, we have:
\begin{itemize}
\item In Case I, there exists $\epsilon_2>0$ such that
\begin{equation}\label{EqnPfSobBasicBoundI}
\LpOpN{2}{T_j^1 T_{j+k_2}^2T_{j+k_3}^3\cdots T_{j+k_K}^K}\lesssim 2^{-\epsilon_2 (|k_2|+\cdots+|k_K|)}.
\end{equation}

\item In Case II, for every $L$, there exists a constant $C_L$ such that
\begin{equation}\label{EqnPfSobBasicBoundII}
\LpOpN{2}{T_j^1 T_{j+k_2}^2T_{j+k_3}^3\cdots T_{j+k_K}^K}\leq C_L 2^{-L (|k_2|+\cdots+|k_K|)}.
\end{equation}
\end{itemize}
In Case II, fix $L\in \N$.  Applying \eqref{EqnPfSobBasicBoundI} (in Case I) and \eqref{EqnPfSobBasicBoundII} (in Case II) and interchanging the norms,
we have
\begin{equation}\label{EqnPfSobVVTTB}
\LplqnuOpN{2}{2}{\sT_{k_2,\ldots, k_K}}\lesssim
\begin{cases}
2^{-\epsilon_2 (|k_2|+\cdots+|k_K|)} & \text{ in Case I},\\
2^{-L (|k_2|+\cdots+|k_K|)} &\text{ in Case II}.
\end{cases}
\end{equation}
We also have the trivial inequality $\LpOpN{1}{T_j^1 T_{j+k_2}^2T_{j+k_3}^3\cdots T_{j+k_K}^K}\lesssim 1$, from which it follows that
\begin{equation}\label{EqnPFSobVVOOB}
\LplqnuOpN{1}{1}{\sT_{k_2,\ldots, k_K}} \lesssim 1. 
\end{equation}
Interpolating \eqref{EqnPfSobVVTTB} and \eqref{EqnPFSobVVOOB} we see for $1<p\leq 2$,
\begin{equation}\label{EqnPfSobVVPPB}
\LplqnuOpN{p}{p}{\sT_{k_2,\ldots, k_K}}\lesssim
\begin{cases}
2^{-(2-\frac{2}{p})\epsilon_2 (|k_2|+\cdots+|k_K|)} & \text{ in Case I},\\
2^{- (2-\frac{2}{p}) L (|k_2|+\cdots+|k_K|)} &\text{ in Case II}.
\end{cases}
\end{equation}

We claim, for $1<p<\infty$,
\begin{equation}\label{EqnPFSobVVPIB}
\LplqnuOpN{p}{\infty}{\sT_{k_2,\ldots, k_K}} \lesssim 1,
\end{equation}
with implicit constant depending on $p$.
Applying Proposition \ref{PropPfMaxMaxFuncBound} to find $\psi_{1,l}',\psi_{2,l}'\in C_0^\infty(\Omega_0)$, $1\leq l\leq K$ (depending only on $\sB_3^l$)
such that
\begin{equation*}
\q|T_{j_l}^l f(x) \w|\lesssim \sM_{(\gamma^l, e^l,N_l), \psi_{1,l}',\psi_{2,l}'} |f| (x).
\end{equation*}
Hence, we have for $1<p<\infty$,
\begin{equation*}
\begin{split}
&\LplqnuN{p}{\infty}{ \sT_{k_2,\ldots, k_K} \{ f_j \}_{j\in \N^{\nu}}} 
= \LpN{p}{\sup_{j} \q| T_{j}^1 T_{j+k_2}^2\cdots T_{j+k_K}^K f_j\w| }
\\&\lesssim \LpN{p}{ \sM_{(\gamma^1, e^1,N_1), \psi_{1,1}',\psi_{2,1}'} \cdots \sM_{(\gamma^K, e^K,N_K), \psi_{1,K}',\psi_{2,K}'} \sup_j |f_j|}
\\&\lesssim \LpN{p}{\sup_j |f_j|} = \LplqnuN{p}{\infty}{\{f_j\}_{j\in \N^{\nu}}},
\end{split}
\end{equation*}
where in the last inequality we have applied Theorem \ref{ThmMaxThm}.  \eqref{EqnPFSobVVPIB} follows.

Interpolating \eqref{EqnPfSobVVPPB} and \eqref{EqnPFSobVVPIB} shows for $1<p\leq 2$,
\begin{equation*}
\LplqnuOpN{p}{2}{\sT_{k_2,\ldots, k_K}}\lesssim
\begin{cases}
2^{-(p-1)\epsilon_2 (|k_2|+\cdots+|k_K|)} & \text{ in Case I},\\
2^{- (p-1) L (|k_2|+\cdots+|k_K|)} &\text{ in Case II}.
\end{cases}
\end{equation*}
Because $L\in \N$ was arbitrary, this completes the proof for $1<p\leq 2$.

Fix $2\leq p<\infty$. 
 For this choice of $p$, we wish to prove \eqref{EqnPropPfSobVVCaseI} in Case I, and \eqref{EqnPropPfSobVVCaseII} in Case II.
Let $\frac{1}{p}+\frac{1}{q}=1$.
The dual of $\Lplqnu{p}{2}$ is $\Lplqnu{q}{2}$.
Let $\sT_{k_2,\ldots, k_K}^{*}$ denote the adjoint of $\sT_{k_2,\ldots, k_K}$, and define
\begin{equation*}
\sR_{k_2,\ldots, k_K}\{f_j \}_{j\in \N^{\nu}}:= \q\{ \q(T_{j}^K\w)^{*} \q(T_{j+k_2}^{K-1}\w)^{*} \cdots \q(T_{j+k_K}^1\w)^{*} f_j \w\}_{j\in \N^{\nu}}.
\end{equation*}
Proposition \ref{PropPfAdjEachScale} combined with Proposition \ref{PropPfAdjGamma} shows that $\sR_{k_2,\ldots, k_K}$ is of the same
form as $\sT_{k_2,\ldots, k_K}$ and by applying the result for $q$ (which we have already proved since $1<q\leq 2$), we have in Case I there exists $\epsilon_p>0$, and in Case II for all $L\in \N$,
\begin{equation}\label{EqnPfAdjVVBoundsR}
\LplqOpN{q}{2}{\sR_{k_2,\ldots, k_K}}
\lesssim
\begin{cases}
2^{-\epsilon_p (|k_2|+\cdots+|k_K|)} & \text{ in Case I},\\
2^{- L (|k_2|+\cdots+|k_K|)} &\text{ in Case II}.
\end{cases}
\end{equation}
Consider, (using the convention that $f_j=0$ for $j\in \Z^{\nu}\setminus \N^{\nu}$),
\begin{equation*}
\begin{split}
&\LplqnuN{q}{2}{\sT_{k_2,\ldots,k_K}^{*} \{f_j\}_{j\in \N^{\nu}}} 
\\&= \LplqnuN{q}{2}{\q\{ \q(T_{j+k_K}^K\w)^{*} \q(T_{j+k_{K-1}}^{K-1}\w)^{*}\cdots \q( T_{j+k_2}^2 \w)^{*} \q(T_j^1\w)^{*} f_j \w\}_{j\in \N^{\nu}}}
\\& =\LplqnuN{q}{2}{ \q\{ \q(T_{j}^K\w)^{*} \q(T_{j+k_{K-1}-k_{K}}^{K-1}\w)^{*}\cdots \q( T_{j+k_2-k_K}^2 \w)^{*} \q(T_{j-k_K}^1\w)^{*} f_{j-k_K} \w\}_{j\in \N^{\nu}} }
\\& = \LplqnuN{q}{2}{ \sR_{k_{K-1}-k_K,k_{K-2}-k_K,\ldots, k_{2}-k_K,-k_K } \{f_{j-k_K} \}_{j\in \N^{\nu}}}
\\&\lesssim
\begin{cases}
2^{-\epsilon_p (|k_{K-1}-k_{K}|+\cdots+|k_2-k_{K}|+|k_K|)} \LplqnuN{q}{2}{\{f_j\}_{j\in \N^{\nu}}} & \text{ in Case I},\\
2^{- L (|k_{K-1}-k_{K}|+\cdots+|k_2-k_{K}|+|k_K|)} \LplqnuN{q}{2}{\{f_j\}_{j\in \N^{\nu}}}&\text{ in Case II}.
\end{cases}
\\&\lesssim
\begin{cases}
2^{-\epsilon_p/2 (|k_{2}|+\cdots+|k_K|)} \LplqnuN{q}{2}{\{f_j\}_{j\in \N^{\nu}}}& \text{ in Case I},\\
2^{- L/2 (|k_{2}|+\cdots+|k_K|)} \LplqnuN{q}{2}{\{f_j\}_{j\in \N^{\nu}}}&\text{ in Case II}.
\end{cases}
\end{split}
\end{equation*}
Hence,
\begin{equation*}
\LplqnuOpN{q}{2}{\sT_{k_2,\ldots,k_K}^{*}}\lesssim
\begin{cases}
2^{-\epsilon_p/2 (|k_{2}|+\cdots+|k_K|)} & \text{ in Case I},\\
2^{- L/2 (|k_{2}|+\cdots+|k_K|)} &\text{ in Case II}.
\end{cases}
\end{equation*}
Therefore,
\begin{equation*}
\LplqnuOpN{p}{2}{\sT_{k_2,\ldots,k_K}}\lesssim
\begin{cases}
2^{-\epsilon_p/2 (|k_{2}|+\cdots+|k_K|)} & \text{ in Case I},\\
2^{- L/2 (|k_{2}|+\cdots+|k_K|)} &\text{ in Case II},
\end{cases}
\end{equation*}
as desired, completing the proof.
\end{proof}

\begin{lemma}\label{LemmaPfSobTwoToOne}
Let $\fD=\q(\nu,(\gamma,e,N,\Omega,\Omega'''), a, \eta,\{\vsig_j\}_{j\in \N^{\nu}}, \psi\w)$ be Sobolev data on $\Omega'$.
Define $D_j=D_j(\fD)$ by \eqref{EqnResSobDefnDjNew}.  We separate our assumptions into two cases:
\begin{enumerate}[\bf{Case} I]
\item $\fD$ is finitely generated on $\Omega'$.
\item $\fD$ is linearly finitely generated on $\Omega'$.
\end{enumerate}
Then, if $a>0$ is sufficiently small (depending on $(\gamma,e,N)$), we have for $1<p<\infty$,
\begin{itemize}
\item In Case I, there exists $\epsilon=\epsilon(p,(\gamma,e,N))>0$ and $A=A(\fD,p)$ such that for all $\delta\in \R^{\nu}$, $\delta_0\in \R$ with $|\delta|,\delta_0<\epsilon$,
we have for $f\in C_0^\infty(\Omega_0),$
\begin{equation}\label{EqnPfSobTwoToOneBoundI}
\sum_{k\in \Z^{\nu}} \LpN{p}{\q( \sum_{j\in \N^{\nu}} \q| 2^{j\cdot \delta+\delta_0|k|} D_j D_{j+k} f \w|^2 \w)^{\frac{1}{2}} }\leq A \NLpN{p}{\delta}{\fD}{f}.
\end{equation}

\item In Case II, for every $\delta\in \R^{\nu}$, $\delta_0\in \R$, there exists $A=A(\fD,p,\delta,\delta_0)$ such that for $f\in C_0^\infty(\Omega_0),$
\begin{equation}\label{EqnPfSobTwoToOneBoundII}
\sum_{k\in \Z^{\nu}} \LpN{p}{\q( \sum_{j\in \N^{\nu}} \q| 2^{j\cdot \delta+\delta_0|k|} D_j D_{j+k} f \w|^2 \w)^{\frac{1}{2}} }\leq A \NLpN{p}{\delta}{\fD}{f}.
\end{equation}
\end{itemize}
\end{lemma}
\begin{proof}
First we prove the result in Case I, then we indicate the modifications necessary to prove the result in Case II.  
Let $p\in (1,\infty)$.  
We define two families vector valued operators: for $k_1,k_2\in \Z^{\nu}$ define
\begin{equation*}
\sR_{k_1,k_2}^1 \{f_j\}_{j\in \N^{\nu}} := \{ D_j D_{j+k_1} D_{j+k_2} f_j \}_{j\in \N^{\nu}}, \quad \sR_{k_1}^2\{f_j\}_{j\in \N^{\nu}}:=\{ D_j D_{j+k_1} f_j\}_{j\in \N^{\nu}}.
\end{equation*}
Case I of Proposition \ref{PropPfSobVVOps} shows that there exists $\epsilon_p>0$ with
\begin{equation}\label{EqnPfSobApplVVBound}
\begin{split}
&\LplqnuOpN{p}{2}{\sR_{k_1,k_2}^1}\lesssim 2^{-\epsilon_p (|k_1|+|k_2|)}, 
\\& \LplqnuOpN{p}{2}{\sR_{k_1}^2}\lesssim 2^{-\epsilon_p|k_1|}.
\end{split}
\end{equation}
Here we have replaced $\Omega_0$ in Proposition \ref{PropPfSobVVOps} with a larger open set $\Omega_0'\Subset \Omega'$
so that $\supp{\psi}\Subset \Omega_0'$.

We prove \eqref{EqnPfSobTwoToOneBoundI} for $\delta\in \R^{\nu}$, $\delta_0\in \R$ with $|\delta|\leq \epsilon_p/8$ and $\delta_0\leq \epsilon_p/4$;
and the result will follow.
In fact, we prove the result for $\delta_0=\epsilon_p/4$ and $|\delta|\leq \epsilon_p/8$, as then the result follows for all smaller $\delta_0$ as well.
Thus, for the rest of the proof of Case I, take $\delta_0=\epsilon_p/4$ and $|\delta|\leq \epsilon_p/8$.

Fix $M\in \N$ to be chosen later.\footnote{In the following inequalities $A\lesssim B$ means $A\leq CB$ where $C$ is allowed
to depend on $p$ and $\fD$, but not on $M$ or the function $f$ under consideration.}
Consider, for $f\in C_0^\infty(\Omega_0)$, using that $\psi f = f$ and $\sum_{j} D_j = \psi^2$,
\begin{equation*}
\begin{split}
&\sum_{k\in \Z^{\nu}} \LpN{p}{\q( \sum_{j\in \N^{\nu}} \q| 2^{j\cdot \delta+(\epsilon_p/4)|k|} D_j D_{j+k} f \w|^2 \w)^{\frac{1}{2}} } 
\\&= \sum_{k\in \Z^{\nu}} \LpN{p}{\q( \sum_{j\in \N^{\nu}} \q| 2^{j\cdot \delta+(\epsilon_p/4)|k|} D_j D_{j+k} \psi^4 f \w|^2 \w)^{\frac{1}{2}} } 
\\& = \sum_{k_1\in \Z^{\nu}} \LpN{p}{\q( \sum_{j\in \N^{\nu}} \q| \sum_{k_2,k_3\in \Z^\nu} 2^{j\cdot \delta+(\epsilon_p/4)|k_1|} D_j D_{j+k_1} D_{j+k_2} D_{j+k_3} f \w|^2 \w)^{\frac{1}{2}} }
\\&\leq \sum_{k_1,k_2,k_3\in \Z^{\nu}} \LpN{p}{\q( \sum_{j\in \N^{\nu}} \q|  2^{j\cdot \delta+(\epsilon_p/4)|k_1|} D_j D_{j+k_1} D_{j+k_2} D_{j+k_3} f \w|^2 \w)^{\frac{1}{2}} }
\\&= \sum_{\substack{k_1,k_2,k_3\in \Z^{\nu} \\ |k_3|\leq M }}  + 
\sum_{\substack{k_1,k_2,k_3\in \Z^{\nu} \\ |k_3|> M \\ |k_2|\geq |k_3|/2 }}  + 
\sum_{\substack{k_1,k_2,k_3\in \Z^{\nu} \\ |k_3|> M \\ |k_3|> 2|k_2| }}
=:(I)+(II)+(III).
\end{split}
\end{equation*}

Set
\begin{equation*}
\sT_{k_1,k_2,k_3}^1 := 2^{-k_3\cdot \delta+(\epsilon_p/4)|k_1|} \sR_{k_1,k_2}^1, \quad \sT_{k_1,k_2,k_3}^2:= 2^{-k_2\cdot \delta+(\epsilon_p/4)|k_1|-(\epsilon_p/4)|k_3-k_2|} \sR_{k_1}^2.
\end{equation*}

We begin with $(I)$.  Using \eqref{EqnPfSobApplVVBound}, we have
\begin{equation*}
\begin{split}
&\LplqnuOpN{p}{2}{\sT_{k_1,k_2,k_3}^1}= 2^{-k_3\cdot \delta+(\epsilon_p/4)|k_1|} \LplqnuOpN{p}{2}{\sR_{k_1,k_2}^1}
\\&\lesssim 2^{-k_3\cdot\delta+(\epsilon_p/4)|k_1|-\epsilon_p(|k_1|+|k_2|)}\leq 2^{|k_3| |\delta| - (3/4)\epsilon_p (|k_1|+|k_2|)}.
\end{split}
\end{equation*}
Thus,
\begin{equation*}
\begin{split}
(I)&=\sum_{\substack{k_1,k_2,k_3\in \Z^{\nu} \\ |k_3|\leq M }} \LpN{p}{\q( \sum_{j\in \N^{\nu}} \q|  2^{j\cdot \delta+(\epsilon_p/4)|k_1|} D_j D_{j+k_1} D_{j+k_2} D_{j+k_3} f \w|^2 \w)^{\frac{1}{2}} }
\\&=\sum_{\substack{k_1,k_2,k_3\in \Z^{\nu} \\ |k_3|\leq M }}\LplqnuN{p}{2}{ \sT_{k_1,k_2,k_3}^1 \q\{ 2^{(j+k_3)\cdot \delta} D_{j+k_3}f\w\}_{j\in \N^{\nu}}}
\\&\lesssim \sum_{\substack{k_1,k_2,k_3\in \Z^{\nu} \\ |k_3|\leq M }} 2^{|k_3||\delta|- (3/4)\epsilon_p (|k_1|+|k_2|)} \LplqnuN{p}{2}{\{2^{j\cdot \delta} D_j f\}_{j\in \N^{\nu}}}
\\&\lesssim 2^{\nu|\delta|M} \NLpN{p}{\delta}{\fD}{f}. 
\end{split}
\end{equation*}

We now turn to $(II)$.  For $(II)$, we restrict attention to $|k_2|\geq |k_3|/2$.  With this restriction, \eqref{EqnPfSobApplVVBound} shows
\begin{equation*}
\begin{split}
&\LplqnuOpN{p}{2}{\sT_{k_1,k_2,k_3}^1} \lesssim 2^{-k_3\cdot \delta + (\epsilon_p/4)|k_1|-\epsilon_p(|k_1|+|k_2|)}
\\&\leq 2^{|k_3|  |\delta|+ (\epsilon_p/4)|k_1|-\epsilon_p|k_1|-(\epsilon_p/4)|k_2|-(3\epsilon_p/8)|k_3|}
\leq 2^{-(\epsilon_p/4)(|k_1|+|k_2|+|k_3|)},
\end{split}
\end{equation*}
where in the last line we have used $|\delta|\leq \epsilon_p/8$.  Thus,
\begin{equation*}
\begin{split}
(II) &=\sum_{\substack{k_1,k_2,k_3\in \Z^{\nu} \\ |k_3|> M \\ |k_2|\geq |k_3|/2 }}\LpN{p}{\q( \sum_{j\in \N^{\nu}} \q|  2^{j\cdot \delta+(\epsilon_p/4)|k_1|} D_j D_{j+k_1} D_{j+k_2} D_{j+k_3} f \w|^2 \w)^{\frac{1}{2}} }
\\&=\sum_{\substack{k_1,k_2,k_3\in \Z^{\nu} \\ |k_3|> M \\ |k_2|\geq |k_3|/2 }} \LplqnuN{p}{2}{\sT_{k_1,k_2,k_3}^1 \q\{ 2^{(j+k_3)\cdot \delta} D_{j+k_3} f\w\}_{j\in \N^{\nu}}}
\\&\lesssim \sum_{\substack{k_1,k_2,k_3\in \Z^{\nu} \\ |k_3|> M \\ |k_2|\geq |k_3|/2 }}  2^{-(\epsilon_p/4)(|k_1|+|k_2|+|k_3|)} \LplqnuN{p}{2}{\{ 2^{j\cdot \delta} D_j f\}_{j\in \N^{\nu}}}
\\&\lesssim  \NLpN{p}{\delta}{\fD}{f}.
\end{split}
\end{equation*}

We turn to $(III)$.  For $(III)$ we restrict attention to $|k_3|>2|k_2|$.  With this restriction, \eqref{EqnPfSobApplVVBound} shows (using $|\delta|\leq \epsilon_p/8$)
\begin{equation*}
\begin{split}
&\LplqnuOpN{p}{2}{\sT_{k_1,k_2,k_3}^2} \lesssim 2^{-k_2\cdot \delta+(\epsilon_p/4)|k_1|-(\epsilon_p/4)|k_3-k_2|-\epsilon_p|k_1|}
\\&\leq 2^{|k_2||\delta| -(3\epsilon_p/4)|k_1| - (\epsilon_p/8)|k_3|} \leq 2^{(\epsilon_p/16)|k_3| -(3\epsilon_p/4)|k_1|-(\epsilon_p/8)|k_3|}
\\&\leq 2^{-(3\epsilon_p/4)|k_1|-(\epsilon_p/16)|k_3|} \leq 2^{-(\epsilon_p/64) (|k_1|+|k_2|+|k_3|)}.
\end{split}
\end{equation*}
Thus,
\begin{equation*}
\begin{split}
&(III) =\sum_{\substack{k_1,k_2,k_3\in \Z^{\nu} \\ |k_3|> M \\ |k_3|> 2|k_2| }}\LpN{p}{\q( \sum_{j\in \N^{\nu}} \q|  2^{j\cdot \delta+(\epsilon_p/4)|k_1|} D_j D_{j+k_1} D_{j+k_2} D_{j+k_3} f \w|^2 \w)^{\frac{1}{2}} }
\\&=\sum_{\substack{k_1,k_2,k_3\in \Z^{\nu} \\ |k_3|> M \\ |k_3|> 2|k_2| }} \LplqnuN{p}{2}{\sT_{k_1,k_2,k_3}^2 \q\{ 2^{ (j+k_2)\cdot \delta +(\epsilon_p/4)|k_3-k_2| } D_{j+k_2} D_{j+k_3} f   \w\}_{j\in \N^{\nu}} }
\\&\lesssim \sum_{\substack{k_1,k_2,k_3\in \Z^{\nu} \\ |k_3|> M \\ |k_3|> 2|k_2| }}2^{-\frac{\epsilon_p}{64} (|k_1|+|k_2|+|k_3|)} \LpN{p}{ \q(\sum_{j\in \N^{\nu}} \q|  2^{(j+k_2)\cdot \delta +(\epsilon_p/4)|k_3-k_2| } D_{j+k_2} D_{j+k_3} f\w|^2\w)^{\frac{1}{2}}  }
\\&\leq \sum_{\substack{k_1,k_2,k_3\in \Z^{\nu} \\ |k_3|> M \\ |k_3|> 2|k_2| }}2^{-\frac{\epsilon_p}{64} (|k_1|+|k_2|+|k_3|)} \LpN{p}{ \q(\sum_{j\in \N^{\nu}} \q|  2^{j\cdot \delta +(\epsilon_p/4)|k_3-k_2| } D_{j} D_{j+k_3-k_2} f\w|^2\w)^{\frac{1}{2}}  }
\\&\leq \sum_{\substack{k_1,k_2,k_3\in \Z^{\nu} \\ |k_3|> M  }}2^{-\frac{\epsilon_p}{64} (|k_1|+|k_3|)} \LpN{p}{ \q(\sum_{j\in \N^{\nu}} \q|  2^{j\cdot \delta +(\epsilon_p/4)|k_3-k_2| } D_{j} D_{j+k_3-k_2} f\w|^2\w)^{\frac{1}{2}}  }
\\&= \sum_{\substack{k_1,k_2',k_3\in \Z^{\nu} \\ |k_3|> M  }}2^{-\frac{\epsilon_p}{64} (|k_1|+|k_3|)} \LpN{p}{ \q(\sum_{j\in \N^{\nu}} \q|  2^{j\cdot \delta +(\epsilon_p/4)|k_2'| } D_{j} D_{j+k_2'} f\w|^2\w)^{\frac{1}{2}}  }
\\&\lesssim 2^{-M\epsilon_p/(64\nu)}\sum_{k_2'\in \Z^{\nu}} \LpN{p}{ \q(\sum_{j\in \N^{\nu}} \q|  2^{j\cdot \delta +(\epsilon_p/4)|k_2'| } D_{j} D_{j+k_2'} f\w|^2\w)^{\frac{1}{2}}  },
\end{split}
\end{equation*}
where in the second to last line we have set $k_2'=k_2-k_3$ in the summation in $k_2$, and in the last line we have summed in $k_1,k_3$ using the restriction $|k_3|>M$.

Combining the above estimates, we have that there exists a constant $C$ which is independent of $f$ and $M$ such that
\begin{equation*}
\begin{split}
&\sum_{k\in \Z^{\nu}} \LpN{p}{\q( \sum_{j\in \N^{\nu}} \q| 2^{j\cdot \delta+\delta_0|k|} D_j D_{j+k} f \w|^2 \w)^{\frac{1}{2}} }
=(I)+(II)+(III)
\\&\leq C2^{\nu|\delta|M} \NLpN{p}{\delta}{\fD}{f} + C2^{-M\epsilon_p/(64\nu)} \sum_{k\in \Z^{\nu}} \LpN{p}{\q( \sum_{j\in \N^{\nu}} \q| 2^{j\cdot \delta+\delta_0|k|} D_j D_{j+k} f \w|^2 \w)^{\frac{1}{2}} }
\end{split}
\end{equation*}
Taking $M$ so large $C2^{-M\epsilon_p/(64\nu)}\leq \frac{1}{2}$ we have
\begin{equation*}
\sum_{k\in \Z^{\nu}} \LpN{p}{\q( \sum_{j\in \N^{\nu}} \q| 2^{j\cdot \delta+\delta_0|k|} D_j D_{j+k} f \w|^2 \w)^{\frac{1}{2}} } \leq 2C2^{\nu|\delta|M} \NLpN{p}{\delta}{\fD}{f}.
\end{equation*}
This completes the proof in Case I.

We now turn to Case II.  In the proof in Case I, we proved the result for $|\delta|\leq \epsilon_p/8$ and $\delta_0\leq \epsilon_p/4$, where $\epsilon_p$
was chosen so that \eqref{EqnPfSobApplVVBound} held.  In Case II, Case II of Proposition \ref{PropPfSobVVOps} shows that
\eqref{EqnPfSobApplVVBound} holds for all $\epsilon_p\in (0,\infty)$.  Thus, the same proof applies to show that \eqref{EqnPfSobTwoToOneBoundII}
holds for all $\delta\in \R^{\nu}$, $\delta_0\in \R$, as desired.
\end{proof}

\begin{proof}[Proof of Theorem \ref{ThmResSobWellDefNew}]
First we prove the result for Case I, then we indicate the modifications necessary to prove the result in Case II.
Let $p\in (1,\infty)$, and let $\fD$ and $\fDt$ be as in the statement of the theorem.
For $j\in \N^{\nu}$ define $D_j=D_j(\fD)$ and $\Dt_j=\Dt_j(\fDt)$ by \eqref{EqnResSobDefnDjNew}.
Our goal is to show that there is an $\epsilon=\epsilon(p,(\gamma,e,N),(\gammat,\et,\Nt))>0$ such that
for $\delta\in \R^{\nu}$ with $|\delta|<\epsilon$,
\begin{equation*}
\NLpN{p}{\delta}{\fD}{f}\approx \NLpN{p}{\delta}{\fDt}{f}, \quad f\in C_0^\infty(\Omega_0).
\end{equation*}
Because the problem is symmetric in $\fD$ and $\fDt$ it suffices to show that 
there exists $\epsilon=\epsilon(p,(\gamma,e,N),(\gammat,\et,\Nt))>0$ such that
for $\delta\in \R^{\nu}$ with $|\delta|<\epsilon$,
\begin{equation*}
\NLpN{p}{\delta}{\fDt}{f}\lesssim \NLpN{p}{\delta}{\fD}{f}, \quad f\in C_0^\infty(\Omega_0).
\end{equation*}

For $k_1\in \Z^\nu$, we define a vector valued operator
\begin{equation*}
\sR_{k_1}^1 \{ f_j\}_{j\in \N^{\nu}} := \{ \Dt_j D_{j+k_1} f_j\}_{j\in \N^{\nu}}.
\end{equation*}
Case I of Proposition \ref{PropPfSobVVOps} and Lemma \ref{LemmaPfSobTwoToOne} show that
there exists $\epsilon_p>0$ such that:\footnote{As in the proof of Lemma \ref{LemmaPfSobTwoToOne}, we have
replaced $\Omega_0$ in our application of Proposition \ref{PropPfSobVVOps} with a larger open set $\Omega_0'\Subset\Omega'$,
so that Proposition \ref{PropPfSobVVOps} applies to $\sR_{k_1}^1$.}
\begin{equation}\label{EqnPfSobFirstVV}
\LplqnuOpN{p}{2}{\sR_{k_1}^1}\lesssim 2^{-\epsilon_p |k_1|},
\end{equation}
and if $|\delta|,\delta_0<\epsilon_p$,
\begin{equation}\label{EqnPfSobUseTwoToOne}
\sum_{k\in \Z^{\nu}} \LpN{p}{\q( \sum_{j\in \N^{\nu}} \q| 2^{j\cdot \delta+\delta_0|k|} D_j D_{j+k} f \w|^2 \w)^{\frac{1}{2}} }\leq A \NLpN{p}{\delta}{\fD}{f}, \quad f\in C_0^\infty(\Omega_0).
\end{equation}
We prove the result for $\q|\delta\w|\leq \epsilon_p/4$.

We have, for $f\in C_0^\infty(\Omega_0)$, using that $f=\psi^2 f$ and the convention that $D_k=0$ for $k\in \Z^{\nu}\setminus \N^{\nu}$,
\begin{equation*}
\begin{split}
&\NLpN{p}{\delta}{\fDt}{f} = \NLpN{p}{\delta}{\fDt}{\psi^4 f}
= \LpN{p}{\q( \sum_{j\in \N^{\nu}}\q| 2^{j\cdot \delta} \Dt_j \psi^4 f \w|^2 \w)^{\frac{1}{2} }}
\\&= \LpN{p}{\q( \sum_{j\in \N^{\nu}}\q| \sum_{k_1,k_2\in \Z^{\nu}}2^{j\cdot \delta} \Dt_j D_{j+k_1} D_{j+k_2} f \w|^2 \w)^{\frac{1}{2} }}
\\&\leq \sum_{k_1,k_2\in \Z^{\nu}}\LpN{p}{\q( \sum_{j\in \N^{\nu}}\q| 2^{j\cdot \delta} \Dt_j D_{j+k_1} D_{j+k_2} f \w|^2 \w)^{\frac{1}{2} }}
\\&= \sum_{\substack{k_1,k_2\in \Z^{\nu} \\ |k_1|\geq |k_2|/2}}
+ \sum_{\substack{k_1,k_2\in \Z^{\nu} \\ |k_2|> 2|k_1|}} =:(I)+(II).
\end{split}
\end{equation*}

We begin by bounding $(I)$.  For $k_1,k_2\in \Z^{\nu}$, let $\sT_{k_1}^1 := 2^{-k_2\cdot \delta} \sR_{k_1}^1$.  For $|k_1|\geq |k_2|/2$, we have
by \eqref{EqnPfSobFirstVV} and the fact that $|\delta|\leq \epsilon_p/4$,
\begin{equation}\label{EqnPfSobSecondVV}
\begin{split}
&\LplqnuOpN{p}{2}{\sT_{k_1}^1} = 2^{-k_2\cdot \delta} \LplqnuOpN{p}{2}{\sR_{k_1}^1} 
\\&\lesssim 2^{-k_2\cdot \delta - \epsilon_p|k_1|} 
\leq 2^{2|k_1||\delta| - \epsilon_p|k_1|}
\leq 2^{|k_1|\epsilon_p/2 -\epsilon_p|k_1|}\leq 2^{-\epsilon_p|k_1|/2}.
\end{split}
\end{equation}
Using \eqref{EqnPfSobSecondVV}, we have
\begin{equation*}
\begin{split}
(I)&= \sum_{\substack{k_1,k_2\in \Z^{\nu} \\ |k_1|\geq |k_2|/2}}\LpN{p}{\q( \sum_{j\in \N^{\nu}}\q| 2^{j\cdot \delta} \Dt_j D_{j+k_1} D_{j+k_2} f \w|^2 \w)^{\frac{1}{2} }}
\\&=\sum_{\substack{k_1,k_2\in \Z^{\nu} \\ |k_1|\geq |k_2|/2}} \LplqnuN{p}{2}{\sT_{k_1}^1 \{ 2^{(j+k_2)\cdot \delta} D_{j+k_2} f \}_{j\in \N^{\nu}} }
\\&\lesssim \sum_{\substack{k_1,k_2\in \Z^{\nu} \\ |k_1|\geq |k_2|/2}} 2^{-\epsilon_p |k_1|/2} \LplqnuN{p}{2}{ \{ 2^{j\cdot \delta} D_j f\}_{j\in \N^{\nu}}}
\\&\lesssim \sum_{\substack{k_1,k_2\in \Z^{\nu} \\ |k_1|\geq |k_2|/2}} 2^{-\epsilon_p |k_1|/4 - \epsilon_p|k_2|/8} \LplqnuN{p}{2}{ \{2^{j\cdot \delta} D_j f\}_{j\in \N^{\nu}}}
\\&\lesssim \LplqnuN{p}{2}{\{2^{j\cdot \delta} D_j f\}_{j\in \N^{\nu}}} = \NLpN{p}{\delta}{\fD}{f}.
\end{split}
\end{equation*}

We now turn to $(II)$.  Define a vector valued operator $\sR \{f_j\}_{j\in \N^{\nu}} := \{\Dt_j f_j \}_{j\in \N^{\nu}}$.
Proposition \ref{PropPfSobVVOps} shows $\LplqnuOpN{p}{2}{\sR}\lesssim 1$.
\begin{equation*}
\begin{split}
&(II)=\sum_{\substack{k_1,k_2\in \Z^{\nu} \\ |k_2|> 2|k_1|}}\LpN{p}{\q( \sum_{j\in \N^{\nu}}\q| 2^{j\cdot \delta} \Dt_j D_{j+k_1} D_{j+k_2} f \w|^2 \w)^{\frac{1}{2} }}
\\&=\sum_{\substack{k_1,k_2\in \Z^{\nu} \\ |k_2|> 2|k_1|}} 2^{-k_1\cdot \delta - 3\frac{\epsilon_p}{4} |k_2-k_1|} \LplqnuN{p}{2}{\sR \{ 2^{(j+k_1)\cdot \delta +3\frac{\epsilon_p}{4}  |k_2-k_1|} D_{j+k_1} D_{j+k_2} f \}_{j\in \N^{\nu}}}
\\&\lesssim \sum_{\substack{k_1,k_2\in \Z^{\nu} \\ |k_2|> 2|k_1|}} 2^{-k_1\cdot \delta - 3\frac{\epsilon_p}{4}  |k_2-k_1|} \LplqnuN{p}{2}{\{ 2^{(j+k_1)\cdot \delta +3\frac{\epsilon_p}{4} |k_2-k_1|} D_{j+k_1} D_{j+k_2} f \}_{j\in \N^{\nu}}}
\\&\leq \sum_{\substack{k_1,k_2\in \Z^{\nu} \\ |k_2|> 2|k_1|}} 2^{|k_1||\delta| - 3\frac{\epsilon_p}{4}  |k_2-k_1|} \LplqnuN{p}{2}{ \{ 2^{j\cdot \delta +3\frac{\epsilon_p}{4}  |k_2-k_1| }  D_j D_{j+k_2-k_1} f\}_{j\in \N^{\nu}}}
\\&\leq \sum_{\substack{k_1,k_2\in \Z^{\nu} \\ |k_2|> 2|k_1|}} 2^{\frac{\epsilon_p}{4}|k_1| - \frac{\epsilon_p}{2}|k_1| - \frac{\epsilon_p}{8}|k_2|}  \LplqnuN{p}{2}{ \{ 2^{j\cdot \delta +3\frac{\epsilon_p}{4} |k_2-k_1| }  D_j D_{j+k_2-k_1} f\}_{j\in \N^{\nu}}}
\\&\lesssim \sum_{l\in \Z^{\nu}} \LplqnuN{p}{2}{ \{ 2^{j\cdot \delta + 3\frac{\epsilon_p}{4}  |l| }D_j D_{j+l} f\}_{j\in \N^{\nu}}}
\lesssim \NLpN{p}{\delta}{\fD}{f},
\end{split}
\end{equation*}
where the last inequality follows from \eqref{EqnPfSobUseTwoToOne}.  Combining the above estimates completes the proof in Case I.

We now turn to Case II.  The proof above in Case I worked for $|\delta|\leq \epsilon_p/4$, where $\epsilon_p$ was such that
\eqref{EqnPfSobFirstVV} and \eqref{EqnPfSobUseTwoToOne} held.  Under the assumptions in Case II,
Case II of Proposition \ref{PropPfSobVVOps} and Lemma \ref{LemmaPfSobTwoToOne} show that
these equations hold for all $\epsilon_p\in (0,\infty)$.  The result for Case II therefore follows from the same proof.
\end{proof}

We now turn to the proof of Proposition \ref{PropResSobLp}.  Let $\fD$ be Sobolev data.  In light of Theorem \ref{ThmResSobWellDefNew},
it suffices to prove Proposition \ref{PropResSobLp} for any Sobolev data $\fDt$ such that $\fD$ and $\fDt$ are finitely generated by the same $\sF$ on $\Omega'$.
The next lemma helps us pick out a choice of Sobolev data which is convenient for proving Proposition \ref{PropResSobLp}.  In it, we
use the notation $\pi_{\mu}$ from Definition \ref{DefnPfLtTechApOPiSigma}.

\begin{lemma}\label{LemmaPfSobChoiceOfGamma}
Let $\sS\subseteq \snuvectone$ be such that $\sL(\sS)$ is finitely generated by $\sF\subset \snuvect$ on $\Omega'$.
For each $\mu\in \nuset$, enumerate
$$\pi_{\mu} \sF = \{ (X_1^\mu, d_1^\mu),\ldots, (X_{q_{\mu}}^{\mu}, d_{q_\mu}^\mu) \}\subset \snuvectone.$$
Let $\hd_j^{\mu}=|d_j^\mu|_1$.  For each $\mu$, and $t_\mu=(t_{\mu,1},\ldots, t_{\mu,q_{\mu}})\in \R^{q_{\mu}}$, define
\begin{equation}\label{EqnPfSobDefinemuParam}
\gamma_{t_{\mu}}^\mu(x) := e^{ t_{\mu,1} X_{1}^{\mu} + \cdots + t_{\mu, q_{\mu}} X_{q_{\mu}}^{\mu}} x.
\end{equation}
Define single parameter dilations (denoted by $\hd^{\mu}$) on $\R^{q_\mu}$ by, for $j_\mu\in \R$,
\begin{equation*}
2^{-j_{\mu}} t_{\mu} = (2^{-j_\mu \hd_1^{\mu}} t_{\mu,1},\ldots, 2^{-j_\mu \hd_{q_\mu}^{\mu}} t_{q_\mu,\mu}).
\end{equation*}
Thus, we have a parameterization $(\gamma^{\mu}, \hd^{\mu}, q_\mu)$.
Let $q=q_{1}+\cdots+q_{\nu}$.  For $t=(t_1,\ldots, t_\nu)\in \R^{q_1}\times \cdots\times \R^{q_\nu}=\R^{q}$ define
\begin{equation*}
\gamma_t(x) := \gamma_{t_{\nu}}^{\nu}\circ \gamma_{t_{\nu-1}}^{\nu-1}\circ\cdots\circ\gamma_{t_{1}}^1(x).
\end{equation*}
We define $\nu$ parameter dilations on $\R^{q}$, which we denote by $e$, by for $j\in \R^{\nu}$,
\begin{equation*}
2^{-j} (t_1,\ldots, t_\nu) = (2^{-j_1} t_1,\ldots, 2^{-j_\nu} t_\nu),
\end{equation*}
where $2^{-j_\mu} t_\mu$ is defined by the single parameter dilations on $\R^{q_{\mu}}$.
Then, $(\gamma,e,q)$ is finitely generated by $\sF$ on $\Omega'$.
Furthermore, if $\sL(\sS)$ is linearly finitely generated by $\sF\subset \snuvectone$ on $\Omega'$,
then $\gamma$ is linearly finitely generated by $\sF$ on $\Omega'$.
\end{lemma}
\begin{proof}
First we show that $\sF$ controls $(\gamma,e,q)$ on $\Omega'$.  Using Proposition \ref{PropFrobControlUnitScaleVsAll}, 
it suffices to show that for $j\in [0,\infty]^{\nu}$, 
if $Z:=\{2^{-j\cdot d} X : (X,d)\in \sF\}$ and if $\gammah_t(x):= \gamma_{2^{-j} t}(x)$,
then $Z$ controls $\gammah$ at the unit scale on $\Omega'$, uniformly in $j\in [0,\infty]^{\nu}$.
Set $\gammah_{t_\mu}^{\mu}(x):= \gamma_{2^{-j_\mu} t_\mu}^\mu(x)$.
By Proposition 12.6 of \cite{SteinStreetI}, it suffices to show $Z$ controls $\gammah^{\mu}$ at the unit scale on $\Omega'$,
uniformly in $j\in [0,\infty]^{\nu}$, $\mu\in \nuset$.
Fix $x_0$, we will show $Z$ controls $\gammah$ at the unit scale near $x_0$ uniformly for $x_0\in \Omega'$ and $j\in [0,\infty]^{\nu}$, $\mu\in \nuset$.
Because $\sF$ satisfies $\sD(\Omega')$ (Lemma \ref{LemmaResCCsLFinGenGivessD}), $Z$ satisfies the conditions of Theorem \ref{ThmQuantFrob}
uniformly for $x_0\in \Omega'$, $j\in [0,\infty]^{\nu}$.  Let $\Phi$ be the map associated to $Z$ given by Theorem \ref{ThmQuantFrob},
with this choice of $x_0$.
By Proposition \ref{PropFrobControlEquivsQ}, it suffices to show that $\gammah^{\mu}$ satisfies $\sQ_2$, with parameters independent
of $j\in [0,\infty]^{\nu}$, $x_0\in \Omega'$.

Let $Y_l^{\mu}$ be the pullback of $2^{-j_\mu \hd_l^{\mu}} X_l^{\mu}$ via $\Phi$.
Theorem \ref{ThmQuantFrob} shows for every $m$, $\| Y_j^\mu \|_{C^m}\lesssim 1$.  Standard theorems from ODEs\footnote{See Appendix B.1 of \cite{StreetMultiParamSingInt} for more on this.}
show that the function
\begin{equation*}
\theta_{t_\mu}^{\mu}(u) := e^{ t_{\mu,1} Y_1^{\mu} + \cdots + t_{q_{\mu},\mu} Y_{q_\mu}^{\mu}} u
\end{equation*}
satisfies $\| \theta^{\mu}\|_{C^m(B^{q_\mu}(a')\times B^{n_0}(\eta') )}\lesssim 1$ for every $m$ (here $a',\eta'\gtrsim 1$).
Because $\theta^{\mu}_{t_\mu} (u)= \Phi^{-1} \circ \gammah_{t_\mu}\circ \Phi(u)$, we have that $\sQ_2$ holds with 
parameters independent of $j\in [0,\infty]^{\nu}$, $x_0\in \Omega'$.   Combining all of the above, we have that $\sF$
controls $(\gamma,e,q)$ on $\Omega'$, as desired.

Let $(\gamma,e,q)$ correspond to the vector field parametrization $(W,e,q)$.  Note that
\begin{equation}\label{EqnPfSobWWithZeros}
W(0,\ldots, 0 , t_\mu,0,\ldots, 0) = t_{\mu,1}X_1^{\mu}+\cdots + t_{\mu,q_{\mu}} X_{q_\mu}^{\mu}.
\end{equation}
It follows that if $\sS$ is given by \eqref{EqnResSurfDefnsS} (with this choice of $W$), then each $\pi_\mu \sF\subseteq \sS$.
Lemma \ref{LemmaPfLtsLControl} shows that $\sL(\bigcup_{\mu\in \nuset} \pi_\mu \sF)$ is equivalent to $\sF$ on $\Omega'$,
and therefore $\sL(\sS)$ controls $\sF$ on $\Omega'$.  Because $\sF$ controls $(\gamma,e,q)$ on $\Omega'$, it follows
that $\sF$ controls $\sL(\sS)$ on $\Omega'$, and therefore $\sL(\sS)$ is finitely generated by $\sF$ on $\Omega'$.
Thus, $(\gamma,e,q)$ is finitely generated by $\sF$ on $\Omega'$.

Finally, if $\sF\subset \snuvectone$, then $\bigcup_{\mu\in \nuset} \pi_\mu \sF=\sF$, and therefore if
$W(t)\sim \sum_{|\alpha|>0} t^{\alpha} X_\alpha$, we have by \eqref{EqnPfSobWWithZeros}
$\sF = \{ (X_\alpha, \deg(\alpha)) : \deg(\alpha)\in \onecompnu\text{ and } |\alpha|=1\}$.
Because $\sF$ controls $(\gamma,e,q)$ on $\Omega'$, it follows that $(\gamma,e,q)$ is linearly finitely generated
by $\sF$ on $\Omega'$, as desired.
\end{proof}

\begin{lemma}\label{LemmaPfSobConvSquare}
Let $\fD$ be Sobolev data on $\Omega'$.  We separate our assumptions into two cases:
\begin{enumerate}[\bf{Case} I:]
\item $\fD$ is finitely generated by $\sF\subset \snuvect$ on $\Omega'$.
\item $\fD$ is linearly finitely generated by $\sF\subset \snuvectone$ on $\Omega'$.
\end{enumerate}
We take all the same notation as in Lemma \ref{LemmaPfSobChoiceOfGamma} with this choice of $\sF$.
Thus, we have a parameterization $(\gamma,d,q)$ which is finitely generated by $\sF$ on $\Omega'$ in Case I,
and linearly finitely generated by $\sF$ on $\Omega'$ in Case II.  For $\mu\in \nuset$, we also have parameterizations
$(\gamma^{\mu},\hd^{\mu},q_{\mu})$ defined by 
\eqref{EqnPfSobDefinemuParam}, so that $\gamma_t(x)=\gamma_{t_\nu}^{\nu}\circ \gamma_{t_{\nu-1}}^{\nu-1}\circ\cdots \circ\gamma_{t_1}^1(x)$.
For $\mu\in \nuset$,
fix $\psi_\mu\in C_0^\infty(\Omega')$ with $\psi_\nu\equiv 1$ on a neighborhood of the closure of $\Omega_0$,
and $\psi_{\mu}\equiv 1$ on a neighborhood $\supp{\psi_{\mu+1}}$, for $\mu<\nu$.
Let $a>0$ be small (to be chosen in the proof).  
For each $\mu\in \nuset$, let 
$\fD_{\mu}=\q(1,(\gamma^{\mu},\hd^{\mu},q_{\mu},\Omega,\Omega'''), a, \eta_{\mu},\{\vsig_{\mu,j}\}_{j\in \N}, \psi_\mu\w)$
be Sobolev data on $\Omega'$.  For $j_{\mu}\in \N$, let $D_j^{\mu}=D_j^{\mu}(\fD^{\mu})$ be given by \eqref{EqnResSobDefnDjNew}.
 For $j=(j_1,\ldots, j_\nu)\in \N^{\nu}$ define
\begin{equation*}
D_{j}=D_{j_1}^1\cdots D_{j_\nu}^{\nu}.
\end{equation*}
Then,
\begin{itemize}
\item In Case I, for $1<p<\infty$, there exists $\epsilon=\epsilon(p,(\gamma,d,q))>0$ such that for $|\delta|<\epsilon$,
\begin{equation}\label{EqnPfSobChangeSquareI}
\NLpN{p}{\delta}{\fD}{f}\approx \LpN{p}{\q(\sum_{j\in \N^{\nu}} \q| 2^{j\cdot \delta} D_j f\w|^2\w)^{\frac{1}{2}}}, \quad f\in C_0^\infty(\Omega_0).
\end{equation}
Here, the implicit constants depend on $p\in (1,\infty)$ and $\fD$.

\item In Case II, for $1<p<\infty$, $\delta\in \R^{\nu}$,
\begin{equation}\label{EqnPfSobChangeSquareII}
\NLpN{p}{\delta}{\fD}{f}\approx \LpN{p}{\q(\sum_{j\in \N^{\nu}} \q| 2^{j\cdot \delta} D_j f\w|^2\w)^{\frac{1}{2}}}, \quad f\in C_0^\infty(\Omega_0).
\end{equation}
Here the implicit constants depend on $p\in (1,\infty)$, $\fD$, and $\delta\in \R^{\nu}$.
\end{itemize}
Furthermore, in either case, fix $\nuh\in \nuset$, set $\nut=\nu-\nuh$, and decompose $\delta\in \R^{\nu}$ as $\delta=(\deltah,\deltat)\in \R^{\nuh}\times \R^{\nut}$
where $\deltah=(\delta_1,\ldots, \delta_{\nuh})$, $\deltat=(\delta_{\nuh+1}, \ldots, \delta_{\nu})$.  Also, for $j\in \N^{\nu}$ decompose
$j=(\jh,\jt)$ in the same way.
Set 
$$\Dt_{\jt}:= D_{j_{\nuh+1}}^{\nuh+1} D_{j_{\nuh+2}}^{\nuh+2}\cdots D_{j_{\nu}}^{\nu}.$$
Then for $1<p<\infty$, $\deltat\in \R^{\nut}$, we have
\begin{equation}\label{EqnPfSobDropParamsInSquare}
\LpN{p}{\q(\sum_{j\in\N^{\nu}} \q|2^{\jt\cdot \deltat} D_j f\w|^2\w)^{\frac{1}{2}}}\approx \LpN{p}{\q(\sum_{\jt\in \N^{\nut}} \q|2^{\jt\cdot \deltat}\Dt_{\jt} f  \w|^2\w)^{\frac{1}{2}}},
\end{equation}
where the implicit constants depend on $p\in (1,\infty)$.
In particular, taking $\nuh=\nu$, we have
\begin{equation}\label{EqnPfSobDropParamsInSquareAll}
\LpN{p}{\q(\sum_{j\in\N^{\nu}} \q|D_j f\w|^2\w)^{\frac{1}{2}}}\approx\LpN{p}{f}.
\end{equation}
\end{lemma}
\begin{proof}
We pick $a>0$ so small for $|t_{\mu}|<a$ and $\mu<\nu$,
$\psi_{\mu}(x) \psi_{\mu+1}(\gamma_{t_{\mu}}^{\mu}(x))=\psi_{\mu+1}(\gamma_{t_{\mu}}^{\mu}(x))$.
Note that $\sum_{j_\mu\in \N} D_{j_\mu}^{\mu} = \psi_{\mu}^2$ and therefore $\sum_{j\in \N^{\nu}} D_j = \prod_{\mu=1}^{\nu} \psi_\mu^{2}=\psi_{\nu}^2$.
Also,
\begin{equation*}
\begin{split}
&D_j f(x) = 
\\&\psi_1(x) \int f(\gamma_t(x))  \psi_{\nu}(\gamma_t(x))\psi_{\nu}(\gamma_{t_{\nu-1}}^{\nu-1}\circ \gamma_{t_{\nu-2}}^{\nu-2}\circ \cdots \circ \gamma_{t_1}^1(x) )\q[ \prod_{\mu=1}^{\nu} \eta_{\mu}(t_\mu) \dil{\vsig_{j_\mu,\mu}}{2^{j_\mu}}(t_{\mu}) \w]\: dt.
\end{split}
\end{equation*}
Because of the above remarks and the fact that $(\gamma,d,q)$ is finitely generated by $\sF$ on $\Omega'$ in Case I, and linearly finitely generated by $\sF$ on $\Omega'$ in Case II,
the same proof as in  Theorem \ref{ThmResSobWellDefNew} yields \eqref{EqnPfSobChangeSquareI} and \eqref{EqnPfSobChangeSquareII} (by possibly shrinking $a>0$).  Strictly speaking, $D_j$ is not exactly of the form covered in the proof of Theorem \ref{ThmResSobWellDefNew}, but the same proof goes through unchanged.

We now turn to proving \eqref{EqnPfSobDropParamsInSquare}.  Fix $p\in (1,\infty)$ and $\deltat\in \R^{\nut}$.  Set $\delta=(0_{\nuh},\deltat)\in \R^{\nuh}\times \R^{\nut}=\R^{\nu}$.
For each $\mu\in \nuset$, pick $\Omega_\mu$ with $\Omega_0\Subset \Omega_\mu \Subset \Omega'$ and $\supp{\psi_\mu}\Subset \Omega_\mu$,
and if $\mu>1$, $\psi_{\mu-1}\equiv 1$ on a neighborhood of the closure of $\Omega_{\mu}$.  
Proposition \ref{PropPfsLp} (applied with $\Omega_0$ replaced by $\Omega_\mu$) shows that if $\epsilon_{j_\mu}^{\mu}$ is a sequence of i.i.d. random variables of mean $0$ taking value $\pm 1$,
we have
\begin{equation}\label{EqnPfSobOneParamDropNew}
\q(\bE \LpN{p}{ \sum_{j_{\mu}\in \N}  \epsilon_{j_\mu}^{\mu} D_{j_\mu}^{\mu} f }^{p} \w)^{\frac{1}{p}} \approx \LpN{p}{f}, \quad f\in C_0^\infty(\Omega_\mu).
\end{equation} 
Pick the sequences $\epsilon_{j_\mu}^{\mu}$ so that they are mutually independent for $\mu\in \nuset$.
For $j=(j_1,\ldots,j_{\nu})\in \N^{\nu}$, set $\epsilon_j :=\epsilon_{j_1}^1\cdots \epsilon_{j_\nu}^\nu$, so that $\{\epsilon_j\}_{j\in \N^{\nu}}$ are i.i.d. random
variables of mean $0$ taking value $\pm 1$.  
Also set $\epsilon_{\jt}= \epsilon_{\jt_1}^{\nuh+1}\cdots \epsilon_{\jt_{\nut}}^{\nuh+\nut}$, so that $\{\epsilon_{\jt}\}_{\jt\in \N^{\nut}}$ are also i.i.d. random
variables of mean $0$ taking values $\pm 1$.
Using the Khintchine inequality and repeated applications of \eqref{EqnPfSobOneParamDropNew}, we have for $f\in C_0^\infty(\Omega_0)$,
\begin{equation*}
\begin{split}
&\LpN{p}{\q(\sum_{j\in \N} \q|2^{\jt\cdot \deltat}D_j f\w|^2\w)^{\frac{1}{2}} } \approx \q(\bE \LpN{p}{\sum_{j\in \N^{\nu}} \epsilon_j 2^{\jt\cdot\deltat} D_j f }^p\w)^{\frac{1}{p}}
\\&=\Bigg(\bE \Bigg\| \q(\sum_{j_1\in \N} \epsilon_{j_1}^1 D_{j_1}^1\w)\cdots \q(\sum_{j_{\nuh}\in \N} \epsilon_{j_{\nuh}}^{\nuh} D_{j_{\nuh}}^{\nuh}\w) 
\\&\quad\quad\quad\times \q(\sum_{j_{\nuh+1}\in \N} \epsilon_{j_{\nuh+1}}^{\nuh+1} 2^{j_{\nuh+1}\delta_{\nuh+1} }D_{j_{\nuh+1}}^{\nuh+1}\w)\cdots\q(\sum_{j_\nu\in \N} \epsilon_{j_\nu}^\nu 2^{j_{\nu}\delta_{\nu}}D_{j_\nu}^\nu\w)  f \Bigg\|_{L^p}^p\Bigg)^{\frac{1}{p}}
\\&\approx\Bigg(\bE \Bigg\| \q(\sum_{j_2\in \N} \epsilon_{j_2}^2 D_{j_2}^2\w)\cdots \q(\sum_{j_{\nuh}\in \N} \epsilon_{j_{\nuh}}^{\nuh} D_{j_{\nuh}}^{\nuh}\w)
\\&\quad\quad\quad\times
 \q(\sum_{j_{\nuh+1}\in \N} \epsilon_{j_{\nuh+1}}^{\nuh+1} 2^{j_{\nuh+1}\delta_{\nuh+1} }D_{j_{\nuh+1}}^{\nuh+1}\w)\cdots\q(\sum_{j_\nu\in \N} \epsilon_{j_\nu}^\nu 2^{j_{\nu}\delta_{\nu}}D_{j_\nu}^\nu\w)  f \Bigg\|_{L^p}^p\Bigg)^{\frac{1}{p}}
\\&\approx\cdots
\approx \q(\bE \LpN{p}{  \q(\sum_{j_{\nuh+1}\in \N} \epsilon_{j_{\nuh+1}}^{\nuh+1} 2^{j_{\nuh+1}\delta_{\nuh+1} }D_{j_{\nuh+1}}^{\nuh+1}\w)\cdots\q(\sum_{j_\nu\in \N} \epsilon_{j_\nu}^\nu 2^{j_{\nu}\delta_{\nu}}D_{j_\nu}^\nu\w)  f }^p\w)^{\frac{1}{p}}
\\&= \q(\bE \LpN{p}{\sum_{\jt\in \N^{\nut}} \epsilon_{\jt} 2^{\jt\cdot \deltat} \Dt_{\jt} f   }^p  \w)
\approx \LpN{p}{\q(\sum_{\jt\in \N^{\nut}} \q|2^{\jt\cdot \deltat}\Dt_{\jt} f  \w|^2\w)^{\frac{1}{2}}},
\end{split}
\end{equation*}
as desired.
\end{proof}

\begin{proof}[Proof of Proposition \ref{PropResSobLp}]
This follows by combining \eqref{EqnPfSobChangeSquareI} and \eqref{EqnPfSobDropParamsInSquareAll}.
\end{proof}

%% file: pfcompare.tex
In this section, we prove the results from Section \ref{SectionResSobCompNew}.

\begin{proof}[Proof of Theorem \ref{ThmResCompSobDropParamsNew}]
We first prove the result for Case I, and then indicate the modifications necessary to adapt the proof for Case II.
Because of the symmetry of the theorem in $\fDt$ and $\fDh$, it suffices to show that there
exists $\epsilon=\epsilon(p,(\gamma,e,N),(\gammat,\et,\Nt))>0$ such that for $\deltat\in \R^{\nut}$
with $|\deltat|<\epsilon$,
\begin{equation*}
\NLpN{p}{(0_{\nuh},\deltat)}{\fD}{f}\approx \NLpN{p}{\deltat}{\fDt}{f}, \quad f\in C_0^\infty(\Omega_0).
\end{equation*}
Write $j\in \N^{\nu}$ as $j=(\jt,\jh)\in \N^{\nut}\times \N^{\nuh}=\N^{\nu}$.

We take $D_j$, $j\in \N^{\nu}$, and $\Dt_{\jt}$, $\jt\in \N^{\nut}$ as in the statement of Lemma \ref{LemmaPfSobConvSquare}, with this choice of
$\fD$, $\nut$, and $\nuh$.
Notice that $D_j$ is to $\fD$ as $\Dt_{\jt}$ is to $\fDt$.
Thus applying \eqref{EqnPfSobChangeSquareI}, there exists $\epsilon=\epsilon(p,(\gamma,e,N),(\gammat,\et,\Nt))>0$
such that for $\deltat\in \R^{\nut}$
with $|\deltat|<\epsilon$,
\begin{equation*}
\NLpN{p}{(0_{\nuh},\deltat)}{\fD}{f}\approx \LpN{p}{\q(\sum_{j\in \N^{\nu}} \q| 2^{\jt\cdot \deltat} D_j f\w|^2\w)^{\frac{1}{2}}}, \quad f\in C_0^\infty(\Omega_0)
\end{equation*}
and
\begin{equation*}
\NLpN{p}{\deltat}{\fDt}{f}\approx \LpN{p}{\q(\sum_{\jt\in \N^{\nut}} \q| 2^{\jt\cdot \deltat} \Dt_{\jt} f\w|^2\w)^{\frac{1}{2}}}, \quad f\in C_0^\infty(\Omega_0).
\end{equation*}
Combining the above with \eqref{EqnPfSobDropParamsInSquare} yields the result.

In Case II, the same proof works except that we use \eqref{EqnPfSobChangeSquareII} in place of \eqref{EqnPfSobChangeSquareI}.
\end{proof}

\begin{lemma}\label{LemmaPfCompareVV}
Let $\fD$ and $\lambda$ be as in the statement of Theorem \ref{ThmResCompSobMainThmNew}--where we separate our assumptions
into the same two cases as in that theorem.
For $j\in \N^{\nu}$ define $D_j=D_j(\fD)$ by \eqref{EqnResSobDefnDjNew}; and decompose $j=(\jh,\jt)\in \N^{\nuh}\times\N^{\nut}$ as in Theorem \ref{ThmResCompSobMainThmNew}.
For $k\in \Z^{\nu}$ and $\deltat\in [0,\infty)^{\nut}$, define two vector valued operators:\footnote{We again use the convention that $D_j=0$ for $j\in \Z^{\nu}\setminus \N^{\nu}$.}
\begin{equation}\label{EqnPfCompareDefnsRo}
\sR_{k,\deltat}^1 \{f_j\}_{j\in \N^{\nu}} := \{  2^{\jt\cdot \deltat - \jh\cdot \lambda^{t}(\deltat)} D_j D_{j+k} f_j  \}_{j\in \N^{\nu}},
\end{equation}
\begin{equation}\label{EqnPfCompareDefnsRt}
\sR_{\deltat}^2 \{f_j\}_{j\in \N^{\nu}} := \{ 2^{\jt\cdot \deltat - \jh\cdot \lambda^{t}(\deltat)} D_j  f_j\}_{j\in \N^{\nu}}.
\end{equation}
Then, for $1<p<\infty$, we have the following:
\begin{itemize}
\item In Case I, there exists $\epsilon=\epsilon(p,(\gamma,e,N),\lambda)>0$ such that for $|\deltat|<\epsilon$ ($\deltat\in [0,\infty)^{\nut}$), we have
\begin{equation}\label{EqnPfCompareToShowLemmaVVI}
\begin{split}
\LplqnuOpN{p}{2}{\sR_{k,\deltat}^1}&\lesssim 2^{-\epsilon|k|},\\ 
\LplqnuOpN{p}{2}{\sR_{\deltat}^2} &\lesssim 1.
\end{split}
\end{equation}
Here, the implicit constants depend on $p$ and $\fD$.

\item In Case II, for every $L$ and $\deltat\in [0,\infty)^{\nut}$, there exists $C_L=C_L(p,\deltat,\lambda,(\gamma,e,N))$ such that
\begin{equation*}
\begin{split}
\LplqnuOpN{p}{2}{\sR_{k,\deltat}^1}&\leq C_L 2^{-L|k|}, 
\\ \LplqnuOpN{p}{2}{\sR_{\deltat}^2}&\leq C_1.
\end{split}
\end{equation*}
\end{itemize}
\end{lemma}
\begin{proof}
We first prove the result in Case I, and then indicate the necessary modifications to prove the result in Case II.
To prove Case I, it suffices to prove \eqref{EqnPfCompareToShowLemmaVVI} for $1<p\leq 2$, and prove \eqref{EqnPfCompareToShowLemmaVVI}
for $1<p\leq 2$ with $\sR_{k,\deltat}^1$ and $\sR_{\deltat}^2$ replaced by $\q(\sR_{k,\deltat}^1\w)^{*}$ and $\q(\sR_{\deltat}^2\w)^{*}$, respectively.
The result then follows, since (for $1<p\leq 2$) the dual of $\Lplqnu{p}{2}$ is $\Lplqnu{p'}{2}$ where $\frac{1}{p}+\frac{1}{p'}=1$.
We exhibit the proof for $\sR_{k,\deltat}^1$ and $\sR_{\deltat}^2$.  A nearly identical proof works for the adjoints
after an application of Propositions \ref{PropPfAdjEachScale} and \ref{PropPfAdjGamma} (see the proof of Proposition \ref{PropPfSobVVOps},
where the same idea is used).  We leave the remainder of the details for the adjoints to the reader.

Case I of Proposition \ref{PropPfLtApTMainProp} shows that there exists $\epsilon>0$ such that
\begin{equation}\label{EqnPfCompareLemmaVVTO}
\LpOpN{2}{D_j}\lesssim 2^{-\epsilon |\jt\vee \lambda(\jh)-\lambda(\jh)|}.
\end{equation}
Case I of Proposition \ref{PropPfLtTechApO} shows that there exists $\epsilon>0$ such that
\begin{equation}\label{EqnPfCompareLemmaVVTTT}
\LpOpN{2}{D_j D_{j+k}}\lesssim 2^{-\epsilon |k|},\quad k\in \Z^{\nu}.
\end{equation}
Combining the above two estimates, using that $ \deltat\in [0,\infty)^{\nut}$, and using the trivial bound $\LpOpN{2}{D_{j+k}}\lesssim 1$,
we have that there exists $\epsilon_2>0$ such that for $|\deltat|<\epsilon_2$, $\deltat\in [0,\infty)^{\nut}$,
\begin{equation}\label{EqnPfCompareLemmaVVCoord}
\LpOpN{2}{ 2^{ \jt\cdot \deltat - \jh\cdot \lambda^{t}\cdot \deltat} D_j} \lesssim 1,
\quad \LpOpN{2}{ 2^{ \jt\cdot \deltat - \jh\cdot \lambda^{t}\cdot \deltat} D_jD_{j+k}} \lesssim 2^{-\epsilon_2|k|}.
\end{equation}

We complexify the variable $\deltat$, which turns $\sR_{k,\deltat}^1$ and $\sR_{\deltat}^2$ into operators which depend holomorphically on $\deltat$.
For a variable $z\in \C^{\nut}$, we write $$\re{z}:=(\re{z_1},\ldots, \re{z_{\nut}})\in \R^{\nut};$$ similarly for $\im{z}$.
When $|\re{\deltat}|<\epsilon_2$, $\re{\deltat}\in [0,\infty)^{\nut}$, \eqref{EqnPfCompareLemmaVVCoord} shows
\begin{equation}\label{EqnPfCompareLemmaVVTT}
\LplqnuOpN{2}{2}{\sR_{k,\deltat}^1}\lesssim 2^{-\epsilon_2 |k|}, \quad \LplqnuOpN{2}{2}{\sR_{\deltat}^2}\lesssim 1,
\end{equation}
merely by interchanging the norms.  Here the bounds are independent of $\im{\deltat}$.
Also, we have the trivial estimates,
\begin{equation*}
\LpOpN{1}{D_j}\lesssim 1, \quad \LpOpN{1}{D_j D_{j+k}}\lesssim 1.
\end{equation*}
Thus, when $|\re{\deltat}|=0$, we have
\begin{equation}\label{EqnPfCompareLemmaVVOO}
\LplqnuOpN{1}{1}{\sR_{k,\deltat}^1}\lesssim 1, \quad \LplqnuOpN{1}{1}{\sR_{\deltat}^2}\lesssim 1,
\end{equation}
again by interchanging the norms, and the bounds are independent of $\im{\deltat}$.

Interpolating \eqref{EqnPfCompareLemmaVVTT} and \eqref{EqnPfCompareLemmaVVOO}, for $1<p\leq 2$
and $|\re{\deltat}|<\q(2-\frac{2}{p}\w)\epsilon_2$, $\re{\deltat}\in [0,\infty)^{\nut}$, we have
\begin{equation}\label{EqnPfCompareLemmaVVPP}
\begin{split}
\LplqnuOpN{p}{p}{\sR_{k,\deltat}^1}&\lesssim 2^{-(2-\frac{2}{p})\epsilon_2 |k|}, 
\\ \LplqnuOpN{p}{p}{\sR_{\deltat}^2}&\lesssim 1.
\end{split}
\end{equation}
Just as in the proof of Proposition \ref{PropPfSobVVOps}, by using the maximal function, we have for $1<p<\infty$, $|\re{\deltat}|=0$,
\begin{equation}\label{EqnPfCompareLemmaVVPI}
\begin{split}
\LplqnuOpN{p}{\infty}{\sR_{k,\deltat}^1}&\lesssim 1, 
\\ \LplqnuOpN{p}{\infty}{\sR_{\deltat}^2} &\lesssim 1.
\end{split}
\end{equation}
Interpolating \eqref{EqnPfCompareLemmaVVPP} and \eqref{EqnPfCompareLemmaVVPI} shows for $1<p\leq 2$,
if $|\re{\deltat}|<(p-1)\epsilon_2$, $\re{\deltat}\in [0,\infty)^{\nut}$, we have
\begin{equation}\label{EqnPfCompareLemmaFinalVV}
\begin{split}
\LplqnuOpN{p}{2}{\sR_{k,\deltat}^1}&\lesssim 2^{-(p-1)\epsilon_2 |k|}, 
\\ \LplqnuOpN{p}{2}{\sR_{\deltat}^2} & \lesssim 1.
\end{split}
\end{equation}
This completes the proof in Case I.

For Case II, we note that we proved \eqref{EqnPfCompareLemmaFinalVV}, where $\epsilon_2$ was as in \eqref{EqnPfCompareLemmaVVCoord}.
Case II of Propositions \ref{PropPfLtApTMainProp} and \ref{PropPfLtTechApO} show that \eqref{EqnPfCompareLemmaVVTO} and \eqref{EqnPfCompareLemmaVVTTT}
hold for all $\epsilon\in (0,\infty)$, and therefore \eqref{EqnPfCompareLemmaVVCoord} holds for all $\epsilon_2\in (0,\infty)$.  From here, the same proof as above
proves the result in Case II.
\end{proof}

\begin{proof}[Proof of Theorem \ref{ThmResCompSobMainThmNew}]
We first prove the result in Case I, and then indicate the modifications necessary to prove the result in Case II.
Let $\fD$ be as in Case I, and
for $j\in \N^{\nu}$ define $D_j=D_j(\fD)$ by \eqref{EqnResSobDefnDjNew}; and decompose $j=(\jh,\jt)\in \N^{\nuh}\times\N^{\nut}$ as in the statement of the theorem.

For $k\in \Z^{\nu}$, $\deltat\in [0,\infty)^{\nut}$ define $\sR_{k,\deltat}^1$ and $\sR_{\deltat}^2$ by \eqref{EqnPfCompareDefnsRo} and \eqref{EqnPfCompareDefnsRt}.
Fix $1<p<\infty$.
Case I of Lemma \ref{LemmaPfCompareVV} and Case I of Lemma \ref{LemmaPfSobTwoToOne} show there exists $\epsilon_p>0$ such that
for $|\deltat|,|\delta|,\delta_0<\epsilon_p$, $\deltat\in [0,\infty)^{\nut}$, $\delta\in \R^{\nu}$, $\delta_0\in \R$,
\begin{equation}\label{EqnPfCompareConcLemmaVV}
\LplqnuOpN{p}{2}{\sR_{k,\deltat}^1}\lesssim 2^{-\epsilon_p|k|}, \quad \LplqnuOpN{p}{2}{\sR_{\deltat}^2}\lesssim 1,
\end{equation}
\begin{equation}\label{EqnPfCompareTwoToOne}
\sum_{k\in \Z^{\nu}} \LpN{p}{ \q( \sum_{j\in \Z^{\nu}} \q|2^{j\cdot \delta+\deltat_0 |k|} D_j D_{j+k} f  \w|^2   \w)^{\frac{1}{2}}}\lesssim \NLpN{p}{\delta}{\fD}{f}, \quad f\in C_0^\infty(\Omega_0).
\end{equation}
We prove \eqref{EqnResCompSobMainThmSmall} for $|\delta|,|\deltat|\leq \epsilon_p/4$, $\deltat\in [0,\infty)^{\nut}$, which will complete the proof in Case I.

Consider, for $f\in C_0^\infty(\Omega')$, and using the fact that $\sum_{j\in \N^\nu} D_j = \psi^2$ where $\psi\equiv 1$ on a neighborhood of the closure of $\Omega_0$,
\begin{equation*}
\begin{split}
\NLpN{p}{\delta+ (-\lambda^{t}(\deltat),\deltat )}{\fD}{f}
&= \LpN{p}{\q(\sum_{j\in \N^{\nu}} \q| 2^{j\cdot \delta +\jt\cdot \deltat - \jh\cdot \lambda^{t}(\deltat) } D_j f \w|^2\w)^{\frac{1}{2}}}
\\&= \LpN{p}{\q(\sum_{j\in \N^{\nu}} \q| 2^{j\cdot \delta +\jt\cdot \deltat - \jh\cdot \lambda^{t}(\deltat) } D_j \psi^4 f \w|^2\w)^{\frac{1}{2}}}
\\&= \LpN{p}{\q(\sum_{j\in \N^{\nu}} \q| \sum_{k_1,k_2\in \Z^{\nu}}2^{j\cdot \delta +\jt\cdot \deltat - \jh\cdot \lambda^{t}(\deltat) } D_j D_{j+k_1} D_{j+k_2} f \w|^2\w)^{\frac{1}{2}}}
\\&\leq  \sum_{k_1,k_2\in \Z^{\nu}}\LpN{p}{\q(\sum_{j\in \N^{\nu}} \q| 2^{j\cdot \delta +\jt\cdot \deltat - \jh\cdot \lambda^{t}(\deltat) } D_j D_{j+k_1} D_{j+k_2} f \w|^2\w)^{\frac{1}{2}}}
\\&= \sum_{\substack{k_1,k_2\in \Z^{\nu} \\ |k_1|\geq |k_2|/2 }} + \sum_{\substack{k_1,k_2\in \Z^{\nu} \\ |k_2|> 2|k_1| }}
=:(I)+(II).
\end{split}
\end{equation*}
We bound the above two terms separately.

We begin with $(I)$.  For $k_1,k_2\in \Z^{\nu}$ set $\sT_{k_1,k_2}^1:= 2^{-k_2\cdot \delta} \sR_{k_1,\deltat}^1$.
Applying \eqref{EqnPfCompareConcLemmaVV}, we have for $|k_1|\geq |k_2|/2$ and using that $|\delta|\leq \epsilon_p/4$,
\begin{equation*}
\LplqnuOpN{p}{2}{\sT_{k_1,k_2}^1}\lesssim 2^{|k_2| |\delta|} 2^{-\epsilon_p |k_1|} \leq 2^{-|k_1| \epsilon_p/2}\leq 2^{- (|k_1|+|k_2|)\epsilon_p/8}.
\end{equation*}
Using this, we have,
\begin{equation*}
\begin{split}
(I)&= \sum_{\substack{k_1,k_2\in \Z^{\nu} \\ |k_1|\geq |k_2|/2 }} \LpN{p}{\q(\sum_{j\in \N^{\nu}} \q| 2^{j\cdot \delta +\jt\cdot \deltat - \jh\cdot \lambda^{t}(\deltat) } D_j D_{j+k_1} D_{j+k_2} f \w|^2\w)^{\frac{1}{2}}}
\\&= \sum_{\substack{k_1,k_2\in \Z^{\nu} \\ |k_1|\geq |k_2|/2 }} \LplqnuN{p}{2}{\sT_{k_1,k_2}^1 \q\{2^{(j+k_2)\cdot \delta} D_{j+k_2} f \w\}_{j\in \N^{\nu}} }
\\&\lesssim \sum_{\substack{k_1,k_2\in \Z^{\nu} \\ |k_1|\geq |k_2|/2 }} 2^{- (|k_1|+|k_2|)\epsilon_p/8} \LplqnuN{p}{2}{ \q\{ 2^{j\cdot \delta} D_j f\w\}_{j\in \N^{\nu}}}
\\&\lesssim \NLpN{p}{\delta}{\fD}{f},
\end{split}
\end{equation*}
as desired.

We turn to $(II)$.  For $k_1,k_2\in \Z^{\nu}$, set $\sT_{k_1,k_2}^2:= 2^{-k_1\cdot \delta - (3\epsilon_p/4) |k_1-k_2|} \sR_{\deltat}^2$.
For $|k_2|>2|k_1|$, we have $-k_1 \cdot \delta - (3\epsilon_p/4)|k_1-k_2|\leq |k_2| (\epsilon_p/4) - |k_2| (3\epsilon_p/8) = -(\epsilon_p/8) |k_2|\leq -(\epsilon_p/16) (|k_1|+|k_2|)$.   Combining this with \eqref{EqnPfCompareConcLemmaVV}, we have for $|k_2|>2|k_1|$,
\begin{equation*}
\LplqnuOpN{p}{2}{\sT_{k_1,k_2}^2}\lesssim 2^{-k_1\cdot \delta - (3\epsilon_p/4) |k_1-k_2|} \lesssim 2^{-(\epsilon_p/16) (|k_1|+|k_2|)}.
\end{equation*}
Using this, we have,
\begin{equation*}
\begin{split}
&(II)= \sum_{\substack{k_1,k_2\in \Z^{\nu} \\ |k_2|> 2|k_1| }} \LpN{p}{\q(\sum_{j\in \N^{\nu}} \q| 2^{j\cdot \delta +\jt\cdot \deltat - \jh\cdot \lambda^{t}(\deltat) } D_j D_{j+k_1} D_{j+k_2} f \w|^2\w)^{\frac{1}{2}}}
\\&=  \sum_{\substack{k_1,k_2\in \Z^{\nu} \\ |k_2|> 2|k_1| }} \LplqnuN{p}{2}{\sT_{k_1,k_2}^2 \q\{ 2^{(j+k_1)\cdot \delta +(3\epsilon_p/4) |k_1-k_2|} D_{j+k_1} D_{j+k_2} f \w\}_{j\in \N^{\nu}}  }
\\&\lesssim\sum_{\substack{k_1,k_2\in \Z^{\nu} \\ |k_2|> 2|k_1| }}  2^{-(\epsilon_p/16) (|k_1|+|k_2|)} \LpN{p}{\q( \sum_{j\in \N^{\nu}} \q| 2^{(j+k_1)\cdot \delta +(3\epsilon_p/4)|k_1-k_2|} D_{j+k_1} D_{j+k_2} f \w|^2\w)^{\frac{1}{2}} }
\\&\leq \sum_{\substack{k_1,k_2\in \Z^{\nu} \\ |k_2|> 2|k_1| }}  2^{-(\epsilon_p/16) (|k_1|+|k_2|)} \LpN{p}{\q( \sum_{j\in \N^{\nu}} \q| 2^{j\cdot \delta +(3\epsilon_p/4)|k_1-k_2|} D_{j} D_{j+k_2-k_1} f \w|^2\w)^{\frac{1}{2}} }
\\&\leq \sum_{l_1,l_2\in \Z^{\nu}} 2^{-(\epsilon_p/32) |l_1|} \LpN{p}{\q(\sum_{j\in \N^{\nu}} \q| 2^{j\cdot \delta + (3\epsilon_p/4) |l_2|} D_j D_{j+l_2} f \w|^2  \w)^{\frac{1}{2}}}
\\&\lesssim \NLpN{p}{\delta}{\fD}{f},
\end{split}
\end{equation*}
where the last line follows from \eqref{EqnPfCompareTwoToOne}.

Combining the above two estimates shows
\begin{equation*}
\NLpN{p}{\delta+ (-\lambda^{t}(\deltat),\deltat )}{\fD}{f}\lesssim \NLpN{p}{\delta}{\fD}{f},
\end{equation*}
as desired, completing the proof in Case I.

For Case II, we note that the above proof proved the result for $|\delta|,|\deltat|\leq \epsilon_p/4$, $\deltat\in [0,\infty)^{\nut}$, where $\epsilon_p$
was such that \eqref{EqnPfCompareConcLemmaVV} and \eqref{EqnPfCompareTwoToOne} held.
In Case II, Case II of Lemma \ref{LemmaPfCompareVV} and Case II of Lemma \ref{LemmaPfSobTwoToOne}
show that \eqref{EqnPfCompareConcLemmaVV} and \eqref{EqnPfCompareTwoToOne} hold for all $\epsilon_p\in (0,\infty)$.
From here, the same proof as above gives the result in Case II.
\end{proof} 

%% file: pfradon.tex
\begin{proof}[Proof of Theorem \ref{ThmResRadonMainThm}]
We first prove the result when $(\gamma,e,N)$ is finitely generated on $\Omega'$; then we outline the changes necessary for when
$(\gamma,e,N)$ is linearly finitely generated on $\Omega'$.

Let $T$ be given by \eqref{EqnResRadonMainOpNew}, so that
\begin{equation*}
Tf(x)=\psi_1(x) \int f(\gamma_t(x)) \psi_2(\gamma_t(x)) \kappa(t,x) K(t)\: dt,
\end{equation*}
where $\psi_1,\psi_1\in C_0^\infty(\Omega_0)$, $\kappa(t,x)\in C^\infty(B^N(a)\times \Omega'')$, $\delta\in \R^{\nu}$,
and $K\in \sK_\delta(N,e,a)$.
Take $\eta\in C_0^\infty(B^N(a))$ and a bounded set $\{\vsig_j : j\in \N^{\nu}\}\subset \schS(\R^N)$ with $\vsig_j\in \schS_{\{\mu : j_\mu\ne 0\}}$ such that
\begin{equation*}
K(t) = \eta(t)\sum_{j\in \N^{\nu}} 2^{j\cdot \delta} \dil{\vsig_j}{2^j}(t).
\end{equation*}
For $j\in \N^{\nu}$, define
\begin{equation*}
T_j f(x) = \psi_1(x) \int f(\gamma_t(x)) \psi_2(\gamma_t(x)) \kappa(t,x) \dil{\vsig_j}{2^j}\: dt,
\end{equation*}
so that $T=\sum_{j\in \N^{\nu} } 2^{j\cdot \delta} T_j$.  Per our usual convention, we take $T_j=0$ for $j\in \Z^{\nu}\setminus \N^{\nu}$.
Let $\fD=\q(\nu,(\gamma,e,N,\Omega,\Omega'''), a, \etat,\{\vsigt_j\}_{j\in \N^{\nu}}, \psi\w)$ be Sobolev data on $\Omega'$.
We wish to show that there exists $\epsilon=\epsilon(p,(\gamma,e,N))>0$ such that for $|\delta|,|\delta'|<\epsilon$,
\begin{equation}\label{EqnPfRadonMainToShow}
\NLpN{p}{\delta'}{\fD}{Tf}\lesssim \NLpN{p}{\delta+\delta'}{\fD}{f}, \quad f\in C_0^{\infty}(\Omega_0).
\end{equation}

For $j\in \N^{\nu}$ let $D_j = D_j(\fD)$ be as in \eqref{EqnResSobDefnDjNew}; and as usual for $j\in \Z^{\nu}\setminus \N^{\nu}$, $D_j=0$.  For $k_1,k_2\in \Z^{\nu}$, we define two vector valued operators
\begin{equation*}
\sR_{k_1,k_2}^1 \{f_j\}_{j\in \N^{\nu}} := \{D_j T_{j+k_1} D_{j+k_2} f_j \}_{j\in \N^{\nu}}, \quad \sR_{k_1}^2 \{ f_j\}_{j\in \N^{\nu}} := \{ D_j T_{j+k_1} f_j\}_{j\in \N^{\nu}}.
\end{equation*}
Case I of Proposition \ref{PropPfSobVVOps} combined with Case I of Lemma \ref{LemmaPfSobTwoToOne} shows
that for $1<p<\infty$, there exists $\epsilon_p>0$ such that for $k_1,k_2\in \Z^{\nu}$, $|\delta|,\delta_0<\epsilon_p$, $\delta\in \R^{\nu}$, and $\delta_0\in \R$,
we have
\begin{equation}\label{EqnPfRadonMainVV}
\begin{split}
\LplqnuOpN{p}{2}{\sR_{k_1,k_2}^1}&\lesssim 2^{-\epsilon_p (|k_1|+|k_2|)}, 
\\ \LplqnuOpN{p}{2}{\sR_{k_1}^2}&\lesssim 2^{-\epsilon_p |k_1|},
\end{split}
\end{equation}
\begin{equation}\label{EqnPfRadonMainTwoToOne}
\sum_{k\in \Z^{\nu}} \LpN{p}{\q( \sum_{j\in \N^{\nu}} \q| 2^{j\cdot \delta+\delta_0 |k|} D_j D_{j+k} f \w|^2\w)^{\frac{1}{2} }}\lesssim \NLpN{p}{\delta}{\fD}{f}, \quad f\in C_0^\infty(\Omega_0).
\end{equation}
Here, we have replaced $\Omega_0$ in the application of Proposition \ref{PropPfSobVVOps} with some $\Omega_0'\Subset \Omega$ such that
$\supp{\psi}\subset \Omega_0'$.
We prove \eqref{EqnPfRadonMainToShow} for $|\delta|,|\delta'|<\epsilon_p/8$, which will complete the proof in Case I.

Let $f\in C_0^\infty(\Omega_0)$.  Using the fact that $\psi f= f$ and $\sum_{j\in \N^{\nu}} D_j = \psi^2$, we see
\begin{equation*}
\begin{split}
&\NLpN{p}{\delta'}{\fD}{Tf} = \LpN{p}{\q(\sum_{j\in \N^{\nu}} \q|2^{j\cdot \delta'} D_j T \psi^4 f\w|\w)^{\frac{1}{2}}}
\\&= \LpN{p}{\q(\sum_{j\in \N^{\nu}} \q|\sum_{k_1,k_2,k_3\in \Z^{\nu}} 2^{j\cdot \delta' + (j+k_1)\cdot \delta} D_j T_{j+k_1} D_{j+k_2} D_{j+k_3} f \w|^2\w)^{\frac{1}{2}} }
\\&\leq \sum_{k_1,k_2,k_3\in \Z^{\nu}} \LpN{p}{\q(\sum_{j\in \N^{\nu}} \q| 2^{j\cdot \delta' + (j+k_1)\cdot \delta} D_j T_{j+k_1} D_{j+k_2} D_{j+k_3} f \w|^2\w)^{\frac{1}{2}} }
\\&= \sum_{\substack{k_1,k_2,k_3\in \Z^{\nu} \\ |k_2|\geq |k_3|/2 }} + \sum_{\substack{k_1,k_2,k_3\in \Z^{\nu} \\ |k_3|> 2|k_2| }} =: (I)+(II).
\end{split}
\end{equation*}
We bound the above two terms separately.

We begin with $(I)$.  For $k_1,k_2,k_3\in \Z^{\nu}$, define a vector valued operator
$$\sT_{k_1,k_2,k_3}^1 := 2^{-k_3\cdot \delta' + (k_1-k_3)\cdot \delta} \sR_{k_1,k_2}^1.$$
Note that \eqref{EqnPfRadonMainVV} implies for $|k_2|\geq |k_3|/2$, using that $|\delta'|, |\delta|<\epsilon_p/8$,
\begin{equation*}
\begin{split}
&\LplqnuOpN{p}{2}{\sT_{k_1,k_2,k_3}^1} \lesssim 2^{-k_3\cdot \delta' + (k_1-k_3)\cdot \delta - \epsilon_p (|k_1|+|k_2|)}
\\&\leq 2^{|k_3||\delta'| + |k_1-k_3| |\delta| - \epsilon_p (|k_1|+k_2|)}
\leq 2^{|k_2|\epsilon_p/2 + |k_1|\epsilon_p/8 -   \epsilon_p (|k_1|+k_2|)} 
\\&\leq 2^{-(|k_1|+|k_2|)\epsilon_p/2} \leq 2^{-(|k_1|+|k_2|+|k_3|)\epsilon_p/8}.
\end{split}
\end{equation*}
Using this, we have
\begin{equation*}
\begin{split}
(I)&=\sum_{\substack{k_1,k_2,k_3\in \Z^{\nu} \\ |k_2|\geq |k_3|/2 }}\LpN{p}{\q(\sum_{j\in \N^{\nu}} \q| 2^{j\cdot \delta' + (j+k_1)\cdot \delta} D_j T_{j+k_1} D_{j+k_2} D_{j+k_3} f \w|^2\w)^{\frac{1}{2}} }
\\&= \sum_{\substack{k_1,k_2,k_3\in \Z^{\nu} \\ |k_2|\geq |k_3|/2 }} \LplqnuN{p}{2}{\sT_{k_1,k_2,k_3}^1\q\{ 2^{(j+k_3)\cdot (\delta'+\delta) }  D_{j+k_3} f\w\}_{j\in \N^{\nu}} }
\\&\lesssim \sum_{\substack{k_1,k_2,k_3\in \Z^{\nu} \\ |k_2|\geq |k_3|/2 }} 2^{-(|k_1|+|k_2|+|k_3|)\epsilon_p/8} \LplqnuN{p}{2}{ \q\{2^{j\cdot (\delta'+\delta) }  D_{j} f\w\}_{j\in \N^{\nu}}}
\\&\lesssim \NLpN{p}{\delta'+\delta}{\fD}{f},
\end{split}
\end{equation*}
as desired.

We now turn to $(II)$.  For $k_1,k_2,k_3\in \Z^{\nu}$ define
\begin{equation*}
\sT_{k_1,k_2,k_3}^2 := 2^{-k_2\cdot \delta' +(k_1-k_2)\cdot \delta - (3\epsilon_p/4) |k_3-k_2| } \sR_{k_1}^2.
\end{equation*}
\eqref{EqnPfRadonMainVV} implies for $|k_3|>2|k_2|$, using that $|\delta'|,|\delta|<\epsilon_p/8$,
\begin{equation*}
\begin{split}
&\LplqnuOpN{p}{2}{\sT_{k_1,k_2,k_3}^2}\lesssim 2^{-k_2\cdot \delta' +(k_1-k_2)\cdot \delta - (3\epsilon_p/4) |k_3-k_2| -\epsilon_p|k_1|}
\\&\leq 2^{ |k_2| (|\delta'|+|\delta|) + |k_1||\delta| - (3\epsilon_p/8) |k_3|- \epsilon_p|k_1|} 
\leq 2^{|k_3| (\epsilon_p/8) +|k_1|(\epsilon_p/8) - (3\epsilon_p/8)|k_3|-\epsilon_p|k_1|}
\\&\leq 2^{-(\epsilon_p/4) (|k_1|+|k_3|)}
\leq 2^{-(\epsilon_p/8)(|k_1|+|k_2|+|k_3|)}.
\end{split}
\end{equation*}
Using this, we have
\begin{equation*}
\begin{split}
&(II)= \sum_{\substack{k_1,k_2,k_3\in \Z^{\nu} \\ |k_3|> 2|k_2| }}\LpN{p}{\q(\sum_{j\in \N^{\nu}} \q| 2^{j\cdot \delta' + (j+k_1)\cdot \delta} D_j T_{j+k_1} D_{j+k_2} D_{j+k_3} f \w|^2\w)^{\frac{1}{2}} }
\\&= \sum_{\substack{k_1,k_2,k_3\in \Z^{\nu} \\ |k_3|> 2|k_2| }} \LplqnuN{p}{2}{\sT_{k_1,k_2,k_3}^2 \q\{ 2^{(j+k_2)\cdot (\delta+\delta') + 3\frac{\epsilon_p}{4}|k_3-k_2| } D_{j+k_2} D_{j+k_3} f \w\}_{j\in \N^{\nu} }}
\\&\lesssim \sum_{\substack{k_1,k_2,k_3\in \Z^{\nu} \\ |k_3|> 2|k_2| }} 2^{-\frac{\epsilon_p}{8}(|k_1|+|k_2|+|k_3|)} \LpN{p}{\q(\sum_{j\in \N^{\nu}} \q|2^{(j+k_2)\cdot (\delta+\delta') + 3\frac{\epsilon_p}{4}|k_3-k_2|} D_{j+k_2} D_{j+k_3} f \w|^2 \w)^{\frac{1}{2}} } 
\\&\leq \sum_{k_1,k_2,k_3\in \Z^{\nu} } 2^{-\frac{\epsilon_p}{8}(|k_1|+|k_2|+|k_3|)} \LpN{p}{\q(\sum_{j\in \N^{\nu}} \q|2^{j\cdot (\delta+\delta') + 3\frac{\epsilon_p}{4}|k_3-k_2|} D_{j} D_{j+k_3-k_2} f \w|^2 \w)^{\frac{1}{2}} } 
\\&\leq \sum_{k_1,l_1,l_2\in \Z^{\nu}} 2^{-\frac{\epsilon_p}{16} (|k_1|+|l_1|)} \LpN{p}{\q(\sum_{j\in \N^{\nu}} \q|2^{j\cdot (\delta+\delta') + 3\frac{\epsilon_p}{4}|l_2|} D_{j} D_{j+l_2} f \w|^2 \w)^{\frac{1}{2}} } 
\\&\lesssim \NLpN{p}{\delta+\delta'}{\fD}{f},
\end{split}
\end{equation*}
where the last line follows by \eqref{EqnPfRadonMainTwoToOne}.
Combining the above estimates proves \eqref{EqnPfRadonMainToShow}, completing the proof in Case I.

We turn to the case when $(\gamma,e,N)$ is linearly finitely generated, we note that in the above we proved \eqref{EqnPfRadonMainToShow} for $|\delta|,|\delta'|\leq\epsilon_p/8$, where $\epsilon_p$ was
so that \eqref{EqnPfRadonMainVV} and \eqref{EqnPfRadonMainTwoToOne} held.
But when $(\gamma,e,N)$ is linearly finitely generated, Case II of Proposition \ref{PropPfSobVVOps} and Case II of Lemma \ref{LemmaPfSobTwoToOne}
show that \eqref{EqnPfRadonMainVV} and \eqref{EqnPfRadonMainTwoToOne} hold for all $\epsilon_p\in (0,\infty)$,
and therefore the above proof shows \eqref{EqnPfRadonMainToShow} holds for all $\delta,\delta'\in \R^{\nu}$, completing the proof.
\end{proof}

\begin{proof}[Proof of Proposition \ref{PropResRadonOtherAddingParamsNew}]
Let $T$ be a fractional Radon transform of order $\deltat\in \R^{\nut}$ corresponding to $(\gammat,\et,\Nt)$ on $B^{\Nt}(a)$, as in the statement of the proposition.
I.e., there exist $\psi_1,\psi_2\in C_0^\infty(\Omega_0)$, $\kappa(t,x)\in C^\infty(B^{\Nt}(a)\times \Omega'')$, and $\Kt\in \sK_{\deltat}(\Nt,\et,a)$ such that
\begin{equation*}
Tf(x) = \psi_1(x) \int f(\gammat_{\vtt}(x)) \psi_2(\gammat_{\vtt}(x)) \kappa(\vtt,x) \Kt(\vtt)\: d\vtt.
\end{equation*}
Let $\etat\in C_0^\infty(B^{\Nt}(a))$ and $\q\{\vsigt_{\jt} : \jt\in \N^{\nut}\w\}\subset \schS(\R^{\Nt})$ be a bounded set with $\vsigt_{\jt}\in \schS_{\{\mu : \jt_{\mu}\ne 0\}}$
and such that $\Kt(\vtt)=\etat(\vtt) \sum_{\jt\in \N^{\nut}} \dil{\vsigt}{2^{\jt}}(\vtt)$.
Because $\supp{\etat}\Subset B^{\Nt}(a)$, we may pick $\at\in (0,a)$ so that $\supp{\etat}\Subset B^{\Nt}(\at)$.

Let $\sSh\subseteq \snuhvectone$ be as in the statement of the proposition.  We know that $\sL(\sSh)$ is finitely generated (resp. linearly finitely generated)
by $\sFh\subset \snuhvect$ (resp. $\sFh\subset \snuhvectone$) on $\Omega'$ in Case I (resp. in Case II).
Enumerate the vector fields in $\sFh:=\{ (\Xh_1,\hd_1),\ldots, (\Xh_{\qh},\hd_{\qh})\}$.  Define $\nuh$ parameter dilations on $\R^{\qh}$
by setting $\rh (\vht_1,\ldots, \vht_{\qh}):=(\rh^{\hd_1} \vht_1,\ldots, \rh^{\hd_{\qh}} \vht_{\qh})$ for $\rh\in \R^{\nuh}$.
Denote these $\nuh$-parameter dilations on $\R^{\qh}$ by $\hd$.

Let $\Nt'=\qh+\Nt$, and define $\nu$-parameter dilations on $\R^{\Nt'}$ by for $t=(\vht, \vtt)\in \R^{\qh}\times \R^{\Nt}=\R^{\Nt'}$ and $r=(\rh,\rt)\in \R^{\nuh}\times\R^{\nut}=\R^{\nu}$,
\begin{equation*}
r (\vht,\vtt) := (\rh \vht,\rt\vtt),
\end{equation*}
where $\rh\vht$ is defined by the above $\nuh$ parameter dilations on $\R^{\qh}$, and $\rt\vtt$ is defined by the given $\nut$-parameter dilations $\et$ on 
$\R^{\Nt}$.  Denote these $\nu$ parameter dilations on $\R^{\Nt'}$ by $\et'$.

Let $\ah\in (0,\frac{a-\at}{2})$ be a small number, to be chosen later.
Let $\delta_0(\vht)$ denote the Dirac $\delta$ function at $0$ in the $\vht$ variable.  By Lemma \ref{LemmaResKerDelta},
$\delta_0(\vht)\in \sK_0(\qh, \hd,\ah)$.  Take $\etah(\vht)\in C_0^\infty(B^{\Nh}(\ah))$ and $\{\vsigh_{\jh} : \jh\in \N^{\nuh}\}\subset \schS(\R^{\Nh})$
a bounded set with $\vsigh_{\jh}\in \schS_{\{ \mu : \jh_\mu\ne 0\}}$ and 
$$\delta_0(\vht) = \etah(\vht)\sum_{\jh\in \N^{\nuh}} \dil{\vsigh_{\jh}}{2^{\jh}}(\vht),$$
where $\dil{\vsigh_{\jh}}{2^{\jh}}$ is defined by the dilations $\hd$, via \eqref{EqnResKerDefFuncDil}.

For $j=(\jh,\jt)\in \N^{\nuh}\times \N^{\nut}$, let $\vsig_j(\vht,\vtt):=\vsigh_{\jh}(\vht) \vsigt_{\jt}(\vtt)$.  Note that
$\dil{\vsig_j}{2^j}(\vht,\vtt) = \dil{\vsigh_{\jh}}{2^{\jh}}(\vht)\dil{\vsigt_{\jt}}{2^{\jt}}(\vtt)$, where $\dil{\vsig_j}{2^j}$ is defined via the $\nu$-parameter dilations
on $\et'$ on $\R^{\Nt'}$.
Thus, we have
\begin{equation*}
\delta_0(\vht)\otimes \Kt(\vtt)  = \etah(\vht)\etat(\vtt)\sum_{j\in \N^{\nu}} \dil{\vsig_j}{2^j}(\vht,\vtt).
\end{equation*}
Because $\vsig_j\in \schS_{\{\mu : j_\mu\ne 0\}}$, we see $K(\vht,\vtt):=\delta_0(\vht)\otimes \Kt(\vtt) \in \sK_{(0,\deltat)}(\Nt',\et',a)$.

Let $(\gammat,\et,\Nt,\Omega,\Omega''')$ correspond to the vector field parameterization $(\Wt,\et,\Nt,\Omega'')$, where $\Omega'\Subset \Omega''\Subset \Omega'''$.
Define a new vector field parameterization:
\begin{equation*}
W(\vht,\vtt,x):= \Wt(\vtt,x) + \sum_{l=1}^{\qh} \vht_l \Xh_l.
\end{equation*}
Let $(\gammat', \et', \Nt')$ denote the parameterization corresponding to $(W, \et', \Nt')$.
Because $\gammat_{0,\vtt}'(x)=\gammat_{\vtt}(x)$, standard existence theorems from ODEs show that $\gammat_{\vht,\vtt}'(x)$ is defined for $|\vht|<\ah$, $|\vtt|<\at$, provided 
$\ah$ is chosen sufficiently small.  Note that if we define $\sS\subset \snuvectone$ in terms of $(W,\et,\Nt')$ by \eqref{EqnResSurfDefnsS}, then $\sS$
is exactly given by \eqref{EqnAssumptionResRadonOtherMainsSNew}.  It follows from the assumptions that $(\gammat',\et',\Nt')$ is finitely generated in Case I
(linearly finitely generated in Case II) by $\sF$ on $\Omega'$.

Finally, we have
\begin{equation*}
\begin{split}
Tf(x) &= \psi_1(x) \int f(\gammat_{\vtt}(x)) \psi_2(\gammat_{\vtt}(x)) \kappa(\vtt,x) \Kt(\vtt)\: d\vtt
\\&=\psi_1(x) \int f(\gammat_{\vht,\vtt}'(x)) \psi_2(\gammat_{\vht,\vtt}'(x)) \kappa(\vtt,x) K(\vht,\vtt)\: d\vht\: d\vtt.
\end{split}
\end{equation*}
This shows that $T$ is a fractional radon transform of order $(0_{\nuh},\deltat)\in \R^{\nu}$ corresponding to $(\gammat',\et',\Nt')$ on $B^{\Nt'}(a)$,
which competes the proof.
\end{proof}

%% file: optimal.tex
In this section we present results concerning optimality.
We focus on the single-parameter case, and in fact discuss only the optimality of the result in Corollary \ref{CorResRadonOtherHor3Main} (and as a special case
we obtain the optimality of Corollary \ref{CorResRadonOtherHorWithEuclidNewer}).\footnote{The methods here apply in some cases to the multi-parameter situation, but we were unable to formulate a short statement of a general result in the multi-parameter case.
We have therefore presented the single-parameter case, and leave any generalizations to the interested reader.}
Fix open sets $\Omega'\Subset \Omega''\Subset \Omega'''\Subset \Omega\subseteq \R^n$. 

\begin{defn}
Let $(\Xb,\bd)=\{(\Xb_1,\bd_1),\ldots, (\Xb_q,\bd_q)\}\subset \sonevect$ be a finite set and let $\lambda>0$.  Suppose
$\sSc\subset \sonevect$.
We say $\sFb$ \textbf{sharply} $\lambda$\textbf{-controls} $\sSc$ \textbf{on} $\Omega'$ if $\sFb$ $\lambda$-controls $\sSc$ on $\Omega'$
and there exists a set $\Omega_0\Subset \Omega'$, $\tau_0>0$ such that for all $0<\tau_1\leq \tau_0$ the following holds.
There exists $(\Xc,\cd)\in \sSc$, $m\in \N$, and a sequence $\delta_k\in (0,1]$ with $\delta_k\rightarrow 0$ such that
for each $x\in \Omega_0$ we have\footnote{It is always possible to write $\delta_k^{\lambda\cd} \Xc$ as in \eqref{EqnOptimalSharpSum} because
$\sFb$ $\lambda$-controls $\sSc$ on $\Omega'$.}
\begin{equation}\label{EqnOptimalSharpSum}
\delta_k^{\lambda \cd} \Xc = \sum_{j=1}^q \delta_k^{\bd_j} c_{x,j}^k \Xb_j\text{ on }\B{\Xb}{\bd}{x}{\tau_1 \delta_k}
\end{equation}
and such that
\begin{equation*}
\liminf_{k\rightarrow \infty} \sup_{x\in \Omega_0} \inf \sum_{|\alpha|\leq m} \sum_{j=1}^q \CjN{ (\delta_k \Xb)^{\alpha} c_{x,j}^k}{0}{\B{\Xb}{\bd}{x}{\tau_1 \delta_k}} >0.
\end{equation*}
Here, $\delta_k \Xb=(\delta^{\bd_1} \Xb_1,\ldots, \delta^{\bd_q} \Xb_q)$ and the infimum is taken over all representations of the form \eqref{EqnOptimalSharpSum}.
\end{defn}

\begin{defn}\label{DefnOptimalSharpControl}
Let $\sSb\subseteq \sonevect$ and let $\lambda>0$ such that $\sSb$ is finitely generated by some $\sFb\subset \sonevect$ on $\Omega'$,
and let $\lambda>0$.  Suppose $\sSc\subset \sonevect$.  We say
$\sSb$ \textbf{sharply} $\lambda$\textbf{-controls} $\sSc$ \textbf{on} $\Omega'$ if $\sFb$ sharply $\lambda$-controls $\sSc$ on $\Omega'$.
\end{defn}

\begin{rmk}
Note that Definition \ref{DefnOptimalSharpControl} is independent of the choice of $\sFb$.
\end{rmk}

\begin{rmk}
It follows immediately from the definitions that if $\sSc$ and $\sSc'$ are equivalent on $\Omega'$, then $\sSb$ sharply $\lambda$-controls
$\sSc$ on $\Omega'$ if and only if $\sSb$ sharply $\lambda$-controls $\sSc'$ on $\Omega'$.
\end{rmk}

We now present the main theorem of the section.
\begin{thm}\label{ThmOptimalMain}
Suppose $(\gammat,\et, \Nt,\Omega,\Omega''')$ and $(\gammah,\eh,\Nh,\Omega,\Omega''')$ are parameterizations with single-parameter dilations $\et$ and $\eh$, and let
 $(\Wt,\et,\Nt)$ and $(\Wh,\eh,\Nh)$ be the corresponding vector field parameterizations.  
 Expand $\Wt(\vtt)$ and $\Wh(\vht)$ as Taylor series in the $\vtt$ and $\vht$ variables:
\begin{equation*}
\Wt(\vtt)\sim \sum_{|\alphat|>0} \vtt^{\alphat} \Xt_{\alphat}, \quad \Wh(\vht)\sim \sum_{|\alphah|>0} \vht^{\alphah} \Xh_{\alphah}.
\end{equation*}
We suppose both $\{\Xt_{\alphat} :|\alphat|>0\}$ and $\{\Xh_{\alphah}:|\alphah|>0\}$ satisfy H\"ormander's condition
on $\Omega''$.\footnote{I.e., we are assuming that $\gammat$ and $\gammah$ satisfy the curvature condition from \cite{ChristNagelSteinWaingerSingularAndMaximalRadonTransforms}.}  
Then, by Corollary \ref{CorResSurfHorSmoothlyFG}, $(\gammat,\et,\Nt)$ and $(\gammah,\eh,\Nh)$ are finitely generated
by some $\sFt\subset \sonevect$ and $\sFh\subset \sonevect$ on $\Omega'$, respectively.  
There exists $a>0$ such that the following holds.
\begin{itemize}
\item Suppose $\sFt$ sharply $\lambda$-controls $\sFh$ on $\Omega'$, for some $\lambda>0$.
Then, for $1<p<\infty$, there exists $\epsilon=\epsilon(p,(\gammat,\et,\Nt), (\gammah,\eh,\Nh),\lambda)>0$ such that
for all $\deltah\in [0,\epsilon)$, $\delta\in(-\epsilon,\epsilon)$, and every fractional Radon transform $T$ of order $-\lambda\deltah$ corresponding to $(\gammat,\et,\Nt)$ on $B^{\Nt}(a)$,
\begin{equation}\label{EqnOptimalKnownBound1}
\NLpN{p}{\delta}{(\gammah,\eh,\Nh)}{T f} \lesssim \NLpN{p}{\delta -\deltah}{(\gammah,\eh,\Nh)}{f}, \quad f\in C_0^\infty(\Omega').
\end{equation}
Furthermore, this is optimal in the sense that there do not exist $p\in (1,\infty)$, $r>0$, $\deltah\in [0,\epsilon)$,\footnote{Recall, $\epsilon$ depends on $p\in (1,\infty)$.} $\delta\in(-\epsilon,\epsilon)$
such that for every fractional Radon transform $T$ of order $-\lambda\deltah$ corresponding to $(\gammat,\et,\Nt)$ on $B^{\Nt}(a)$ we have
\begin{equation}\label{EqnOptimalFalseBound1}
\NLpN{p}{\delta+r}{(\gammah,\eh,\Nh)}{T f} \lesssim \NLpN{p}{\delta -\deltah}{(\gammah,\eh,\Nh)}{f}, \quad f\in C_0^\infty(\Omega').
\end{equation}

\item  Suppose $\sFh$ sharply $\lambda$-controls $\sFt$ on $\Omega'$, for some $\lambda>0$.  Then,
for $1<p<\infty$, there exists $\epsilon=\epsilon(p, (\gammat,\et, \Nt),(\gammah,\eh,\Nh),\lambda)>0$
such that for all $\deltat\in [0,\epsilon)$, $\delta\in (-\epsilon,\epsilon)$, and every fractional Radon transform $T$ of order $\deltat$
corresponding to $(\gammat,\et,\Nt)$ on $B^{\Nt}(a)$,
\begin{equation}\label{EqnOptimalKnownBound2}
\NLpN{p}{\delta}{(\gammah,\eh,\Nh)}{Tf}\lesssim \NLpN{p}{\delta+\lambda \deltat}{(\gammah,\eh,\Nh)}{f},\quad f\in C_0^\infty(\Omega').
\end{equation}
Furthermore, this is optimal in the sense that there do not exist $p\in (1,\infty)$, $r>0$, $\deltat\in [0,\epsilon)$, $\delta\in (-\epsilon,\epsilon)$
such that for every fractional Radon transform $T$ of order $\deltat$ corresponding to $(\gammat,\et,\Nt)$ on $B^{\Nt}(a)$ we have
\begin{equation}\label{EqnOptimalFalseBound2}
\NLpN{p}{\delta+r}{(\gammah,\eh,\Nh)}{Tf}\lesssim \NLpN{p}{\delta+\lambda \deltat}{(\gammah,\eh,\Nh)}{f},\quad f\in C_0^\infty(\Omega').
\end{equation}
\end{itemize}
\end{thm}

Before we prove Theorem \ref{ThmOptimalMain} we need to further study the notion of sharp $\lambda$-control.  Suppose $\sSb\subset \sonevect$ is such that $\sL(\sSb)$ is finitely generated by $\sFb=\{(\Xb_1,\bd_1),\ldots, (\Xb_q,\bd_q) \}\subset \sonevect$ on $\Omega'$.
We also assume that for all $x\in \Omega'$, $\dim \Span{ X(x) :  \exists d, (X,d)\in \sL(\sSb)}=n$.
By Lemma \ref{LemmaResCCsLFinGenGivessD}, $\sFb$ satisfies $\sD(\Omega')$.
By Proposition \ref{PropFrobControlUnitScaleVsAll}, the vector fields $\sZ_\delta:=\q\{ \delta^{\bd} \Xb : (\Xb,\bd)\in \sFb\w\}$
satisfy the conditions of Theorem \ref{ThmQuantFrob}, uniformly for $\delta\in [0,1]$, $x_0\in \Omega'$.
Thus Theorem \ref{ThmQuantFrob} applies to give $\eta>0$ and a map
$\Phi_{x_0,\delta}: B^{n}(\eta)\rightarrow B_{\sZ_\delta}(x_0,\xi_2)$ (for $x_0\in \Omega'$, $\delta\in[0,1]$)
satisfying the conclusions of that theorem (uniformly in $x_0$ and $\delta$) with $Z$ replaced by $\sZ_\delta$. 
The next lemma helps to elucidate the notion of sharp $\lambda$-control in this setting.

\begin{lemma}\label{LemmaOptimalSubseqVect}
Let $\sSb$ be as above so that we have the maps $\Phi_{x_0,\delta}$ for $x_0\in \Omega'$, $\delta\in [0,1]$.
Let $\sSc\subset \sonevect$ and suppose $\sL(\sSb)$ sharply $\lambda$-controls $\sL(\sSc)$ on $\Omega'$.
Fix $\eta_0>0$.
Then, there exists $\Omega_0\Subset \Omega'$ and sequences $x_k\in \Omega_0$, $\delta_k\in [0,1]$ with $\delta_k\rightarrow 0$,
and vector fields with  formal degrees $(\Xc_0,\cd_0)\in \sSc$, $(\Xb_0,\bd_0)\in \sSb$, such that the following holds.
Let $\Yc_k$ be the pullback of $\delta_k^{\lambda \cd_0} \Xc_0$ via $\Phi_{x_k,\delta_k}$ to $B^n(\eta)$
and let $\Yb_k$ be the pullback of $\delta_k^{\bd_0} \Xb_0$ via $\Phi_{x_k,\delta_k}$ to $B^{n}(\eta)$.
Then $\Yc_k$ and $\Yb_k$ converge in $C^\infty(B^n(\eta))$, $\Yc_k\rightarrow \Yc_\infty$ and $\Yb_k\rightarrow \Yb_\infty$.
Furthermore, there is a nonempty open set $U\subseteq B^n(\eta_0)$ such that for all $u\in U$, $\Yb_\infty(u)\ne 0$ and $\Yc_\infty(u)\ne 0$.
\end{lemma}
\begin{proof}
Pick $\tau_1>0$ so small 
$B_{\sFb}(x,{\tau_1\delta})\subset \Phi_{x,\delta}(B^{n}(\eta_0))$ 
for $x\in \Omega'$, $\delta\in [0,1]$.
Let $\Omega_0\Subset \Omega'$, $\delta_k\in [0,1]$, $x_k\in \Omega_0$, $m\in \N$, and $(\Xc,\cd)\in \sL(\sSc)$ be as in the definition of sharp $\lambda$-control, with this choice of 
$\tau_1$.
Let $\Yc_k'$ be the pullback of $\delta_k^{\lambda \cd} \Xc$ via $\Phi_{x_k,\delta_k}$ to $B^n(\eta)$.
Because $(\Xb,\bd)$ $\lambda$-controls $\sSc$ on $\Omega'$, it follows from \eqref{EqnFrobYsSmooth} and
\eqref{EqnFrobYsSpan} that $\CjN{\Yc_k'}{L}{B^{n}(\eta_2)}\lesssim 1$, for every $L\in \N$ with implicit constant depending on $L$.
By the definition of sharp $\lambda$-control and Proposition \ref{PropFrobControlEquivsP} and Remark \ref{RmkExtraPropFrobControlEquivsP}, we have
that  $\CjN{\Yc_k'}{m}{B^n(\eta_0)}\gtrsim 1$.
Replacing $\delta_k$ and $x_k$ with a subsequence shows that $\Yc_k'\rightarrow \Yc_\infty'$ in $C^{\infty}$ with $\Yc_\infty'$
not the zero vector field on $B^{n}(\eta_0)$.
For $(\Xc,\cd)\in \sL(\sSc)$, $\Xc$ can be written as an iterated commutator of vector fields $\Xc_{1},\ldots, \Xc_{r}$,
where $(\Xc_{1},\cd_{1}),\ldots, (\Xc_{r},\cd_{r})\in \sSc$ with $\cd_{1}+\cdots+\cd_{r}=\cd$.
Thus, if $\Yc_{k,s}$ is the pullback of $\delta_k^{\lambda \cd_{s}} \Xc_{s}$ via $\Phi_{x_k,\delta_k}$ to $B^n(\eta)$,
we have that $\Yc_{k}'$ can be written as an iterated commutator of $\Yc_{k,1},\ldots, \Yc_{k,r}$.
Because $\Yc_{k}'\rightarrow  \Yc_{\infty}'$ where $\Yc_{\infty}'$ is not the zero vector field, and because
$\Yc_{k,1},\ldots, \Yc_{k,r}$ are uniformly in $C^\infty$ (Proposition \ref{PropFrobControlEquivsP})
we must that that $\Yc_{k,s}$ does not tend to the zero vector field on $B^n(\eta_0)$ in $C^\infty$ for some $s$.
Moving to a subsequence, we see $\Yc_{k,s}\rightarrow \Yc_{\infty,s}$ in $C^\infty$, where $\Yc_{\infty,s}$ is not the zero
vector field on $B^{n}(\eta_0)$.
Because $\Yc_{\infty,s}$ is smooth, there is a nonempty open set $U\subseteq B^n(\eta_0)$ on which $\Yc_{\infty,s}$ is nonzero.
  This completes the proof for $\sSc$ with $(\Xc_0,\cd_0)=(\Xc_{s}, \cd_{s})\in \sSc$.

Fix $u_0\in B^{n}(\eta_0)$ so that $\Yc_{\infty,s}(u_0)\ne 0$.
Let $\Yb_{k,1},\ldots, \Yb_{k,q}$ be the pullbacks of $\delta^{\bd_1}\Xb_1,\ldots, \delta^{\bd_q} \Xb_q$ via $\Phi_{x_k,\delta_k}$ to 
$B^{n}(\eta)$.  Combining \eqref{EqnFrobYsSmooth} and \eqref{EqnFrobYsSpan}, we see that if we move to a subsequence,
there exists $l\in \{1,\ldots, q\}$ such that $\Yb_{k,l}\rightarrow \Yb_{\infty,l}$ in $C^\infty$ with $\Yb_{\infty,l}(u_0)\ne 0$.
Because $(\Xb_l,\bd_0)\in \sL(\sSb)$, $\Xb$ can be written as an iterated commutator of vector fields $\Xb_{l,1},\ldots, \Xb_{l,r}$,
where $(\Xb_{l,1},\bd_{l,1}),\ldots, (\Xb_{l,r},\bd_{l,r})\in \sSb$ with $\bd_{l,1}+\cdots+\bd_{l,r}=\bd_l$.
Thus, if $\Yb_{k,l,s}$ is the pullback of $\delta_k^{\bd_{s}} \Xb_{l,s}$ via $\Phi_{x_k,\delta_k}$ to $B^n(\eta)$,
we have that $\Yb_{k,l}$ can be written as an iterated commutator of $\Yb_{k,l,1},\ldots, \Yb_{k,l,r}$.
Because $\Yb_{k,l}\rightarrow  \Yb_{\infty,l}$ where $\Yb_{\infty,l}(u_0)\ne 0$, 
and because
$\Yb_{k,l,1},\ldots, \Yb_{k,l,r}$ are uniformly in $C^\infty$ (Proposition \ref{PropFrobControlEquivsP})
we must that 
there exists an $s$ such that 
that $\Yb_{k,l,s}(u_0)$ does not tend to the zero for some $s$.
By moving to a subseqence, we have $\Yb_{k,l,s}$ converges in $C^{\infty}$ to some
vector field $\Yb_{\infty,l,s}$ with 
$\Yb_{\infty,l,s}(u_0)\ne 0$.  
We take $(\Xb_0,\bd_0)=(\Xb_{l,s},\bd_{l,s})$.
Because both $\Yb_{\infty,l,s}(u_0)\ne 0$  and $\Yc_{\infty,s}(u_0)\ne 0$, and both vector fields are smooth, 
they are both nonzero on some open set.  This completes the proof.
%
\end{proof}

\begin{rmk}
The same proof as in Lemma \ref{LemmaOptimalSubseqVect} can be used to show various facts about sharp control.  For instance,
it can be used to show the following.
\begin{itemize}
\item Let $\sSb\subset \sonevect$ be such that $\sL(\sSb)$ is finitely generated on $\Omega'$.  Let $\sSc\subset \sonevect$.
Then $\sL(\sSb)$ sharply $\lambda$-controls $\sL(\sSc)$ on $\Omega'$ if and only if $\sL(\sSb)$ sharply $\lambda$-controls $\sSc$ on $\Omega'$.
\end{itemize}
\end{rmk}

Now suppose we have two parameterizations with single parameter dilations $(\gammab,\eb,\Nb,\Omega,\Omega''')$ and
$(\gammac, \ec, \Nc, \Omega,\Omega''')$.  We suppose $(\gammab,\eb,\Nb)$ is finitely generated by $\sFb$ on $\Omega'$
and $(\gammac,\ec,\Nc)$ is finitely generated by $\sFc$ on $\Omega'$.  Finally, we suppose $\sFb$ sharply $\lambda$-controls $\sFc$ on $\Omega'$, for some
$\lambda>0$.

As before, Lemma \ref{LemmaResCCsLFinGenGivessD} shows $\sFb$ satisfies $\sD(\Omega')$ and
Proposition \ref{PropFrobControlUnitScaleVsAll} shows the vector fields $\sZ_\delta:=\q\{ \delta^{\bd} \Xb : (\Xb,\bd)\in \sFb\w\}$
satisfy the conditions of Theorem \ref{ThmQuantFrob}, uniformly for $\delta\in [0,1]$, $x_0\in \Omega'$.
Theorem \ref{ThmQuantFrob} applies to give $\eta>0$ and a map
$\Phi_{x_0,\delta}: B^{n}(\eta)\rightarrow B_{\sZ_\delta}(x_0,\xi_2)$ (for $x_0\in \Omega'$, $\delta\in[0,1]$)
satisfying the conclusions of that theorem (uniformly in $x_0$ and $\delta$) with $Z$ replaced by $\sZ_\delta$.
For each $x\in \Omega'$ and $\delta\in [0,1]$ define
\begin{equation*}
\thetab_{\tb}^{x,\delta}(u) := \Phi_{x,\delta}^{-1} \circ \gammab_{\delta \tb}\circ \Phi_{x,\delta}(u), \quad \thetac_{\tc}^{x,\delta}(u):=\Phi_{x,\delta}^{-1}\circ \gammac_{\delta^{\lambda} \tc}\circ \Phi_{x,\delta}(u),
\end{equation*}
where $\delta \tb$ is defined by the single-parameter dilations $\eb$ and $\delta^{\lambda} \tc$ is defined by the single-parameter dilations $\ec$.

\begin{lemma}\label{LemmaOptimalNonzeroParams}
Under the above hypotheses, there exists $\Omega_0\Subset \Omega'$, $\eta_0>0$, $a>0$, and sequences $x_k\in \Omega_0$ and $\delta_k\in (0,1]$ with $\delta_k\rightarrow 0$
such that the following holds.  For $x\in \Omega'$, $\delta\in (0,1]$, $\thetab^{x,\delta}_{\tb}(u)\in C^{\infty}\q(B^{\Nb}(a)\times B^n(\eta_0)\w)$ (where $\tb\in B^{\Nb}(a)$, $u\in B^n(\eta_0)$) and $\thetac^{x,\delta}_{\tc}(u)\in C^{\infty}\q(B^{\Nc}(a)\times B^{n}(\eta_0)\w)$ (where $\tc\in B^{\Nc}(a)$, $u\in B^n(\eta_0)$).  Also
$\thetab^{x_k,\delta_k}\rightarrow \thetab^{\infty}$ in  $C^{\infty}\q(B^{\Nb}(a)\times B^n(\eta_0)\w)$ and $\thetac^{x_k,\delta_k}\rightarrow \thetac^{\infty}$ in  $C^{\infty}\q(B^{\Nc}(a)\times B^n(\eta_0)\w)$, and there exists an open set $U\subseteq B^{n}(\eta)$ such that
for $u\in U$ neither $\thetab^{\infty}_{\tb}(u)$ nor $\thetac^{\infty}_{\tc}(u)$ are constant in $\tb$ or $\tc$, respectively, on any neighborhood of $0$ in the $\tb$
or $\tc$ variable, respectively.
\end{lemma}
\begin{proof}
That $\{ \thetac^{x,\delta} : x\in \Omega', \delta\in [0,1]\}\subset C^{\infty}\q(B^{\Nc}(a)\times B^{n}(\eta_0)\w)$
and $\{\thetab^{x,\delta} : x\in \Omega', \delta\in [0,1]\}\subset C^{\infty}\q(B^{\Nb}(a)\times B^n(\eta_0)\w)$ are bounded sets
(for some $\eta_0>0$) follows from Proposition \ref{PropFrobControlEquivsQ}.
Let $(\Wb,\eb,\Nb)$ and $(\Wc,\ec,\Nc)$ be the vector field parameterizations corresponding to the parameterizations $(\gammab,\eb,\Nb)$
and $(\gammac,\ec,\Nc)$, respectively.  We expand $\Wb(\tb)$ and $\Wc(\tc)$ into Taylor series:
\begin{equation*}
\Wb(\tb)\sim \sum_{|\alphab|>0} \tb^{\alphab} \Xb_\alpha, \quad \Wc(\tc)\sim \sum_{|\alphac|>0} \tc^{\alphac} \Xc_\alpha.
\end{equation*}
Define
\begin{equation*}
\sSb:=\{(\Xb_{\alphab}, \deg(\alphab)) : |\alphab|>0 \}, \quad \sSc:=\{ (\Xc_{\alphac},\deg(\alphac)) : |\alphac|>0 \},
\end{equation*}
where $\deg(\alphab)$ and $\deg(\alphac)$ are defined using the single-parameter dilations $\eb$ and $\ec$, respectively; see Definition \ref{DefnResKerDegree}.
Note that, by our assumptions $\sL(\sSb)$ is finitely generated by $\sFb$ on $\Omega'$ and $\sL(\sSc)$ is finitely generated by $\sFc$ on $\Omega'$.
Because $\sFb$ sharply $\lambda$-controls $\sFc$ on $\Omega'$, we have $\sL(\sSb)$ sharply $\lambda$-controls $\sL(\sSc)$ on $\Omega'$.

Define the vector fields, for $x\in \Omega', \delta\in [0,1]$,
\begin{equation*}
\Vb_{x,\delta}(\tb,u):= \frac{\partial}{\partial \epsilon}\bigg|_{\epsilon=1} \thetab_{\epsilon\tb}^{x,\delta}\circ\q( \thetab_{\tb}^{x,\delta} \w)^{-1}(u), \quad 
\Vc_{x,\delta}(\tc,u):= \frac{\partial}{\partial \epsilon}\bigg|_{\epsilon=1} \thetac_{\epsilon\tc}^{x,\delta}\circ\q( \thetac_{\tc}^{x,\delta} \w)^{-1}(u),
\end{equation*}
where $\epsilon \tb$ and $\epsilon \tc$ are defined using standard multiplication, and do not reference the single parameter dilations $\eb$ and $\ec$.
Note that $\Vb_{x,\delta}(\tb)$ and $\Vc_{x,\delta}(\tc)$ are the pullbacks, via $\Phi_{x,\delta}$, of $\Wb(\delta \tb)$ and $\Wc(\delta^{\lambda} \tc)$, respectively;
here, $\delta \tb$ and $\delta^{\lambda} \tc$ are defined using the single-parameter dilations $\eb$ and $\ec$, respectively.
Expand $\Vb_{x,\delta}(\tb)$ and $\Vc_{x,\delta}(\tc)$ as Taylor series in the $\tb$ and $\tc$ variables:
\begin{equation*}
\Vb_{x,\delta}(\tb)\sim \sum_{|\alphab|>0} \tb^{\alphab} \Yb_{\alphab}^{x,\delta}, \quad 
\Vc_{x,\delta}(\tc)\sim \sum_{|\alphac|>0} \tb^{\alphac} \Yc_{\alphac}^{x,\delta}.
\end{equation*}
Note that $\Yb_{\alphab}^{x,\delta}$ and $\Yc_{\alphac}^{x,\delta}$ are the pullbacks, via $\Phi_{x,\delta}$,
of $\delta^{\deg(\alphab)} \Xb_{\alphab}$ and $\delta^{\lambda\deg(\alphac)} \Xc_{\alphac}$, respectively.

Now let $\Omega_0\Subset\Omega'$, $x_k\in \Omega_0$, and $\delta_k\rightarrow 0$ be as in Lemma \ref{LemmaOptimalSubseqVect}.
Because $\{\thetab^{x_k,\delta_k}\}$ and $\{\thetac^{x_k,\delta_k}\}$ are bounded subsets of $C^\infty$, as discussed above,
by moving to a subsequence, we have that $\thetab^{x_k,\delta_k}$ and $\thetac^{x_k,\delta_k}$ converge in $C^\infty$.
Say, $\thetab^{x_k,\delta_k}\rightarrow \thetab^{\infty}$ and $\thetac^{x_k,\delta_k}\rightarrow \thetac^{\infty}$. 
Note that $\Vb_{x_k,\delta_k}\rightarrow \Vb_\infty$ and $\Vc_{x_k,\delta_k}\rightarrow \Vc_{\infty}$, where
\begin{equation*}
\Vb_{\infty}(\tb,u):= \frac{\partial}{\partial \epsilon}\bigg|_{\epsilon=1} \thetab_{\epsilon\tb}^{\infty}\circ\q( \thetab_{\tb}^{\infty} \w)^{-1}(u), \quad 
\Vc_{\infty}(\tc,u):= \frac{\partial}{\partial \epsilon}\bigg|_{\epsilon=1} \thetac_{\epsilon\tb}^{\infty}\circ\q( \thetac_{\tc}^{\infty} \w)^{-1}(u).
\end{equation*}
By Lemma \ref{LemmaOptimalSubseqVect}, there are multi-indices $\alphab$ and $\alphac$ such that
$\Yb_{\alphab}^{x_k,\delta_k}\rightarrow \Yb_{\alphab}^{\infty}$ and $\Yc_{\alphac}^{x_k,\delta_k}\rightarrow \Yc_{\alphac}^{\infty}$,
where there is an open set $U\subset B^{n_0}(\eta_0)$ such that $\Yb_{\alphab}^{\infty}$ and $\Yc_{\alphac}^{\infty}$ are never $0$ in $U$.
But $\Yb_{\alphab}^{\infty}$ and $\Yc_{\alphac}^{\infty}$ appear as Taylor coefficients of $\Vb_{\infty}(\tb)$ and $\Vc_\infty(\tc)$, respectively.
It follows that, for any $u\in U$, $\thetab_{\tb}(u)$ and $\thetac_{\tc}(u)$ are not constant in $\tb$ or $\tc$, respectively, on any neighborhood of $0$.
This completes the proof.
\end{proof}

\begin{lemma}\label{LemmaOptimalNonzeroInts}
Let $\thetab^{\infty}$ and $\thetac^{\infty}$ be as in Lemma \ref{LemmaOptimalNonzeroParams}.
Then, for every $M\in \N$, $a>0$, there exist $\vsigb_0\in C_0^{\infty}\q(B^{\Nb}(a)\w), \vsigc_0\in C_0^\infty\q(B^{\Nc}(a)\w)$, multi-indices
$\alphab\in \N^{\Nb}$, $\alphac\in \N^{\Nc}$ with $|\alphab|=|\alphac|=M$, and functions $f_1,f_2\in C_0^{\infty}(B^{n}(\eta_0))$
such that the following holds.  Let $\vsigb:=\partial_{\tb}^{\alphab} \vsigb_0$ and $\vsigc:=\partial_{\tc}^{\alphac} \vsigc_0$ and define
\begin{equation*}
g_1(u):=\int f_1\q( \thetac_{\tc_1}^{\infty}\circ \thetab_{\tb}^{\infty}\circ \thetac_{\tc_2}^{\infty}(u) \w) \vsigc(\tc_1) \vsigb(\tb) \vsigc(\tc_2)\: d\tc_1 \:d\tb\: d\tc_2,
\end{equation*}
\begin{equation*}
g_2(u):=\int f_2\q( \thetab_{\tb_1}^{\infty}\circ \thetac_{\tc}^{\infty}\circ \thetab_{\tb_2}^{\infty}(u) \w)\vsigb(\tb_1) \vsigc(\tc) \vsigb(\tb_2)\: d\tb_1\: d\tc\: d\tb_2.
\end{equation*}
Then, there is an open set $U\subseteq B^{n}(\eta_0)$ such that $g_1$ and $g_2$ are never zero on $U$.
\end{lemma}
\begin{proof}
This follows immediately from the conclusion of Lemma \ref{LemmaOptimalNonzeroParams}.
\end{proof}

\begin{lemma}\label{LemmaOptimalDefineTkandSk}
Let $M\in \N$.  Take $\Omega_0\Subset\Omega'$, $\eta_0>0$, $x_k\in \Omega_0$, and $\delta_k\rightarrow 0$ as in Lemma \ref{LemmaOptimalNonzeroParams}.  
Take $a>0$ less than or equal to the choice of $a$ in Lemma \ref{LemmaOptimalNonzeroParams}.
Take $\vsigb$, $\vsigc$ as in Lemma \ref{LemmaOptimalNonzeroInts} (with these choices of $M$ and $a$).
Let $\delta_k=2^{-j_k}$ where $j_k\in [0,\infty)$ so that $j_k\rightarrow \infty$, and let $\psi\in C_0^{\infty}(\Omega')$ equal $1$ on a neighborhood
of the closure of $\Omega_0$.  Define
\begin{equation*}
T_{j_k} f(x) := \psi(x) \int f(\gammab_{\tb}(x)) \psi(\gammab_{\tb}(x)) \dil{\vsigb}{2^{j_k}}(\tb)\: d\tb,
\end{equation*}
\begin{equation*}
S_{\lambda j_k} f(x) := \psi(x) \int f(\gammac_{\tc}(x)) \psi(\gammac_{\tc}(x)) \dil{\vsigc}{2^{\lambda j_k}}(\tb)\: d\tb,
\end{equation*}
\begin{equation*}
R_k^1 := S_{\lambda j_k} T_{j_k} S_{\lambda j_k},\quad R_k^2:= T_{j_k} S_{\lambda j_k} T_{j_k}.
\end{equation*}
Then, for $1\leq p\leq \infty$,
\begin{equation*}
\liminf_{k\rightarrow \infty} \LpOpN{p}{R_k^1}>0, \quad \liminf_{k\rightarrow \infty} \LpOpN{p}{R_k^2}>0.
\end{equation*}
\end{lemma}
\begin{proof}
We begin with the result for $R_k^1$.  
Suppose $\liminf_{k\rightarrow \infty} \LpOpN{p}{R_k^1}=0$.  By moving to a subsequence, we may assume $\lim_{k\rightarrow \infty} \LpOpN{p}{R_k^1}=0$.
Define $\Phi_{x,\delta}^{\#} f (u):= f\circ \Phi_{x,\delta}(u)$, and let
\begin{equation*}
U_k:= \Phi_{x_k,\delta_k}^{\#} R_k^1 \q(\Phi_{x_k,\delta_k}^{\#}\w)^{-1},
\end{equation*}
where we think of $U_k$ as an operator acting on functions on $B^n(\eta_0)$.  By \eqref{EqnQuantFrobJacVol}, we have
\begin{equation*}
\LpOpN{p}{U_k}\lesssim \LpOpN{p}{R_k^1},
\end{equation*}
and therefore $\lim_{k\rightarrow \infty} \LpOpN{p}{U_k}=0$.
But, we have
\begin{equation*}
\lim_{k\rightarrow \infty} U_k f(u) = \int f\q( \thetac_{\tc_1}^{\infty}\circ \thetab_{\tb}^{\infty}\circ \thetac_{\tc_2}^{\infty}(u) \w) \vsigc(\tc_1) \vsigb(\tb) \vsigc(\tc_2)\: d\tc_1 \:d\tb\: d\tc_2.
\end{equation*}
By taking $f=f_1$, where $f_1$ is as in Lemma \ref{LemmaOptimalNonzeroInts}, we see that $\lim_{k\rightarrow \infty} U_k f_1(u)$ is nonzero on a set of
positive measure, which contradicts  the fact that $$\lim_{k\rightarrow \infty} \LpOpN{p}{U_k}=0$$ and completes the proof for $R_k^1$.
The same proof works for $R_k^2$, where we use $f_2$ from Lemma \ref{LemmaOptimalNonzeroInts} in place of $f_1$.
\end{proof}

\begin{rmk}\label{RmkOptimalUniform}
Fix $\deltab, \deltac\in \R$ and $a>0$.  Take $M=M(\deltab,\deltac)\geq 1$ large.  Let $T_{j_k}$, $S_{\lambda j_k}$ be as in Lemma \ref{LemmaOptimalDefineTkandSk},
with this choice of $M$.
Write,
\begin{equation*}
T_{j_k} f(x) := \psi(x) \int f(\gammab_{\tb}(x)) \psi(\gammab_{\tb}(x)) \dil{\vsigb}{2^{j_k}}(\tb)\: d\tb,
\end{equation*}
\begin{equation*}
S_{\lambda j_k} f(x) := \psi(x) \int f(\gammac_{\tc}(x)) \psi(\gammac_{\tc}(x)) \dil{\vsigc}{2^{\lambda j_k}}(\tb)\: d\tb,
\end{equation*}
as in Lemma \ref{LemmaOptimalDefineTkandSk}.  Lemma \ref{LemmaPfKerDecompSingleParam} shows
that $2^{\deltab j_k} T_{j_k}$ is a fractional Radon transform corresponding to $(\gammab,\eb,\Nb)$ on $B^{\Nb}(a)$,
and $2^{\deltac \lambda j_k} S_{\lambda j_k}$ is a fractional Radon transform corresponding to $(\gammac,\ec,\Nc)$ on $B^{\Nc}(a)$.
Furthermore, this is true \textit{uniformly in} $k$.  Indeed, fix $\etab\in C_0^{\infty}\q(B^{\Nb}(a)\w)$ and $\etac\in C_0^\infty\q(B^{\Nc}(a)\w)$,
with $\etab\equiv 1$ on a neighborhood of the support of $\vsigb$ and $\etac\equiv 1$ on a neighborhood of the support of $\vsigc$.
Lemma \ref{LemmaPfKerDecompSingleParam} shows that we may write
\begin{equation*}
\begin{split}
&2^{\deltab j_k} \dil{\vsigb}{2^j}=  2^{\deltab j_k} \etab \dil{\vsigb}{2^j}=\etab \sum_{\substack{l\leq j_k\\l\in \N}} 2^{k\deltab} \dil{\vsigb_{l,k}}{2^l}, 
\\&2^{\deltac \lambda j_k} \dil{\vsigc}{2^{\lambda j_k}}=2^{\deltac \lambda j_k}\etac \dil{\vsigc}{2^{\lambda j_k}} = \etac\sum_{\substack{l\leq \lambda j_k\\l\in \N}}2^{k\deltab} \dil{\vsigc_{l,k}}{2^l},
\end{split}
\end{equation*}
where $\vsigb_{l,k}\in \schS_0(\R^{\Nb})$ and $\vsigc_{l,k}\in \schS_0(\R^{\Nc})$ for $l>0$ and
\begin{equation*}
\{ \vsigb_{l,k} : k\in \N, l\leq j_k, l\in \N \} \subset \schS(\R^{\Nb}), \quad \{ \vsigc_{l,k} : k\in \N, l\leq \lambda j_k, l\in \N \} \subset \schS(\R^{\Nc})
\end{equation*}
are bounded sets.  Because of this, the Baire category theorem implies the following.  If $B_1$ and $B_2$ are function spaces
with norms $\|\cdot\|_{B_1}$ and $\|\cdot\|_{B_2}$, respectively, then we have:
\begin{itemize}
\item If for every fractional Radon transform, $T$, of order $\deltab$ corresponding to $(\gammab,\eb,\Nb)$ on $B^{\Nb}(a)$ we have $T$ extends
to a bounded operator $T:B_1\rightarrow B_2$, then there is a constant $C$, independent of $k$, such that
$\| 2^{j_k\deltab} T_{j_k} \|_{B_1\rightarrow B_2} \leq C$.

\item If for every fractional Radon transform, $S$, of order $\deltac$ corresponding to $(\gammac,\ec,\Nc)$ on $B^{\Nc}(a)$ we have $S$ extends
to a bounded operator $S:B_1\rightarrow B_2$, then there is a constant $C$, independent of $k$, such that
$\| 2^{\lambda j_k \deltac}S_{\lambda j_k} \|_{B_1\rightarrow B_2} \leq C$.
\end{itemize}
\end{rmk}

\begin{proof}[Proof of Theorem \ref{ThmOptimalMain}]
Because sharp $\lambda$-control implies $\lambda$-control, \eqref{EqnOptimalKnownBound1} and \eqref{EqnOptimalKnownBound2}
follow from Corollary \ref{CorResRadonOtherDownBelowCor} (that the conditions of Corollary \ref{CorResRadonOtherDownBelowCor} hold in this case follows from 
Remark \ref{RmkResRadonOtherSmoothWorks}, Proposition \ref{PropResCCHorSmoothFG}, and Corollary \ref{CorResSurfHorSmoothlyFG}).

We now turn to showing that \eqref{EqnOptimalFalseBound1} cannot hold if $a>0$ is chosen sufficiently small.  
Suppose \eqref{EqnOptimalFalseBound1} holds for some choice of
$p\in (1,\infty)$, $r>0$, $\deltah\in [0,\epsilon)$, and $\delta\in (-\epsilon,\epsilon)$.  Because \eqref{EqnOptimalFalseBound1} for $r>0$ implies the result for any smaller $r$ and
by possible shrinking $\epsilon$,
we may assume that $r$, $\delta$, and $\deltah$ are as small as we like in what follows.
Fix $M$ large, and apply Lemma \ref{LemmaOptimalDefineTkandSk} with $(\gammab,\eb,\Nb)$ replaced by $(\gammat,\et,\Nt)$ and
$(\gammac,\ec,\Nc)$ replaced by $(\gammah,\eh,\Nh)$, with this choice of $M$, to obtain $T_{j_k}$ and $S_{\lambda j_k}$ as in that lemma.
If $M=M(\delta,\deltah,r,\lambda)$ is chosen sufficiently large, we see by the discussion in Remark \ref{RmkOptimalUniform} that
$2^{(\delta+r)\lambda j_k} S_{\lambda j_k}$ is a fractional Radon transform of order $\delta+r$ corresponding to $(\gammah,\eh,\Nh)$ on $B^{\Nh}(a)$,
$2^{-(\delta-\deltah)\lambda j_k} S_{\lambda j_k}$  is a fractional Radon transform of order $-(\delta-\deltah)$ corresponding to $(\gammah,\eh,\Nh)$ on $B^{\Nh}(a)$,
and $2^{-\lambda \deltah j_k} T_{j_k}$ is a fractional Radon transform of order $-\lambda \deltah$ corresponding to $(\gammat,\et,\Nt)$ on $B^{\Nt}(a)$.
Furthermore, this is all true uniformly in $k$ in the sense made precise in Remark \ref{RmkOptimalUniform}.

Applying Theorem \ref{ThmResRadonMainThm} to $2^{(\delta+r)\lambda j_k} S_{\lambda j_k}$ and $2^{-(\delta-\deltah)\lambda j_k} S_{\lambda j_k}$,
\eqref{EqnOptimalFalseBound1} to  $2^{-\lambda \deltah j_k} T_{j_k}$, and using the uniformity discussed in Remark \ref{RmkOptimalUniform}, for $1<p<\infty$ we have,
for $f\in C^\infty$,
\begin{equation*}
\begin{split}
2^{r\lambda j_k} \LpN{p}{S_{\lambda j_k} T_{j_k} S_{\lambda j_k} f}
&=\LpN{p}{ \q( 2^{(\delta+r)\lambda j_k} S_{\lambda j_k}  \w) \q( 2^{-\lambda \deltah j_k} T_{j_k}  \w) \q( 2^{-(\delta-\deltah) \lambda j_k } S_{\lambda j_k} \w)  f}
\\&\lesssim \NLpN{p}{\delta+r}{(\gammah,\eh,\Nh)}{\q( 2^{-\lambda \deltah j_k} T_{j_k}  \w) \q( 2^{-(\delta-\deltah) \lambda j_k } S_{\lambda j_k} \w)  f }
\\&\lesssim \NLpN{p}{\delta-\deltah}{(\gammah,\eh,\Nh)}{\q( 2^{-(\delta-\deltah) \lambda j_k } S_{\lambda j_k} \w)  f }
\\&\lesssim \LpN{p}{f}.
\end{split}
\end{equation*}
We conclude $\LpOpN{p}{S_{\lambda j_k} T_{j_k} S_{\lambda j_k}}\lesssim 2^{-r \lambda j_k}$.  Since $r,\lambda>0$ and since $j_k\rightarrow \infty$,
we have $\lim_{k\rightarrow \infty} \LpOpN{p}{S_{\lambda j_k} T_{j_k} S_{\lambda j_k}}=0$.  This contradicts the conclusion of Lemma \ref{LemmaOptimalDefineTkandSk},
which achieves the contradiction and completes the proof that \eqref{EqnOptimalFalseBound1} cannot hold.

We finish the proof by showing that \eqref{EqnOptimalFalseBound2} cannot hold if $a>0$ is chosen sufficiently small.  
Suppose \eqref{EqnOptimalFalseBound2} holds for some $p\in (1,\infty)$, $r>0$, $\deltat\in [0,\epsilon)$, and $\delta\in (-\epsilon,\epsilon)$.
Because \eqref{EqnOptimalFalseBound2} for $r>0$ implies the result for any smaller $r$ and
by possible shrinking $\epsilon$,
we may assume that $r$, $\delta$, and $\deltat$ are as small as we like in what follows.
Fix $M$ large, and apply Lemma \ref{LemmaOptimalDefineTkandSk} with $(\gammab,\eb,\Nb)$ replaced by $(\gammah,\eh,\Nh)$ and
$(\gammac,\ec,\Nc)$ replaced by $(\gammat,\et,\Nt)$, with this choice of $M$, to obtain $T_{j_k}$ and $S_{\lambda j_k}$ as in that lemma.
If $M=M(\delta,\deltat, r,\lambda)$ is chosen sufficiently large, we see by the discussion in Remark \ref{RmkOptimalUniform} that
$2^{(\delta+r)j_k} T_{j_k}$ is a fractional Radon transform of order $\delta+r$ corresponding to $(\gammah,\eh,\Nh)$ on $B^{\Nh}(a)$,
$2^{-(\delta+\lambda \deltat)j_k} T_{j_k}$ is a fractional Radon transform of order $-\delta-\lambda \deltat$ corresponding to $(\gammah,\eh,\Nh)$
on $B^{\Nh}(a)$, and $2^{\deltat \lambda j_k} S_{\lambda j_k}$ is a fractional Radon transform of order $\deltat$ corresponding to $(\gammat,\et,\Nt)$
on $B^{\Nt}(a)$.
Furthermore, this is all true uniformly in $k$ in the sense made precise in Remark \ref{RmkOptimalUniform}.

Applying Theorem \ref{ThmResRadonMainThm} to $2^{(\delta+r)j_k} T_{j_k}$ and $2^{-(\delta+\lambda \deltat)j_k} T_{j_k}$,
\eqref{EqnOptimalFalseBound2} to $2^{\deltat \lambda j_k} S_{\lambda j_k}$, and using the uniformity discussed
in Remark \ref{RmkOptimalUniform}, for $1<p<\infty$ we have, for $f\in C^\infty$,
\begin{equation*}
\begin{split}
&2^{r j_k}\LpN{p}{ T_{j_k} S_{\lambda j_k} T_{j_k} f} =
\LpN{p}{ \q(2^{(\delta+r)j_k} T_{j_k}\w)  \q(2^{\deltat \lambda j_k} S_{\lambda j_k}  \w)\q(2^{-(\delta+\lambda \deltat) j_k } T_{j_k} \w) f  }
\\&\lesssim \NLpN{p}{\delta+r}{(\gammah,\eh,\Nh)}{    \q(2^{\deltat \lambda j_k} S_{\lambda j_k}  \w)\q(2^{-(\delta+\lambda \deltat) j_k } T_{j_k} \w) f}
\lesssim \NLpN{p}{\delta+\lambda\deltat}{p}{ \q(2^{-(\delta+\lambda \deltat) j_k } T_{j_k} \w) f}
\\&\lesssim \LpN{p}{f}.
\end{split}
\end{equation*}
We conclude $\LpOpN{p}{ T_{j_k} S_{\lambda j_k} T_{j_k}}\lesssim 2^{-rj_k}$.  Since $r>0$ and $j_k\rightarrow \infty$, we
see $\lim_{k\rightarrow \infty} \LpOpN{p}{T_{j_k} S_{\lambda j_k} T_{j_k}}=0$.  This contradicts the conclusion of Lemma \ref{LemmaOptimalDefineTkandSk},
which achieves the contradiction and completes the proof.
\end{proof}

\begin{proof}[Proof of the optimality in Corollary \ref{CorResRadonOtherHorWithEuclidNewer}]
In the proof of the bounds in Corollary \ref{CorResRadonOtherHorWithEuclidNewer} we saw
that if $\Omega'$ is chosen to be a sufficiently small neighborhood of $x_0$,
then $\sFt$ $\lambda_2'$-controls $(\partial, 1)$ on $\Omega'$ and $(\partial,1)$ $\lambda_1'$-controls $\sFt$ on $\Omega'$;
where $\sFt$, $\lambda_2'$, and $\lambda_1'$ are in the proof of the bounds in Corollary \ref{CorResRadonOtherHorWithEuclidNewer}.
It is immediate to verify that, in fact, $\sFt$ sharply $\lambda_2'$-controls $(\partial, 1)$ on $\Omega'$ and $(\partial,1)$ sharply $\lambda_1'$-controls $\sFt$ on $\Omega'$.  The optimality now follows from Theorem \ref{ThmOptimalMain}.
\end{proof}

%% file: arxiv.bbl
\providecommand{\bysame}{\leavevmode\hbox to3em{\hrulefill}\thinspace}
\providecommand{\MR}{\relax\ifhmode\unskip\space\fi MR }
\providecommand{\MRhref}[2]{%
  \href{http://www.ams.org/mathscinet-getitem?mr=#1}{#2}
}
\providecommand{\href}[2]{#2}
\begin{thebibliography}{NRW76b}

\bibitem[CF85]{ChangFeffermanSomeRecentDevelopmentsInFourierAnalysisAndHpTheory}
Sun-Yung~A. Chang and Robert Fefferman, \emph{Some recent developments in
  {F}ourier analysis and {$H^p$}-theory on product domains}, Bull. Amer. Math.
  Soc. (N.S.) \textbf{12} (1985), no.~1, 1--43. \MR{766959 (86g:42038)}

\bibitem[CGGP92]{ChirstGellerGlowackiPolinPseudodifferentialOperatorsOnGroupsWithDilations}
Michael Christ, Daryl Geller, Pawe{\l} G{\l}owacki, and Larry Polin,
  \emph{Pseudodifferential operators on groups with dilations}, Duke Math. J.
  \textbf{68} (1992), no.~1, 31--65. \MR{1185817}

\bibitem[Chr85]{ChristHilbertTransformsAlongCurvesI}
Michael Christ, \emph{Hilbert transforms along curves. {I}. {N}ilpotent
  groups}, Ann. of Math. (2) \textbf{122} (1985), no.~3, 575--596. \MR{819558}

\bibitem[Chr88]{ChristEndpointBoundsForSingularFractionalIntegralOperators}
\bysame, \emph{Endpoint bounds for singular fractional integral operators},
  unpublished, 1988.

\bibitem[CNS92]{ChangNagelSteinEstaimtesForDbarNeumannProblem}
D.-C. Chang, A.~Nagel, and E.~M. Stein, \emph{Estimates for the
  {$\overline\partial$}-{N}eumann problem in pseudoconvex domains of finite
  type in {${\bf C}^2$}}, Acta Math. \textbf{169} (1992), no.~3-4, 153--228.
  \MR{1194003 (93k:32025)}

\bibitem[CNSW99]{ChristNagelSteinWaingerSingularAndMaximalRadonTransforms}
Michael Christ, Alexander Nagel, Elias~M. Stein, and Stephen Wainger,
  \emph{Singular and maximal {R}adon transforms: analysis and geometry}, Ann.
  of Math. (2) \textbf{150} (1999), no.~2, 489--577. \MR{MR1726701
  (2000j:42023)}

\bibitem[Cuc96]{CuccagnaSobolevEstimatesForFractionalAndSingularRadonTransforms}
Scipio Cuccagna, \emph{Sobolev estimates for fractional and singular {R}adon
  transforms}, J. Funct. Anal. \textbf{139} (1996), no.~1, 94--118. \MR{1399687
  (98f:42010)}

\bibitem[Fab67]{FabesSingularIntegralsAndPartialDifferentialEquationsOfParabolicType}
Eugene~B. Fabes, \emph{Singular integrals and partial differential equations of
  parabolic type}, Studia Math. \textbf{28} (1966/1967), 81--131. \MR{0213744
  (35 \#4601)}

\bibitem[Fef81]{FeffermanSingularIntegralsOnProductDomains}
Robert Fefferman, \emph{Singular integrals on product domains}, Bull. Amer.
  Math. Soc. (N.S.) \textbf{4} (1981), no.~2, 195--201. \MR{MR598687
  (83i:42014)}

\bibitem[Fef87]{FeffermanHarmonicAnalysisOnProductSpaces}
\bysame, \emph{Harmonic analysis on product spaces}, Ann. of Math. (2)
  \textbf{126} (1987), no.~1, 109--130. \MR{898053 (90e:42030)}

\bibitem[Fef88]{FeffermanSomeRecentDevelopmentsInFourierAnalysisII}
\bysame, \emph{Some recent developments in {F}ourier analysis and {$H^p$}
  theory on product domains. {II}}, Function spaces and applications ({L}und,
  1986), Lecture Notes in Math., vol. 1302, Springer, Berlin, 1988, pp.~44--51.
  \MR{942256 (89g:42036)}

\bibitem[Fol75]{FollandSubellipticEstimatesAndFunctionSpacesOnNilpotent}
G.~B. Folland, \emph{Subelliptic estimates and function spaces on nilpotent
  {L}ie groups}, Ark. Mat. \textbf{13} (1975), no.~2, 161--207. \MR{0494315 (58
  \#13215)}

\bibitem[FP83]{FeffermanPhongSubellipticEigenValueProblems}
C.~Fefferman and D.~H. Phong, \emph{Subelliptic eigenvalue problems},
  Conference on harmonic analysis in honor of {A}ntoni {Z}ygmund, {V}ol. {I},
  {II} ({C}hicago, {I}ll., 1981), Wadsworth Math. Ser., Wadsworth, Belmont, CA,
  1983, pp.~590--606. \MR{730094 (86c:35112)}

\bibitem[FS74]{FollandSteinEstimatesForTheDbarComplex}
G.~B. Folland and E.~M. Stein, \emph{Estimates for the {$\bar \partial _{b}$}
  complex and analysis on the {H}eisenberg group}, Comm. Pure Appl. Math.
  \textbf{27} (1974), 429--522. \MR{0367477 (51 \#3719)}

\bibitem[FS82]{FeffermanSteinSingularIntegralsOnProductSpaces}
Robert Fefferman and Elias~M. Stein, \emph{Singular integrals on product
  spaces}, Adv. in Math. \textbf{45} (1982), no.~2, 117--143. \MR{664621
  (84d:42023)}

\bibitem[G{\l}o10a]{GlowackiCompositionAndLtBoundednessOfFlagKernels}
Pawe{\l} G{\l}owacki, \emph{Composition and {$L^2$}-boundedness of flag
  kernels}, Colloq. Math. \textbf{118} (2010), no.~2, 581--585. \MR{2602167
  (2011h:42015)}

\bibitem[G{\l}o10b]{GlowackiCompositionAndLtBoundednessOfFlagKernelsCorrection}
\bysame, \emph{Correction to ``{C}omposition and {$L^2$}-boundedness of flag
  kernels'' [mr2602167]}, Colloq. Math. \textbf{120} (2010), no.~2, 331.
  \MR{2679042 (2011h:42016)}

\bibitem[G{\l}o13]{GlowackiLpBoundednessOfFlagKernelsOnHomogeneousGroups}
\bysame, \emph{{$L^p$}-boundedness of flag kernels on homogeneous groups via
  symbolic calculus}, J. Lie Theory \textbf{23} (2013), no.~4, 953--977.
  \MR{3185206}

\bibitem[Goo76]{GoodmanNilpotentLieGroups}
Roe~W. Goodman, \emph{Nilpotent {L}ie groups: structure and applications to
  analysis}, Lecture Notes in Mathematics, Vol. 562, Springer-Verlag,
  Berlin-New York, 1976. \MR{0442149 (56 \#537)}

\bibitem[Gre07]{GreenblattAnAnalogueToATheoremOfFeffermanAndPhong}
Michael Greenblatt, \emph{An analogue to a theorem of {F}efferman and {P}hong
  for averaging operators along curves with singular fractional integral
  kernel}, Geom. Funct. Anal. \textbf{17} (2007), no.~4, 1106--1138.
  \MR{2373012 (2008m:42018)}

\bibitem[GS82]{GellerSteinSingularConvolutionOperatorsOnTheHeisenbergGroup}
D.~Geller and E.~M. Stein, \emph{Singular convolution operators on the
  {H}eisenberg group}, Bull. Amer. Math. Soc. (N.S.) \textbf{6} (1982), no.~1,
  99--103. \MR{634441}

\bibitem[GS84]{GellerSteinEstimatesForSingularConvolutionOperatorsOnTheHeisenbergGroup}
\bysame, \emph{Estimates for singular convolution operators on the {H}eisenberg
  group}, Math. Ann. \textbf{267} (1984), no.~1, 1--15. \MR{737332}

\bibitem[GSW99]{GreenleafSeegerWaingerOnXrayTransformsForRigidLineComplexes}
Allan Greenleaf, Andreas Seeger, and Stephen Wainger, \emph{On {X}-ray
  transforms for rigid line complexes and integrals over curves in {${\bf R}\sp
  4$}}, Proc. Amer. Math. Soc. \textbf{127} (1999), no.~12, 3533--3545.
  \MR{MR1670367 (2001a:44002)}

\bibitem[GU90]{GreenleafUhlmannEstimatesForSingularRadonTransforms}
Allan Greenleaf and Gunther Uhlmann, \emph{Estimates for singular {R}adon
  transforms and pseudodifferential operators with singular symbols}, J. Funct.
  Anal. \textbf{89} (1990), no.~1, 202--232. \MR{1040963}

\bibitem[Jou85]{JourneCalderonZygmundOperatorsOnProductSpaces}
Jean-Lin Journ{\'e}, \emph{Calder\'on-{Z}ygmund operators on product spaces},
  Rev. Mat. Iberoamericana \textbf{1} (1985), no.~3, 55--91. \MR{836284
  (88d:42028)}

\bibitem[Koe02]{KoenigOnMaximalSobolevAndHolderEstimatesForTheTangentialCR}
Kenneth~D. Koenig, \emph{On maximal {S}obolev and {H}\"older estimates for the
  tangential {C}auchy-{R}iemann operator and boundary {L}aplacian}, Amer. J.
  Math. \textbf{124} (2002), no.~1, 129--197. \MR{MR1879002 (2002m:32061)}

\bibitem[MRS95]{MullerRicciSteinMarcinkiewiczMulipliersAndMulitparameterStructure}
Detlef M{\"u}ller, Fulvio Ricci, and Elias~M. Stein, \emph{Marcinkiewicz
  multipliers and multi-parameter structure on {H}eisenberg (-type) groups.
  {I}}, Invent. Math. \textbf{119} (1995), no.~2, 199--233. \MR{1312498
  (96b:43005)}

\bibitem[MRS96]{MullerRicciSteinMarcinkiewiczMultipliersAndMultiParameterStructuresOnHeisenbergTypeGroupsII}
\bysame, \emph{Marcinkiewicz multipliers and multi-parameter structure on
  {H}eisenberg (-type) groups. {II}}, Math. Z. \textbf{221} (1996), no.~2,
  267--291. \MR{1376298 (97c:43007)}

\bibitem[M{\"u}l83]{MullerCalderonZygmundKernelsCarriedByLinearSubsubspace}
D.~M{\"u}ller, \emph{Calder\'on-{Z}ygmund kernels carried by linear subspaces
  of homogeneous nilpotent {L}ie algebras}, Invent. Math. \textbf{73} (1983),
  no.~3, 467--489. \MR{718942}

\bibitem[M{\"u}l84]{MullerTwistedConvolutionsWithCalderonZygmundKernels}
Detlef M{\"u}ller, \emph{Twisted convolutions with {C}alder\'on-{Z}ygmund
  kernels}, J. Reine Angew. Math. \textbf{352} (1984), 133--150. \MR{758698}

\bibitem[M{\"u}l85]{MullerSingularKernelsSupportedByHomogeneousSubmanifolds}
\bysame, \emph{Singular kernels supported by homogeneous submanifolds}, J.
  Reine Angew. Math. \textbf{356} (1985), 90--118. \MR{779377}

\bibitem[NRS01]{NagelRicciSteinSingularIntegralsWithFlagKernels}
Alexander Nagel, Fulvio Ricci, and Elias~M. Stein, \emph{Singular integrals
  with flag kernels and analysis on quadratic {CR} manifolds}, J. Funct. Anal.
  \textbf{181} (2001), no.~1, 29--118. \MR{MR1818111 (2001m:22018)}

\bibitem[NRSW89]{NagelRosaySteinWaingerEstimatesForTheBergmanAndSzegoKernels}
A.~Nagel, J.-P. Rosay, E.~M. Stein, and S.~Wainger, \emph{Estimates for the
  {B}ergman and {S}zeg{\H o} kernels in {${\bf C}\sp 2$}}, Ann. of Math. (2)
  \textbf{129} (1989), no.~1, 113--149. \MR{MR979602 (90g:32028)}

\bibitem[NRSW12]{NagelRicciSteinWaingerSingularIntegralWithFlagKernelsOnHomogeneousGroupsI}
Alexander Nagel, Fulvio Ricci, Elias Stein, and Stephen Wainger, \emph{Singular
  integrals with flag kernels on homogeneous groups, {I}}, Rev. Mat. Iberoam.
  \textbf{28} (2012), no.~3, 631--722. \MR{2949616}

\bibitem[NRW74]{NagelRiviereWaingerI}
Alexander Nagel, N{\'e}stor Rivi{\`e}re, and Stephen Wainger, \emph{On
  {H}ilbert transforms along curves}, Bull. Amer. Math. Soc. \textbf{80}
  (1974), 106--108. \MR{0450899}

\bibitem[NRW76a]{NagelRiviereWaingerIII}
Alexander Nagel, Nestor Riviere, and Stephen Wainger, \emph{A maximal function
  associated to the curve {$(t, t^{2})$}}, Proc. Nat. Acad. Sci. U.S.A.
  \textbf{73} (1976), no.~5, 1416--1417. \MR{0399389}

\bibitem[NRW76b]{NagelRiviereWaingerII}
Alexander Nagel, N{\'e}stor~M. Rivi{\`e}re, and Stephen Wainger, \emph{On
  {H}ilbert transforms along curves. {II}}, Amer. J. Math. \textbf{98} (1976),
  no.~2, 395--403. \MR{0450900}

\bibitem[NSW79]{NagelSteinWaingerHilbertTransformsAndMaximalFunctionsRelatedToVariableCurves}
Alexander Nagel, Elias~M. Stein, and Stephen Wainger, \emph{Hilbert transforms
  and maximal functions related to variable curves}, Harmonic analysis in
  {E}uclidean spaces ({P}roc. {S}ympos. {P}ure {M}ath., {W}illiams {C}oll.,
  {W}illiamstown, {M}ass., 1978), {P}art 1, Proc. Sympos. Pure Math., XXXV,
  Part, Amer. Math. Soc., Providence, R.I., 1979, pp.~95--98. \MR{545242}

\bibitem[NSW85]{NagelSteinWaingerBallsAndMetricsDefinedByVectorFields}
\bysame, \emph{Balls and metrics defined by vector fields. {I}. {B}asic
  properties}, Acta Math. \textbf{155} (1985), no.~1-2, 103--147. \MR{MR793239
  (86k:46049)}

\bibitem[NW77]{NagelWaingerL2BoundednessOfHilbertTransformsMultiParameterGroup}
Alexander Nagel and Stephen Wainger, \emph{{$L^{2}$} boundedness of {H}ilbert
  transforms along surfaces and convolution operators homogeneous with respect
  to a multiple parameter group}, Amer. J. Math. \textbf{99} (1977), no.~4,
  761--785. \MR{MR0450901 (56 \#9192)}

\bibitem[PS86]{PhongSteinHilbertIntegrals}
D.~H. Phong and E.~M. Stein, \emph{Hilbert integrals, singular integrals, and
  {R}adon transforms. {I}}, Acta Math. \textbf{157} (1986), no.~1-2, 99--157.
  \MR{857680 (88i:42028a)}

\bibitem[RS76]{RothschildSteinHypoellipticDifferentialOperatorsAndNilpotentGroups}
Linda~Preiss Rothschild and E.~M. Stein, \emph{Hypoelliptic differential
  operators and nilpotent groups}, Acta Math. \textbf{137} (1976), no.~3-4,
  247--320. \MR{MR0436223 (55 \#9171)}

\bibitem[RS88]{RicciSteinHarmonicAnalysisOnNilpotentGroupsAndSingularIntegralsII}
Fulvio Ricci and Elias~M. Stein, \emph{Harmonic analysis on nilpotent groups
  and singular integrals. {II}. {S}ingular kernels supported on submanifolds},
  J. Funct. Anal. \textbf{78} (1988), no.~1, 56--84. \MR{937632}

\bibitem[RS89]{RicciSteinHarmonicAnalysisOnNilpotentGroupsAndSingularIntegralsIII}
\bysame, \emph{Harmonic analysis on nilpotent groups and singular integrals.
  {III}. {F}ractional integration along manifolds}, J. Funct. Anal. \textbf{86}
  (1989), no.~2, 360--389. \MR{1021141 (90m:22027)}

\bibitem[RS92]{RicciSteinMultiparameterSingularIntegralsAndMaximalFunctions}
F.~Ricci and E.~M. Stein, \emph{Multiparameter singular integrals and maximal
  functions}, Ann. Inst. Fourier (Grenoble) \textbf{42} (1992), no.~3,
  637--670. \MR{1182643 (94d:42020)}

\bibitem[SS11]{SteinStreetA}
Elias~M. Stein and Brian Street, \emph{Multi-parameter singular {R}adon
  transforms}, Math. Res. Lett. \textbf{18} (2011), no.~2, 257--277.
  \MR{2784671 (2012b:44007)}

\bibitem[SS12]{SteinStreetIII}
\bysame, \emph{Multi-parameter singular {R}adon transforms {III}: {R}eal
  analytic surfaces}, Adv. Math. \textbf{229} (2012), no.~4, 2210--2238.
  \MR{2880220}

\bibitem[SS13]{SteinStreetII}
\bysame, \emph{Multi-parameter singular {R}adon transforms {II}: {T}he {$L^p$}
  theory}, Adv. Math. \textbf{248} (2013), 736--783. \MR{3107526}

\bibitem[Ste76a]{SteinMaximalFunctionsIIHomogeneousCurves}
Elias~M. Stein, \emph{Maximal functions. {II}. {H}omogeneous curves}, Proc.
  Nat. Acad. Sci. U.S.A. \textbf{73} (1976), no.~7, 2176--2177. \MR{0420117}

\bibitem[Ste76b]{SteinMaximalFunctionsPoissonIntegralsOnSymmetricSpaces}
\bysame, \emph{Maximal functions: {P}oisson integrals on symmetric spaces},
  Proc. Nat. Acad. Sci. U.S.A. \textbf{73} (1976), no.~8, 2547--2549.
  \MR{0420118}

\bibitem[Ste93]{SteinHarmonicAnalysis}
\bysame, \emph{Harmonic analysis: real-variable methods, orthogonality, and
  oscillatory integrals}, Princeton Mathematical Series, vol.~43, Princeton
  University Press, Princeton, NJ, 1993, With the assistance of Timothy S.
  Murphy, Monographs in Harmonic Analysis, III. \MR{MR1232192 (95c:42002)}

\bibitem[Str11]{StreetMultiParameterCCBalls}
Brian Street, \emph{Multi-parameter {C}arnot-{C}arath\'eodory balls and the
  theorem of {F}robenius}, Rev. Mat. Iberoam. \textbf{27} (2011), no.~2,
  645--732. \MR{2848534 (2012m:53064)}

\bibitem[Str12]{SteinStreetI}
\bysame, \emph{Multi-parameter singular radon transforms {I}: {T}he {$L^2$}
  theory}, J. Anal. Math. \textbf{116} (2012), 83--162. \MR{2892618}

\bibitem[Str14]{StreetMultiParamSingInt}
\bysame, \emph{{M}ulti-parameter {S}ingular {I}ntegrals}, Annals of Mathematics
  Studies, vol. 189, Princeton University Press, Princeton, NJ, 2014.

\bibitem[SW70]{SteinWaingerTheEstimationOfAnIntegralArising}
Elias~M. Stein and Stephen Wainger, \emph{The estimation of an integral arising
  in multiplier transformations.}, Studia Math. \textbf{35} (1970), 101--104.
  \MR{0265995}

\bibitem[SW78]{SteinWaingerProblemsInHarmonicAnalysisRelatedToCurvature}
\bysame, \emph{Problems in harmonic analysis related to curvature}, Bull. Amer.
  Math. Soc. \textbf{84} (1978), no.~6, 1239--1295. \MR{508453}

\bibitem[SW03]{SeegerWaingerBoundsForSingularFractionalIntegralsAndRelatedFIO}
Andreas Seeger and Stephen Wainger, \emph{Bounds for singular fractional
  integrals and related {F}ourier integral operators}, J. Funct. Anal.
  \textbf{199} (2003), no.~1, 48--91. \MR{1966823 (2005d:42017)}

\bibitem[Tr{\`e}67]{TrevesTopologicalVectorSpaces}
Fran{\c{c}}ois Tr{\`e}ves, \emph{Topological vector spaces, distributions and
  kernels}, Academic Press, New York-London, 1967. \MR{0225131 (37 \#726)}

\bibitem[TW03]{TaoWrightLpImprovingBounds}
Terence Tao and James Wright, \emph{{$L^p$} improving bounds for averages along
  curves}, J. Amer. Math. Soc. \textbf{16} (2003), no.~3, 605--638. \MR{1969206
  (2004j:42005)}

\end{thebibliography}
